\documentclass[a4paper, reqno]{amsart}
\usepackage[T1]{fontenc}
\usepackage{microtype}
\usepackage{lmodern}
\usepackage{amsmath,amsthm,amssymb,amscd}
\usepackage[shortlabels]{enumitem}
\usepackage[nobysame, alphabetic]{amsrefs}
\usepackage[margin=30mm]{geometry}
\usepackage{relsize}
\makeatletter
\g@addto@macro\bfseries{\boldmath}
\makeatother
\usepackage{caption,subcaption}
\usepackage{tikz}
\usetikzlibrary{cd, arrows}
\usepackage{tikz-cd}

\usepackage{xcolor}

\definecolor{darkblue}{rgb}{0,0,0.6}

\usepackage{xpatch}
\usepackage[breaklinks, pdftex, ocgcolorlinks,colorlinks=true, citecolor=darkblue, filecolor=darkblue, linkcolor=darkblue, urlcolor=darkblue]{hyperref}
\usepackage[capitalize,noabbrev]{cleveref}

\usepackage{footnotebackref}
\makeatletter
\xpretocmd{\@adminfootnotes}{\let\@makefntext\BHFN@OldMakefntext}{}{}
\renewcommand\@makefntext[1]{%
  \@ifundefined{@makefnmark}
    {}
    {%
     \renewcommand\@makefnmark{%
       \mbox{%
         \textsuperscript{%
           \normalfont
           \hyperref[\BackrefFootnoteTag]{\@thefnmark}%
         }%
       }\,%
     }%
     \BHFN@OldMakefntext{#1}%
  }%
}
\makeatother

\makeatletter
\newtheorem*{rep@theorem}{\rep@title}
\newcommand{\newreptheorem}[2]{%
\newenvironment{rep#1}[1]{%
\def\rep@title{#2 \ref{##1}}%
\begin{rep@theorem}}%
{\end{rep@theorem}}}
\makeatother

\numberwithin{equation}{section}

\newtheorem{proposition}[equation]{Proposition}
\newtheorem{theorem}[equation]{Theorem}
\newtheorem{corollary}[equation]{Corollary}
\newtheorem{lemma}[equation]{Lemma}

\newreptheorem{theorem}{Theorem}
\newreptheorem{lemma}{Lemma}
\newreptheorem{proposition}{Proposition}
\newreptheorem{corollary}{Corollary}

\theoremstyle{definition}
\newtheorem{definition}[equation]{Definition}
\newtheorem{question}[equation]{Question}
\newtheorem{example}[equation]{Example}

\theoremstyle{remark}
\newtheorem{remark}[equation]{Remark}
\newtheorem*{claim}{Claim}
\newtheorem*{remark*}{Remark}

\AddToHook{env/theorem/begin}{\crefalias{equation}{theorem}}
\AddToHook{env/proposition/begin}{\crefalias{equation}{proposition}}
\AddToHook{env/corollary/begin}{\crefalias{equation}{corollary}}
\AddToHook{env/definition/begin}{\crefalias{equation}{definition}}
\AddToHook{env/lemma/begin}{\crefalias{equation}{lemma}}
\AddToHook{env/question/begin}{\crefalias{equation}{question}}
\AddToHook{env/example/begin}{\crefalias{equation}{example}}
\AddToHook{env/conjecture/begin}{\crefalias{equation}{conjecture}}
\AddToHook{env/remark/begin}{\crefalias{equation}{remark}}
\AddToHook{env/const/begin}{\crefalias{equation}{const}}

\crefname{theorem}{Theorem}{Theorems}
\crefname{proposition}{Proposition}{Propositions}
\crefname{corollary}{Corollary}{Corollaries}
\crefname{definition}{Definition}{Definitions}
\crefname{lemma}{Lemma}{Lemmas}
\crefname{question}{Question}{Questions}
\crefname{example}{Example}{Examples}
\crefname{conjecture}{Conjecture}{Conjectures}
\crefname{remark}{Remark}{Remarks}
\crefname{const}{Construction}{Constructions}

\renewcommand{\epsilon}{\varepsilon}
\renewcommand{\phi}{\donotusephi}

\newcommand{\bbF}{\mathbb{F}}
\newcommand{\mB}{\mathcal{B}}

\newcommand{\N}{\mathbb{N}}

\newcommand{\Z}{\mathbb{Z}}

\newcommand{\B}{\mathcal{B}}

\newcommand{\cN}{\mathcal{N}}

\newcommand{\im}{\operatorname{Im}}
\newcommand{\Id}{\operatorname{Id}}
\newcommand{\id}{\operatorname{Id}}

\newcommand{\sS}{\mathcal{S}}
\DeclareMathOperator{\hAut}{hAut}
\DeclareMathOperator{\Tor}{Tor}
\DeclareMathOperator{\UNil}{UNil}
\DeclareMathOperator{\Wh}{Wh}
\DeclareMathOperator{\G}{G}
\DeclareMathOperator{\Top}{Top}
\DeclareMathOperator{\TopSpin}{\mathrm{TopSpin}}
\newcommand{\medoplus}{\mathop{\mathsmaller \bigoplus}}

\newcommand{\onto}{\twoheadrightarrow}

\newcommand{\ol}{\overline}
\newcommand{\wt}{\widetilde}
\newcommand{\wh}{\widehat}

\newcommand{\heq}{\simeq}
\newcommand{\ks}{\operatorname{ks}}

\newcommand{\proj}{\mathrm{proj}}
\newcommand{\CP}{\mathbb{CP}}
\newcommand{\RP}{\mathbb{RP}}
\newcommand{\st}{\mathop{}\!\star}%%prevents kerning between star and manifold

\DeclareMathOperator{\Out}{Out}
\DeclareMathOperator{\Aut}{Aut}
\DeclareMathOperator{\Hom}{Hom}
\DeclareMathOperator{\ev}{ev}
\DeclareMathOperator{\Her}{Her}
\DeclareMathOperator{\Sesq}{Sesq}

\DeclareMathOperator{\coker}{coker}

\DeclareMathOperator{\red}{red}
\DeclareMathOperator{\Ext}{Ext}

\DeclareMathOperator*{\bighash}{\scalebox{1.5}{\#}}

\newcommand{\propH}{Property~4HL}
\newcommand{\HtFA}{P2FA}
\newcommand{\HtFAs}{P2FA$^*$}

\DeclareMathOperator{\Res}{Res}
\DeclareMathOperator{\Ind}{Ind}
\DeclareMathOperator{\Coind}{Coind}

\begin{document}
\title[Homotopy classification of $4$-manifolds]{Homotopy classification of $4$-manifolds with\\$3$-manifold fundamental group}

\author{Jonathan Hillman}
\address{School of Mathematics and Statistics\\ University of Sydney\\ Australia}
\email{jonathanhillman47@gmail.com}

\author{Daniel Kasprowski}
\address{School of Mathematical Sciences\\ University of Southampton\\ United Kingdom}
\email{d.kasprowski@soton.ac.uk}

\author{Mark Powell}
\address{School of  Mathematics and Statistics\\ University of Glasgow\\ United Kingdom}
\email{mark.powell@glasgow.ac.uk}

\author{Arunima Ray}
\address{School of Mathematics and Statistics\\ The University of Melbourne\\ Australia}
\email{aru.ray@unimelb.edu.au}

\def\subjclassname{\textup{2020} Mathematics Subject Classification}
\expandafter\let\csname subjclassname@1991\endcsname=\subjclassname

\begin{abstract}
We give a criterion on a group $\pi$ and a homomorphism $w \colon \pi \to C_2$ under which closed
$4$-manifolds with fundamental group $\pi$ and orientation character $w$ are classified up to homotopy equivalence by their quadratic $2$-types.
We verify the criterion for a large class of $3$-manifold groups and orientation characters, in particular for the fundamental group $\pi$ of any closed, orientable $3$-manifold whose finite subgroups are cyclic, provided $w$ vanishes on every element of $\pi$ of finite order.
We deduce a homeomorphism classification of closed, orientable $4$-manifolds with infinite dihedral fundamental group $\Z/2 * \Z/2$.
\end{abstract}
\maketitle

\section{Introduction}

The study of $4$-manifolds up to homotopy equivalence began with the work of Whitehead and Milnor~\cites{Whitehead-4-complexes,Milnor-simply-connected-4-manifolds}, who gave a full classification in the closed, simply connected case.
An appropriate generalisation of Whitehead and Milnor's classification to aim for was formulated by Hambleton--Kreck~\cite{Hambleton-Kreck:1988-1}, as follows.

\begin{question}\label{question:driving}
For which fundamental groups can we classify closed $4$-manifolds up to homotopy equivalence in terms of the  quadratic $2$-type?
\end{question}

Smooth structures will play no role here, so we work in the generality of topological $4$-manifolds.
The \emph{quadratic $2$-type} $Q(M)$ of a closed, connected, based $4$-manifold $M$, with a local orientation at the basepoint, consists of the data
\[Q(M):=(\pi_1(M),\pi_2(M),k_M,w_1(M),\lambda_M).\]
Here $\pi_2(M)$ is considered as a $\Z[\pi_1(M)]$-module, $k_M\in H^3(\pi_1(M);\pi_2(M))$ is the $k$-invariant classifying the Postnikov $2$-type of $M$, $w_1(M)\in H^1(M;\Z/2)$ is the first Stiefel--Whitney class, and $\lambda_M\colon \pi_2(M)\times \pi_2(M)\to \Z[\pi_1(M)]$ is the equivariant intersection form.
In this article, instead of $w_1(M)$ we will consider equivalently the orientation character $w_M\colon \pi_1(M)\to \{\pm 1\}$. An \emph{isomorphism} between the quadratic $2$-types $Q(M)$ and $Q(M')$ of $4$-manifolds $M$ and $M'$ consists of a pair of isomorphisms $g_i\colon \pi_i(M)\to \pi_i(M')$ for $i=1,2$, that respect the $k$-invariant, and such that $g_1$ intertwines the orientation characters, and $g_2$ induces an isomorphism of the intersection forms.

There has been considerable progress on an affirmative answer to \cref{question:driving} in special cases, for example certain finite fundamental groups by Hambleton--Kreck~\cite{Hambleton-Kreck:1988-1}, Kasprowski--Powell--Ruppik~\cite{KPR}, and Kasprowski--Nicholson--Ruppik~\cite{KNR}, all in the oriented setting.  For geometrically 2-dimensional fundamental groups that satisfy the Farrell--Jones conjecture,  Hambleton--Kreck--Teichner~\cite{HKT09} showed that the homotopy classification is determined by the quadratic $2$-type together with the $w_2$-type.

A key tool in the articles~\cites{Hambleton-Kreck:1988-1,KPR,KNR} is a criterion due to Hambleton--Kreck~\cite{Hambleton-Kreck:1988-1}*{Theorem~1.1~(i)} for finite fundamental groups, which when satisfied implies that an isomorphism of quadratic $2$-types $\smash{Q(M) \xrightarrow{\cong} Q(M')}$ is induced by a homotopy equivalence $M \simeq M'$. We generalise the Hambleton--Kreck criterion in \cref{thm:homotopyclass}. Our new criteria can be applied to all fundamental groups, not just finite groups. In other words, we show that for $4$-manifolds satisfying the conditions in \cref{thm:homotopyclass}, the homotopy type is determined by the quadratic $2$-type.

Then we investigate whether the conditions of \cref{thm:homotopyclass} are satisfied by \emph{$3$-manifold groups}, namely those groups that arise as the fundamental group of some (not necessarily orientable) closed $3$-manifold.
To state our main theorem we introduce the following terminology.

\begin{definition}
   Let $\pi$ be a $3$-manifold group, i.e.~$\pi=\pi_1(Y)$ for some closed $3$-manifold $Y$, and let $w \colon \pi \to C_2$ be a homomorphism (which need not be the orientation character of $Y$).
   Suppose that $Y$ has prime decomposition $Y_1 \# \cdots \# Y_m$.
   We say the pair $(\pi,w)$ is \emph{admissible} if $w$ vanishes on every element of $\pi$ of finite order, and if whenever $\pi_1(Y_i)$ contains finite order elements, $Y_i$ is either a lens space or $S^1 \times \RP^2$.
\end{definition}

The second condition in the definition means that summands $Y_i$ with fundamental group a non-cyclic finite $3$-manifold group are inadmissible, as are summands~$Y_i$, other than $S^1 \times \RP^2$, that contain a two-sided~$\RP^2$.

\begin{remark}
  In particular, if $\pi$ is the fundamental group of a closed, orientable $3$-manifold whose finite subgroups are cyclic, and $w$ vanishes on every element of finite order, then $(\pi,w)$ is admissible.
\end{remark}

\begin{theorem}\label{thm:main-homotopy}
Let $\pi$ be a $3$-manifold group and let $w \colon \pi \to C_2$ be a homomorphism, such that $(\pi,w)$ is admissible.
Let $M$ and $M'$ be closed $4$-manifolds, locally oriented at the basepoints. Suppose that $\pi_1(M)$ and $\pi_1(M')$ are both isomorphic to $\pi$, via isomorphisms that pull back $w$ to $w_M$ and $w_{M'}$ respectively.

Then every isomorphism $\smash{Q(M) \xrightarrow{\cong} Q(M')}$ between the quadratic $2$-types of $M$ and $M'$ is realised by a homotopy equivalence.
In particular, $M$ and $M'$ are homotopy equivalent if and only if they have isomorphic quadratic $2$-types.
Homotopy equivalences are assumed to be basepoint and local orientation preserving.
\end{theorem}

\cref{thm:main-homotopy} includes the homotopy classification for oriented $4$-manifolds with the following fundamental groups.
\begin{enumerate}
    \item Fundamental groups of closed, oriented, aspherical $3$-manifolds, i.e.~\emph{COAT groups}.
    \item Free products of finitely many cyclic groups. In particular, the infinite dihedral group $D_\infty\cong \Z/2*\Z/2$.
    \item The group $\Z \times \Z/2$.
\end{enumerate}
Note that Hambleton--Kreck's result~\cite{Hambleton-Kreck:1988-1}*{Theorem~A} already covered finite groups with 4-periodic cohomology, so in particular finite fundamental groups of closed $3$-manifolds.
\cref{thm:main-homotopy} also includes many nonorientable cases.

The proof of \cref{thm:main-homotopy} relies on the following foundational results of Baues--Bleile and Hambleton--Kreck.
Let $f \colon M \to P_2(M)$ and $f' \colon M' \to P_2(M')$ be $3$-connected maps to the respective Postnikov $2$-types, and suppose there is a homotopy equivalence $g \colon P_2(M) \xrightarrow{\simeq} P_2(M')$. Let $\pi := \pi_1(P_2(M'))$, and identify all relevant fundamental groups with $\pi$ using $f$, $f'$, and $g$. Suppose that there is a homomorphism $w \colon \pi \to \{\pm 1\}$ that determines the orientation characters of $M$ and $M'$.  Then by \cite{Hambleton-Kreck:1988-1}*{Theorem~1.1} and
\cite{Baues-Bleile}*{Corollary~3.2} (stated below as \cref{prop:fundclass}) there is a homotopy equivalence $h \colon M \to M'$, with $h_* = (f'_*)^{-1} \circ g_* \circ f_* \colon \pi_i(M) \to \pi_i(M')$, for $i=1,2$, if and only if there are $($twisted$)$ fundamental classes $[M]$ and $[M']$ such that $g_* \circ f_*([M]) = f'_*([M']) \in H_4(P_2(M');\Z^w)$.

Our conditions from \cref{thm:homotopyclass}, if satisfied, imply that an isomorphism  $Q(M) \xrightarrow{\cong} Q(M')$ gives rise to a homotopy equivalence $g \colon P_2(M) \xrightarrow{\simeq} P_2(M')$ as above such that $g_* \circ f_*([M]) = f'_*([M'])$, and hence by \cref{prop:fundclass} to a homotopy equivalence between $M$ and $M'$.
The majority of the proof of \cref{thm:main-homotopy}, which starts in \cref{section:pi_2-(M)} and culminates in \cref{sec:mainhomotopy-proof}, consists of careful verification of the criteria from \cref{thm:homotopyclass}.

As another application of our criteria from \cref{thm:homotopyclass},
in \cref{cor:2-dim}  we give a new proof of \cite{HKT09}*{Theorem~5.13}. This was part of the proof of the classification of $4$-manifolds with geometrically 2-dimensional fundamental groups from that article.
See \cref{sec:geom2d} for details.

\begin{remark}
A homotopy classification via $Q(M)$ as in \cref{thm:main-homotopy} does not hold for all fundamental groups. In particular, while it holds for $\Z \times \Z/2$ in the orientable case, the analogous statement does not hold for fundamental group $\Z\times\Z/p$ whenever there are non-homotopy equivalent $3$-dimensional lens spaces with fundamental group $\Z/p$. To see this consider $S^1 \times L$ and $S^1 \times L'$, where $L$ and $L'$ are lens spaces with $\pi_1(L) \cong \Z/p \cong \pi_1(L')$. Then by considering $\Z$-covers we see that $S^1 \times L$ and $S^1 \times L'$ are homotopy equivalent if and only if $L$ and $L'$ are homotopy equivalent.  However $\pi_2(S^1 \times L)=0=\pi_2(S^1 \times L')$ and so $S^1 \times L$ and $S^1 \times L'$ have isomorphic quadratic $2$-types, even if $L$ and $L'$ are not homotopy equivalent.

We also check that the criteria  from \cref{thm:homotopyclass} do not hold in this case. Conditions \eqref{item:0}, \eqref{item:2}, and \eqref{item:3} hold trivially since $\pi_2(S^1 \times L)$ is trivial. More precisely, using the notation from \cref{thm:homotopyclass}, we can take $B = B(\Z \times \Z/p)$, so in particular $H_4(B;\Z[\Z \times \Z/p]) =0$, and $\im \varphi_B =\{0\}$. However the images of $[S^1 \times L]$ and $[S^1 \times L']$ in $H_4(\Z \times \Z/p;\Z)$ are distinct, so in particular their difference does not lie in $\im \varphi_B$.  Hence condition~\eqref{item:1} does not hold.
\end{remark}

\begin{remark}
   The restrictions on the orientation character in \cref{thm:main-homotopy} are necessary. Even for fundamental group $\Z/2$ and nontrivial orientation character, the homotopy type of $M$ is not determined by $Q(M)$ as can be seen from work of Kim, Kojima, and Raymond \cite{kim-kojima-raymond}.
\end{remark}

\begin{remark}
In this paper we use the term \emph{$PD_n$-complex} to refer to a finite $n$-dimensional Poincar\'{e} duality complex. A group $G$ is said to be a \emph{$PD_n$-group} if the classifying space $K(G,1)$ is a $PD_n$-complex.
The proof of \cref{thm:main-homotopy} extends to the case that $\pi$ is a free product $*_{i=1}^sG_i$, where the factors $G_i$ are either $PD_3$-groups or cyclic.
This may be a spurious generalisation, as it remains an open question whether every $PD_3$-group is the fundamental group of a closed $3$-manifold.

Note that the finiteness conditions on $PD_n$-complexes and $PD_n$-groups in the literature sometimes differ. For example the algebraic definition of a $PD_n$-group from \cite{bieri-eckmann} does not require that the group be finitely presentable, and for every $n \neq 4$ there are examples of such groups that are not finitely presentable by \cite{davis}*{Theorem~C}, and so do not have finite classifying spaces.
\end{remark}

\subsection{Stable homeomorphism for COAT and \texorpdfstring{$\Z \times \Z/2$}{ZxZ/2} fundamental groups}
Recall that $4$-manifolds $M$ and $M'$ are said to be \emph{stably homeomorphic} if there exists some $k\geq 0$ such that $M\#k(S^2\times S^2)$ is homeomorphic to $M'\#k(S^2\times S^2)$.
Since closed, oriented $4$-manifolds with COAT or $\Z \times \Z/2$ fundamental group are stably homeomorphic if they are homotopy equivalent and have equal Kirby--Siebenmann invariant~\citelist{\cite{KLPT}*{Corollary~1.6}\cite{KNV}*{Theorem~B}} we obtain the following statement.

\begin{corollary}\label{cor:stab-homeo-ZZ2}
Let $M$ and $M'$ be closed, oriented $4$-manifolds with $\pi_1(M) \cong \pi_1(M')$ either a COAT group or isomorphic to $\Z \times \Z/2$. Then $M$ and $M'$ are orientation preserving stably homeomorphic if and only if they have equal Kirby--Siebenmann invariant and their quadratic $2$-types are stably isomorphic.
\end{corollary}

Similar stable classification statements for $4$-manifolds with COAT fundamental groups were obtained in \cite{KLPT}*{Theorem~9.1} and \cite{Hambleton-Hildum}*{Theorem~B}.
The latter result of Hambleton--Hildum covers more 3-dimensional groups than COAT groups, but when restricted to COAT groups their result is covered by \cite{KLPT}*{Theorem~9.1}.  To compare \cite{KLPT}*{Theorem~9.1} with \cref{cor:stab-homeo-ZZ2} note that the former used the $w_2$-type, but did not consider the $k$-invariant, whereas in the latter we use the quadratic $2$-type, including the $k$-invariant, but the $w_2$-type does not appear.

A stable classification for manifolds with fundamental group $\Z\times \Z/2$  was obtained in \cite{KPT-long}*{Theorem~1.2} using modified surgery over the normal 1-type~\cite{kreck}, but using a thoroughly different set of invariants.

\subsection{Homeomorphism classification for \texorpdfstring{$4$-manifolds}{4-manifolds} with infinite dihedral fundamental group}

Given a homotopy classification as in  \cref{thm:main-homotopy}, a natural question is whether this can be upgraded to a homeomorphism classification using surgery theory. The first and most famous instance of this is due to Freedman~\cite{F}; once he established that surgery theory could be applied topologically in dimension four for trivial fundamental groups, he improved Whitehead and Milnor's homotopy classification~\cites{Whitehead-4-complexes,Milnor-simply-connected-4-manifolds} of closed, simply connected $4$-manifolds to a homeomorphism classification. (Initially, Freedman's classification was stated for $4$-manifolds that can be smoothed away from a point.  However shortly afterwards Quinn showed~\cite{Quinn-annulus} that this holds for all connected $4$-manifolds.)

For non-simply-connected $4$-manifolds, a prerequisite for applying surgery theoretic methods is that the fundamental group be \emph{good}, a class of groups that contains finite groups and solvable groups, and is closed under subgroups, quotients, extensions, and colimits~\cites{Freedman-Teichner:1995-1,DET-book-goodgroups}.

In particular the infinite dihedral group $D_\infty:= \Z/2 * \Z/2$ is good, since it fits into an extension $0\to \Z \to D_{\infty} \to \Z/2 \to 0$.  As $\Z/2 * \Z/2 \cong \pi_1(\RP^3 \# \RP^3)$ it is one of the groups covered by \cref{thm:main-homotopy}, with respect to the trivial orientation character.
In addition the Whitehead group $\Wh(D_\infty)=0$, since $\Wh(\Z/2)=0$ and $\Wh(\Z/2 * \Z/2) = \Wh(\Z/2) \oplus \Wh(\Z/2)$ by~\cite{Stallings-Whitehead-additive}, which simplifies the application of surgery theory.
We focus on $D_\infty$, obtaining a homeomorphism classification for closed, oriented $4$-manifolds with fundamental group $D_\infty$, in terms of the quadratic $2$-type, the Kirby--Siebenmann invariant, and an additional invariant $s(M,\alpha)$. This invariant is due to Kreck, Lück, and Teichner \cite{KLT}*{Definition~2.1}, although we give a slightly different description that applies to topological $4$-manifolds.

\begin{definition}\label{def:s-invt}
Let $M$ be a closed, oriented $4$-manifold and let $\alpha\colon \pi_1(M)\xrightarrow{\cong} D_\infty$ be an isomorphism.
If the universal cover $\wt{M}$ is not spin, then set $s(M,\alpha) =0$.
We assume for the rest of the definition that $\wt M$ is spin. The isomorphism $\alpha$ induces a map $f\colon M\to \RP^\infty\vee\RP^\infty\simeq \RP^\infty\cup[0,1]\cup\RP^\infty$, which is well-defined up to homotopy. Let $S$ be a regular preimage of $\tfrac{1}{2}\in [0,1]$ under $f$. Then the inclusion $S\subseteq M$ lifts to an inclusion $S\subseteq \wt M$ and the unique spin structure of $\wt M$ induces a spin structure on $S$. Let $N$ be a spin $4$-manifold with spin boundary $S$. The submanifold $S\subseteq M$ decomposes $M$ into two parts, i.e.\ $M=M_L\cup_SM_R$. We define
\[s(M,\alpha):=(\sigma(M_L\cup_S-N)/8+\ks(M_L\cup_S-N),\sigma(M_R\cup_SN)/8+\ks(M_R\cup_SN))\in \Z/2\times\Z/2.\]
For more details see \cite{KLT}*{Section~2}.
In particular, $s(M,\alpha)$ is a stable homeomorphism invariant of $(M,\alpha)$ by \cite{KLT}*{Lemma~2.2}.
\end{definition}

We can now state the promised homeomorphism classification.

\begin{theorem}\label{thm:homeo-class}
Let $M_1$ and $M_2$ be closed, oriented $4$-manifolds with isomorphisms $\alpha_i\colon \pi_1(M_i)\xrightarrow{\cong} D_\infty$, for $i=1,2$. Then $M_1$ and $M_2$ are orientation preserving homeomorphic over $D_\infty$ if and only if
\begin{enumerate}[leftmargin=1cm,font=\normalfont]
\item\label{item-D-infty-1} $M_1$ and $M_2$ have isomorphic quadratic $2$-types over $D_\infty$,
\item\label{item-D-infty-2} $\ks(M_1)=\ks(M_2)$, and
\item\label{item-D-infty-3}
$s(M_1,\alpha_1)=s(M_2,\alpha_2) \in \Z/2 \times \Z/2$.
\end{enumerate}
Moreover, if  conditions \eqref{item-D-infty-2} and \eqref{item-D-infty-3} hold, then every isomorphism of the quadratic $2$-types over $D_\infty$
is realised by a homeomorphism $M_1 \to M_2$.
\end{theorem}

Here,  we say that $M_1$ and $M_2$ are \emph{homeomorphic over $D_\infty$} if there exists a homeomorphism $f\colon M_1\to M_2$ such that $\alpha_1= \alpha_2\circ f$. Similarly an isomorphism of $Q(M_1)$ and $Q(M_2)$ over $D_\infty$ is an isomorphism of the quadratic $2$-types where the constituent isomorphism $g\colon \pi_1(M_1)\xrightarrow{\cong}\pi_1(M_2)$ intertwines $\alpha_1$ and $\alpha_2$, i.e.~$\alpha_1=\alpha_2\circ g$.

\begin{remark}
Let $M_1$ and $M_2$ be closed, oriented $4$-manifolds with isomorphisms $\alpha_i\colon \pi_1(M_i)\xrightarrow{\cong} D_\infty$, for $i=1,2$. Let $x,y\in H^1(D_\infty;\Z/2)$ be the elements obtained from pulling back  the generator of $H^1(\Z/2;\Z/2)$ using the map on group cohomology induced by the standard projections $D_{\infty}=\Z/2*\Z/2 \to \Z/2$, to the first or second factor respectively.
\leavevmode
\begin{enumerate}[label=(\roman*)]
    \item Note that condition~\eqref{item-D-infty-1} in \cref{thm:homeo-class} implies, by \cref{thm:main-homotopy}, that $M_1$ and $M_2$ are homotopy equivalent over $D_\infty$.  The Stiefel--Whitney classes are homotopy invariant, which follows from the Wu formulae.  Thus  $\wt{M}_1$ is spin if and only if $\wt{M}_2$ is spin, and also $w_2(M_1) =\alpha_1^*(x^2+y^2)$ if and only if $w_2(M_2) =\alpha_2^*(x^2+y^2)$.
    \item Of course if the $\wt{M}_i$ are both not spin, then \cref{thm:homeo-class}\,\eqref{item-D-infty-3} holds automatically by definition of $s(M_i,\alpha_i)$.
    \item Let $M$ be a closed, oriented $4$-manifold and let $\alpha\colon \pi_1(M)\xrightarrow{\cong} D_\infty$ be an isomorphism. We will show in \cref{khldflkhfglufdalkgualkguvf} that if $\wt{M}$ is spin and $w_2(M)\neq\alpha^*(x^2+y^2)$, then $s(M,\alpha)$ is determined by the signature $\sigma(M)$ and the Kirby--Siebenmann invariant $\ks(M)$. So again \cref{thm:homeo-class}\,\eqref{item-D-infty-3} holds automatically in such cases as a consequence of conditions \eqref{item-D-infty-1} and \eqref{item-D-infty-2}.
    \item\label{item:w2-type-x2y2} If each $\wt{M_i}$ is spin and $w_2(M_i) =\alpha_i^*(x^2+y^2)$, then \cref{thm:homeo-class}\,\eqref{item-D-infty-3} is an important extra condition. For example, let $E$ denotes the unique $S^2$-bundle over $\RP^2$ with orientable but not spin total space. There exists a unique $4$-manifold $\st E$, which is homotopy equivalent to $E$, but not (stably) homeomorphic to $E$. Then Teichner~\cite{teichner-star} showed that the closed, oriented $4$-manifolds $E\#E$ and $\st E \# \st E$ are homotopy equivalent, have vanishing Kirby--Siebenmann invariants, but are not (stably) homeomorphic. See~\cite{table}*{Example~5.11}  for additional discussion of these examples.
    \item Let $M$ be a closed, oriented $4$-manifold together with an isomorphism $\alpha\colon \pi_1(M)\xrightarrow{\cong} D_\infty$. Let $\beta \colon \Z/2 \times \Z/2 \to \Z/2$ be given by $(a,b) \mapsto a+b$.
    Then if $w_2(M) =\alpha^*(x^2+y^2)$, we have that \[\ks(M) = \beta(s(M,\alpha)) + \sigma(M)/8.\]
   This follows by Novikov additivity of the signature and additivity of the Kirby--Siebenmann invariant.  So the invariants are not independent.
    \item Item~\ref{item:w2-type-x2y2} also shows that the analogue of \cref{cor:stab-homeo-ZZ2} does not hold for fundamental group $D_{\infty}$.
    \item We would like to remove `over $D_{\infty}$' from the statement of \cref{thm:homeo-class}, but we have not been able to do so.
  This would follow if we knew that every closed $4$-manifold $M$ with fundamental group $D_\infty$ admits a homotopy self-equivalence that sends $a \mapsto b$ and $b \mapsto a$.
\end{enumerate}
\end{remark}

In some cases we can simplify the classification. In  \cref{subsection:D-infty-replacing-k-invariant}, specifically in~\cref{prop:D-infty-k-invariant-free-statement}, we show that in some cases we do not need to control the $k$-invariant, and instead it suffices to control the $w_2$-type. A special case is the following corollary, in which we try to minimise the algebraic topological computations needed to apply it.

\begin{corollary}\label{cor:intro-smooth-dihedral}
    Let $M$ and $M'$ be closed, oriented, smooth $4$-manifolds with fundamental group $\pi:=D_\infty$ and equivariant intersection forms both isomorphic to $H(I\pi)\oplus\lambda$, where $\lambda$ is a nonsingular Hermitian form on a stably free $\Z\pi$-module.
    Then $M\#\CP^2$ and $M'\#\CP^2$ are homeomorphic, as are $M\#\ol{\CP}^2$ and $M'\#\ol{\CP}^2$.
\end{corollary}

Here $H(I\pi)$ denotes the hyperbolic form on $I\pi \oplus I\pi^\dagger$, where $I\pi := \ker(\varepsilon \colon \Z\pi \to \Z)$ is the augmentation ideal of the group ring and $I\pi^\dagger \cong \Hom_{\Z\pi}(I\pi,\Z\pi)$; see the conventions below.

\begin{remark}
Other known homeomorphism classifications of closed, oriented $4$-manifolds are due to Freedman--Quinn~\cite{FQ} for fundamental group $\Z$, Hambleton--Kreck~\cites{Hambleton-Kreck:1988-1,Hambleton-Kreck-93} for finite cyclic fundamental groups, Hambleton--Kreck--Teichner \cite{HKT09} for some geometrically $2$-dimensional groups, especially solvable Baumslag--Solitar groups, and Hambleton--Hildum~\cite{Hambleton-Hildum} for some cases (spin${}^+$ and pre-stabilised) involving cohomological dimension 3 groups.

In the closed, nonorientable case we have classifications by Wang~\cite{wang95} for fundamental group $\Z$, Hambleton--Kreck--Teichner~\cite{HKT94} for fundamental group $\Z/2$,
$4$-manifolds homotopy equivalent to $\RP^4\#\RP^4$ by Brookman--Davis--Khan~\cite{BDK07} (see also Jahren--Kwasik~\cite{jahren-kwasik}), and Hambleton--Hillman~\cite{Hambleton-Hillman} for $4$-manifolds homotopy equivalent to quotients of $S^2 \times S^2$.
\end{remark}

\subsection{Homeomorphism classification for \texorpdfstring{$3$-manifold}{3-manifold} fundamental groups}
In \cref{section:realisation-pi-2,section:realisation-int-forms} we place limitations on the $\Z\pi$-modules, and the sesquilinear Hermitian forms on them, that can arise as the invariants of $4$-manifolds.
In \cref{sec:homeo-s-cob-classification} we consider what the surgery exact sequence tells us about the classification of homotopy equivalent $4$-manifolds whose fundamental group is a torsion-free $3$-manifold group.
\cref{thm:s-cob-classn-3-mfld-group} applies to all such groups, giving an upper bound on the number of $s$-cobordism classes of such $4$-manifolds, in both the smooth and topological categories.
Then specialising again to solvable groups we obtain the following result on the homeomorphism classification.

\begin{corollary}\label{cor:homeo-classn-3-mfld-groups}
Let $M$ be a closed $4$-manifold whose fundamental group $\pi$ is a torsion-free, solvable $3$-manifold group.
There are at most two homeomorphism classes $($including that of $M)$ of closed $4$-manifolds with quadratic $2$-type isomorphic to $Q(M)$
and with the same Kirby--Siebenmann invariant.
\end{corollary}

\subsection*{Conventions and notation} All manifolds are assumed to be closed, connected, and based. They are considered as topological manifolds by default. Our $n$-manifolds $X$ are also assumed to be endowed with a local orientation at the basepoint, which determines a (twisted) fundamental class $[X] \in H_n(X;\Z^w)$.
Homotopy equivalences, homeomorphisms, and stable homeomorphisms are assumed to respect the basepoint and the local orientation, and hence in the oriented case to be orientation preserving.

We use the symbol $C_2$ to denote the multiplicative group $\{\pm 1\}$. The symbol $\Z/2$ denotes the additive group $\{0,1\}$. The symbol $\Sigma_2$ denotes the symmetric group on two elements. The symbol $\bbF_2$ denotes the field of two elements. The symbol $\Z^-$ denotes the integers as a $\Z[\Z/2]$-module, where $1\in \Z/2$ acts by multiplication by $-1$.

We will always assume that $\pi$ is a finitely presented group and that $w \colon \pi \to C_2$ is a homomorphism. When we say that a $4$-manifold $M$ has fundamental group $\pi$ and orientation character $w$, we mean that $M$ comes with an identification $\pi_1(M) \xrightarrow{\cong} \pi$, such that $\pi_1(M) \to \pi \xrightarrow{w} C_2$ is the orientation character of $M$.

For a topological space $X$ homotopy equivalent to a CW complex, the Postnikov $2$-type is denoted by $P_2(X)$. We always choose a model CW complex for $P_2(X)$.

For a left $\Z\pi$-module $A$ we write $A^\dagger$ for the left $\Z\pi$-module given by $\Hom_{\Z\pi}(A,\Z\pi)$, turned into a left module using the involution of $\Z\pi$ sending $g \mapsto w(g)g^{-1}$. Similarly, for a map $f\colon A\to B$ of $\Z\pi$-modules, we have the induced left module homomorphism $f^\dagger\colon B^\dagger\to A^\dagger$.
When $G\leq \pi$ is a subgroup, for a $\Z G$-module $A$, we will denote the module $\Hom_{\Z G}(A,\Z G)$ by $A^\star$.

\subsection*{Outline}
\cref{sec:generalities} gives our promised criteria in \cref{thm:homotopyclass}. This can be combined with \cites{Hambleton-Kreck:1988-1,Baues-Bleile} to provide homotopy classifications. This combination is presented as \cref{cor:main-strategy}, which states that homotopy classifications can be proven for $4$-manifolds satisfying conditions \eqref{item:0} -- \eqref{item:3} of \cref{thm:homotopyclass}.

\cref{section:pi_2-(M),sec:H2FA,sec:inj-ev*,sec:injectiveB,section:proving-thm-22(3),section:Property-H} are concerned with showing that these conditions are satisfied by the $4$-manifolds considered in \cref{thm:main-homotopy}. More precisely, in \cref{section:pi_2-(M)} we prove general results about the second homotopy groups of $4$-manifolds and in \cref{sec:H2FA} we show that \cref{thm:homotopyclass}\,\eqref{item:0} holds in our setting. Next, in \cref{sec:inj-ev*} we consider the injectivity of the map $\ev^*$, which is condition~\eqref{item:3} of~\cref{thm:homotopyclass}.
In \cref{sec:injectiveB} we give general criteria under which the map $\B_A \colon \Z^w\otimes_{\Z\pi}\Gamma(A)\to\Her^w(A^\dagger)$ is injective, where $A$ is a $\Z\pi$-module. Condition~\eqref{item:2} of~\cref{thm:homotopyclass} calls for the kernel of $\mB_{H_2(B;\Z\pi)}$ to be contained in the kernel of the map $\varphi_B\colon \Z^w\otimes_{\Z\pi} H_4(B;\Z\pi)\to H_4(B;\Z^w)$. That this holds for the $4$-manifolds in \cref{thm:main-homotopy} is shown in \cref{section:proving-thm-22(3)}. \cref{section:Property-H} introduces a property of groups, that we call {\propH}, short for \emph{$4$th Homology Lifting}. This property is useful to establish \cref{thm:homotopyclass}\,\eqref{item:1}. We show that the groups considered in \cref{thm:main-homotopy} have {\propH}.
With these results in hand, we give the proof of \cref{thm:main-homotopy} in \cref{sec:mainhomotopy-proof}.

In~\cref{sec:oldresults} we present two more applications of \cref{cor:main-strategy}, using it  to recover \cite{HKT09}*{Theorem~5.13} and \cite{Hambleton-Kreck:1988-1}*{Theorem~1.1(i)}.
Then in \cref{sec:homeo} we prove \cref{thm:homeo-class}.

In \cref{section:additional-information} we provide additional information on various classifications of $4$-manifolds with torsion-free $3$-manifold fundamental groups.

\subsection*{Acknowledgements}
We thank the Max Planck Institute for Mathematics in Bonn, the University of Glasgow, and the University of Southampton, where some of this research was carried out. We thank Peter Teichner for helpful conversations.

Preliminary versions of some of the results of this paper appeared in \cite{hillman-COAT}, by JH.  Our collaboration was formed when we learnt that the two teams had independently obtained similar results, and we continued to develop the theory together.

MP was partially supported by EPSRC New Investigator grant EP/T028335/2 and EPSRC New Horizons grant EP/V04821X/2.

\section{General statements on homotopy classification}\label{sec:generalities}

In this section we state and prove \cref{thm:homotopyclass}, which will be our main technical tool going forward. Before doing so we recall the result of  \cites{Hambleton-Kreck:1988-1,Baues-Bleile} and the definition of Whitehead's universal quadratic functor. We end the section by giving a more detailed outline of the upcoming proof of \cref{thm:main-homotopy}.

\subsection{Classification via the fundamental triple}
In general, the homotopy classification of closed $4$-manifolds is given by the \emph{fundamental triple}:  the Postnikov $2$-type $B$, the orientation character $w$, and the image of the fundamental class in $H_4(B;\Z^w)$. This was shown by Baues and Bleile~\cite{Baues-Bleile}*{Corollary~3.2}, extending work of Hambleton and Kreck \cite{Hambleton-Kreck:1988-1}*{Theorem~1.1}, in the following result.

\begin{theorem}[\citelist{\cite{Baues-Bleile}\cite{Hambleton-Kreck:1988-1}}]\label{prop:fundclass}
Let $M$ and $M'$ be closed $4$-manifolds, both with fundamental group $\pi$. Let~$B$ be a connected, $3$-coconnected $CW$ complex, also with fundamental group identified with $\pi$. Fix a homomorphism $w \colon \pi \to C_2$. Assume there are $3$-connected maps $f\colon M\to B$ and $f'\colon M'\to B$ inducing the identity maps on fundamental groups, and such that $w \circ f_*,w \circ f'_*\colon \pi\to C_2$ give the orientation characters of $M$ and $M'$ respectively.  If $f_*([M])=f'_*([M'])\in H_4(B;\Z^w)$, then $M\heq M'$, via a $($basepoint and local orientation preserving$)$ homotopy equivalence that induces the identity on fundamental groups and induces $(f'_*)^{-1} \circ f_* \colon \pi_2(M) \xrightarrow{\cong} \pi_2(M')$.
\end{theorem}

The main difficulty in applying \cref{prop:fundclass} lies in computing the image of the fundamental class. In the case of finite fundamental groups, it is known in some cases that the image of the fundamental class in the homology of the Postnikov $2$-type is determined by the equivariant intersection form. More precisely, by work of Hambleton and Kreck \cite{Hambleton-Kreck:1988-1}*{Theorem~1.1~(i)} together with an improvement by Teichner \cite{teichner-phd} (see \cite{KT}*{Corollary~1.5} for the published version), for finite fundamental groups the quadratic $2$-type of $M$ determines its homotopy type if the abelian group $\Z^w\otimes_{\Z\pi}\Gamma(H_2(M;\Z\pi))$ is torsion-free.
Here $\Gamma$ denotes Whitehead's universal quadratic functor, which we recall presently. For $w=0$, this condition was shown to hold for finite groups that are cyclic \cite{Hambleton-Kreck:1988-1}, abelian with 2 generators \cite{KPR} and dihedral \cite{KNR}. In \cref{sub:criterion} we state and prove \cref{thm:homotopyclass}, which gives a generalisation of the Hambleton--Kreck criterion that can be applied to all fundamental groups.

\subsection{Whitehead's \texorpdfstring{$\Gamma$}{Gamma} groups}\label{sub:gamma-groups-recall}
In \cite{whitehead-certainsequence}, Whitehead defined the universal quadratic functor $\Gamma$.
Here, a function $f\colon A\to B$ between abelian groups is said to be \emph{quadratic} if $f(-a)=f(a)$ for all $a\in A$ and if the function $A\times A\to B$, given by $(a,b)\mapsto f(a+b)-f(a)-f(b)$ is bilinear. The functor $\Gamma$ is the universal quadratic functor in the sense that there exists a quadratic map $\gamma\colon A\to \Gamma(A)$ with the property that for every quadratic map $f\colon A\to B$ there exists a unique linear map $\Gamma(f)\colon \Gamma(A)\to B$ with $f=\Gamma(f)\circ \gamma$.

For a free abelian group $A$, the group $\Gamma(A)$ is isomorphic to the group of fixed points of $A\otimes_{\Z} A$ under the $\Sigma_2$-action permuting the two copies of $A$.
If $A$ is a $\Z\pi$-module then the diagonal action of $\pi$ gives $A\otimes_{\Z}A$
the structure of  a $\Z\pi$-module, and $\Gamma(A)\subseteq A\otimes_{\Z}A$
	is a $\Z\pi$-submodule.
A direct consequence of~\cite{whitehead-certainsequence}*{Sections~10 and 13} is that for a connected, $3$-coconnected CW complex $B$ with fundamental group~$\pi$, we have an isomorphism $\Gamma(H_2(B;\Z\pi)) \cong H_4(B;\Z\pi)$.

\subsection{A criterion for homotopy classification via the quadratic \texorpdfstring{$2$-type}{2-type}}\label{sub:criterion}

Let $M$ be a closed $4$-manifold with fundamental group $\pi$ and orientation character $w\colon \pi\to C_2$. We consider the equivariant intersection form $\lambda_M$ as a Hermitian form on $H^2(M;\Z\pi)$. For a left $\Z\pi$-module $A$ we denote the left $\Z\pi$-module given by $\Hom_{\Z\pi}(A,\Z\pi)$, turned into a left module using the involution of $\Z\pi$ sending $g \mapsto w(g)g^{-1}$, by $A^\dagger$. We denote the group of Hermitian forms on a $\Z\pi$-module $C$  by $\Her^w(C)$ and we denote sesquilinear forms by $\Sesq^w(C)$. When $w$ is the constant map, we may suppress the superscript. When $A$ is free as an abelian group, we have the homomorphism
\begin{align}
\label{eq:B-def}
\mB_A\colon \Z^w\otimes_{\Z\pi}\Gamma(A) &\to \Her^w(A^\dagger) \\
a\otimes b &\mapsto ((f,g)\mapsto \overline{f(a)}g(b)), \nonumber
\end{align}
as in \cite{hillman}*{Section~7}.

\begin{remark}
This is a special case of the following, which defines $\mB_A$ also when $A$ is not necessarily free as an abelian group.
Consider the quadratic function $f \colon A \to \Her^w(A^\dagger)$ given by $a \mapsto ((f,g)\mapsto \overline{f(a)}g(a))$. By the universal property of $\Gamma$, there is an induced map $\Gamma(A) \to \Her^w(A^\dagger)$, and it is straightforward to check that this factors through $\Z^w\otimes_{\Z\pi}\Gamma(A)$, yielding a homomorphism $\Z^w\otimes_{\Z\pi}\Gamma(A) \to \Her^w(A^\dagger)$. We omit the details, since henceforth we will only consider cases where $A$ is free as an abelian group.
\end{remark}

Recall from the previous section that by~\cite{whitehead-certainsequence}*{Sections~10 and 13}, for a connected, $3$-coconnected CW complex $B$ with fundamental group~$\pi$, we have an isomorphism
\[\Upsilon \colon \Z^w \otimes_{\Z\pi} H_4(B;\Z\pi)  \xrightarrow{\cong}  \Z^w \otimes_{\Z\pi} \Gamma(H_2(B;\Z\pi)) .\]
For any connected CW complex $B$ with fundamental group $\pi$, let
\[\ev^*\colon \Her^w(H_2(B;\Z\pi)^\dagger)\to \Her^w(H^2(B;\Z\pi))\]
denote the homomorphism induced by the evaluation map $\ev\colon H^2(B;\Z\pi)\to H_2(B;\Z\pi)^\dagger$ taking $\alpha \mapsto (x\mapsto \alpha \cap x)$. Let
\[\varphi_B \colon \Z^w\otimes_{\Z\pi}H_4(B;\Z\pi)\to H_4(B;\Z^w)\]
denote the homomorphism given by reduction of coefficients.
Now we state our main technical theorem.

\begin{theorem}\label{thm:homotopyclass}
Let $M$ and $M'$ be closed $4$-manifolds with fundamental group $\pi$ and orientation character $w \colon \pi \to C_2$. Let $B$ be a connected, $3$-coconnected $CW$ complex with fundamental group also identified with $\pi$, and  let $f\colon M\to B$ and $f'\colon M' \to B$ be maps that induce the identity on fundamental groups. Assume further that, for some $k\in \Z$, the following holds.
\begin{enumerate}[leftmargin=1cm,font=\normalfont]
\item\label{item:0} The module $H_2(B;\Z\pi)$ is free as an abelian group.
\item\label{item:1} The difference $k(f_*[M]-(f')_*[M'])\in H_4(B;\Z^w)$ lies in the image of the map $\varphi_B \colon \Z^w\otimes_{\Z\pi} H_4(B;\Z\pi) \to H_4(B;\Z^w)$.
\item\label{item:2} The kernel of $\mB_{H_2(B;\Z\pi)} \circ \Upsilon$ is contained in the kernel of $\varphi_B \colon \Z^w\otimes_{\Z\pi}H_4(B;\Z\pi)\to H_4(B;\Z^w)$.
\item\label{item:3} The map $\ev^*\colon\Her^w(H_2(B;\Z\pi)^\dagger)\to \Her^w(H^2(B;\Z\pi))$ is injective.
\end{enumerate}
Then $kf_*[M]=kf'_*[M']\in H_4(B;\Z^w)$ if and only if $f_*\lambda_M=f'_*\lambda_{M'}\in \Her^w(H^2(B;\Z\pi))$.
\end{theorem}

\begin{remark}
    We will only apply \cref{thm:homotopyclass} (in \cref{sec:oldresults,sec:mainhomotopy-proof}) in the case $k=1$. We prove the result for general $k$ in case this version is useful in the future.
\end{remark}

\begin{remark}
We will see in the upcoming proof of \cref{thm:homotopyclass} that rather than conditions \eqref{item:2} and \eqref{item:3} above, we only need the kernel of the composition  $\ev^* \circ \mB_{H_2(B;\Z\pi)}\circ \Upsilon$ to be contained in the kernel of $\varphi_B$. We prefer the current formulation since in our applications of this result 
we will verify the conditions individually. As we will show in \cref{sec:inj-ev*}, there are several general methods to conclude the injectivity of $\ev^*$, and in the setting of \cref{thm:main-homotopy} the map $\ev^*$ is in fact an isomorphism.
\end{remark}	

\begin{proof}[Proof of \cref{thm:homotopyclass}]		
Consider the diagram
\begin{equation}\label{eq:main-square}
\begin{tikzcd}
\Z^w\otimes_{\Z\pi}H_4(B;\Z\pi)	\arrow[r,"\varphi_B"]\arrow[d, "\mB_{H_2(B;\Z\pi)} \circ \Upsilon"]	& H_4(B;\Z^w)\arrow[d,"\Theta_B"]\\
\Her^w(H_2(B;\Z\pi)^\dagger)\arrow[r,"\ev^*"]	&\Her^w(H^2(B;\Z\pi)).
\end{tikzcd}
\end{equation}		
Here we used \eqref{item:0} to define the map $\mB_{H_2(B;\Z\pi)}$.
The map $\Theta_B$ is defined using the cap product as
\[
x \mapsto \big( (\alpha,\beta) \mapsto \langle \beta,\alpha\cap x\rangle \big).
\]
Commutativity of the diagram follows as in the proof of \cite{hillman}*{Lemma~10}. Since $\lambda_M$ and $\lambda_{M'}$ are induced by the cap product with the corresponding fundamental class, it follows from naturality of the cap product that $k(f_*[M]-f'_*[M'])\in H_4(B;\Z^w)$ maps to $k(f_*\lambda_M-f'_*\lambda_{M'})\in \Her^w(H^2(B;\Z\pi))$ under $\Theta_B$.

The only if direction can now be proven without using \eqref{item:1}, \eqref{item:2}, or \eqref{item:3}. Suppose that \[kf_*[M]=kf'_*[M']\in H_4(B;\Z^w).\] Then $k(f_*[M] -f'_*[M'])=0$ and hence $k(f_*\lambda_M-f'_*\lambda_{M'}) =0 \in \Her^w(H^2(B;\Z\pi))$.
However, $\Her^w(H^2(B;\Z\pi))$ is torsion-free since $\Z\pi$ is torsion-free, and so \[f_*\lambda_M=f'_*\lambda_{M'}\in \Her^w(H^2(B;\Z\pi)).\]

To prove the if direction, we suppose that $f_*\lambda_M=f'_*\lambda_{M'}\in \Her^w(H^2(B;\Z\pi))$.
By~\eqref{item:1}, there is some $x\in\Z^w\otimes_{\Z\pi}H_4(B;\Z\pi)$ such that \[\varphi_B(x)=k(f_*[M]-(f')_*[M'])\in H_4(B;\Z^w).\]
Then
\[\Theta_B \circ \varphi_B(x) = \Theta_B(k(f_*[M]-(f')_*[M'])) = k(f_*\lambda_M-f'_*\lambda_{M'})=0,\] and so
by commutativity \[\ev^*\circ \mB_{H_2(B;\Z\pi)} \circ \Upsilon(x)=0.\] By~\eqref{item:3} $\ev^*$ is injective, so $x \in \ker (\mB_{H_2(B;\Z\pi)} \circ \Upsilon)$. By \eqref{item:2}, $\ker (\mB_{H_2(B;\Z\pi)} \circ \Upsilon) \subseteq \ker (\varphi_B)$, and thus $k(f_*[M]-(f')_*[M']) = \varphi_B(x)=0$. It follows that $kf_*[M]=kf'_*[M']\in H_4(B;\Z^w)$ as desired.
\end{proof}

We will frequently apply the combination of \cref{prop:fundclass,thm:homotopyclass} to obtain our homotopy classifications. To make this easier, we give the statement that we will use in the following corollary.

\begin{corollary}\label{cor:main-strategy}
Let $M$ and $M'$ be closed $4$-manifolds with fundamental group $\pi$ and orientation character $w \colon \pi \to C_2$. Let $B$ be a connected, $3$-coconnected $CW$ complex with fundamental group also identified with $\pi$.  Let $f\colon M\to B$ and $f'\colon M'\to B$ be $3$-connected maps inducing the identity on fundamental groups. Assume that, for $k=1$, the conditions \eqref{item:0} -- \eqref{item:3} of \cref{thm:homotopyclass} hold. Let $g_i := (f')_*^{-1} \circ f_* \colon \pi_i(M) \to \pi_i(M')$, and suppose that $(g_1,g_2)$ induces an isomorphism between the quadratic $2$-type of $M$ and that of $M'$. Then there exists a $($basepoint and local orientation preserving$)$ homotopy equivalence $h \colon M \xrightarrow{\simeq} M'$ such that $h_*=g_i \colon \pi_i(M) \to \pi_i(M')$ for $i=1,2$.
\end{corollary}

\begin{proof}
By assumption, conditions \eqref{item:0} -- \eqref{item:3} of \cref{thm:homotopyclass} hold, so we can apply that theorem, for $k=1$. The assumption that $(g_1,g_2)$ induces an isomorphism between the quadratic $2$-type of $M$ and that of $M'$ implies that $f_*\lambda_M=f'_*\lambda_{M'}\in \Her^w(H^2(B;\Z\pi))$. Hence  \cref{thm:homotopyclass} implies that $f_*[M]=f'_*[M']\in H_4(B;\Z^w)$. Now \cref{prop:fundclass} implies that the desired homotopy equivalence $h \colon M \to M'$ exists.
\end{proof}

\subsection{Outline of the proof of \texorpdfstring{\cref{thm:main-homotopy}}{Theorem~1.1}}
We will use \cref{cor:main-strategy} to prove \cref{thm:main-homotopy}, with the culmination of the proof presented in \cref{sec:mainhomotopy-proof}. That is, we will show that conditions \eqref{item:0} -- \eqref{item:3} of \cref{thm:homotopyclass} hold for closed $4$-manifolds with $3$-manifold fundamental group $\pi$ and orientation character $w$, such that $(\pi,w)$ is admissible.
We sketch the proof for each condition individually next.

\begin{enumerate}
    \item In \cref{sec:H2FA} we introduce a property of groups which we call \emph{\HtFA} (\cref{def:H2FA}). By definition if a $4$-manifold $M$ has \HtFA\ fundamental group,  then $\pi_2(M)$ is free as an abelian group. We will show in \cref{prop:H2FA-3groups} that $3$-manifold groups are \HtFA, which shows that \cref{thm:homotopyclass}\,\eqref{item:0} holds in our setting.
    \item In \cref{section:Property-H} we introduce the \emph{$4$th homology lifting property} of a pair $(\pi,w)$, which we abbreviate to {\propH}, where $\pi$ is a group and $w\colon \pi\to C_2$ is a homomorphism. Roughly, this property gives a criterion to decide whether an element in the codomain of $\varphi_X \colon \Z^w\otimes_{\Z\pi} H_4(X;\Z^w)\to H_4(X;\Z^w)$ lies in the image of $\varphi_X$, for an arbitrary CW complex $X$ with $\pi_1(X)=\pi$ and homomorphism $w\colon \pi\to C_2$. This is helpful for proving \cref{thm:homotopyclass}\,\eqref{item:1} since we can apply the criterion to the element $f_*[M]-(f')_*[M']$, using the notation of \cref{thm:homotopyclass}. We will show in \cref{cor:propH-3mfd} that $(\pi,w)$ has {\propH} when $\pi$ is as in \cref{thm:main-homotopy}, enabling the application of the criterion in our setting. For the proof we first address the cases of finite groups, $PD_3$-groups, the infinite cyclic group, and the group $\Z\times \Z/2$ individually. In these cases the primary tool is the Leray--Serre spectral sequence for the fibration $\widetilde{B}\to B\to BH$, where $B$ is a connected, 2-coconnected CW complex and $H=\pi_1(B)$. With these individual cases established, the final result is proven by considering how {\propH} behaves under free products.
    \item In \cref{sec:injectiveB} we give general criteria under which the map $\B_A \colon \Z^w\otimes_{\Z\pi}\Gamma(A)\to\Her^w(A^\dagger)$ is injective, where $A$ is a $\Z\pi$-module.
    When $\mB_{H_2(B;\Z\pi)}$ is injective, for $B$ and $\pi$ as in \cref{thm:homotopyclass}, since $\Upsilon$ is an isomorphism it follows that condition~\eqref{item:2} is satisfied. However, in our setting, the map $\mB_{H_2(B;\Z\pi)}$ is not necessarily injective. Therefore in \cref{section:proving-thm-22(3)} we specialise to the case of admissible $3$-manifold groups and orientation characters. Here we use our previous results on the second homotopy groups from \cref{section:pi_2-(M)} and once again use the Leray--Serre spectral sequence. Special care is needed for the case of $\Z\times \Z/2$ fundamental group. Finally we show in \cref{cor:kernelB-equals-kernelphi} that \cref{thm:homotopyclass}\,\eqref{item:2} holds in our setting.
    \item We will show in \cref{prop:ev-inj} that for $4$-manifolds with $3$-manifold fundamental group the map $\ev^*$ is in fact an isomorphism, showing that \cref{thm:homotopyclass}\,\eqref{item:3} holds. The proof of \cref{prop:ev-inj} is rather general, depending on the homology and cohomology groups of~$\pi$ with $\Z\pi$ coefficients, and consisting of an analysis of certain exact sequences arising from the universal coefficient spectral sequence.
\end{enumerate}

We end this section by remarking that while \cref{sec:H2FA,section:proving-thm-22(3)} are rather specific to our setting, \cref{section:pi_2-(M),sec:inj-ev*,sec:injectiveB,section:Property-H} contain a number of general results that are likely to be useful to those interested in proving homotopy classification results for other fundamental groups.

\section{The second homotopy group
for \texorpdfstring{$4$-manifolds with $3$-manifold}{4-manifolds with 3-manifold} fundamental group}\label{section:pi_2-(M)}

In this section we study $\pi_2(M)\cong H_2(M;\Z\pi)$ for $4$-manifolds $M$ whose fundamental group $\pi$ is a $3$-manifold group.
In \cref{sub:general-result} we recall a general result on the stable isomorphism class of $\pi_2(M)$, without restriction on the fundamental group.
In \cref{subsection:irreducible-3-mfld-gps} we consider cases for which  $\pi_1(M)$ is the fundamental group of an irreducible $3$-manifold.
In \cref{subsection-pi_2-where-pi1-is-free-product}, we investigate~$\pi_2(M)$ when $\pi_1(M)$ is a $3$-manifold group that is an admissible nontrivial free product.
Finally in~\cref{sub:second-hom-gp-determines-image-fund-class} we prove that the stable isomorphism class of $\pi_2(M)$ determines the image of the fundamental class of $M$ in $H_4(\pi;\Z^w)$, where $\pi=\pi_1(M)$.

We will express the stable isomorphism classes of the second homotopy groups in terms of twisted augmentation ideals, which we define next.

\begin{definition}\label{defn:twisted-augmentation-ideals}
Let $v\colon \pi\to C_2$ be a homomorphism. Let $I\pi^v \unlhd  \Z\pi$ denote the \emph{twisted augmentation ideal}, i.e.~the kernel of the twisted augmentation map  \[\varepsilon_v\colon \Z\pi\to \Z^v\] sending $g\in \pi$ to $v(g)$. When $v$ is the constant map to $1\in C_2$, we use the symbol $I\pi$ for the (untwisted) augmentation ideal.
\end{definition}

We will use the following elementary lemma several times, so we record it here.

\begin{lemma}\label{lemma:tensoring-with-Zpi-w-gives-iso}
  Let $\pi$ be a group and let $w \colon \pi \to C_2$ be a homomorphism.
  The map
\begin{align}\label{eqn:map-for-changing-w}
 \omega \colon \Z\pi &\xrightarrow{} \Z\pi;\;\;\;
 \sum_{g\in \pi} n_g g  \mapsto  \sum_{g\in \pi}  w(g) n_g g.
\end{align}
 induces a left $\Z\pi$-module isomorphism $\omega\colon\Z\pi^w \xrightarrow{\cong} \Z\pi$.
\end{lemma}

\subsection{A general result on \texorpdfstring{$\pi_2(M)$}{pi2(M)}}\label{sub:general-result}

We start this section by recording the following fact regarding the stable isomorphism types of second homotopy groups of $4$-manifolds. Recall that two $4$-manifolds are said to be \emph{$\CP^2$-stably homeomorphic} if they become homeomorphic after connected sum with copies of $\CP^2$ and $\ol{\CP^2}$.

\begin{lemma}\label{lem:pi2-stable}
Let $M$ and $M'$ be closed $4$-manifolds with fundamental group $\pi$ and orientation character $w \colon \pi \to C_2$, along with classifying maps $c_M\colon M\to B\pi$ and $c_{M'}\colon M'\to B\pi$ $($not necessarily inducing the identity on $\pi_1)$. If $(c_M)_*([M])=(c_{M'})_*([M'])\in H_4(\pi;\Z^w)/\pm \Aut(\pi)$, then $\pi_2(M)$ and $\pi_2(M')$ are stably isomorphic as $\Z\pi$-modules.
\end{lemma}

\begin{proof}
By~\cite{kreck} (see also \cite{KPT}*{Theorem~1.2, Section~1.5}), in the case that $(c_M)_*[M]=(c_{M'})_*[M']\in H_4(\pi;\Z^w)/\pm \Aut(\pi)$, the manifolds $M$ and $M'$ are $\CP^2$-stably homeomorphic if and only if they have equal Kirby--Siebenmann invariants. In other words, possibly after connected sum with a copy of $\star\CP^2$, they become $\CP^2$-stably homeomorphic. 
Connected sum with $\CP^2$, $\ol{\CP^2}$, or $\star\CP^2$ changes the second homotopy group by direct sum with $\Z\pi$. This completes the proof.
\end{proof}

\subsection{Irreducible \texorpdfstring{$3$-manifold}{3-manifold} groups}\label{subsection:irreducible-3-mfld-gps}

In this section we study $\pi_2(M)$ for $4$-manifolds $M$ such that~$\pi_1(M)$ is
either $\Z\times \Z/2$, infinite cyclic, a finite $3$-manifold group, or a $PD_3$-group. We will consider these cases individually.

In the following special case of \citelist{\cite{Hambleton-Kreck:1988-1}*{Proposition~2.4}\cite{Ha09}*{Theorem 4.2}\cite{KPT}*{Proposition 1.10}},
we use the fact that $H_4(\pi;\mathbb{Z})=0$ for the group $\pi$ in the statement.  When $\pi$ is a torsion-free $3$-manifold group this follows from the fact that the cohomological dimension $\operatorname{cd}\pi\leq3$.

\begin{lemma}[\cites{Hambleton-Kreck:1988-1,Ha09,KPT}]\label{lem:pi-2-stably-iso-ker-coker}
Let $M$ be a closed $4$-manifold such that $\pi=\pi_1(M)$ is cyclic or a torsion-free $3$-manifold group and let $w\colon \pi\to C_2$ be the orientation character of $M$. Assume that $w$ is trivial if $\pi$ is finite cyclic. Let $F\xrightarrow{d_2} G\onto I\pi$ be a presentation for the augmentation ideal $I\pi$ and write $d_2^\dagger \colon G^\dagger \xrightarrow{} F^\dagger$ for the dual map, where $F$ and $G$ are free $\Z\pi$-modules. Recall that here we use the involution $g\mapsto w(g)g^{-1}$ to turn $G^\dagger$ and $F^\dagger$ into left modules.
Then $\pi_2(M)$ is stably isomorphic to $\ker(d_2)\oplus\coker(d_2^\dagger)$.
\end{lemma}

Recall that an $R$-module $A$ is said to be \emph{stably free} if there exist $m,n\geq 0$ such that $A\oplus R^n\cong R^m$.

\begin{lemma}
\label{lem:pi2-z}
Let $M$ be a closed $4$-manifold with infinite cyclic fundamental group. Then $\pi_2(M)$ is free.
\end{lemma}
\begin{proof}
The augmentation ideal $I\Z$ is free as a $\Z[\Z]$-module. Then $0\to I\Z\xrightarrow{\id} I\Z$ is a presentation for $I\Z$, and hence $\pi_2(M)$ is stably free by \cref{lem:pi-2-stably-iso-ker-coker}. For every $k$, any projective $\Z[\Z^k]$-module is free \cite{swan-laurent-pols}*{Theorem~1.1}, see also \cite{lam}*{Corollary~V.4.12}. Hence $\pi_2(M)$ is free as claimed.
\end{proof}

For $4$-manifolds with fundamental group a finite $3$-manifold group, we have the following special case of \cite{Hambleton-Kreck:1988-1}*{Remark~2.5}. The application requires the fact that an arbitrary finite $3$-manifold group has 4-periodic cohomology and that for finite $3$-manifold groups the projective modules $P_i$ in \cite{Hambleton-Kreck:1988-1}*{Remark~2.5} can be chosen to be free.

\begin{lemma}[\cite{Hambleton-Kreck:1988-1}]
\label{lem:pi2-fin-3-manifold-group}
Let $\pi$ be a finite $3$-manifold group. Let $M$ be a closed, orientable $4$-manifold with
fundamental group $\pi$.
Then $\pi_2(M)$ is stably isomorphic to $I\pi\oplus \Hom_{\Z}(I\pi,\Z)$.
\end{lemma}	

In the special case of an orientable $4$-manifold with finite cyclic fundamental group $\pi$, we deduce  that $\pi_2(M)$ is isomorphic to $I\pi\oplus I\pi$, as we see next.

\begin{lemma}\label{lem:pi2-fin-cyclic}
    For the group $\pi:=\Z/n$, the dual $\Hom_{\Z}(I\pi,\Z) = I\pi^*$ of the augmentation ideal $I\pi$ is isomorphic to $I\pi$. In particular, for a closed, orientable $4$-manifold $M$ with fundamental group $\pi$, the second homotopy group $\pi_2(M)$ is stably isomorphic to $I\pi\oplus I\pi$.
\end{lemma}

\begin{proof}
Let $\pi:=\Z/n=\langle T\mid T^n\rangle$. Dualise the sequence $0 \to I\pi \to \Z\pi \xrightarrow{\varepsilon} \Z \to 0$, where $\varepsilon \colon \Z\pi\to \Z$ is the augmentation map, to obtain \[(\Z \cong \Hom_{\Z}(\Z,\Z)) \xrightarrow{N} (\Hom_{\Z}(\Z\pi,\Z)\cong \Z\pi) \to \Hom_\Z(I\pi,\Z) \to \Ext^1_{\Z}(\Z,\Z) =0.\]
It follows that the cokernel of the norm map $\Z\xrightarrow{N}\Z\pi$ sending $1 \mapsto 1+T+\cdots + T^{n-1}$ is isomorphic to $\Hom_\Z(I\pi,\Z)$.
From the short exact sequence
\[0\to \Z\xrightarrow{N}\Z\pi\xrightarrow{1-T}\Z\pi\xrightarrow{\varepsilon}\Z,\]
we see that $\Hom_\Z(I\pi,\Z) \cong \Z\pi/\im N \cong  \Z\pi/\ker(1-T) \cong \im (1-T) \cong \ker \epsilon = I\pi$, as needed. The second statement now follows from \cref{lem:pi2-fin-3-manifold-group} since $\Z/n$ is a finite $3$-manifold group.
\end{proof}

The analogues of the previous lemmas do not hold for nonorientable $4$-manifolds as can be seen by considering $\RP^4$. 
This is a further reason that we restrict ourselves in \cref{thm:main-homotopy} to cases where the orientation character is trivial on the finite cyclic subgroups.

\begin{lemma}
\label{lem:pi2-ZZ2}
Let $\pi=\Z\times \Z/2=\langle t,T\mid [T,t],T^2\rangle$ and let $v'\colon \pi\to C_2$ be given by $v'(t)=1$ and $v'(T)=-1$. Let $M$ be a closed $4$-manifold with fundamental group $\pi$ and orientation character~$w$ such that $w(T)=1$. Let $v:=wv'$. Let $c\colon M\to B\pi$ induce an isomorphism on fundamental groups. Then $\pi_2(M)$ is stably free if $c_*([M]) \neq 0$ in $H_4(\pi;\Z^w)\cong \Z/2$ and otherwise $\pi_2(M)$ is stably isomorphic to $I\pi\oplus I\pi^{v}$ if $c_*([M])=0$.
\end{lemma}

\begin{proof}
By \cref{lem:pi2-stable}, for a fixed $w$, the stable isomorphism class of $\pi_2(M)$ only depends on $c_*[M]$.

By \cite{KPT-long}*{Lemma~7.5}, the module $\pi_2(M)$ is stably free if $c_*[M]\neq 0$ and $w$ is trivial. For $w(t)=-1$, the same holds, because we can construct a model with trivial $\pi_2$, as follows.
Let $\tau\colon \RP^3\to \RP^3$ be an orientation reversing self-homeomorphism, e.g.\ induced by reflection across the equator of $S^3$.
Then the mapping torus $T_\tau$ of $\tau$ has orientation character $w$ and $c_*[T_\tau] \neq 0$. Also it has trivial $\pi_2$, by the long exact sequence of the fibration $\RP^3 \to T_\tau \to S^1$. It follows that $\pi_2(M)$ is stably free for all $4$-manifolds with  orientation character $w$ and $c_*[M] \neq 0$, as needed.

In \cite{KPT}*{Section~5}, it was shown that if $c_*[M]=0$, then $\pi_2(M)$ is stably isomorphic to $\ker d_2 \oplus \coker d^2_{w}$, where $(C_*,d_*)$ is the standard 2-periodic free $\Z\pi$-module resolution of $\Z$:
	\[\begin{tikzcd}[row sep=small, column sep=large]
\cdots \ar[r] & \Z\pi \ar[r,"{1-T}"]\ar[r,bend right, "d_3", pos=0.35,phantom]  & \Z\pi \ar[r,"{1+T}"]\ar[r,bend right, "d_2", pos=0.35,phantom] & \Z \pi \ar[r,"1-T"]\ar[r,bend right, "d_1", pos=0.35,phantom] & \Z\pi \ar[r,"\varepsilon"] & \Z \\
		&\oplus & \oplus & \oplus &  & \\
		\cdots \ar[r] &  \Z\pi \ar[r,"1+T"] \ar[uur,"1-t"',sloped] & \Z\pi \ar[r,"1-T"] \ar[uur,"t-1"', sloped] & \Z\pi \ar[uur,"1-t"',sloped] & &
\end{tikzcd}
\]
and $d^2_w$ denotes the cochain map from the cochain complex $\Hom_{\Z\pi}(C_*,\Z\pi^w)$.

 The summand $\ker d_2$ is independent of $w$, and as in \cite{KPT-long}*{Lemma~7.11} we have
\[\ker d_2 \cong \im d_3 \cong C_3/\ker d_3 \cong C_3/\im d_4 \cong C_1/\im d_2 \cong C_1/\ker d_1 \cong \im d_1 \cong I\pi.\]
The map $d^2_w$ is given by
	\[\begin{tikzcd}[row sep=small, column sep=large]
	\Z\pi \ar[r,"{1-T}"] & \Z\pi \\
	\oplus & \oplus & \\
  \Z\pi \ar[r,"1+T"] \ar[uur,"1-w(t)t"', sloped, pos=0.65] & \Z\pi .
\end{tikzcd}\]
We show that the cokernel of this map is $I\pi^v$.
Recall the $\Z\pi$-module isomorphism
$\Z\pi^v \xrightarrow{\cong} \Z\pi$ from \cref{lemma:tensoring-with-Zpi-w-gives-iso}.
Tensoring the above free resolution of $\Z$ with $\Z\pi^v$ over $\Z\pi$ and applying this isomorphism, we obtain the free resolution
	\[\begin{tikzcd}[row sep=small, column sep=large]
	\cdots \ar[r] & \Z\pi \ar[r,"{1+T}"]\ar[r,bend right, "d_3^v", pos=0.35,phantom] & \Z\pi \ar[r,"{1-T}"]\ar[r,bend right,pos=0.35, "d_2^v", phantom] & \Z \pi \ar[r,"1+T"]\ar[r,bend right,pos=0.35, "d_1^v", phantom] & \Z\pi \ar[r,"\varepsilon_v"] & \Z^v. \\
	&\oplus & \oplus & \oplus &  & \\
	\cdots \ar[r] &  \Z\pi \ar[r,"1-T"] \ar[uur,"1-w(t)t"', sloped, pos=0.65] & \Z\pi \ar[r,"1+T"] \ar[uur,"w(t)t-1"', sloped, pos=0.65] & \Z\pi \ar[uur,"1-w(t)t"',sloped, pos=0.65] & &
\end{tikzcd}\]
Thus we have
\[\coker d^2_w\cong \coker d_2^v\cong C_1/\im d_2^v\cong C_1/\ker d_1^v\cong \im d_1^v\cong \ker \varepsilon_v=: I\pi^v.\]
So $\pi_2(M)$ is stably isomorphic to $I\pi \oplus I\pi^v$, as required.
\end{proof}

Finally we consider the case of $PD_3$-groups. Below note that since any two aspherical $PD_3$-complexes with the same fundamental group are homotopy equivalent, the orientation character for a $PD_3$ group is well-defined. Also note that in this case the orientation character is determined by $\pi$, via the natural right action of $\pi$ on $H^3(\pi;\Z\pi) \cong \Z$. Hence $v'$ in the next lemma is defined in terms of $\pi$.

\begin{lemma}
\label{lem:pi2-pd3}
Let $\pi$ be a $PD_3$-group. Let $M$ be a closed $4$-manifold with fundamental group $\pi$ and orientation character $w\colon \pi \to C_2$. Let $v'$ be the orientation character of the aspherical $PD_3$-complex associated to $\pi$.   
Let $v:=wv'$.
Then $\pi_2(M)$ is stably isomorphic to $I\pi^{v}$.
\end{lemma}

\begin{proof}
Let $X$ be an aspherical $PD_3$-complex with fundamental group $\pi$. We can assume that $X$ has a single 0- and $3$-cell. Consider the cellular $\Z\pi$-chain complex
\[C_3(\wt{X})\xrightarrow{d_3}C_2(\wt{X})\xrightarrow{d_2}C_1(\wt{X})\xrightarrow{d_1}C_0(\wt{X})\]
of $X$. Since $X$ is aspherical, $C_1/\im d_2 \cong C_1/\ker d_1 \cong \im d_1 = \ker \varepsilon = I\pi$, and so $d_2$ is a presentation homomorphism of $I\pi$. Moreover $\ker d_2=C_3(\wt{X})$ is free. By \cref{lem:pi-2-stably-iso-ker-coker}, $\pi_2(M)$ is stably isomorphic to $\coker d_2^\dagger$. Since $X$ is a $PD_3$-complex with orientation character $v'$, $d_2^\dagger$ is a presentation homomorphism for $I\pi^{v}$ by Poincar\'e duality. It follows that $\pi_2(M)$ is stably isomorphic to $I\pi^{v}$.
\end{proof}

\begin{corollary}
\label{cor:pd3-free-dual-pi2}
    Let $\pi$ be a $PD_3$-group. Let $M$ be a closed $4$-manifold with fundamental group $\pi$.
    Then $\pi_2(M)^\dagger$ is finitely generated and stably free.
\end{corollary}

\begin{proof} First note that $\pi_2(M)$ is finitely generated as a $\Z\pi$-module.  Let $w\colon \pi \to C_2$ be the  orientation character  of $M$ and let $v'$ be the orientation character of the aspherical $PD_3$-complex associated to $\pi$. By \cref{lem:pi2-pd3}, $\pi_2(M)$ is stably isomorphic to the twisted augmentation ideal $I\pi^{v}$, where $v:=wv'$. Hence it suffices to show that $(I\pi^v)^\dagger$ is finitely generated and free. Dualising the short exact sequence $I\pi^v\to \Z\pi\to \Z^v$, we obtain the exact sequence
    \[(\Z^v)^\dagger\to (\Z\pi)^\dagger\to (I\pi^v)^\dagger \to \Ext^1_{\Z\pi}(\Z^v,\Z\pi).\]
    Since $\pi$ is infinite, $(\Z^v)^\dagger=0$. Furthermore we have
    \[\Ext^1_{\Z\pi}(\Z^v,\Z\pi)\cong H^1(\pi;\Z\pi^v)\cong H_2(\pi;\Z\pi^{vv'})=0.\]
    Here we used that $\Z\pi^{vv'} = \Z\pi^w$ is a free $\Z\pi$-module, because it is isomorphic to $\Z\pi$ via the map $\omega$ from \eqref{eqn:map-for-changing-w}.
Hence $(I\pi^v)^\dagger\cong (\Z\pi)^\dagger$ is finitely generated and free, as required.
\end{proof}

\subsection{\texorpdfstring{$3$-manifold}{3-manifold} groups that are free products}\label{subsection-pi_2-where-pi1-is-free-product}

Now we start considering $\pi_2(M)$ in the case that $\pi_1(M)$ is a $3$-manifold group that is a nontrivial free product. The first few results hold generally, but we quickly specialise to the case that the pair $(\pi_1(M),w_M)$ is admissible, as in \cref{thm:main-homotopy}.

\begin{lemma}
\label{lem:adding-I}
Let $G$ and $H$ be finitely presented groups. Let $w\colon G\to C_2$ and $w'\colon H\to C_2$ be homomorphisms.
Then the twisted augmentation ideal $I(G*H)^{w*w'}$ of $G*H$ is isomorphic to the direct sum $\Ind_G^{G*H} IG^w \oplus \Ind_H^{G*H} IH^{w'}$.
\end{lemma}

\begin{proof}	
	Choose a model for $BG$ and tensor its $\Z G$-chain complex with $\Z G^w$ over $\Z G$ to obtain a free resolution
	\[\cdots\to C_2^{G,w}\xrightarrow{d_2^{G,w}}C_1^{G,w}\xrightarrow{d_1^{G,w}}\Z G^w \xrightarrow{\varepsilon^{G,w}}\Z^w.\]
	Similarly, a model for $BH$ gives a free resolution
	\[\cdots\to C_2^{H,w'}\xrightarrow{d_2^{H,w'}}C_1^{H,w'}\xrightarrow{d_1^{H,w'}} \Z H\xrightarrow{\varepsilon^{H,w'}}\Z^{w'}.\]
Here the module homomorphism $\varepsilon^{G,w} \colon \Z G^w \to \Z^w$ is the unique left $\Z\pi$-module homomorphism sending $1 \mapsto 1$, which sends $\sum n_g g \mapsto \sum n_g$.
It factors as
\[\varepsilon^{G,w}  \colon \Z\pi^w \xrightarrow{\omega,\cong} \Z\pi \xrightarrow{\varepsilon_w} \Z^w,\]
using the left $\Z\pi$-module isomorphism $\omega$ from \cref{lemma:tensoring-with-Zpi-w-gives-iso} and the map $\varepsilon_w$ from \cref{defn:twisted-augmentation-ideals}.
It follows that $\omega$ induces a left $\Z\pi$-module isomorphism $\ker \varepsilon^{G,w} \cong \ker \varepsilon_w = IG^w$.
So by exactness  $IG^w \cong \ker \varepsilon^{G,w} \cong \coker(d_2^{G,w})$ and $IH^{w'} \cong \ker \varepsilon^{H,w'} \cong \coker(d_2^{H,w'})$.
	
Write $\pi := G*H$.
	The space $BG\vee BH$ is a model for $B(G*H)$. So we can take its $\Z\pi$-chain complex and tensor with $\Z\pi^{w*w'}$ over $\Z\pi$ (which does not affect exactness) to obtain a free resolution of $\Z^{w*w'}$. To compare it with the previous resolutions for $G$ and $H$, we need the following.
	
	For every $\Z G$-module $A$, there is an isomorphism $\Ind_G^\pi(A)^{w*w'}\cong \Ind_G^\pi(A^w)$ given by $\gamma\otimes a\mapsto (w*w')(\gamma)\gamma\otimes a$ for all $\gamma\in\pi$ and $a\in A$. Similarly, $\Ind_H^\pi(A')^{w*w'}\cong \Ind_H^\pi((A')^{w'})$ for every $\Z H$-module $A'$.
To use this consider \[C_*(B\pi;\Z\pi) \cong C_*(BG \vee BH;\Z\pi) \cong \Ind_G^\pi C_*^{G} \oplus \Ind_H^\pi C_*^{H}.\]
Tensoring with $\Z\pi^{w*w'}$ over $\Z\pi$ yields
\[(\Ind_G^\pi C_i^G)^{w*w'} \oplus (\Ind_H^\pi C_i^{H})^{w*w'} \cong \Ind_G^\pi C_i^{G,w} \oplus \Ind_H^\pi C_i^{H,w'}.\]
	Hence the $\Z\pi$-chain complex of $BG\vee BH$ tensored with $\Z\pi^{w*w'}$ over $\Z\pi$ gives the free resolution
	\begin{align*}
		\cdots\to \Ind_G^{\pi} C_2^{G,w} \oplus\Ind_H^{\pi} C_2^{H,w'} \xrightarrow{\Ind_G^{\pi} d_2^{G,w} \oplus \Ind_H^{\pi} d_2^{H,w'} }\Ind_G^{\pi} C_1^{G,w} \oplus\Ind_H^{\pi} C_1^{H,w'}  \\ \xrightarrow{\Ind_G^{\pi} d_1^{G,w} +\Ind_H^{\pi} d_1^{H,w'} }\Z\pi^{w*w'} \xrightarrow{\varepsilon^{\pi,w*w'}} \Z^{w*w'}.
	\end{align*}
	Hence by exactness
	\begin{align*}
		I(G*H)^{w*w'} \cong I\pi^{w*w'}& \cong \coker(\Ind_G^{\pi} d_2^{G,w} \oplus \Ind_H^{\pi} d_2^{H,w'})\cong \Ind_G^{\pi}\coker (d_2^{G,w}) \oplus \Ind_H^{\pi} \coker(d_2^{H,w'}) \\
		&\cong \Ind_G^{\pi} IG^w \oplus \Ind_H^{\pi} IH^{w'} \cong \Ind_G^{G*H} IG^w \oplus \Ind_H^{G*H} IH^{w'}
	\end{align*}
as desired.
\end{proof}

We will use the following useful fact several times, which is due to the first-named author~\cite{hillman-stable-connsum}; an alternative proof was given later in~\cite{KLT-stable-connsum}. When \cref{lemma:stable-splitting} applies we say that $M$ \emph{stably splits as a connected sum}.

\begin{lemma}[\cite{hillman-stable-connsum}]\label{lemma:stable-splitting}
Let $M$ be a closed $4$-manifold and suppose that $\pi_1(M) \cong *_{i=1}^n G_i$. Then up to connected sum with copies of $S^2\times S^2$, $M$ is homeomorphic to a connected sum of $4$-manifolds $\#_{i=1}^n M_i$ with fundamental groups $\pi_1(M_i) \cong G_i$.
\end{lemma}

The previous lemma is helpful since for a connected sum $M_1\#M_2$ of manifolds $M$ and $N$, we have that
\begin{equation}\label{eq:pi2-connsum-decomp}
\pi_2(M_1\#M_2)\cong \Ind_{\pi_1(M_1)}^{\pi_1(M_1)*\pi_1(M_2)} \pi_2(M_1) \oplus \Ind_{\pi_1(M_2)}^{\pi_1(M_1)*\pi_1(M_2)} \pi_2(M_2).
\end{equation}
Therefore, for a $4$-manifold $M$ with $\pi_1(M)=\pi$, connected sum with $S^2\times S^2$ adds copies of $\Z\pi\oplus \Z\pi$ to the second homotopy group, so the stable isomorphism type does not change.

For the remainder of this section, let $(\pi,w)$ be admissible, with $\pi$ a $3$-manifold group and $w$ an orientation character. Then, by the prime decomposition theorem for $3$-manifolds, there is a decomposition of the form
\begin{equation}
\label{eq:pi-decomp}
\pi\cong F*\big(\ast_{i=1}^rZ_i\big) *\big(\ast_{j=1}^sG_j\big)*\big(\ast_{k=1}^tH_k\big),
\end{equation}
for some $r,s,t\geq 0$, with $F$ a free group, $Z_i$ a finite cyclic group for each $i$, $G_j$ a $PD_3$-group for each $j$ and $H_k\cong \Z\times \Z/2$ for each $k$.

In the following proposition, and elsewhere, we will use that for either admissible $w \colon \Z \times \Z/2 \cong H \to C_2$, we have $H_4(H;\Z^w)\cong \Z/2$. When $w$ is trivial, this is straightforward to see from the K\"{u}nneth theorem. When $w$ is nontrivial, this can be obtained by writing out the standard $\Z H$-module resolution of $\Z$, tensoring with $\Z^w$, and computing the homology.

\begin{proposition}
\label{prop:pi2-decomp}
Fix a decomposition for $\pi$ as in \eqref{eq:pi-decomp}.
Let $M$ be a closed $4$-manifold with fundamental group $\pi$ and orientation character $w\colon \pi\to C_2$ such that $(\pi,w)$ is admissible.
By reordering the factors of $\pi$ if needed, we assume that there exists $0 \leq t' \leq t$ such that the image of the fundamental class $[M]$ in
\[H_4(\pi;\Z^w)\cong \bigoplus_{k=1}^tH_4(H_k;\Z^w)\cong (\Z/2)^t\]
is trivial in the first $t'$ summands and nontrivial for $k>t'$.
Then $\pi_2(M)$ is stably isomorphic to \[\Ind_\Gamma^\pi I\Gamma^v \oplus \Ind_{\Gamma'}^\pi I\Gamma',\] where
\begin{equation}\label{eqn:Gamma-splitting-big}
\Gamma = \big(\ast_{i=1}^rZ_i\big)* \big(\ast_{j=1}^sG_j\big)* \big(\ast_{k=1}^{t'}H_k\big)
\end{equation}
and
\begin{equation}\label{eqn:Gamma-prime-splitting-big}
\Gamma' = \big(\ast_{i=1}^rZ_i\big) *  \big(\ast_{k=1}^{t'}H_k\big)
\end{equation}
 are subgroups of $\pi$ in the canonical way.  Here $v=wv'\colon \Gamma \to C_2$, where $v'$ is trivial on each $Z_i$, is the projection onto the second factor on $\Z\times \Z/2$ $($followed by the canonical isomorphism $\Z/2\to C_2)$, and on each $G_j$ factor it is the orientation character $u_j$ of the aspherical $PD_3$-complex with fundamental group $G_j$.
\end{proposition}

\begin{proof}
By \cref{lemma:stable-splitting}, $M$ is stably homeomorphic to
\[
M^F\#\Big(\bighash_{i=1}^r M^{Z_i}\Big)\#\Big(\bighash_{j=1}^s M^{G_j}\Big)\#\Big(\bighash_{k=1}^t M^{H_k}\Big),
\]
where each of $M^F$, $M^{Z_i}$, $M^{G_j}$, and $M^{H_k}$ has fundamental group $F$, $Z_i$, $G_j$, and $H_k$, respectively.
Therefore, we know that $\pi_2(M)$ is stably isomorphic to
\begin{equation}\label{eq:pi2-connsum-decomp-big}
\Ind_F^\pi \pi_2\big(M^F\big) \oplus \bigoplus_{i=1}^r \Ind_{Z_i}^\pi \pi_2\big(M^{Z_i}\big) \oplus\bigoplus_{j=1}^s \Ind_{G_j}^\pi \pi_2\big(M^{G_j}\big) \oplus\bigoplus_{k=1}^t \Ind_{H_k}^\pi \pi_2\big(M^{H_k}\big).
\end{equation}
Next we will consider each summand individually.

By \cref{lem:pi2-z,lemma:stable-splitting}, and using \eqref{eq:pi2-connsum-decomp}, we see that $\pi_2(M^F)$ is stably free. The induction is also stably free, so up to stable isomorphism, the first summand can be ignored.

By \cref{lem:pi2-fin-cyclic} we know that each $\pi_2(M^{Z_i})$ is stably isomorphic to $IZ_i\oplus IZ_i$. Similarly, by \cref{lem:pi2-pd3}, each $\pi_2(M^{G_j})$ is stably isomorphic to $IG_j^{v_j}$, where $v_j=wu_j\colon G_j\to C_2$ and $u_j$ is the orientation character of the aspherical $PD_3$-complex with fundamental group $G_j$. The $M^{H_k}$ factors are slightly more complicated. We saw in \cref{lem:pi2-ZZ2} that $\pi_2(M^{H_k})$ is isomorphic to $IH_k\oplus IH_k^u$, with $u=wu'$ and $u'$ the projection onto the second factor of $H_k= \Z\times \Z/2$, followed by the isomorphism $\Z/2\xrightarrow{\cong}C_2$, in the case of $k\leq t'$, or $\pi_2(M^{H_k})$ is stably free for $t \geq k >t'$. Applying these facts to \eqref{eq:pi2-connsum-decomp-big}, we see that $\pi_2(M)$ is stably isomorphic to
\begin{align}\label{eq:pi2-decomp-lessbig}
&\bigoplus_{i=1}^r \Ind_{Z_i}^\pi (IZ_i\oplus IZ_i)\oplus\bigoplus_{j=1}^s \Ind_{G_j}^\pi IG_j^{v_j} \oplus\bigoplus_{k=1}^{t'} \Ind_{H_k}^\pi(IH_k\oplus IH_k^u) \\
\cong  & \bigoplus_{i=1}^r \Ind_{Z_i}^\pi IZ_i   \oplus  \bigoplus_{i=1}^r \Ind_{Z_i}^\pi IZ_i \oplus \bigoplus_{j=1}^s \Ind_{G_j}^\pi IG_j^{v_j} \oplus \bigoplus_{k=1}^{t'} \Ind_{H_k}^\pi IH_k \oplus \bigoplus_{k=1}^{t'} \Ind_{H_k}^\pi IH_k^u.  \nonumber
\end{align}
Here we use that the induction of a direct sum is the direct sum of the inductions.
We claim that this decomposition equals $\Ind_\Gamma^\pi I\Gamma^v \oplus \Ind_{\Gamma'}^\pi I\Gamma'$.  To see this,
start with each of $\Ind_\Gamma^\pi I\Gamma^v$ and $\Ind_{\Gamma'}^\pi I\Gamma'$, and
apply \cref{lem:adding-I} iteratively, splitting off one factor at a time in the free product \eqref{eqn:Gamma-splitting-big} or \eqref{eqn:Gamma-prime-splitting-big} respectively. Use that $\Ind_K^\pi \Ind_H^K  = \Ind_H^\pi$ for subgroups $H \leq K \leq \pi$.
This shows that
\[
\Ind_\Gamma^\pi I\Gamma^v \cong \bigoplus_{i=1}^r
\Ind_{Z_i}^\pi IZ_i \oplus \bigoplus_{j=1}^s
\Ind_{G_j}^\pi IG_j^{v_j} \oplus \bigoplus_{k=1}^{t'}
\Ind_{H_k}^\pi IH_k^u,
\]
by definition of $v$. Similarly, we see that
\[
\Ind_{\Gamma'}^\pi I\Gamma'\cong \bigoplus_{i=1}^r
\Ind_{Z_i}^\pi IZ_i \oplus\bigoplus_{k=1}^{t'}
\Ind_{H_k}^\pi IH_k.
\]
To complete the proof note that each summand of \eqref{eq:pi2-decomp-lessbig} appears precisely once as a summand of precisely one of $\Ind_\Gamma^\pi I\Gamma^v$ or $\Ind_{\Gamma'}^\pi I\Gamma'$.
\end{proof}

\begin{corollary}
    \label{cor:pi2-decomp}
    Let $\pi$ be a torsion-free $3$-manifold group. Let $M$ be a closed $4$-manifold with fundamental group $\pi$ and orientation character $w\colon \pi\to C_2$ such that $(\pi,w)$ is admissible.
Then $\pi_2(M)$ is stably isomorphic to $I\pi^v$ where $v$ is defined as follows. Since $\pi$ is torsion-free, there is a decomposition $\pi\cong F\ast(\ast_{j=1}^sG_j)$, where $F$ is a free group and each $G_j$ is a $PD_3$-group. Then we have $v=wv'\colon \pi \to C_2$, where $v'$ is trivial on $F$, and on each $G_j$ factor it is the orientation character $u_j$ of the aspherical $PD_3$-complex with fundamental group $G_j$.
\end{corollary}

\begin{proof}
    By \cref{prop:pi2-decomp}, $\pi_2(M)$ is stably isomorphic to $\Ind_\Gamma^\pi I\Gamma^v$, where $\Gamma=\ast_{j=1}^sG_j$. Since $IF\cong \Z F$, $\Ind_\Gamma^\pi I\Gamma^v$ is stably isomorphic to $I\pi^v$ by \cref{lem:adding-I}.
\end{proof}

\subsection{The second homotopy group \texorpdfstring{$\pi_2(M)$}{} determines the image of the fundamental class in \texorpdfstring{$H_4(\pi;\Z^w)$}{group homology}}
\label{sub:second-hom-gp-determines-image-fund-class}

We conclude the section with a result showing that the image $c_*([M]) \in H_4(\pi;\Z^w)$ of the fundamental class in the group homology is determined by the stable isomorphism class of~$\pi_2(M)$, for a map $c\colon M\to B\pi$ inducing an isomorphism on fundamental groups. We obtain a corollary (\cref{cor:images-fund-classes-in-group-homology}) that we will  use  in the proof of \cref{thm:main-homotopy} in \cref{sec:mainhomotopy-proof} in order to be able to apply \propH; see \cref{section:Property-H} for details on the latter.

\begin{proposition}
\label{prop:pi2-determines-H4-image}
Fix a decomposition for $\pi$ as in \eqref{eq:pi-decomp}.
Let $M$ be a closed $4$-manifold with fundamental group $\pi$ and orientation character $w\colon \pi\to C_2$ such that $(\pi,w)$ is admissible.
Let $c\colon M\to B\pi$ be a continuous map inducing the identity on fundamental groups.
By reordering the factors of $\pi$ if needed, we assume that there exists $t' \leq t$ such that $c_*([M])$ in
$H_4(\pi;\Z^w)\cong \bigoplus_{k=1}^tH_4(H_k;\Z^w)\cong (\Z/2)^t$ is trivial in the first $t'$ summands and nontrivial for $k>t'$.

Then $\Hom_{\Z\pi}(\pi_2(M),\Z H_k)$ is stably free as a $\Z H_k$-module if and only if $k>t'$.
Here the $\pi$-action on $\Z H_k$ is given by the projection to $H_k$ and left multiplication.
\end{proposition}

\begin{proof}
By \cref{prop:pi2-decomp} we know that $\pi_2(M)$ is stably isomorphic to \[\Ind_\Gamma^\pi I\Gamma^v \oplus \Ind_{\Gamma'}^\pi I\Gamma',\]
where \[\Gamma= \big(\ast_{i=1}^rZ_i\big)* \big(\ast_{j=1}^sG_j\big)* \big(\ast_{k=1}^{t'}H_k\big)\] and \[\Gamma'= \big(\ast_{i=1}^rZ_i\big) *  \big(\ast_{k=1}^{t'}H_k\big)\]
are subgroups of $\pi$ in the canonical way.  Here $v=wv'\colon \Gamma \to C_2$, where $v'$ is trivial on each $Z_i$, is the projection onto the second factor on $\Z\times \Z/2$ followed by an isomorphism $\Z/2\xrightarrow{\cong}C_2$, and on the $G_j$ it is the orientation character of the aspherical $PD_3$-complex with fundamental group $G_j$. Then, for each $k$, the module $\Hom_{\Z\pi}(\pi_2(M),\Z H_k)$ is stably isomorphic to
\[\Hom_{\Z\pi}(\Ind_\Gamma^\pi I\Gamma^v \oplus \Ind_{\Gamma'}^\pi I\Gamma',\Z H_k)\cong \Hom_{\Z\pi}(\Ind_\Gamma^\pi I\Gamma^v, \Z H_k)\oplus \Hom_{\Z\pi}(\Ind_{\Gamma'}^\pi I\Gamma', \Z H_k).\]
Let us consider the two summands separately. We know by iteratively applying \cref{lem:adding-I} that
\[
\Ind_\Gamma^\pi I\Gamma^v \cong \bigoplus_{\substack{L \text{ a factor}\\ \text{ of } \Gamma \text{ in } \eqref{eqn:Gamma-splitting-big}}} \Ind_L^\pi IL^v.
\]
Then, by adjunction, for each $k$ the module $\Hom_{\Z\pi}(\Ind_\Gamma^\pi I\Gamma^v, \Z H_k)$ is isomorphic to
\[
\bigoplus_{\substack{L \text{ a factor}\\ \text{ of } \Gamma}} \Hom_{\Z L}(IL^v,\Res_L^\pi \Z H_k).
\]
Similarly, using \cref{lem:adding-I} and adjunction, $\Hom_{\Z\pi}(\Ind_\Gamma^\pi I\Gamma', \Z H_k)$ is isomorphic to
\[
\bigoplus_{\substack{L' \text{ a factor}\\ \text{ of } \Gamma'}} \Hom_{\Z L'}(IL',\Res_{L'}^\pi \Z H_k).
\]
By hypothesis the $\pi$ action on $\Z H_k$ is given by projection to $H_k$. Therefore,
\[
\Res_L^\pi \Z H_k= \begin{cases}
   \Z\otimes_\Z \Z H_k,   &\text{if }L\neq H_k\\
    \Z H_k,  &\text{if }L=H_k.
\end{cases}
\]
Here we write $\Z\otimes_\Z \Z H_k$ to emphasise that the $L$ action on $\Z H_k$ is trivial when $L\neq H_k$.
Therefore in the case that $k>t'$, the module $\Hom_{\Z\pi}(\pi_2(M),\Z H_k)$ is stably isomorphic to a direct sum of copies of $\Z H_k\otimes_{\Z} \Hom_{\Z L}(IL^v, \Z)\cong \Hom_{\Z L}(IL^v,\Z\otimes_{\Z}\Z H_k)$ and $\Z H_k\otimes_{\Z} \Hom_{\Z L'}(IL', \Z)\cong \Hom_{\Z L'}(IL',\Z\otimes_{\Z}\Z H_k)$, which are both free, as needed. On the other hand, if $k\leq t'$, then $\Hom_{\Z\pi}(\pi_2(M),\Z H_k)$ is stably isomorphic to $\Hom_{\Z H_k}(IH_k^v,\Z H_k)\oplus \Hom_{\Z H_k}(IH_k,\Z H_k)$.
The proof will be completed by showing that the latter is not a free module.

From \cite{KPT-long}*{Lemma~7.14} we know that $\Hom_{\Z H_k}(IH_k, \Z H_k)\cong IH_k^v$. We will show that $IH_k^v$ is not projective, which implies that
$\Hom_{\Z H_k}(IH_k^v,\Z H_k)\oplus \Hom_{\Z H_k}(IH_k,\Z H_k)$
is not free. Suppose that $IH_k^v$ were projective. Then $IH_k^v\to \Z H_k\to \Z^v$ is a projective resolution for $\Z^v$, implying that $H_i(H_k;\Z^v)=0$ for all $i>1$. This is a contradiction since $H_k=\Z\times \Z/2$, and its group homology with $\Z^v$ coefficients can be computed by hand, writing out a 2-periodic free resolution and tensoring with $\Z^v$. This completes the proof.
\end{proof}

\begin{corollary}\label{cor:images-fund-classes-in-group-homology}
Let $\pi$ be a $3$-manifold group and let $w \colon \pi \to C_2$ be a homomorphism such that $(\pi,w)$ is admissible.
Let $M$ and $M'$ be closed $4$-manifolds with fundamental group $\pi$, and common orientation character $w\colon \pi\to C_2$.
Let $c \colon M \to B\pi$ and $c' \colon M' \to B\pi$ be maps inducing the given identifications $\pi_1(M) \cong \pi$ and $\pi_1(M') \cong \pi$.  Suppose that $M$ and $M'$ have isomorphic quadratic $2$-types. Then $c_*([M]) - (c')_*([M']) =0 \in H_4(\pi;\Z^w)$.
\end{corollary}

As in \cref{prop:pi2-decomp,prop:pi2-determines-H4-image}, $H_4(\pi;\Z^w) \cong \bigoplus_t H_4(H_k;\Z^w) \cong (\Z/2)^t$. By \cref{prop:pi2-determines-H4-image}, the image $c_*[M]$ is nontrivial in the $k$th summand if and only if $\Hom_{\Z\pi}(\pi_2(M),\Z H_k)$ is stably free.

\begin{proof}
Since $M$ and $M'$ have isomorphic quadratic $2$-types, in particular $\pi_2(M) \cong \pi_2(M')$, and hence for each $H_k$ factor of $\pi$, we know $\Hom_{\Z\pi}(\pi_2(M),\Z H_k)$ is stably free if and only if $\Hom_{\Z\pi}(\pi_2(M'),\Z H_k)$ is stably free.
Hence by \cref{prop:pi2-determines-H4-image} the images  $c_*([M])$ and $(c')_*([M'])$ agree in $H_4(\pi;\Z^w)$, and so their difference vanishes.
\end{proof}

\section{P2FA and \texorpdfstring{\HtFAs}{P2FA*} groups}\label{sec:H2FA}

The main goal of this section is to prove \cref{prop:H2FA-3groups}, showing that if a $4$-manifold $M$ has a $3$-manifold group as fundamental group, then $\pi_2(M)$ is free an abelian group.
We do not need to restrict to admissible $(\pi,w)$ here.
We begin by introducing a relevant property of groups.

\begin{definition}\label{def:H2FA}
Let $\pi$ be a finitely presented group.
\begin{enumerate}[label=(\roman*)]
    \item  A group $\pi$ is \emph{\HtFA} if, for all closed $4$-manifolds $M$ with fundamental group $\pi$ and for all orientation characters, $\pi_2(M)$ is free as an abelian group.
    \item Further, a \HtFA\ group is said to be \emph{\HtFAs} if $\pi_2(M)^\dagger$ is free as an abelian group for all closed $4$-manifolds $M$ with fundamental group $\pi$ and for all orientation characters.
\end{enumerate}
\end{definition}

The following exact sequence will be useful in the proof of the next lemma.
\begin{remark}\label{remark:ucss-subgroup-pi2}
By the universal coefficient spectral sequence \cite{levine:modules}*{Theorem~2.3},
the sequence
\begin{equation}\label{eq:ucss-one zero}
0\to H^2(\pi;\Z\pi)\xrightarrow{c^*} H^2(X;\Z\pi)\xrightarrow{\ev}H_2(X;\Z\pi)^\dagger\to H^3(\pi;\Z\pi)\xrightarrow{c^*} H^3(X;\Z\pi)
\end{equation}
is exact for any space $X$ and any $2$-connected map $c\colon X\to B\pi$; see e.g.~\cite{FMGK}*{Lemma 3.3} or \cite{KPT-long}*{Proposition~3.3} for the derivation. 

If $X$ is a $3$-manifold with fundamental group $\pi$ and orientation character $w$, then $H^2(X;\Z\pi)\cong H_1(X;\Z\pi^w)=0$ by Poincar\'e duality. In this case, it follows from the exactness of \eqref{eq:ucss-one zero} that $H^2(\pi;\Z\pi)=0$.
If $X=M$ is a $4$-manifold then $H^3(M;\Z\pi)=0$, by Poincar\'e duality.
\end{remark}

\begin{lemma}\label{lem:P2TA-H2free}
		A group $\pi$ is \HtFA\ if and only if $H^2(\pi;\Z\pi)$ is free as an abelian group.
\end{lemma}

It is not known whether $H^2(\pi;\Z{\pi})$ is free as an abelian group for every finitely presentable group $\pi$.
There are finitely generated groups for which $H^2(\pi;\Z{\pi})$ is not free~\cite{Geoghegan}*{Chapter~13}.

\begin{proof}
By \eqref{eq:ucss-one zero}, $H^2(\pi;\Z\pi)$ is a subgroup of  $H^2(M;\Z\pi)\cong H_2(M;\Z\pi^w)$, which is isomorphic to $\pi_2(M)$ as an abelian group. So if $\pi$ is \HtFA, then $H^2(\pi;\Z\pi)$ is free as an abelian group.
	
For the converse, let $(C_*,d_*)$ be a free resolution of $\Z$ as a $\Z\pi$-module. Let $M$ be a closed $4$-manifold with fundamental group $\pi$.
By \citelist{\cite{Hambleton-Kreck:1988-1}*{Proposition~2.4}\cite{Ha09}*{Theorem 4.2}\cite{KPT}*{Proposition 1.10}}, $\pi_2(M)$ is stably an extension of $\ker d_2$ and $\coker d^2$. The module $\ker d_2$ is always free as an abelian group since it is a submodule of the free module $C_2$.
The extension
\[0\to \ker d^3/\im d^2\to C^2/\im d^2 \to C^2/\ker d^3\to 0,\]
together with the fact that $C^2/\ker d^3\cong \im d^3$, gives rise to the extension
\begin{equation}\label{eqn:an-extension} 0\to H^2(C^*)\to \coker d^2\to \im d^3\to 0.\end{equation}
Since $\im d^3$ is also a submodule of the free module $C^3$, it is free as an abelian group, and so $\coker d^2$, and hence $\pi_2(M)$, is free as an abelian group if and only if $H^2(C^*)\cong H^2(\pi;\Z\pi)$ is free as an abelian group.
\end{proof}

Recall that a group $\pi$ is said to be $FP_n$ if there is a resolution of $\Z$ by $\Z\pi$-modules such that the first $n$ terms are finitely generated projective modules, and that $\pi$ is $FP$ if there is a resolution with nonzero groups in only finitely many degrees.

\begin{lemma}
	\label{FP3+}
Let $\pi$ be a \HtFA\ group. Suppose that $H^3(\pi;\Z\pi)$ is free as an abelian group or that $\pi$ is $FP_3$.  Then $\pi$ is \HtFAs.
\end{lemma}

\begin{proof}
    Let $M$ be a closed $4$-manifold with fundamental group $\pi$.
    Then $M$ is homotopy equivalent to a finite 4-dimensional CW complex $X$ (see \cite{guide}*{Theorem~3.17} for references).
    We first assume that $\pi$ is~$FP_3$. Then there is a resolution $(P_*,d_*^\pi)$ with $P_i$ a finitely generated projective $\Z\pi$-module for $i \leq 3$. In particular $\ker d_2^\pi = \im d_3^\pi$ is finitely generated.
	Let $C_*=(C_*(X;\Z{\pi}),d_*^X)$ be the cellular chain complex of the universal cover of $X$, considered as a complex of finitely generated free left $\Z{\pi}$-modules.
	Since $0 \to  \ker d_2^\pi \to P_2 \to P_1 \to P_0 \to \Z \to 0$ and $0 \to \ker d_2^X \to C_2 \to C_1 \to C_0 \to \Z \to 0$ are both exact, with $\ker d_2^\pi$, $P_i$, and $C_i$ finitely generated, it follows from Schanuel's lemma that $\ker d_2^X$ is finitely generated.
	Hence $\pi_2(M)\cong H_2(M;\Z{\pi})\cong {H_2(C_*)}$ is finitely generated as a  $\Z\pi$-module.
	Let $k$ be such that $\pi_2(M)$ is a quotient of $(\Z{\pi})^k$. Then by left exactness $\pi_2(M)^\dagger$ is a subgroup of
	$((\Z{\pi})^{k})^{\dagger}\cong(\Z{\pi})^k$, so is free as an abelian group.
	Thus $\pi$ is \HtFAs.
	
	Now we assume that $H^3(\pi;\Z\pi)$ is free as an abelian group.
    As above, let $C_*=(C_*(X;\Z{\pi}),d_*^X)$ be the cellular chain complex of the universal cover of $X$, where $X$ is a finite 4-dimensional CW complex homotopy equivalent to $M$.
    Let $K := X^{(2)}$ be the $2$-skeleton of~$X$, and consider a CW model for $B\pi$ with $2$-skeleton $K$.
	Let $(C_*^\pi,d_*^\pi)$ be the free resolution of $\Z$ as a $\Z\pi$-module obtained from the cellular chain complex of $E\pi$. Let $c\colon X\to B\pi$ be a map that is the inclusion $K\to B\pi$ on the $2$-skeleton. Then $c$ induces a map of the chain complexes that is the identity in degrees up to two.
    There is a short exact sequence
    \[0\to \ker d^3_X\xrightarrow{(p,i)} (\ker d^3_X/\im d^2_X)\oplus C^2\xrightarrow{j-q} \coker d^2_X\to 0,\]
    where $i,j,p,$ and $q$ are the canonical inclusion and projection maps. Using that $d^2_\pi=d^2_X$, this fits into the following commutative diagram, where $j',p', q',$ and all unlabelled arrows are again the canonical inclusion and projection maps.
    \[\begin{tikzcd}
		&&&0\ar[d]&\\
		&&\ker d^3_\pi/\im d^2_X\ar[r,"="]\ar[d,"{(c^*,0)}"]&\ker d^3_\pi/\im d^2_X\ar[d]&\\
		0\ar[r]&\ker d^3_X\ar[r,"{(p,i)}"]\ar[d,"="]& (\ker d^3_X/\im d^2_X)\oplus C^2\ar[d]\ar[r,"j-q"]&\coker d^2_X\ar[r]\ar[d]&0\\
		0\ar[r]&\ker d^3_X\ar[r,"{(p',i)}"]&(\ker d^3_X/\ker d^3_\pi)\oplus C^2\ar[d]\ar[r,"j'-q'"]&C^2/\ker d^3_\pi\ar[d]\ar[r]&0\\
		&&0&0&
	\end{tikzcd}\]
    Since $\ker d^3_X\subseteq C^2$ and $C^2/\ker d^3_\pi\cong \im d^3_\pi \subseteq C^3_\pi$ are free as abelian groups, it follows from the bottom exact sequence that $(\ker d^3_X/\ker d^3_\pi)\oplus C^2$ is free as an abelian group. The middle vertical sequence identifies this term with $\coker c^*\oplus C^2$. Hence $\coker c^*$ is free as an abelian group.

By \eqref{eq:ucss-one zero} we have the short exact sequence
\[0\to \coker c^*\to H_2(M;\Z\pi)^\dagger\to H^3(\pi;\Z\pi)\to 0.\]
Hence $H_2(M;\Z\pi)^\dagger$ is free as an abelian group if $H^3(\pi;\Z\pi)$ is free as an abelian group.
\end{proof}

\begin{corollary}\label{ex:H2FA}
   Suppose that $\pi$ is a $PD_n$-group for some $n$. Then $\pi$ is \HtFAs. In particular, the trivial group, $\Z$, and all $PD_3$-groups are \HtFAs.
\end{corollary}

\begin{proof}
 If $\pi$ is a $PD_n$-group, then $H^k(\pi;\Z\pi)\cong H_{n-k}(\pi;\Z\pi)\in \{0,\Z\}$ is free as an abelian group for all $k$.
Thus $\pi$ is \HtFA\ by \cref{lem:P2TA-H2free} and
is \HtFAs\ by \cref{FP3+}.
\end{proof}

Let $\pi$ be a finitely presented group. Recall that for $G\leq \pi$ and a $\Z G$-module $A$, we denote the module $\Hom_{\Z G}(A,\Z G)$ by $A^\star$. As usual, for a $\Z\pi$-module $A$, we denote the module $\Hom_{\Z\pi}(A,\Z\pi)$ by $A^\dagger$.

\begin{lemma}\label{lem:ind-star-dagger-isom}
 Let $\pi$ be a finitely presented group, and let $G\leq \pi$ be a subgroup, let $A$ be a $\Z G$-module, and let $w\colon \pi\to C_2$ be a  homomorphism. 
    Then the map
    \begin{align*}
        \alpha\colon \Ind_G^\pi (A^\star)   &\to (\Ind_G^\pi A)^\dagger\\
        \gamma\otimes f   &\mapsto \big(\gamma\otimes a \mapsto \gamma f(a) \ol{\gamma}\big)
    \end{align*}
    is an isomorphism.
    As usual, we view $(\Ind_G^\pi A)^\dagger$ as a left module using the involution $\gamma\mapsto \overline{\gamma}=w(\gamma)\gamma^{-1}$.
\end{lemma}

\begin{proof}
    The map $\alpha$ is the composition
    \[
    \begin{tikzcd}[row sep=2mm, column sep=scriptsize]
    \Ind_G^\pi   (A^\star) \ar[r,"\cong"]  &\Hom_{\Z G}(A, \Z G)\otimes_{\Z G} \Z\pi \ar[r,"\cong"]    &\Hom_{\Z G} (A, \Z G\otimes_{\Z G} \Z\pi)\ar[r,"\cong"] &(\Ind_G^\pi A)^\dagger\\
    \gamma\otimes f \ar[r, mapsto]  &(a\mapsto f(a))\otimes \ol{\gamma}\ar[r, mapsto]  & (a\mapsto f(a)\otimes \ol{\gamma})\ar[r, mapsto]  &(\gamma \otimes a\mapsto \gamma f(a)\ol{\gamma}).
    \end{tikzcd}
    \]
    Here the first isomorphism follows from the definition, switching the left module structure to a right module structure using the involution on $\Z\pi$. The second isomorphism uses the fact that $\Z\pi$ is a free $\Z G$-module. The third isomorphism uses adjunction, and can be factored into the following three isomorphisms:
    \[(\Ind_G^\pi A)^\dagger \cong \Hom_{\Z\pi}(\Ind_G^\pi A,\Z\pi) \cong \Hom_{\Z G}(A,\Res_G^\pi \Z\pi) \cong \Hom_{\Z G}(A,\Z G \otimes_{\Z G} \Z\pi). \qedhere\]
\end{proof}

\begin{lemma}\label{lemma:Mayer-vietoris-forH3}
Let $Y$ be an orientable $3$-manifold and let $\pi := \pi_1(Y)$.
By the prime decomposition theorem for orientable $3$-manifolds, $\pi \cong \big(*_{i=1}^k L_i\big) * \big(*_{j=1}^sG_j\big)$, for some $k,s$, where each $L_i$ is infinite cyclic or finite, and each $G_j$ is a $PD_3$-group.
Then $H^3(\pi;\Z\pi)\cong\oplus_{j=1}^s E_j$, where $E_j:=\mathbb{Z}\pi\otimes_{\Z G_j}\Z = \Ind_{G_j}^{\pi} (\Z)$.
\end{lemma}

\begin{proof}
The classifying space $B\pi$ splits as a wedge of the classifying spaces $BL_i$ and $BG_j$. A Mayer--Vietoris argument therefore shows that
\begin{align*}
H^3(\pi;\Z\pi) &\cong\medoplus_{i=1}^k H^3(L_i;\Res_{L_i}^\pi\Z\pi) \oplus \medoplus_{j=1}^s H^3(G_j;\Res_{G_j}^\pi\Z\pi) \\
&\cong \medoplus_{i=1}^k H^3\big(L_i;\medoplus_{\pi/L_i} \Z L_i\big)\oplus \medoplus_{j=1}^s H^3\big(G_j;\medoplus_{\pi/G_j} \Z G_j\big)\\
&\cong \medoplus_{i=1}^k \medoplus_{\pi/L_i} H^3(L_i; \Z L_i)\oplus \medoplus_{j=1}^s \medoplus_{\pi/G_j} H^3(G_j; \Z G_j).
\end{align*}
For each $i,j$, we have $H^3(L_i;\Z L_i)=0$ and $H^3(G_j;\Z G_j)\cong \Z$, and so
\[H^3(\pi;\Z\pi) \cong \medoplus_{j=1}^s \medoplus_{\pi/G_j} \Z \cong \medoplus_{j=1}^s (\mathbb{Z}\pi\otimes_{\mathbb{Z}G_j}\mathbb{Z}) = \medoplus_{j=1}^s E_j\]
as required.
\end{proof}

Finally we prove the main result of this section showing that $3$-manifold groups are \HtFAs.

\begin{proposition}
	\label{prop:H2FA-3groups}
	Every $3$-manifold group is \HtFAs.
\end{proposition}
\begin{proof}
	Let $\pi$ be the fundamental group of a $3$-manifold $Y$. Then $H^2(\pi;\Z\pi)=0$ by \cref{remark:ucss-subgroup-pi2}, so $\pi$ is \HtFA\ by \cref{lem:P2TA-H2free}.
    Let $\pi':= \pi_1(\wh Y)$, where $\wh Y$ is the orientation double cover of $Y$. By Shapiro's lemma, $H^3(\pi;\Z\pi) \cong H^3(\pi';\Z\pi')$ as $\Z\pi'$-modules, and thus in particular as abelian groups. By \cref{lemma:Mayer-vietoris-forH3}, $H^3(\pi';\Z\pi')$ is free as an abelian group, and hence so is $H^3(\pi;\Z\pi)$.
	It now follows that $\pi$ is \HtFAs\ by \cref{FP3+}.
\end{proof}

By \cref{prop:H2FA-3groups}, \cref{thm:main-homotopy}\,\eqref{item:0} holds for the $4$-manifolds considered in \cref{thm:main-homotopy}, and the map $\mB_{H_2(M;\Z\pi)}$ is defined. For these inferences we only need that $3$-manifold groups are \HtFA. We will use the fact that they are moreover \HtFAs\ later in \cref{lem:res-dual-is-free}.

\section{Injectivity of \texorpdfstring{$\ev^*$}{ev}}
\label{sec:inj-ev*}

For this section, fix a finitely presented group $\pi$, a map $w\colon \pi\to C_2$, and a connected CW complex~$B$ with fundamental group $\pi$. We consider the map
\[\ev^*\colon \Her^w(H_2(B;\Z\pi)^\dagger)\to \Her^w(H^2(B;\Z\pi)),\]
induced by the evaluation map $\ev\colon H^2(B;\Z\pi)\to H_2(B;\Z\pi)^\dagger$. The following lemma gives a general method for showing that $\ev^*$ is injective. Recall that $\ev^*$ being injective is condition~\eqref{item:3} of~\cref{thm:homotopyclass}.

\begin{lemma}\label{lem:ev-inj}
If $H^3(\pi;\Z\pi)=0$, then $\ev^*\colon \Her^w(H_2(B;\Z\pi)^\dagger)\to \Her^w(H^2(B;\Z\pi))$ is injective.
\end{lemma}

\begin{proof}
As in \eqref{eq:ucss-one zero}, the universal coefficient spectral sequence 
gives the exact sequence
\[H^2(\pi;\Z\pi)\to H^2(B;\Z\pi)\xrightarrow{\ev} H_2(B;\Z\pi)^\dagger\to H^3(\pi;\Z\pi).\]
Since we assume $H^3(\pi;\Z\pi)=0$, the map $\ev$ is surjective. We show this implies the induced map \[\ev^*\colon \Her^w(H_2(B;\Z\pi)^\dagger)\to \Her^w(H^2(B;\Z\pi))\] is injective.
Let $\theta \in \Her^w(H_2(B;\Z\pi)^\dagger)$ be such that $\ev^*(\theta) =0$. In other words, $\theta(\ev(x),\ev(y))=0$ for all $x,y \in H^2(B;\Z\pi)$. For any $a,b \in H_2(B;\Z\pi)^\dagger$, we have $x,y\in H^2(B;\Z\pi)$ with $a=\ev(x)$ and $b=\ev(y)$, since $\ev$ is surjective. Hence $\theta(a,b) = \theta(\ev(x),\ev(y))=0$ as claimed. Thus $\theta=0$ and $\ev^*$ is injective.
\end{proof}

\begin{corollary}\label{cor:ev*-pi-finite}
If $\pi$ is virtually free, then $\ev^*\colon \Her^w(H_2(B;\Z\pi)^\dagger)\to \Her^w(H^2(B;\Z\pi))$ is injective.
\end{corollary}

\begin{proof}
If $\pi$ is virtually free then it has a free normal subgroup $\rho$ of finite index,
and $H^q(\pi;\Z\pi)\cong{H^q(\rho;\Z\rho)}$ as abelian groups,
by Shapiro's Lemma \cite{brown:cohomology-group}*{p.73}. 
Since $\rho$ is free these groups are 0 for all $q>1$.
Hence \cref{lem:ev-inj} applies.
\end{proof}

However, for general $3$-manifold groups $\pi$, it can happen that $H^3(\pi;\Z\pi)\neq 0$. Nonetheless, the conclusion still holds when $B$ is the Postnikov $2$-type of a $4$-manifold with $3$-manifold fundamental group, as we show below. In fact we will show in \cref{prop:ev-inj} that in this case $\ev^*$ is an isomorphism. For the proof we need the following preliminary results.

\begin{lemma}
\label{lem:cohomology-free-product}
Suppose $\pi\cong G*H$ and let $n>1$. Assume that $H^n(G;\Z G)=H^n(H;\Z H)=0$. Then $H^n(\pi;\Z\pi)=0$.
\end{lemma}

\begin{proof}
Since $B\pi\simeq BG\vee BH$, we can see from the Mayer--Vietoris sequence that
\begin{align*}
H^n(\pi;\Z \pi)&\cong H^n(G;\Res_G^\pi \Z\pi)\oplus H^n(H;\Res_H^\pi\Z\pi)\\
&\cong H^n\Big(G;\bigoplus_{\pi/G} \Z G\Big)\oplus H^n\Big(H;\bigoplus_{\pi/H} \Z H\Big)\\
&\cong \bigoplus_{\pi/G} H^n(G;\Z G)\oplus \bigoplus_{\pi/H} H^n(H;\Z H)=0. \qedhere
\end{align*}
\end{proof}

Let $\pi$ be a group. Recall that, for a $\Z\pi$-module $A$, we denote the module $\Hom_{\Z\pi}(A,\Z\pi)$ by $A^\dagger$. For $G\leq \pi$ and a $\Z G$-module $A$, we denote the module $\Hom_{\Z G}(A,\Z G)$ by $A^\star$. In the proof below we will have two subgroups $G,H\leq \pi$, and we will consider both $\Z G$-duals and $\Z H$-duals. To avoid a proliferation of symbols, we will use the same notation in both cases.

\begin{lemma}
\label{lem:res-dual-is-free}
    Let $G$ be a $PD_3$-group and let $\pi= G*H$ for some $3$-manifold group $H$. Let $M$ be a closed $4$-manifold with fundamental group $\pi$. Then $\Res_G^\pi\big(\pi_2(M)^\dagger\big)$ is projective. 
\end{lemma}

\begin{proof}
    By \cref{lemma:stable-splitting}, $M$ is stably homeomorphic to $M_1\#M_2$ with $\pi_1(M_1)\cong G$ and $\pi_1(M_2)\cong H$. Hence $\pi_2(M)$ is stably isomorphic to
    \begin{equation*}
        \Ind_G^\pi \pi_2(M_1) \oplus \Ind_H^\pi \pi_2(M_2).
    \end{equation*} 
Therefore, $\Res_G^\pi\big(\pi_2(M)^\dagger\big)$ is stably isomorphic to
\begin{equation}
\label{eqn:stable-sum-decompn}
\Res^\pi_G\Big(\big(\Ind^\pi_G\pi_2(M_1)\oplus \Ind^\pi_H\pi_2(M_2)\big)^\dagger\Big)
\cong \Res^\pi_G\big((\Ind^\pi_G\pi_2(M_1))^\dagger\big) \oplus\Res^\pi_G \big((\Ind^\pi_H\pi_2(M_2))^\dagger\big)
\end{equation}
We consider the first summand. By \cref{lem:ind-star-dagger-isom} we have that $(\Ind_G^\pi\pi_2(M_1))^\dagger\cong \Ind_G^\pi(\pi_2(M_1)^\star)$, and hence
    \[ \Res_G^\pi (\Ind_G^\pi\pi_2(M_1))^\dagger \cong \Res_G^\pi \Ind_G^\pi(\pi_2(M_1)^\star),\]
    where $\pi_2(M_1)^\star:=\Hom_{\Z G}(\pi_2(M_1),\Z G)$.
   By \cref{cor:pd3-free-dual-pi2}, $\pi_2(M_1)^\star$ is a stably free $\Z G$-module, hence $\Res_G^\pi\Ind_G^\pi(\pi_2(M_1)^\star)$ is projective (but not necessarily finitely generated).
    This deals with the first summand of \eqref{eqn:stable-sum-decompn}, and so it remains to consider $\Res_G^\pi\big((\Ind_H^\pi\pi_2(M_2))^\dagger\big)$.

    Again using \cref{lem:ind-star-dagger-isom}, $\Res_G^\pi\big((\Ind_H^\pi \pi_2(M_2))^\dagger\big) \cong \Res_G^\pi \Ind_H^\pi(\pi_2(M_2)^\star)$,     where $\pi_2(M_2)^\star:=\Hom_{\Z H}(\pi_2(M_2),\Z H)$.  By \cref{prop:H2FA-3groups} we know that $\pi_2(M_2)^\star$ is free as an abelian group. We will finish the proof by showing that $\Res_G^\pi\Ind_H^\pi A$ is free for every $\Z H$-module $A$ that is free as an abelian group.
 By \cite{brown:cohomology-group}*{Proposition~III.5.6(b)},
    $\Res_G^\pi \Ind_H^\pi A $ is isomorphic to a direct sum of modules of the form $\Ind_{\{e\}}^G \Res_{\{e\}}^H A$, which are again free, since $\Res_{\{e\}}^H A$ is just $A$ considered as an abelian group.
    Therefore we have shown that each summand of \eqref{eqn:stable-sum-decompn} is projective, which completes the proof.
\end{proof}

The following lemma concerns general $3$-manifold groups, without an admissibility assumption.

\begin{lemma}\label{lem:hom-inj}
Suppose that $\pi$ is a $3$-manifold group. 
Let $B=P_2(M)$, where $M$ is a closed $4$-manifold with fundamental group $\pi$.
Let $A=\Z\pi$ or $A=\pi_2(M)^\dagger$.  Then
\begin{equation*}
\Hom_{\Z\pi}(\ev,A)\colon \Hom_{\Z\pi}(H_2(B;\Z\pi)^\dagger,A)\to \Hom_{\Z\pi}(H^2(B;\Z\pi),A)
\end{equation*}
is injective.
\end{lemma}

\begin{proof}
Since $B$ can be constructed from $M$ by only adding cells of dimension four and higher, the map $M \to P_2(M)=B$ induces isomorphisms $\pi_1(M) \xrightarrow{\cong} \pi_1(B)$, $H_2(B;\Z\pi)^\dagger \xrightarrow{\cong} H_2(M;\Z\pi)^\dagger$, and $H^2(B;\Z\pi)\xrightarrow{\cong} H^2(M;\Z\pi)$.  By naturality of the evaluation map, we obtain the exact sequence
\begin{equation}\label{eq:ucss-zeros-Bs}
0\to H^2(B;\Z\pi)\xrightarrow{\ev}H_2(B;\Z\pi)^\dagger\to H^3(\pi;\Z\pi)\to 0
\end{equation}
from the exact sequence \eqref{eq:ucss-one zero}, where we used $H^2(\pi;\Z\pi)=0$ and $H^3(M;\Z\pi)=0$ as mentioned.
Apply the $\Hom_{\Z\pi}(-,A)$ functor to obtain the exact sequence
\begin{equation*}\label{eq:ucss-zeros-Bs-hommed}
	\begin{tikzcd}[row sep=huge]
		0\to \Hom_{\Z\pi}(H^3(\pi;\Z\pi),A)\arrow[r]&\Hom_{\Z\pi}(H_2(B;\Z\pi)^\dagger,A)\arrow[rr,"{\Hom_{\Z\pi}(\ev,A)}"]&&\Hom_{\Z\pi}(H^2(B;\Z\pi),A).
	\end{tikzcd}
\end{equation*}
We need to show that $\Hom_{\Z\pi}(\ev,A)$ is an injection. For this, we will show that the left term $\Hom_{\Z\pi}(H^3(\pi;\Z\pi),A)$ vanishes. 

Let $Y$ be a $3$-manifold such that $\pi\cong\pi_1(Y)$ and let $\pi'= \pi_1(\wh Y)$, where $\wh Y$ is the orientation double cover of $Y$. By Shapiro's lemma, $H^3(\pi;\Z\pi) \cong H^3(\pi';\Z\pi')$ as $\Z\pi'$-modules.  Hence $\Hom_{\Z\pi}(H^3(\pi;\Z\pi),A)$ is a submodule of $\Hom_{\Z\pi'}(H^3(\pi';\Z\pi'),A')$, where $A':=\Res_{\pi'}^\pi (A)$.
By the prime decomposition theorem for orientable $3$-manifolds, $\pi'\cong \big(*_{i=1}^k L_i\big) * \big(*_{j=1}^sG_j\big)$, for some $k,s$, where each $L_i$ is infinite cyclic or finite, and each $G_j$ is a $PD_3$-group.

By \cref{lemma:Mayer-vietoris-forH3},  $H^3(\pi';\Z\pi')\cong\oplus_{j=1}^s E_j$, where $E_j=\mathbb{Z}\pi'\otimes_{\mathbb{Z}G_j}\mathbb{Z} = \Ind_{G_j}^{\pi'}(\Z)$. Hence
\begin{align*}
    \Hom_{\Z\pi'}(H^3(\pi';\Z\pi'),A') &\cong \Hom_{\Z\pi'}(\oplus_{j=1}^s E_j,A')\cong \prod_{j=1}^s\Hom_{\Z\pi'}(E_j,A') \\ & \cong \prod_{j=1}^s\Hom_{\Z\pi'}(\Ind_{G_j}^{\pi'} \Z,A')\cong \prod_{j=1}^s\Hom_{\Z G_j}(\Z,\Res_{G_j}^{\pi} A).\end{align*}
Let $\wh M$ be the double cover of $M$ corresponding to $\pi'\leq \pi$.
As $\Z \pi'$-modules we have
    \[\pi_2(M)^\dagger \cong \Hom_{\Z\pi}(\pi_2(M),\Ind^\pi_{\pi'}\Z \pi'))\cong \Hom_{\Z\pi}(\pi_2(M),\Coind_{\pi'}^\pi \Z \pi')\cong \Hom_{\Z\pi'}(\pi_2(\wh M),\Z \pi'),\]
    where $\Coind_{\pi'}^\pi \Z \pi'\cong \Hom_{\Z \pi'}(\Z\pi,\Z \pi')$ is the coinduction of $\Z \pi'$, and for the middle isomorphism we used that  $\Ind^\pi_{\pi'}\Z \pi'\cong \Coind_{\pi'}^\pi \Z \pi'$~\cite{brown:cohomology-group}*{Proposition~III.5.9}. The last isomorphism used the $\Res-\Coind$ adjunction and that $\Res^\pi_{\pi'}  \pi_2(M) \cong \pi_2(\wh M)$.
For $A=\pi_2(M)^\dagger$ it follows that
\[\Res_{G_j}^\pi A\cong\Res_{G_j}^{\pi'}A'\cong\Res_{G_j}^{\pi'} (\Hom_{\Z \pi}(\pi_2(\wh M),\Z\pi')).\]
 Hence $\Res_{G_j}^\pi A$ is projective by \cref{lem:res-dual-is-free}. For $A=\Z\pi$, the module $\Res_{G_j}^\pi A$ is projective as well.

Since $G_j$ is infinite, the only $G_j$-fixed point in $\Z G_j$ is the trivial element. 
Since $\Res^{\pi}_{G_j} A$ is projective, $\Hom_{\Z G_j}(\Z,\Res_{G_j}^{\pi}A)\subseteq \Hom_{\Z G_j}(\Z,\bigoplus \Z G_j)=0$. Thus $\Hom_{\Z\pi'}(H^3(\pi';\Z\pi'),A')=0$ and hence also $\Hom_{\Z\pi}(H^3(\pi;\Z\pi),A)=0$ as claimed.
\end{proof}

\begin{lemma}
    \label{lem:hom-surj}
    Let $\pi\cong H*\big(*_{j=1}^sG_j\big)$ such that $H^2(H;\Z H)=H^3(H;\Z H)=0$ and each $G_j$ is a $PD_3$-group. Let $B=P_2(M)$, where $M$ is a closed $4$-manifold with fundamental group $\pi$.  Let $A=\Z\pi$ or $A=\pi_2(M)^\dagger$,
then
\begin{equation*}\label{eq:hom-inj}
\Hom_{\Z\pi}(\ev,A)\colon \Hom_{\Z\pi}(H_2(B;\Z\pi)^\dagger,A)\to \Hom_{\Z\pi}(H^2(B;\Z\pi),A)
\end{equation*}
is surjective.
\end{lemma}
\begin{proof}
    As in the proof of \cref{lem:hom-inj}, there is a short exact sequence
    \begin{equation*}
0\to H^2(B;\Z\pi)\xrightarrow{\ev}H_2(B;\Z\pi)^\dagger\to H^3(\pi;\Z\pi)\to 0.
\end{equation*}
Apply the $\Hom_{\Z\pi}(-,A)$ functor to obtain the exact sequence
\begin{equation*}
\begin{tikzcd}[row sep=huge]
\Hom_{\Z\pi}(H_2(B;\Z\pi)^\dagger,A)\arrow[rr,"{\Hom_{\Z\pi}(\ev,A)}"]&&\Hom_{\Z\pi}(H^2(B;\Z\pi),A)\ar[r]&\Ext^1_{\Z\pi}(H^3(\pi;\Z\pi),A).
\end{tikzcd}
\end{equation*}
We need to show that $\Hom_{\Z\pi}(\ev,A)$ is a surjection. For this, we will show that the term $\Ext^1_{\Z\pi}(H^3(\pi;\Z\pi),A)$ vanishes. We have $H^3(H;\Z H)=0$ and $H^3(G_j;\Z G_j)\cong \Z^{w(G_j)}$, where $w(G_j)\colon G_j\to C_2$ is the orientation character of the aspherical $PD_3$-complex with fundamental group $G_j$. A Mayer--Vietoris argument as in the proof of \cref{lemma:Mayer-vietoris-forH3} shows that $H^3(\pi;\Z\pi)\cong\oplus_{j=1}^s E_j$, where $E_j=\mathbb{Z}\pi\otimes_{\mathbb{Z}G_j}\mathbb{Z}^{w(G_j)} = \Ind_{G_j}^{\pi} \big(\Z^{w(G_j)}\big)$. Thus $\Ext^1_{\Z\pi}(H^3(\pi;\Z\pi),A)$ is isomorphic to
\begin{align*}
    \Ext^1_{\Z\pi}(H^3(\pi;\Z\pi),A) &\cong\Ext^1_{\Z\pi}(\oplus_{j=1}^sE_j,A)\cong
\prod_{j=1}^s\Ext^1_{\Z\pi}(E_j,A) \\
&\cong \prod_{j=1}^s\Ext^1_{\Z\pi}\big(\Ind_{G_j}^\pi \Z^{w(G_j)},A\big)\cong \prod_{j=1}^s\Ext^1_{\Z G_j}\big(\Z^{w(G_j)},\Res_{G_j}^\pi A\big),
\end{align*}
where we again used the Ind-Res adjunction.
We consider each factor separately. By Poincar\'e  duality, we have
\[\Ext^1_{\Z G_j}\big(\Z^{w(G_j)},\Res_{G_j}^\pi A\big)\cong H^1\big(G_j;(\Res_{G_j}^\pi A)^{w(G_j)}\big)\cong H_2(G_j; \Res_{G_j}^\pi A).\]
Since $\Res_{G_j}^\pi A$ is projective as in the proof of \cref{lem:hom-inj} this is trivial. Hence $\Hom_{\Z\pi}(\ev,A)$ is surjective as claimed.
\end{proof}

\begin{proposition}\label{prop:ev-inj}
Let $\pi$ be a $3$-manifold group and let $w \colon \pi \to C_2$ be a homomorphism.  Let $B=P_2(M)$, where $M$ is a closed $4$-manifold with fundamental group $\pi$.
Then $\ev^* \colon \Her^w(H^2(B;\Z\pi)^\dagger)\to \Her^w(H_2(B;\Z\pi))$ is injective. If moreover $(\pi,w)$ is admissible, then $\ev^*$ is an isomorphism.
\end{proposition}

\begin{proof}
We know
\[\ev^\dagger:=\Hom_{\Z\pi}(\ev,\Z\pi) \colon H_2(B;\Z\pi)^{\dagger\dagger} \to H^2(B;\Z\pi)^{\dagger} \]
is injective by setting $A=\Z\pi$ in \cref{lem:hom-inj}. Apply the  functor $\Hom_{\Z\pi}(H_2(B;\Z\pi)^\dagger, -)$, which is covariant and left-exact, to see that
\[(\ev^\dagger)_*\colon \Hom_{\Z\pi}(H_2(B;\Z\pi)^\dagger,H_2(B;\Z\pi)^{\dagger\dagger})\to \Hom_{\Z\pi}(H_2(B;\Z\pi)^\dagger,H^2(B;\Z\pi)^{\dagger})\]
is injective. On the other hand, we can apply \cref{lem:hom-inj} using $A=H^2(B;\Z\pi)^\dagger\cong \pi_2(M)^\dagger$ to see that
\[\Hom_{\Z\pi}(\ev,H^2(B;\Z\pi)^\dagger)\colon \Hom_{\Z\pi}(H_2(B;\Z\pi)^\dagger,H^2(B;\Z\pi)^{\dagger})\to \Hom_{\Z\pi}(H^2(B;\Z\pi),H^2(B;\Z\pi)^{\dagger})\] is injective.
As a consequence, the composition $\Hom_{\Z\pi}(\ev,H^2(B;\Z\pi)^\dagger)\circ (\ev^\dagger)_*$ in the top row of the commuting diagram
\begin{equation}\label{eq:sesq}
\begin{tikzcd}
\Hom_{\Z\pi}(H_2(B;\Z\pi)^\dagger,H_2(B;\Z\pi)^{\dagger\dagger}) \ar[r]  \ar[d,phantom,sloped,"\cong"]&\Hom_{\Z\pi}(H^2(B;\Z\pi),H^2(B;\Z\pi)^{\dagger})\ar[d,phantom,sloped,"\cong"]\\
\Sesq^w(H_2(B;\Z\pi)^\dagger)\arrow[r]   & \Sesq^w(H^2(B;\Z\pi))
\end{tikzcd}
\end{equation}
is injective. Hence the bottom horizontal map is also injective.

Recall that Hermitian forms are the $\Sigma_2$-fixed points of sesquilinear forms. So $\ev^*$ is the restriction of the bottom horizontal map in diagram~\eqref{eq:sesq} to Hermitian forms. Since the bottom horizontal map is $\Sigma_2$-equivariant, the first part of the proposition, that $\ev^*$ is injective, follows by taking $\Sigma_2$-fixed points.

Now suppose that $(\pi,w)$ is admissible. Then $\pi$ is a free product of groups that are
cyclic, isomorphic to $\Z\times \Z/2$, or $PD_3$-groups, as in \eqref{eq:pi-decomp}. If $G$ is cyclic or isomorphic to $\Z\times \Z/2$, then $H^2(G;\Z G)=H^3(G;\Z G)=0$. Write $\pi=H*\big(*_{j=1}^sG_j\big)$, where each $G_j$ is a $PD_3$-group and $H$ is the free product of all the remaining factors. By \cref{lem:cohomology-free-product} we see that $H^2(H;\Z H)=H^3(H;\Z H)=0$. Thus $\ev^\dagger$ and $\Hom_{\Z\pi}(\ev,H^2(B;\Z\pi)^\dagger)$ are surjective by \cref{lem:hom-surj}. Arguing as above, this implies that $(\ev^\dagger)_*$, and therefore $\ev^*$, is an isomorphism, as claimed.
\end{proof}

In particular, \cref{prop:ev-inj} shows that \cref{thm:homotopyclass}\,\eqref{item:3} holds for $4$-manifolds with $3$-manifold fundamental group, with no admissibility condition on subgroups or the orientation character.

\section{Injectivity results for \texorpdfstring{$\B$}{B}}\label{sec:injectiveB}

Again, we fix a finitely presented group $\pi$ and a map $w\colon \pi\to C_2$. In this section we give general criteria under which the map $\B_A \colon \Z^w\otimes_{\Z\pi}\Gamma(A)\to\Her^w(A^\dagger)$ is injective, where $A$ is a $\Z\pi$-module that is free as an abelian group. The map $\B_A$ was defined in~\eqref{eq:B-def}.
Note that if $A$ is finitely generated and projective it is a subgroup of $\oplus^k \Z\pi$ for some $k$, and hence is free as an abelian group. So the map $\mB_A$ is defined in this case.

We start the section with the following special case of \cite{hillman-matrix}*{Theorem~1}.

\begin{proposition}[\cite{hillman-matrix}*{Theorem~1}]\label{prop:free}
Let $A$ be a finitely generated projective $\Z\pi$-module. Assume that there is no $g\in\pi$ of order two with $w(g)=-1$. Then the map $\mB_A\colon \Z^w\otimes_{\Z\pi}\Gamma(A)\to\Her^w(A^\dagger)$ is injective.
\end{proposition}

The following definition appeared in~\citelist{\cite{hambleton-stabilityrange}*{Definition~6.4}\cite{bass-finitistic}*{\S 4.4, pp.~476-477}}.

\begin{definition}\label{def:torsionless}
A $\Z\pi$-module $L$ is called \emph{torsionless} if there exists a $\Z\pi$-embedding $L\to F$ where $F$ is a finitely generated, free $\Z\pi$-module. The module $L$ is called \emph{$w$-strongly torsionless} if additionally the induced map $\Z^w\otimes_{\Z\pi}\Gamma(L)\to \Z^w\otimes_{\Z\pi}\Gamma(F)$ is injective.
\end{definition}

Note that an arbitrary torsionless $\Z\pi$-module $A$ is free as an abelian group, since it is a subgroup of $\oplus^k \Z\pi$ for some $k$, which is free as an abelian group. Therefore the map $\B_A$ is defined. We will use this fact in the sequel without comment.

\begin{proposition}\label{prop:torsionless-inj}
Let $L$ be a $w$-strongly torsionless $\Z\pi$-module. Assume that there is no $g\in\pi$ of order two with $w(g)=-1$. Then $\mB_L\colon \Z^w\otimes_{\Z\pi}\Gamma(L)\to\Her^w(L^\dagger)$ is injective.
\end{proposition}

\begin{proof}
Let $L\xrightarrow{\theta} F$ be an embedding as in \cref{def:torsionless}. Consider the commutative diagram
\[
\begin{tikzcd}
\Z^w\otimes_{\Z\pi}\Gamma(L)\ar[d,"\theta_*"]\ar[r,"\mB_L"]&\Her^w(L^\dagger)\ar[d,"\theta_*"]\\
\Z^w\otimes_{\Z\pi}\Gamma(F)\ar[r,"\mB_F"]&\Her^w(F^\dagger).
\end{tikzcd}\]
The map $\theta_*\colon \Z^w\otimes_{\Z\pi}\Gamma(L)\to \Z^w\otimes_{\Z\pi}\Gamma(F)$ is injective by assumption and the map $\mB_F$ is injective by \cref{prop:free}. Hence $\mB_L$ is injective as claimed.
\end{proof}

It is immediate from the definition (\cref{defn:twisted-augmentation-ideals}) that twisted augmentation ideals are torsionless.  We now show that they are moreover $w$-strongly torsionless.	The case of $\pi$ finite and $v$ trivial was done in~\cite{Hambleton-Kreck:1988-1}*{Lemma~2.3}.

\begin{lemma}\label{lem:gamma-I}
Let $v\colon \pi\to C_2$ be a homomorphism. There is an isomorphism
\[
\Gamma(I\pi^v)\oplus \Z\pi\cong \Gamma(\Z\pi),
\]
which on $\Gamma(I\pi^v)$ is induced by the inclusion $I\pi^v\to \Z\pi$. In particular, $I\pi^v$ is $w$-strongly torsionless.
\end{lemma}	

\begin{proof}
As $\Z\pi$ is free as an abelian group, $\Gamma(\Z\pi)$ is isomorphic to the group of symmetric elements of $\Z\pi\otimes \Z\pi$. Therefore $\Gamma(\Z\pi)$ has a $\Z$-basis $\{g\otimes g,g\otimes h+h\otimes g\mid g,h\in \pi,g\neq h\}$. We also have an inclusion $\Z\pi\to \Gamma(\Z\pi)$ induced by $1\mapsto 1\otimes 1$. Similarly $\Gamma(I\pi^v)$ consists of the symmetric elements of $I\pi^v\otimes I\pi^v$.

By sending $g\otimes h+h\otimes g$ to $-v(gh)\big(v(g)g - v(h)h)\otimes (v(g)g-v(h)h)\big)$ we obtain a map
\[\theta\colon \Gamma(\Z\pi)/\Z\pi\to \Gamma(I\pi^v).\]
To see this, since $g\otimes h+h\otimes g$ for $g\neq h$ is a $\Z$-basis of $\Gamma(\Z\pi)/\Z\pi$, it suffices to show that the map is $\pi$-equivariant, which is a straightforward verification. Recall that the $\pi$ action on both $\Gamma$ groups is diagonal, and the action is not twisted by $v$.

The map $\Gamma(I\pi^v)\to \Gamma(\Z\pi)$ induced by the inclusion $I\pi^v\to \Z\pi$ yields a map $\psi\colon \Gamma(I\pi^v)\to \Gamma(\Z\pi)/\Z\pi$. We have \[-v(gh)\big((v(g)g-v(h)h)\otimes (v(g)g-v(h)h)\big)=h\otimes g+g\otimes h-v(gh)h\otimes h-v(gh)g\otimes g.\]
Hence $\psi\circ \theta$ is the identity on $\Gamma(\Z\pi)/\Z\pi$ because $g\otimes g$ and $h\otimes h$ lie in $\Z\pi\subseteq \Gamma(\Z\pi)$. So $\theta$ is injective.

We will show below that $\theta$ is also surjective. This will imply that we have the following commuting diagram
\[
\begin{tikzcd}
0\arrow[r]	&\Z\pi\arrow[r]	&\Gamma(\Z\pi)\arrow[r]	&\Gamma(\Z\pi)/\Z\pi\arrow[d, "\theta", "\cong"']\arrow[r]	&0\\
&	&	&\Gamma(I\pi^v)\arrow[lu, "i_*"]
\end{tikzcd}
\]		
where $i_*$ denotes the inclusion-induced map.  Thus $i_*$ provides a splitting of the short exact  sequence, and we deduce that $\Gamma(\Z\pi) \cong \Gamma(I\pi^v) \oplus \Z\pi$.

We complete the argument by showing that the map $\theta$ is surjective. Note that $\{v(g)g-v(h)h\mid g,h\in\pi\}$ gives a $\Z$-basis for $I\pi^v$, and therefore
\begin{align*}
&\{(v(g)g-v(h)h)\otimes v(g)g-v(h)h\mid g,h\in \pi\}\\
&\cup \{(v(g)g-v(h)h)\otimes (v(g')g'-v(h')h')+(v(g')g'-v(h')h')\otimes(v(g)g-v(h)h) \mid g,g',h,h'\in \pi\}
\end{align*}
is a $\Z$-basis for $\Gamma(I\pi^v)$. By definition of $\theta$,
\[\theta\big(-v(gh)(g \otimes h + h \otimes g)\big)= v(g)g-v(h)h\otimes v(g)g-v(h)h,\] for $g,h\in \pi$. So it remains to show that elements of the form
\begin{equation}\label{eqn:some-Z-basis-element-of-gamma}
    (v(g)g-v(h)h)\otimes (v(g')g'-v(h')h')+(v(g')g'-v(h')h')\otimes(v(g)g-v(h)h)
\end{equation}
lie in the image of $\theta$.
Set $\wt{g} := h^{-1}g$, $\wt{g}' := h^{-1}g'$, and $\wt{h}' := h^{-1}h'$. Then we rewrite \eqref{eqn:some-Z-basis-element-of-gamma} as
\begin{equation*}\label{eqn:another-Z-basis-element-of-gamma}
(v(h\wt{g})h\wt{g}-v(h)h)\otimes (v(h\wt{g}')h\wt{g}'-v(h\wt{h}')h\wt{h}')+(v(h\wt{g}')h\wt{g}'-v(h\wt{h}')h\wt{h}')\otimes(v(h\wt{g})h\wt{g}-v(h)h).
\end{equation*}
Since the image of $\theta$ is a $\Z\pi$-submodule, we can act by $v(h) h^{-1}$, and it suffices to show that the resulting elements, which are of the following form, lie in $\im \theta$:
\[(v(\wt{g})\wt{g}-1)\otimes (v(\wt{g}')\wt{g}'-v(\wt{h}')\wt{h}')+(v(\wt{g}')\wt{g}'-v(\wt{h}')\wt{h}')\otimes(v(\wt{g})\wt{g}-1).\]
For readability, and since we consider arbitrary $g$, $g'$, and $h'$, we drop the tildes from the notation and consider:
\begin{align*}\label{eqn:a-displayed-equation}
(v(g)&g-1)\otimes (v(g')g'-v(h')h')+(v(g')g'-v(h')h')\otimes(v(g)g-1)\\
\nonumber =&(v(g)g-1)\otimes (v(g')g'-1)+(v(g)g-1)\otimes(1-v(h')h')\\
\nonumber &+(v(g')g'-1)\otimes(v(g)g-1)+(1-v(h')h')\otimes(v(g)g-1).
\end{align*}
Hence it suffices to show that elements of the form
\begin{equation}\label{eqn:another-thing-it-suffices-to-consider}
    (v(g)g-1)\otimes (v(g')g'-1)+(v(g')g'-1)\otimes(v(g)g-1)
\end{equation}
lie in $\im \theta$. Since
\begin{align*}
    (1-v(g)g)\otimes(v(g)g-1) &= \theta\big(-v(g)(1\otimes g+g\otimes 1)\big) \text{ and }\\
(1-v(g')g')\otimes(v(g')g'-1) &= \theta\big(-v(g')(1\otimes g'+g'\otimes 1)\big)
\end{align*}
are in the image of $\theta$ by definition, we can add these to \eqref{eqn:another-thing-it-suffices-to-consider}, to obtain the following expression, which lies in $\im \theta$ if and only if \eqref{eqn:another-thing-it-suffices-to-consider} does.
\begin{align*}
(v(g)&g-1)\otimes (v(g')g'-1)+(v(g')g'-1)\otimes(v(g)g-1)\\
+&(1-v(g')g')\otimes(v(g')g'-1)+(1-v(g)g)\otimes(v(g)g-1)\\
&=(v(g)g-v(g')g')\otimes(v(g')g'-1)+(v(g')g'-v(g)g)\otimes(v(g)g-1)\\
&=(v(g)g-v(g')g')\otimes(v(g')g'-v(g)g).
\end{align*}
But $\theta\big(-v(gg')(g \otimes g' + g' \otimes g)\big) = (v(g)g-v(g')g')\otimes(v(g')g'-v(g)g)$, so this lies in the image of $\theta$, and thus indeed \eqref{eqn:another-thing-it-suffices-to-consider} lies in $\im \theta$.
This completes the proof that $\theta$ is surjective, and hence an isomorphism.

For the final sentence of the lemma note that since $\Gamma(I\pi^v)$ is a summand of $\Gamma(\Z\pi)$, it follows that $\Z^w\otimes_{\Z\pi} \Gamma(I\pi^v)$ is also a summand of $\Z^w\otimes_{\Z\pi} \Gamma(\Z\pi)$. So $I\pi^v$ is $w$-strongly torsionless, with $F=\Z\pi$, as desired.
\end{proof}

\begin{corollary}
\label{lem:B-I-injective}
Assume that there is no $g\in\pi$ of order two with $w(g)=-1$.
Let $v\colon \pi\to C_2$ be a homomorphism.
The map $\B_{I\pi^v}\colon \Z^w\otimes_{\Z\pi}\Gamma(I\pi^v)\to \Her^w((I\pi^v)^\dagger)$ is injective.
\end{corollary}

\begin{proof}
By \cref{lem:gamma-I}, $I\pi^v$ is $w$-strongly torsionless. Hence $\B_{I\pi^v}$ is injective by \cref{prop:torsionless-inj}.
\end{proof}

\cref{lem:B-I-injective} shows that $\B_A$ is injective for $A$ a twisted augmentation ideal, assuming that $w$ is trivial on elements of order two. Soon, in \cref{cor:B-inj-stab-augideals}, we generalise this to show that $\B_A$ is also injective whenever $A$ is stably isomorphic to such a twisted augmentation ideal.

We will use the following lemma due to Baues.
In the statement, to form $A\otimes_{\Z\pi}A'$, we consider $A$ as a right $\Z\pi$-module using the involution on $\Z\pi$ given by $g\mapsto w(g)g^{-1}$.

\begin{lemma}[{\cite{baues}*{(1.2.7)}}] \label{lemma:baues-splitting-of-Gamma-A-plusA'}
Let $\pi$ be a group and let $A$ and $A'$ be $\Z\pi$-modules that are free as abelian groups. Let $\iota_A \colon A \to A \oplus A'$ and $\iota_{A'} \colon A' \to A \oplus A'$ be the canonical inclusions, and define
\begin{align*}
    \psi \colon A\otimes_{\Z}A' &\to \Gamma(A\oplus A') \\
    a\otimes a' &\mapsto a\otimes a'+a'\otimes a
\end{align*}
where we use the description of $\Gamma(A\oplus A')$ as the symmetric elements in $(A\oplus A')\otimes (A\oplus A')$.
Then
\[\Gamma(\iota_A) \oplus \Gamma(\iota_{A'}) \oplus \psi \colon \Gamma(A)\oplus \Gamma(A')\oplus A\otimes_{\Z}A'\xrightarrow{\cong} \Gamma(A\oplus A')\]
is an isomorphism.
\end{lemma}

\begin{remark}\label{remark:splitting-of-Her-for-A-plus-A'}
Furthermore the map
\[\Her^w((A\oplus A')^{\dagger}) \xrightarrow{\cong} \Her^w(A^\dagger)\oplus \Her^w((A')^\dagger)\oplus \Hom_{\Z\pi}((A')^\dagger,A^{\dagger\dagger}),\]
induced by the restrictions along the inclusions of $A^\dagger$ and $(A')^\dagger$ into $(A\oplus A')^\dagger$, is an isomorphism.
\end{remark}

\begin{lemma}
\label{lem:splitting-of-B}
Let $A$ and $A'$ be $($left$)$ $\Z\pi$-modules that are free as abelian groups. Then with respect to the decomposition of $\Gamma(A \oplus A')$ from \cref{lemma:baues-splitting-of-Gamma-A-plusA'}, and the decomposition of  $\Her^w((A\oplus A')^{\dagger})$ from \cref{remark:splitting-of-Her-for-A-plus-A'}, the map $\B_{A\oplus A'}$ is isomorphic to the direct sum of $\B_A$, $\B_{A'}$, and the map
\begin{align*}
A\otimes_{\Z\pi} A'&\to \Hom_{\Z\pi}((A')^\dagger,A^{\dagger\dagger}) \\
a\otimes a' &\mapsto (f\mapsto (g\mapsto \overline{g(a)}f(a'))).
\end{align*}
\end{lemma}

\begin{remark}\label{remark:involution-for-otimes-and-Tor}
Here and throughout we adopt the convention that the tensor product $A\otimes_{\Z\pi} A'$ of two left $\Z\pi$-modules $A$ and $A'$ is formed by converting $A$ into a right $\Z\pi$-module using the involution, and then taking the tensor product.
We will use the same convention shortly when forming $\Tor_1^{\Z\pi}(A,A')$.
\end{remark}

\begin{proof}
The splittings of the domain and codomain, from \cref{lemma:baues-splitting-of-Gamma-A-plusA'} and \cref{remark:splitting-of-Her-for-A-plus-A'} respectively, induce the claimed splitting of $\B_{A\oplus A'}$.  We only show this for the image of $A\otimes_{\Z\pi}A'$ and leave the other cases to the reader. Here $a\otimes a'$ maps under $\B_{A\oplus A'}$ to the Hermitian form on $(A\oplus A')^\dagger$ given by $((f_1,g_1),(f_2,g_2))\mapsto \overline{g_1(a')}f_2(a)$. In particular, this is trivial when restricted to $\Her^w(A^\dagger)$ and $\Her^w((A')^\dagger)$. Hence it maps to $\Hom_{\Z\pi}((A')^\dagger,A^{\dagger\dagger})$, and this is given by the map  claimed.
\end{proof}

\begin{lemma}
\label{lem:stably-inj}
Assume that there is no $g\in\pi$ of order two with $w(g)=-1$.
Let $A$ be a torsionless $\Z\pi$-module and let $k\geq 0$. Then the canonical map from the kernel of $\B_A\colon \Z^w\otimes_{\Z\pi} \Gamma(A)\to \Her^w(A^\dagger)$ to the kernel of $\B_{A\oplus \Z\pi^k}\colon \Z^w\otimes_{\Z\pi} \Gamma(A\oplus \Z\pi^k)\to \Her^w((A\oplus \Z\pi^k)^\dagger)$ is an isomorphism.
\end{lemma}

\begin{proof}
Since $A\oplus \Z\pi^k$ is again torsionless, it suffices to consider the case $k=1$ by induction. Let $A\to F$ be an embedding into a finitely generated, free $\Z\pi$-module.
By the commutative square
\[\begin{tikzcd}
A\ar[r,hook]\ar[d,"\ev_A"]&F\ar[d,"\ev_F","\cong"']\\
A^{\dagger\dagger}\ar[r]&F^{\dagger\dagger},
\end{tikzcd}\]
the map $\ev_A$ sending $a\in A$ to $(f\mapsto \overline{f(a)})\in A^{\dagger\dagger}$ is injective.

By \cref{lem:splitting-of-B}, $\B_{A\oplus \Z\pi}$ is given by the direct sum of $\B_A$, $\B_{\Z\pi}$, and $\ev_A$.
By the previous paragraph $\ev_A$ is injective and $\B_{\Z\pi}$ is injective by \cref{prop:free}. It follows that the kernel of $\B_A$ is isomorphic to the kernel of $\B_{A\oplus \Z\pi^k}$, as claimed.
\end{proof}

\begin{corollary}\label{cor:B-inj-stab-augideals}
Assume that there is no $g\in\pi$ of order two with $w(g)=-1$.
Let $v\colon \pi\to C_2$ be a homomorphism and let $A$ be a $\Z\pi$-module. Suppose there exists $k,j\geq 0$ such that $A\oplus \Z\pi^k\cong I\pi^v\oplus \Z\pi^j$. Then $\B_A$ is defined and injective.
\end{corollary}

\begin{proof}
Note that by hypothesis $A$ is a subgroup of $(\Z\pi)^{j+1}$ and is therefore free as an abelian group. Thus the map $\B_A$ is defined. By \cref{lem:stably-inj}, we know that the kernel of $\B_A$ is isomorphic to the kernel of $\B_{A\oplus \Z\pi^k}$, via a natural map. By hypothesis this kernel can be identified with that of $\B_{I\pi^v\oplus \Z\pi^k}$. Another application of \cref{lem:stably-inj} shows this is isomorphic to the kernel of $\B_{I\pi^v}$, which is trivial by \cref{lem:B-I-injective}.
\end{proof}

\begin{lemma}\label{lem:sum-kernel}
Let $A$ and $A'$ be torsionless $\Z\pi$-modules and suppose that $\B_A \colon \Z^w\otimes_{\Z\pi}\Gamma(A)\to\Her^w(A^\dagger)$ and $\B_{A'} \colon \Z^w\otimes_{\Z\pi}\Gamma(A')\to\Her^w((A')^\dagger)$ are injective. Assume that $0\to A\xrightarrow{j} F\to Q\to 0$ is exact with $F$ a finitely generated, free $\Z\pi$-module. Then the kernel of \[\B_{A\oplus A'} \colon \Z^w\otimes_{\Z\pi}\Gamma(A \oplus A')\to\Her^w((A\oplus A')^\dagger)\] is contained in the image of \[\Tor_1^{\Z\pi}(Q,A')\to A\otimes_{\Z\pi} A'\leq \Z^w\otimes_{\Z\pi}\Gamma(A\oplus A').\]
\end{lemma}

\begin{proof}
By \cref{lem:splitting-of-B}, $\B_{A\oplus A'}$ is the direct sum of $\B_A$, $\B_{A'}$, and $A\otimes_{\Z\pi}A'\to\Hom((A')^\dagger,A^{\dagger\dagger})$.
Since $\B_A$ and $\B_{A'}$ are assumed to be injective, the kernel of $\B_{A\oplus A'}\colon \Z^w\otimes_{\Z\pi} \Gamma(A\oplus A')\to \Her^w((A\oplus A')^\dagger)$ is the kernel of $A\otimes_{\Z\pi}A'\to\Hom_{\Z\pi}((A')^\dagger,A^{\dagger\dagger})$.
Consider the following commuting diagram, where the latter map is the left vertical map:
\[\begin{tikzcd}[column sep = large]
\Tor_1^{\Z\pi}(Q,A')\ar[r]&A\otimes_{\Z\pi} A'\ar[r,"j \otimes \Id_{A'}"] \ar[d] & F\otimes_{\Z\pi} A'\ar[d,hook]\\
&\Hom_{\Z\pi}((A')^\dagger,A^{\dagger\dagger})\ar[r,"{\Hom (-,j^{\dagger\dagger})}"] & \Hom_{\Z\pi}((A')^\dagger,F^{\dagger\dagger}).
\end{tikzcd}\]
Since $F$ is free, the right hand vertical map can be identified with a sum of copies of $\ev_{A'}$. 
Since $\ev_{A'}\colon A'\to (A')^{\dagger\dagger}$ is injective as in the proof of \cref{lem:stably-inj}, the right hand vertical map is injective. The proof is then completed by a diagram chase.
\end{proof}

We end this section with the following result, which we will need in \cref{section:Property-H,sec:oldresults}.

\begin{proposition}
\label{prop:finitepi}
Let $\pi$ be a finite group and let $A$ be a $\Z\pi$-module that is free as a $\Z$-module. Then the kernel of $\mB_A$ equals the torsion subgroup of $\Z^w\otimes_{\Z\pi}\Gamma(A)$.
\end{proposition}

\begin{proof}
For any $\Z\pi$-module $B$ that is free as a $\Z$-module, there is a map
\[
\Phi_B\circ \overline{\omega}\colon\Z^w\otimes_{\Z\pi}\Gamma(B^\dagger)\to \Her^w(B)
\]
induced by $f\otimes g\mapsto ((a,b)\mapsto f(a)\overline{g(b)})$. Our notation corresponds to \cite{KPT-long}*{Lemma~4.19}, where it was shown that the kernel of $\Phi_B\circ \overline{\omega}$ is the torsion in $\Z^w\otimes_{\Z\pi}\Gamma(B^\dagger)$. Note that \cite{KPT-long}*{Lemma~4.19} is stated for a special $\Z\pi$-module and $w=0$. However this proof can be adapted in a straightforward way to show the statement claimed by replacing the standard norm element with the twisted norm element $\sum_{g\in\pi}w(g)g$. We apply this with $B = A^\dagger$ below.

By \cite{Nicholson}*{Lemma~2.5}, $A$ is reflexive, i.e.\ the map $\ev\colon A\to A^{\dagger\dagger}$ given by $a\mapsto (f\mapsto\overline{f(a)})$ is an isomorphism. Considering the commutative diagram
\[\begin{tikzcd}
\Z^w\otimes_{\Z\pi}\Gamma(A)\ar[dr,"\B_A"']\ar[r,"\ev","\cong"']&\Z^w\otimes_{\Z\pi}\Gamma(A^{\dagger\dagger})\ar[d,"\Phi_{A^\dagger}\circ \overline{\omega}"]\\
&\Her^w(A^\dagger)
\end{tikzcd}\]
we see that the kernel of $\B_A$ is the torsion  in $\Z\otimes_{\Z\pi}\Gamma(A)$ as claimed.
\end{proof}

\section{Proving \texorpdfstring{\cref{thm:homotopyclass}\,\eqref{item:2}}{Theorem 2.2 (3)} for \texorpdfstring{$3$}{3}-manifold fundamental groups}\label{section:proving-thm-22(3)}

Let $\pi$ be a $3$-manifold group and let $w \colon \pi \to C_2$ be a homomorphism with $(\pi,w)$ admissible.
The goal of this section is to work towards proving \cref{thm:homotopyclass}\,\eqref{item:2} for $4$-manifolds $M$ with fundamental group $\pi$ and orientation character~$w$.
Recall that this condition states: \emph{the kernel of $\mB_{H_2(B;\Z\pi)}\circ \Upsilon$ is contained in the kernel of $\varphi_B \colon \Z^w\otimes_{\Z\pi}H_4(B;\Z\pi)\to H_4(B;\Z^w)$}, where $B$ is the Postnikov $2$-type of $M$. We recall that by \cref{prop:H2FA-3groups}, $H_2(B;\Z\pi)$ is free as an abelian group, so the map $\mB_{H_2(B;\Z\pi)}$ is defined.

In \cref{subsection:bounds-on-kernel-B} we obtain upper bounds on the size of the kernel of $\B_{H_2(M;\Z\pi)}$.
In \cref{sec:ZZ2} we focus on the group $\Z \times \Z/2$, and we obtain lower bounds on the size of the kernel of $\varphi_B$, when $B$ is the $2$-type of a $4$-manifold with fundamental group $\Z \times \Z/2$, for a specific stable isomorphism class of $\pi_2(M)$.
In \cref{subsection:bounds-kernel-B-general-3-mfld-groups} we again consider $4$-manifolds $M$ of type $(\pi,w)$. For $B = P_2(M)$ we use the results from \cref{sec:ZZ2} to obtain  analogous lower bounds on the size of the kernel of $\varphi_B$ for all the fundamental groups considered in \cref{thm:main-homotopy}.

The output of this section is summarised in \cref{cor:kernelB-equals-kernelphi}. As explained in the proof of that corollary, we know that $\ker \varphi_B$ is contained in  $\ker \B_{H_2(M;\Z\pi)} \circ \Upsilon$ in general.  Our upper bounds from \cref{subsection:bounds-on-kernel-B} and lower bounds from \cref{subsection:bounds-kernel-B-general-3-mfld-groups} coincide, and we deduce in \cref{cor:kernelB-equals-kernelphi} that $\ker (\B_{H_2(M;\Z\pi)}  \circ \Upsilon) = \ker \varphi_B$ in our setting. Of course it then follows that $\ker (\B_{H_2(M;\Z\pi)}  \circ \Upsilon) \subseteq \ker \varphi_B$, so \cref{thm:homotopyclass}\,\eqref{item:2} is satisfied.

\subsection{Upper bounds on the kernel of \texorpdfstring{$\B_{H_2(M;\Z\pi)}$}{B} for \texorpdfstring{$3$}{3}-manifold fundamental groups}\label{subsection:bounds-on-kernel-B}

As in \cref{subsection-pi_2-where-pi1-is-free-product}, for the remainder of this section, let $\pi$ be a $3$-manifold group and let $w \colon \pi \to C_2$ be a homomorphism with $(\pi,w)$ admissible.  Then, by the prime decomposition theorem for $3$-manifolds, we know there is a decomposition of the form
\begin{equation}
\label{eq:pi-decomp-new}
\pi\cong F*\big(\ast_{i=1}^rZ_i\big)*\big(\ast_{j=1}^sG_j\big)*\big(\ast_{k=1}^tH_k\big),
\end{equation}
for some $r,s,t\geq 0$, with $F$ a free group, $Z_i$ a finite cyclic group for each $i$, $G_j$ a $PD_3$-group for each $j$ and $H_k\cong \Z\times \Z/2$ for each $k$.

\begin{proposition}
\label{prop:kernelB}
Fix a decomposition for $\pi$ as in \eqref{eq:pi-decomp-new}. Let $M$ be a closed $4$-manifold with fundamental group $\pi$
and orientation character $w \colon \pi \to C_2$ such that $(\pi,w)$ is admissible.
By reordering the factors of $\pi$ if needed, we assume that there exists $t' \leq t$ such that the image of the fundamental class $[M]$ in $H_4(\pi;\Z^w)\cong \bigoplus_{k=1}^tH_4(H_k;\Z^w)\cong (\Z/2)^t$ is trivial in the first $t'$ summands and nontrivial for $k>t'$.
Then the kernel of $\B_{H_2(M;\Z\pi)}$ is isomorphic to $(\Z/2)^{\tau}$ for some $\tau \leq t'$.
\end{proposition}

\begin{proof}
By \cref{prop:pi2-decomp}, $H_2(M;\Z\pi)\cong\pi_2(M)$ is stably isomorphic to $\Ind_\Gamma^\pi (I\Gamma^v)\oplus \Ind_{\Gamma'}^\pi (I\Gamma')$, where
\[\Gamma = \big(\ast_{i=1}^rZ_i\big)* \big(\ast_{j=1}^sG_j\big)* \big(\ast_{k=1}^{t'}H_k\big)\]
and
\[\Gamma' = \big(\ast_{i=1}^rZ_i\big) *  \big(\ast_{k=1}^{t'}H_k\big)\] are subgroups of $\pi$ in the canonical way.  Here $v=wv'\colon \Gamma \to C_2$, where $v'$ is trivial on each $Z_i$, is the projection onto the second factor on $\Z\times \Z/2$ (followed by the canonical isomorphism $\Z/2\to C_2$), and on each $G_j$ factor it is the orientation character $u_j$ of the aspherical $PD_3$-complex with fundamental group $G_j$. By \cref{lem:stably-inj} we therefore have
\[\ker \B_{H_2(M;\Z\pi)} \cong \ker \B_{\Ind_\Gamma^\pi I\Gamma^v \oplus \Ind_{\Gamma'}^\pi I\Gamma'}. \]

Extend $v \colon \Gamma \to C_2$ to $\ol{v} \colon \pi \to C_2$ by taking the trivial map on the factors of $\pi$ not in $\Gamma$.
By \cref{lem:adding-I}, $\Ind_\Gamma^\pi I\Gamma^v \leq I\pi^{\ol{v}}$ and $\Ind_{\Gamma'}^\pi I\Gamma' \leq I\pi$, where each inclusion is of a summand. Similarly, for the inclusion $\Ind_{\Gamma'}^\pi I\Gamma' \leq I\pi$, we extend the trivial map on $\Gamma'$ to the trivial map on $\pi$. Since $\Ind_\Gamma^\pi I\Gamma^v \oplus \Ind_{\Gamma'}^\pi I\Gamma'$ is a summand of $I\pi^{\ol{v}} \oplus \Ind_{\Gamma'}^\pi I\Gamma'$, by \cref{lem:splitting-of-B} we have that
\[\ker \B_{\Ind_\Gamma^\pi I\Gamma^v\oplus \Ind_{\Gamma'}^\pi I\Gamma'} \leq \ker \B_{I\pi^{\ol{v}} \oplus \Ind_{\Gamma'}^\pi I\Gamma'}\]
under the natural inclusion from that lemma.

By \cref{lem:B-I-injective}, the map $\B_{I\pi^u}$ is injective for every twisted augmentation ideal $I\pi^u$, where $u \colon \pi \to C_2$ is an arbitrary homomorphism that vanishes on elements of $\pi$ of order two.
Applying this to $u=\ol{v}$, we see that $\B_{I\pi^{\ol{v}}}$ is injective. Also, we have that $\B_{I\pi}$ is injective, by setting $u$ to be the trivial map. It then follows from \cref{lem:splitting-of-B} and the fact that $\Ind_{\Gamma'}^\pi I\Gamma'$ is a summand of $I\pi$ that $\B_{\Ind_{\Gamma'}^\pi I\Gamma'}$ is also injective.

Now we want to apply \cref{lem:sum-kernel} with $0 \to A \to F \to Q \to 0$ as $0 \to I\pi^{\ol{v}} \to \Z\pi \to \Z^{\ol{v}} \to 0$ and $A' = \Ind_{\Gamma'}^\pi I\Gamma'$. The
lemma applies because $\B_{I\pi^{\ol{v}}}$ and $\B_{\Ind_{\Gamma'}^\pi I\Gamma'}$ are injective. We deduce using this and the two previous displayed equations that
\[
    \ker \B_{H_2(M;\Z\pi)}\leq \ker B_{I\pi^{\ol{v}}\oplus \Ind_{\Gamma'}^\pi I\Gamma'}\leq \im\big(\Tor_1^{\Z\pi}(\Z^{\ol{v}},\Ind_{\Gamma'}^\pi I\Gamma') \to I\pi^{\ol{v}} \otimes_{\Z\pi} \Ind_{\Gamma'}^\pi I\Gamma'\big).
\]
Then by the definition, Shapiro's lemma, and dimension shifting, we have
\begin{align*}
\Tor_1^{\Z\pi}(\Z^{\ol{v}},\Ind_{\Gamma'}^\pi I\Gamma') &\cong H_1(\pi; (\Ind_{\Gamma'}^\pi I\Gamma')^{w\ol{v}})\cong H_1(\Gamma'; (I\Gamma')^{v'})\cong H_2(\Gamma';\Z^{v'})
\\
&\cong \bigoplus_{i=1}^rH_2(Z_i;\Z)\oplus \bigoplus_{k=1}^{t'}H_2(H_k;\Z^{v'})
    \cong 0\oplus \bigoplus_{k=1}^{t'}\Z/2
    \cong (\Z/2)^{t'}.
\end{align*}
Here the first isomorphism makes use of \cref{remark:involution-for-otimes-and-Tor}, which says that the $w$-twisted involution is implicitly used in defining $\Tor_1^{\Z\pi}(\Z^{\ol{v}},\Ind_{\Gamma'}^\pi I\Gamma')$; to incorporate this into the group homology $H_1(\pi; (\Ind_{\Gamma'}^\pi I\Gamma')^{w\ol{v}})$, we must give the coefficients an extra $w$-twisting.
Since the image of $\Tor_1^{\Z\pi}(\Z^{\ol{v}},\Ind_{\Gamma'}^\pi I\Gamma') \cong (\Z/2)^{t'}$ under a homomorphism is $(\Z/2)^{\tau}$ for some $\tau \leq t'$, this completes the proof.
\end{proof}

\subsection{
Lower bounds on the kernel of \texorpdfstring{$\varphi_B$}{phi-B} for fundamental group
\texorpdfstring{$\Z\times \Z/2$}{Z x Z/2}}
\label{sec:ZZ2}

For this subsection let $\pi=\Z\times \Z/2=\langle t,T\mid [T,t],T^2\rangle$ and let $v'\colon \pi\to C_2$ be given by $v'(t)=1$ and $v'(T)=-1$. In other words, $v'$ is trivial on the $\Z$ factor and nontrivial on the $\Z/2$ factor.

We first consider the case of a specific Postnikov $2$-type $P$ with fundamental group $\Z/2$. Note that this need not be the Postnikov $2$-type of a $4$-manifold. This will a useful ingredient in \cref{lem:map-B-P,lem:nontriv-kernel} where we consider a Postnikov $2$-type $B$ with fundamental group $\pi$, by comparing with~$P$.

Let $\Z^-$ denote the integers as a $\Z/2$-module, so that $1\in \Z/2$ acts by multiplication by $-1$. We use $\Z$ to denote the integers considered as a trivial $\Z/2$-module.

\begin{lemma}
\label{lem:ker-P}
Let $P$ be a connected, $3$-coconnected CW complex with $\pi_1(P)\cong \Z/2$, $\pi_2(P)\cong \Z^-\oplus \Z$, and nontrivial $k$-invariant $0\neq k_P\in H^3(\Z/2,\Z^-\oplus \Z)\cong \Z/2\oplus 0$.
Then we have isomorphisms
\[
\begin{tikzcd}
    \Z\otimes_{\Z[\Z/2]}H_4(P;\Z[\Z/2])\ar[r,"\varphi_P"]\ar[d,"\cong"] &H_4(P;\Z)\ar[d,"\cong"]\\
    \Z^2\oplus\Z/2\ar[r,"\proj_1"]  &\Z^2
\end{tikzcd}
\]
\end{lemma}

\begin{proof}
Since $P$ is $3$-coconnected, $\pi_3(P) =\pi_4(P)=0$.
Thus $H_4(P;\Z[\Z/2])\cong \Gamma(H_2(P;\Z[\Z/2]))\cong\Gamma(\pi_2(P))\cong \Gamma(\Z^-\oplus \Z)$.
By \cref{lemma:baues-splitting-of-Gamma-A-plusA'}
we also have
\[\Gamma(\Z^-\oplus \Z)\cong \Gamma(\Z^-)\oplus\Gamma(\Z)\oplus (\Z^-\otimes_{\Z} \Z)\cong \Z^2\oplus \Z^-\]
and hence
\[\Z\otimes_{\Z[\Z/2]}H_4(P;\Z[\Z/2])\cong \Z\otimes_{\Z[\Z/2]}\Gamma(\Z^-\oplus \Z)\cong \Z^2\oplus\Z/2.\]
Next we build the following model for $P$, where the $4$-skeleton is explicit. Let $\alpha\in\pi_3(\RP^2)\cong \Z$ be a generator and $Y:=\RP^2\cup_\alpha D^4$. Then $X:=Y\times \CP^2$ has
\begin{align*}
\pi_1(X) \cong \Z/2;\quad
\pi_2(X) \cong\pi_2(\RP^2)\oplus\pi_2(\CP^2)\cong \Z^-\oplus\Z; \text{ and }
\pi_3(X)=0.
\end{align*}
Furthermore, the inclusion $\RP^2 \to Y$ induces an isomorphism \[H^3(\Z/2;\pi_2(\RP^2)) \xrightarrow{\cong} H^3(\Z/2;\pi_2(Y))\] sending  the $k$-invariant of $\RP^2$ to the $k$-invariant of $Y$. Note that the first $k$-invariant of $\RP^2$ is nontrivial; if not, there would be a retraction of the inclusion $\RP^2 \to \RP^{\infty}$ over the $3$-skeleton, i.e.\ there would be a retraction $\RP^3 \to \RP^2$, which by a cup product argument in $\Z/2$-cohomology cannot exist.

Using the projection $X \to Y$, we have a map $H^3(\Z/2;\pi_2(X)) \to H^3(\Z/2;\pi_2(Y))$ sending  the $k$-invariant of $X$ to the $k$-invariant of $Y$.
It follows that the $k$-invariant of $X$ is nontrivial.
So it follows that up to homotopy equivalence we can build $P$ from $X$ by attaching cells of dimensions $5$ and higher. This finishes the construction of $X$.

By the Künneth theorem we have
\[H_4(X;\Z)\cong H_4(Y;\Z)\oplus H_2(Y;\Z)\oplus H_0(Y;\Z)\cong \Z\oplus 0\oplus \Z\cong \Z^2.\]
Hence $H_4(P;\Z)$ is some quotient of $\Z^2$. We will now show that it is in fact isomorphic to $\Z^2$.

Consider the diagram
\[
\begin{tikzcd}
\Z\otimes_{\Z[\Z/2]}H_4(P;\Z[\Z/2])	\arrow[r,"\varphi_P"]\arrow[d, "\mB_{H_2(P;\Z[\Z/2])}  \circ \Upsilon"]	&H_4(P;\Z)\arrow[d,"{\Theta_P}"]\\
\Her(H_2(P;\Z[\Z/2])^\dagger)\arrow[r,hook,"\ev^*"]	&\Her(H^2(P;\Z[\Z/2])),
\end{tikzcd}
\]
where the map $\Theta_P$ is given by $x \mapsto \big( (\alpha,\beta) \mapsto \langle \beta,\alpha\cap x\rangle \big)$.
The diagram commutes as in \cite{hillman}*{Lemma~10}.
Here we know that $\ev^*$ is injective by~\cref{cor:ev*-pi-finite}. We also know from \cref{prop:finitepi}, and the fact that $\Upsilon$ is an isomorphism, that the kernel of the map $\mB_{H_2(P;\Z[\Z/2])}  \circ \Upsilon$ equals the torsion in $\Z\otimes_{\Z[\Z/2]}H_4(P;\Z[\Z/2])$. Here we used the fact that $H_2(P;\Z[\Z/2])\cong \pi_2(P)\cong \Z^-\oplus \Z$ is free as an abelian group. Thus we know that the composition $\Theta_P\circ \varphi_P$ is precisely the quotient by the torsion subgroup. Therefore $\Theta_P$ has domain a quotient of $\Z^2$ and image isomorphic to $\Z^2$. Hence, it follows that $H_4(P;\Z)\cong \Z^2$, as claimed.
\end{proof}

\begin{lemma}\label{lem:claim:map}
For $\pi = \Z \times \Z/2$, let $w\colon \pi\to C_2$ be a homomorphism with $w(T)=1$, and let $v=wv'$.
Let $B$ be a connected, $3$-coconnected CW complex with $\pi_1(B)= \pi$, $\pi_2(B)\cong I\pi\oplus I\pi^v$, and $k$-invariant given by $0\neq k_B\in \Z/2\cong  H^3(\pi;I\pi)\leq H^3(\pi;I\pi\oplus I\pi^v)$.  Let $P$ be a CW complex as in \cref{lem:ker-P}. Then there exist maps $P\xrightarrow{f} B\xrightarrow{g} P$ such that the composition is a self-homotopy equivalence of $P$.
\end{lemma}

\begin{proof}
We will define the map $f\colon P \to B$, and will sketch the construction of the map $B \to P$.
Let $\Z/2\xrightarrow{i} \Z\times\Z/2\xrightarrow{p} \Z/2$ be the inclusion and the projection maps.

We begin the construction of a map $f\colon P \to B$.
Define the module homomorphism  $f_1\colon \Z^-\to I\pi$ determined by $1\mapsto 1-T$, and the homomorphism $f_2\colon \Z\to I\pi^v$ determined by $1\mapsto 1+T$. To show that $i \colon \pi_1(P) \to \pi_1(B)$ and $(f_1,f_2) \colon \pi_2(P) \to \pi_2(B)$ induce a map $P\to B$, note that since both $P$ and $B$ are connected and $3$-coconnected, to determine a map $P \to B$, it suffices to exhibit homomorphisms $\pi_j(P) \to \pi_j(B)$ for $j=1,2$ that respect the $k$-invariants.
More precisely, we have to show that
\[i^*k_B=(f_1,f_2)_*k_P\in H^3(\Z/2;\Res^{\pi}_{\Z/2}(I\pi\oplus I\pi^v)).\]
For this, since both $k$-invariants are nontrivial and live in a group isomorphic to $\Z/2$, it suffices to show that the maps $i^*\colon H^3(\pi;I\pi)\to H^3(\Z/2;\Res^{\pi}_{\Z/2} I\pi)$ and $(f_1)_*\colon H^3(\Z/2;\Z^-)\to H^3(\Z/2;\Res^{\pi}_{\Z/2} I\pi)$ are isomorphisms. Here we used the fact that $k_B$ lies in $H^3(\pi;I\pi)\leq H^3(\pi;I\pi\oplus I\pi^v)$. We have a diagram as follows from dimension shifting using the Bockstein homomorphism and the sequence $0 \to I\pi \to \Z\pi \to \Z \to 0$:
\[\begin{tikzcd}
H^2(\pi;\Z)\ar[d,"i^*","\cong"']\ar[r,"\cong"]&H^3(\pi;I\pi)\ar[d,"i^*"]\\
H^2(\Z/2;\Z)\ar[r,"\cong"]&H^3(\Z/2;\Res^{\pi}_{\Z/2} I\pi)
\end{tikzcd}\]
To conclude that the horizontal maps are isomorphisms we used that $\Res^{\pi}_{\Z/2} \Z\pi$ is a free $\Z[\Z/2]$-module, and $H^k(\Z/2;Q)=0$ for every free $\Z[\Z/2]$-module $Q$. Similarly we can compute directly that $H^i(\pi;\Z\pi)=0$ for $i=2,3$.
The diagram commutes by naturality of the Bockstein map. That the left vertical map is an isomorphism is a straightforward computation in group cohomology, which can be performed by comparing standard free resolutions of $\Z$.
We deduce that the right vertical map $i^*$ is an isomorphism as well.

It remains to show that $(f_1)_*$ is an isomorphism. For this consider the commutative diagram of $\Z[\Z/2]$-modules:
\[
\begin{tikzcd}
0\ar[r]&\Z^-\ar[r,"1-T"]\ar[d,"f_1"]&\Z[\Z/2]\ar[d,hook]\ar[r]&\Z\ar[d,"="]\ar[r]&0\\
0\ar[r]&\Res^{\pi}_{\Z/2} I\pi\ar[r]&\Res^{\pi}_{\Z/2} \Z\pi \ar[r]&\Z\ar[r]&0.
\end{tikzcd}
\]
From this we obtain the following commuting diagram, where again the horizontal  Bockstein maps are isomorphisms because both $\Z[\Z/2]$ and $\Res^{\pi}_{\Z/2} (\Z\pi)$ are free $\Z[\Z/2]$-modules.
\[\begin{tikzcd}
H^2(\Z/2;\Z)\ar[d,"="]\ar[r,"\cong"]&H^3(\Z/2;\Z^-)\ar[d,"(f_1)_*"]\\
H^2(\Z/2;\Z)\ar[r,"\cong"]	&H^3(\Z/2;\Res^{\pi}_{\Z/2} I\pi).
\end{tikzcd}\]
Thus $(f_1)_*$ is an isomorphism and so $i$ together with $(f_1,f_2)$ induces a map $f\colon P\to B$.

Now we sketch the construction of a map $g\colon B \to P$.
Define an equivariant maps $g_1\colon I\pi\to \Z^-$ given by sending $1-T$ to $1$ and $1-t$ to $0$, and define $g_2\colon I\pi^v\to \Z$ by sending $1+T$ to $1$ and $1-t$ to $0$. A very similar argument to that in the construction of the map $f\colon P \to B$, which we omit, shows that $p$ and $(g_1,g_2)$ induces a map $B\to P$.
It is straightforward to check that the composition is the identity on $\pi_1$ and $\pi_2$, and hence, since $P$ is $3$-coconnected and a CW complex, by Whitehead's theorem the composition $P \xrightarrow{f} B \xrightarrow{g} P$ is a self-homotopy equivalence of~$P$.
\end{proof}

\begin{lemma}
\label{lem:nontriv-kernel}\label{lem:map-B-P}
For $\pi=\Z \times \Z/2$, let $w\colon \pi\to C_2$ be a homomorphism with $w(T)=1$ and let $v=wv'$.
Let $B$ be a connected, $3$-coconnected CW complex with $\pi_1(B)= \pi$, $\pi_2(B)\cong I\pi\oplus I\pi^v$, and $k$-invariant given by $0\neq k_B\in \Z/2\cong  H^3(\pi;I\pi)\leq H^3(\pi;I\pi\oplus I\pi^v)$. Then the kernel of the map $\varphi_B \colon \Z^w\otimes_{\Z\pi} H_4(B;\Z\pi)\to H_4(B;\Z^w)$ contains a $\Z/2$ subgroup that injects into $\bbF_2\otimes_{\bbF_2[\pi]}H_4(B;\bbF_2[\pi])$ under the change of coefficients map \[\red_2^B \colon \Z^w\otimes_{\Z\pi} H_4(B;\Z\pi) \to \bbF_2\otimes_{\bbF_2[\pi]}H_4(B;\bbF_2[\pi]).\]
\end{lemma}

\begin{proof}
We first show that the kernel of the map $\Z\otimes_{\Z[\Z/2]} H_4(B;\Z[\Z/2]^w)\to H_4(B;\Z^w)$ contains a $\Z/2$ subgroup. Let $P$ be a CW complex as in \cref{lem:ker-P}. By \cref{lem:claim:map}, there exist maps $P\xrightarrow{f} B\xrightarrow{g} P$ such that the composition is a self-homotopy equivalence.
Then we have the following commutative diagram, from considering the latter maps as well as the change of coefficients map.
\[
\begin{tikzcd}
 \bbF_2\otimes_{\bbF_2[\Z/2]} H_4(P;\bbF_2[\Z/2]) \ar[d, hook, "f_*"] &  \Z\otimes_{\Z[\Z/2]}H_4(P;\Z[\Z/2])  \ar[r, "\varphi_P"] \ar[d, "f_*"]\ar[l]  &  H_4(P;\Z)   \ar[d,  "f_*"]    \\
\bbF_2\otimes_{\bbF_2[\Z/2]} H_4(B;\bbF_2[\Z/2]) \ar[d, "g_*"]    & \Z\otimes_{\Z[\Z/2]}H_4(B;\Z[\Z/2]^w) \ar[r]\ar[l]   &   H_4(B;\Z^w) \ar[from=u]
   \\
 \bbF_2\otimes_{\bbF_2[\Z/2]}H_4(P;\bbF_2[\Z/2]) & &
\end{tikzcd}
\]
The composition $g\circ f$ is a homotopy equivalence, and therefore the induced map $g_* \circ f_*$, which is the left vertical composition, is an isomorphism. It follows that the left vertical map $\bbF_2\otimes_{\bbF_2[\Z/2]} H_4(P;\bbF_2[\Z/2])\xrightarrow{f_*} \bbF_2\otimes_{\bbF_2[\Z/2]} H_4(B;\bbF_2[\Z/2])$ is injective.  The top right horizontal arrow was considered in \cref{lem:ker-P} and we know that the kernel, which we denote by $K$, is precisely the $\Z/2$ summand in $\Z^2\oplus \Z/2\cong \Z\otimes_{\Z[\Z/2]} H_4(P;\Z[\Z/2])$. Note that in this map we do not see a contribution from $w$, since $w(T)=1$. As $K$ is a summand of $\Z\otimes_{\Z[\Z/2]} H_4(P;\Z[\Z/2])$, it maps nontrivially to $\bbF_2\otimes_{\bbF_2[\Z/2]} H_4(P;\bbF_2[\Z/2])$, which then maps also nontrivially to $\bbF_2\otimes_{\bbF_2[\Z/2]} H_4(B;\bbF_2[\Z/2])$ since the map $f_*$ is injective. By the commutativity of the diagram, we see that $K$ must also map nontrivially to $\Z\otimes_{\Z[\Z/2]}H_4(B;\Z[\Z/2]^w)$, along the central vertical map $f_*$. By definition $K$ maps trivially to $H_4(B;\Z^w)$ under the composition $f_*\circ \varphi_P$. Therefore, by commutativity of the diagram, the image $f_*(K)$ in $\Z\otimes_{\Z[\Z/2]} H_4(B;\Z[\Z/2]^w)$, which we know is nontrivial, must map trivially to $H_4(B;\Z^w)$. In other words the kernel of the map $\Z\otimes_{\Z[\Z/2]} H_4(B;\Z[\Z/2]^w)\to H_4(B;\Z^w)$ contains a $\Z/2$ subgroup as claimed. Moreover, we have shown that this subgroup maps nontrivially to $\bbF_2\otimes_{\bbF_2[\Z/2]}H_4(B;\bbF_2[\Z/2])$.

Let $\wh B$ be the covering of $B$ with respect to the projection $\pi_1(B)=\Z\times \Z/2\to \Z/2$, and let $\wt{B}$ denote the universal cover. We consider the fibration $\wt B\to \wh B\to S^1$ and the associated Leray--Serre spectral sequence $H_p(S^1;H_q(\wt B;\Z)^w)\Rightarrow H_{p+q}(\wh{B};\Z^w) \cong H_{p+q}(B;\Z[\Z/2]^w)$. We have that $\pi_3(\wt{B}) =0 \to H_3(\wt{B};\Z)$ is surjective by the Hurewicz theorem, and hence $H_1(S^1;H_3(\wt{B};\Z))=0$. From the spectral sequence we see that $E^2_{0,4} = \Z^w\otimes_{\Z[\Z]} H_4(B;\Z\pi)\to H_4(B;\Z[\Z/2]^w)$ is an isomorphism, where the codomain is the homology of the total space. This isomorphism is preserved under tensoring with $\Z$, and is the left vertical map in the commuting diagram
\[
\begin{tikzcd}
\Z\otimes_{\Z[\Z/2]}(\Z^w\otimes_{\Z[\Z]} H_4(B;\Z\pi))\arrow[r,"\cong"]\arrow[d,"\cong"]	&\Z^w\otimes_{\Z\pi} H_4(B;\Z\pi)\arrow[d, "\varphi_B"]\\
\Z\otimes_{\Z[\Z/2]} H_4(B;\Z[\Z/2]^w)\arrow[r]	&H_4(B;\Z^w).
\end{tikzcd}
\]
The kernel of the right vertical map $\varphi_B$ has a $\Z/2$ subgroup, as needed, since this is true for the kernel of the bottom horizontal map by our previous argument. Repeating this argument with $\bbF_2$ coefficients shows that this subgroup of $\ker \varphi_B$ is mapped injectively to $\bbF_2\otimes_{\bbF_2[\pi]}H_4(B;\bbF_2[\pi])$ under the change of coefficients map. This completes the proof of \cref{lem:nontriv-kernel}.
\end{proof}

Now we deduce the fact corresponding to \cref{lem:nontriv-kernel} for the Postnikov $2$-type $B_M$ of a closed $4$-manifold $M$ with fundamental group $\pi = \Z \times \Z/2$ and $\pi_2(M)$ stably isomorphic to $I\pi \oplus I\pi^v$.

\begin{corollary}
\label{cor:ker-B}
For $\pi=\Z \times \Z/2$, let $w\colon \pi\to C_2$ be a homomorphism with $w(T)=1$ and let $v=wv' \colon \pi \to C_2$.
Let $M$ be a closed $4$-manifold with fundamental group $\pi$ and orientation character~$w$.
Assume further that $\pi_2(M)$ is stably isomorphic to $I\pi\oplus I\pi^v$. Let $B_M:= P_2(M)$ be the Postnikov $2$-type of~$M$. Then the kernel of the map $\varphi_{B_M} \colon \Z^w\otimes_{\Z\pi} H_4(B_M;\Z\pi)\to H_4(B_M;\Z^w)$ contains a $\Z/2$ subgroup that injects into $\bbF_2\otimes_{\bbF_2[\pi]}H_4(B_M;\bbF_2[\pi])$ under the change of coefficients map \[\red_2^M \colon \Z^w\otimes_{\Z\pi} H_4(B_M;\Z\pi) \to \bbF_2\otimes_{\bbF_2[\pi]}H_4(B_M;\bbF_2[\pi]).\]
\end{corollary}

\begin{proof}
Let \[M':= M\#\pm \CP^2.\]
We have maps $\iota$ and $\rho$ defined by
\[B_M:=P_2(M)\xrightarrow{\iota} P_2(M\vee S^2)\simeq P_2(M')=: B_{M'} \xrightarrow{\rho} P_2(M),\]
where $\rho$ collapses the $S^2$ wedge-summand, and whose composition is homotopic to the identity.
This induces a sequence of homomorphisms
\[\Z^w \otimes_{\Z\pi} H_4(B_M;\Z\pi) \xrightarrow{\iota_*} \Z^w \otimes_{\Z\pi}H_4(B_{M'};\Z\pi) \xrightarrow{\rho_*} \Z^w \otimes_{\Z\pi} H_4(B_M;\Z\pi),\]
whose composition is the identity. It follows that $\iota_*$ is injective.
Similarly,
\[\iota_*^{\bbF_2} \colon \bbF_2 \otimes_{\bbF_2[\pi]} H_4(B_M;\bbF_2[\pi]) \xrightarrow{} \bbF_2 \otimes_{\bbF[\pi]}H_4(B_{M'};\bbF_2[\pi])\]
is injective, and we have a map
 \[\rho_*^{\bbF_2} \colon \bbF_2 \otimes_{\bbF[\pi]}H_4(B_{M'};\bbF_2[\pi])  \xrightarrow{} \bbF_2 \otimes_{\bbF_2[\pi]} H_4(B_M;\bbF_2[\pi]).\]

\begin{claim}
The kernel of $\varphi_{B_M}$ contains a $\Z/2$ subgroup that maps injectively under the reduction of coefficients map $\red_2^M$ if and only if the kernel of $\varphi_{B_{M'}}$  contains a $\Z/2$ subgroup that maps injectively under~$\red_2^{M'}$.
\end{claim}

The only if direction is easier, and we prove this first.
By naturality of $\varphi$ and since $\iota_*$ is injective, if $x \in  \Z^w\otimes_{\Z\pi} H_4(B_M;\Z\pi)$ is an order two element in the kernel of $\varphi_{B_M}$, then $\iota_*(x) \in \Z^w \otimes_{\Z\pi}H_4(B_{M'};\Z\pi)$ is an element of order two in the kernel of $\varphi_{B_{M'}} \colon \Z^w \otimes_{\Z\pi}H_4(B_{M'};\Z\pi) \to H_4(B_{M'};\Z^w)$.
Since $\iota_*^{\bbF_2}$ is injective, the image of $x$ under the reduction of coefficients map $\red_2^M$ is nontrivial if and only if this is true for $\iota_*(x)$ with respect to $\red_2^{M'}$, which can be seen from the following commuting diagram.
\[
\begin{tikzcd}
\Z^w\otimes_{\Z\pi} H_4(B_M;\Z\pi) \ar[r,"\iota_*", hook] \ar[d,"\red_2^M"] &  \Z^w \otimes_{\Z\pi}H_4(B_{M'};\Z\pi) \ar[d,"\red_2^{M'}"] \\
\bbF_2 \otimes_{\bbF_2[\pi]} H_4(B_M;\bbF_2[\pi]) \ar[r,"\iota_*^{\bbF_2}", hook] & \bbF_2 \otimes_{\bbF[\pi]}H_4(B_{M'};\bbF_2[\pi])
\end{tikzcd}
\]
This completes the proof of the only if direction of the claim.

Conversely, we start with an element $x \in \ker \varphi_{B_{M'}} \subseteq  \Z^w\otimes_{\Z\pi} H_4(B_{M'};\Z\pi)$ of order two,
and assume that the image $\red_2^{M'}(x) \in \bbF_2 \otimes_{\bbF[\pi]}H_4(B_{M'};\bbF_2[\pi])$ of $x$ is nontrivial.
Then by naturality $\rho_*(x) \in \Z^w\otimes_{\Z\pi} H_4(B_M;\Z\pi)$ lies in the kernel of $\varphi_{B_M}$ and since $\rho_*$ is a homomorphism $\rho_*(x)$ has order at most two.   We need to show that $\rho_*(x)$ has order exactly two, and that $\red_2^M(\rho_*(x))$ is nontrivial.

For $N =M,M'$, let $\varphi_{B_N}^{\bbF_2} \colon \bbF_2\otimes_{\bbF[\pi]}H_4(B_N;\bbF_2[\pi])\to H_4(B_N;\bbF_2)$
be the version of $\varphi_{B_N}$ with $\bbF_2$-coefficients.
We assert for the moment that $\iota_*^{\bbF_2}$ restricts to an isomorphism \[\iota_*^{\bbF_2}| \colon \ker \varphi_{B_{M}}^{\bbF_2}  \xrightarrow{\cong} \ker \varphi_{B_{M'}}^{\bbF_2}.\]
Since $\rho_*^{\bbF_2}| \circ \iota_*^{\bbF_2}| = \Id_{\ker \varphi_{B_{M}}^{\bbF_2}}$, we deduce that $\rho_*^{\bbF_2}$ restricts to an isomorphism
\[\rho_*^{\bbF_2} | \colon  \ker \varphi_{B_{M'}}^{\bbF_2} \xrightarrow{\cong} \ker \varphi_{B_{M}}^{\bbF_2}.\]
Since $\red_2^{M'}(x)$ is assumed nontrivial, it follows that $\rho_*^{\bbF_2} (\red_2^{M'}(x))$ is nontrivial. Observe that $\rho^{\bbF_2}_*(\red_2^{M'}(x)) = \red_2^{M}\circ \rho_*(x)$ by naturality of reduction of coefficients.
Since the former is nontrivial, so is the latter, i.e.\ $\rho_*(x)$ has nontrivial image under the reduction of coefficients map.  In particular $\rho_*(x)$ is nontrivial, and so it has order exactly two.
Thus to prove the if direction of the claim, it remains to prove that $\iota_*^{\bbF_2}|$ is an isomorphism, which we do next.

For both $B=B_M$ and $B= B_{M'}$, we consider the fibration $\wt{B} \to B \to B\pi$, where $\wt{B}$ denotes the universal cover, and the associated Leray--Serre spectral sequence $H_p(B\pi;H_q(\wt{B};\bbF_2))\Rightarrow H_{p+q}(B;\bbF_2)$. We have $E^2_{0,4} \cong \bbF_2 \otimes_{\bbF_2[\pi]} H_4(B;\bbF_2[\pi])$. Since $H_1(\wt{B};\bbF_2) =0 = H_3(\wt{B};\bbF_2)$, several differentials are necessarily trivial. In particular, there is no $d_2$- or $d_4$-differential with image in $E^2_{0,4}$, so we have homomorphisms as follows, which decompose $\varphi_B^{\bbF_2}$.
\begin{align}\label{eqn:long-comp-from-LRSS}
\varphi_B^{\bbF_2} \colon \bbF_2 \otimes_{\bbF_2[\pi]} H_4(B;\bbF_2[\pi]) & \cong E^2_{0,4} \xrightarrow{\cong} E^3_{0,4} \xrightarrow{-/\im d_3} E^4_{0,4} \xrightarrow{\cong} E^5_{0,4}  \xrightarrow{-/\im d_5} E^6_{0,4} \\  &\xrightarrow{\cong} E^{\infty}_{0,4} \hookrightarrow H_4(B;\bbF_2).\nonumber
\end{align}
We will investigate the $d_3$- and $d_5$-differentials presently.
First, we have  $d_3$-differentials as shown below, with a commuting diagram with exact rows by naturality of the spectral sequence with respect to $\iota$:
\[
\begin{tikzcd}
    0\ar[r] &\ker{d_3^M}\ar[r]\ar[d]  &H_5(\pi;\bbF_2)\ar[r,"d_3^M"]\ar[d,"="]    &H_2(\pi;H_2(\wt{B}_{M};\bbF_2))\ar[d,"\cong", "\iota_*"']\\
    0\ar[r] &\ker{d_3^{M'}}\ar[r] &H_5(\pi;\bbF_2)\ar[r,"d_3^{M'}"]    &H_2(\pi;H_2(\wt{B}_{M'};\bbF_2)).
\end{tikzcd}
\]
Here the right-most isomorphism uses that $H_2(\wt{B}_{M};\bbF_2)$ and $H_2(\wt{B}_{M'};\bbF_2)$ are stably isomorphic.
It follows that the induced map $\ker d_3^M \xrightarrow{\cong} \ker d_3^{M'}$ is an isomorphism.
Next we return to the analysis of $\varphi_{B_M}^{\bbF_2}$ and $\varphi_{B_{M'}}^{\bbF_2}$ from \eqref{eqn:long-comp-from-LRSS}.
Both maps are visible in the following commutative diagram, where the straight and the diagonal rows are exact. We already showed that $\iota_*^{\bbF_2}$ is injective.
The fact that moreover
$\iota_*^{\bbF_2}$ restricts to an isomorphism \[\iota_*^{\bbF_2}| \colon \ker \varphi_{B_{M}}^{\bbF_2}  \xrightarrow{\cong} \ker \varphi_{B_{M'}}^{\bbF_2}.\]
follows from chasing this diagram.
\[
\begin{tikzcd}[row sep=tiny, column sep=tiny, trim left={([xshift=4cm]current bounding box.west)}, trim right={([xshift=-3cm]current bounding box.east)}]
        &   &\ker{d_3^M}\ar[dr, "d_5^M"]\ar[dd,near start, "\cong"]\\
    H_3(\pi;H_2(\wt{B}_M;\bbF_2))\ar[r,"d_3^M"']\ar[dd, "\cong"]  &\bbF_2 \otimes_{\bbF_2[\pi]}H_4(B_M;\bbF_2[\pi])\ar[rr,crossing over]\ar[dd, hook, "\iota_*^{\bbF_2}"]   &    &\coker{d_3^M}\ar[rr]\ar[dr]\ar[dd]    &   &0&\\
        &   &\ker{d_3^{M'}}\ar[dr,"d_5^{M'}"]   &   &E^\infty_{0,4}(M)\ar[rr,hook]    & &|[xshift=5mm, overlay]|H_4(B_M;\bbF_2)\ar[dd]\\
    H_3(\pi;H_2(\wt{B}_{M'};\bbF_2))\ar[r,"d_3^{M'}"']  &\bbF_2 \otimes_{\bbF_2[\pi]}H_4(B_{M'};\bbF_2[\pi])\ar[rr]   &    &\coker{d_3^{M'}}\ar[rr]\ar[dr]    &   &0&\\
        &   &   &   &E^\infty_{0,4}(M')\ar[rr, hook]\ar[from=uu, crossing over]    &    &|[xshift=5mm, overlay]|H_4(B_{M'};\bbF_2)
\end{tikzcd}
\]
To chase the diagram, let $x \in \ker \varphi_{B_{M'}}^{\bbF_2} \subseteq \bbF_2 \otimes_{\bbF_2[\pi]}H_4(B_{M'};\bbF_2[\pi])$. By assumption this maps to $0$ in $H_4(B_{M'};\bbF_2)$, and hence by injectivity of the bottom right horizontal map goes to $0$ in $E^\infty_{0,4}(M')$.  Thus the image of $x$ in $\coker{d_3^{M'}}$ lifts to $\ker d_3^{M'}$ and hence to $\ker d_3^{M}$. The image of this lift in $\coker d_3^M$ maps by exactness of the diagonal top row to $0$ in $H_4(B_M;\bbF_2)$, and also by exactness of the top straight row  lifts to  $y \in \ker \varphi_{B_{M}}^{\bbF_2}  \subseteq \bbF_2 \otimes_{\bbF_2[\pi]}H_4(B_M;\bbF_2[\pi])$.  By commutativity $\iota_*^{\bbF_2}(y) - x$ maps to zero in $\coker d_3^{M'}$ and so lifts to  $H_3(\pi;H_2(\wt{B}_{M'};\bbF_2))$ and hence to $H_3(\pi;H_2(\wt{B}_{M};\bbF_2))$. Use the image of this lift in   $\bbF_2 \otimes_{\bbF_2[\pi]}H_4(B_M;\bbF_2[\pi])$ to alter $y$ to $y'$ with $\iota_*^{\bbF_2}(y') = x$. By exactness we still have that $y' \in \ker \varphi_{B_{M}}^{\bbF_2}$. Hence $\iota_*^{\bbF_2}|$ is surjective as required.
This completes the proof of the claim that $\iota_*^{\bbF_2}|$ is an isomorphism.

It follows that for a $4$-manifold $X$ that is $\CP^2$-stably homeomorphic to $M$,
the kernel of $\varphi_M \colon \Z^w\otimes_{\Z\pi} H_4(B_M;\Z\pi)\to H_4(B_M;\Z^w)$ contains a $\Z/2$ subgroup that injects into $\bbF_2\otimes_{\bbF_2[\pi]}H_4(B_M;\bbF_2[\pi])$ under $\red_2^M$ if and only if, writing $B_X := P_2(X)$, the kernel of $\varphi_X \colon \Z^w\otimes_{\Z\pi} H_4(B_X;\Z\pi)\to H_4(B_X;\Z^w)$ contains a $\Z/2$ subgroup that injects into $\bbF_2\otimes_{\bbF_2[\pi]}H_4(B_X;\bbF_2[\pi])$ under $\red_2^X$.

By \cref{lem:pi2-ZZ2}, $c_*([M])=0\in H_4(\pi;\Z^w)\cong \Z/2$, for a classifying map $c\colon M\to B\pi$ inducing an isomorphism on fundamental groups. Suppose we have a $4$-manifold $X$ with fundamental group $\pi$ and $c_*([X])=0\in H_4(\pi;\Z^w)$.
Then such an $X$ is $\CP^2$-stably homeomorphic to our $M$ by \cite{KPT}*{Theorem~1.2} (which is due to Kreck~\cite{kreck}; however the citation we provided gives the explicit statement).
So it suffices to show that the kernel of $\varphi_{B_X}$ contains a $\Z/2$ subgroup that injects into $\bbF_2\otimes_{\bbF_2[\pi]}H_4(B_X;\bbF_2[\pi])$ under the change of coefficients map $\red_2^X$, and then by the previous paragraph we deduce the analogous fact for $M$.

To construct such a $4$-manifold, let $X$ be a \emph{double} for $\pi$ and orientation character $w$, obtained by definition as the boundary of a $5$-dimensional thickening $W$ of the standard $2$-complex corresponding to a presentation for the group $\pi$, with the thickening chosen to have orientation character~$w$. Since $X$ is the boundary of a $5$-manifold with fundamental group $\pi$ and orientation character $w$, and therefore $(c_X)_*([X])=0\in H_4(\pi;\Z^w)$ under a classifying map $c_X\colon X\to B\pi$. By \cite{brunner-ratcliffe}*{Theorem~1}, the first $k$-invariant of a 2-complex $K$ is trivial if and only if $\pi_1(K)$ has cohomological dimension at most~2. Thus the $k$-invariant of $W$, and therefore of $X$, is nontrivial. By \cite{KPT-long}*{Lemmas~7.11, 7.12, and~7.14}, $\pi_2(X)\cong I\pi\oplus I\pi^v$. Hence $B_X$ is homotopy equivalent to a CW complex $B$ as in \cref{lem:map-B-P}, and we deduce from that lemma that the kernel of the map $\varphi_{X} \colon \Z^w\otimes_{\Z\pi} H_4(B_X;\Z\pi)\to H_4(B_X;\Z^w)$ contains a $\Z/2$ subgroup that injects into $\bbF_2\otimes_{\bbF_2[\pi]}H_4(B_X;\bbF_2[\pi])$ under the change of coefficients map.
This completes the proof of the corollary.
\end{proof}

\subsection{Lower bounds on the kernel of \texorpdfstring{$\varphi_B$}{phi-B} for \texorpdfstring{$3$}{3}-manifold fundamental groups that are free products}\label{subsection:bounds-kernel-B-general-3-mfld-groups}

With the results of the previous subsection in hand, we can now prove the following result bounding the size of the kernel of the map $\varphi_B$, for $B$ the Postnikov $2$-type of a $4$-manifold with fundamental group $\pi$ and orientation character $w$ as in \eqref{eq:pi-decomp-new}. First we need a lemma showing how the kernel of the reduction of coefficients map $\varphi$ changes under stabilisation.

\begin{lemma}\label{lem:kernel-of-phi-and-stabilisation}
    Let $M$ be a closed $4$-manifold with fundamental group $\pi$ and orientation character $w\colon \pi\to C_2$. Let $B:=P_2(M)$ and let $B^s:=P_2(M\#(S^2\times S^2)) \simeq P_2(M \vee S^2 \vee S^2)$. The map $M \to M \vee S^2 \vee S^2$ induces an isomorphism between $\ker\varphi_B$ and $\ker\varphi_{B^s}$.
\end{lemma}

\begin{proof}
The $E^2_{0,4}$ term of the Leray--Serre spectral sequence for the fibration $\wt B\to B\to B\pi$ converging to $H_*(B;\Z^w)$ is $H_0(\pi;H_4(\wt B;\Z)^w)\cong H_0(\pi;H_4(B;\Z\pi)^w)$.
Quotienting by the images of the iterated differentials in the spectral sequence with codomain the terms $E^k_{0,4}$ induces a map $H_0(\pi;H_4(B;\Z\pi)^w) \to F_{0,4}^B$. The  codomain~$F_{0,4}^B$ is then a subgroup of the output $H_4(B;\Z^w)$ of the spectral sequence, and, under the identification $H_0(\pi;H_4(B;\Z\pi)^w) \cong \Z^w \otimes_{\Z\pi} H_4(B;\Z\pi)$, the kernel is precisely the kernel of $\varphi_B$. We have a similar quotient map $H_0(\pi;H_4(B^s;\Z\pi)^w) \to F_{0,4}^{B^s}$ for $B^s$. The idea of the proof is to compare the kernels of these two maps using naturality of the Leray-Serre spectral sequence.

The inclusion map $M \to M \vee S^2 \vee S^2$ and the collapse map $M \vee S^2 \vee S^2\to M$ induce isomorphisms on fundamental groups, and hence induce isomorphisms between $\pi_1(B)$ and $\pi_1(B^s)$. These maps also induce respectively the inclusion $\pi_2(B) \cong \pi_2(M)\to \pi_2(M)\oplus \Z\pi^2 \cong \pi_2(B^s)$ and the projection $\pi_2(B^s) \cong \pi_2(M)\oplus \Z\pi^2\to \pi_2(M) \cong \pi_2(B)$ on second homotopy groups. We have that $H_p(\pi;\Z\pi)=0$ for all $p>0$, and that $H_3(B;\Z\pi) \cong H_3(\wt{B};\Z)=0$, and similarly for $B^s$. The latter claim holds because $0=\pi_3(\wt{B}) \to H_3(\wt{B};\Z)$ is surjective in degree three by the Hurewicz theorem.  It follows that the induced maps $B\to B^s$ and $B^s\to B$ induce isomorphisms on $H_p(\pi;H_q(-;\Z\pi))$ for all $p>0$ and all $q<4$. Whence naturality of the Leray-Serre spectral sequence implies that the iterated images of the differentials in $H_0(\pi;H_4(-;\Z\pi)^w)$ and its iterated quotients, for $B$ and $B^s$, are identified by the inclusion and collapse maps. Thus the kernels of $\varphi_B$ and $\varphi_{B^s}$ are isomorphic as claimed.
\end{proof}

\begin{lemma}
\label{lem:kernel}
Let $M$ be a closed $4$-manifold with fundamental group $\pi$ as in \eqref{eq:pi-decomp-new}, and orientation character $w \colon \pi \to C_2$, such that $(\pi,w)$ is admissible.
Let $t' \leq t$ be such that, for some identification as in \eqref{eq:pi-decomp-new}, the image of the fundamental class $[M]$ in $H_4(\pi;\Z^w)\cong \bigoplus_{k=1}^tH_4(H_k;\Z^w)\cong (\Z/2)^t$ is trivial in the first $t'$ summands and nontrivial for $k>t'$.
Let $B:=P_2(M)$. Then the kernel of $\varphi_B\colon \Z^w\otimes_{\Z\pi}H_4(B;\Z\pi)\to H_4(B;\Z^w)$ is an abelian group that needs at least $t'$ generators.
\end{lemma}	

\begin{proof}
As $M$ is stably a connected sum by \cref{lemma:stable-splitting},
there exists $m\geq 0$ and a connected sum decomposition
\[M \# \bighash^m S^2 \times S^2 \cong N \# M_1 \# \cdots \# M_{t}\]
with $\pi_1(N) \cong F * \big(\ast_{i=1}^rZ_i\big) * \big(\ast_{j=1}^sG_j\big)$ and $\pi_1(M_k) \cong H_k \cong \Z \times \Z/2$.  Let $B^s := P_2(M \# m(S^2 \times S^2))\simeq P_2(M \vee \bigvee^{2m} S^2)$. The map $M \to M \vee \bigvee^{2m} S^2$ induces a map $B \to B^s$. By \cref{lem:kernel-of-phi-and-stabilisation}, this map induces an isomorphism $\ker \varphi_B \xrightarrow{\cong} \ker \varphi_{B^s}$. Hence it suffices to prove the statement of the lemma with $B^s$ in place of $B$.

Let $\pi':=\Z\times \Z/2=\langle t,T\mid [T,t],T^2\rangle$ and $v'\colon \pi' \to C_2$ be given by $v'(t)=1$ and $v'(T)=-1$. Let $\ol{w}$ denote the composition $\pi'\hookrightarrow \pi\xrightarrow{w}C_2$, where the first map is the inclusion as one of the factors in \eqref{eq:pi-decomp-new}. Let $v=\ol{w}v'$.
Let $M'$ be a closed $4$-manifold with fundamental group $\pi':=\Z\times \Z/2$ and $\pi_2(M')\cong I\pi\oplus I\pi^v$. Let $B'$ be the Postnikov $2$-type of $M'$.

For each factor $H_k$ of $\pi$, $k = 1,\dots,t'$, by \cref{lem:pi2-ZZ2} there are maps $B'\xrightarrow{\iota_k} B^s$ and $B^s\xrightarrow{p_k} B'$ such that $p_k\circ \iota_k\simeq \id_{B'}$ and $p_k\circ\iota_{k'}$ is homotopic to the constant map for $k\neq k'$.

Consider the commutative diagram
\[\begin{tikzcd}
\Z^w\otimes_{\Z\pi}H_4(B';\Z\pi)\ar[r,"\cong"]\ar[d,"(\iota_k)_*"]&\Z^w\otimes_{\Z\pi'}H_4(B';\Z\pi')\ar[d,"(\iota_k)_*"]\ar[r]&H_4(B';\Z^w)\ar[d,"(\iota_k)_*"]\\
\Z^w\otimes_{\Z\pi}H_4(B^s;\Z\pi)\ar[r]&\Z^w\otimes_{\Z\pi'}H_4(B^s;\Z\pi')\ar[r]&H_4(B^s;\Z^w)
\end{tikzcd}\]
Here $\Z\pi'$ is defined to be a $\Z\pi$-module via the map $p_k$. Our goal is to show that the composition of the two bottom horizontal maps has kernel needing at least $t'$ generators. By our condition above on the maps $\iota_k$ and $p_k$, it will suffice to show that the bottom right horizontal map has nontrivial kernel. The result will then follow since we have $t'$ distinct $H_k$ factors in $\pi$.

For this final step, consider the following commutative diagram, where we have used the change of coefficients map corresponding to $\Z\to \bbF_2$.
\[
\begin{tikzcd}
\bbF_2\otimes_{\bbF_2[\pi']} H_4(B';\bbF_2[\pi']) \ar[d, hook, "(\iota_k)_*"] & \Z^w\otimes_{\Z\pi'}H_4(B';\Z\pi')\ar[r]\ar[d, "(\iota_k)_*"]\ar[l]  &   H_4(B';\Z^w)\ar[d,  "(\iota_k)_*"]
   \\
\bbF_2\otimes_{\bbF_2[\pi']}H_4(B^s;\bbF_2[\pi']) \ar[d, "(p_k)_*"]    & \Z^w\otimes_{\Z\pi'}H_4(B^s;\Z\pi')\ar[r]\ar[l]   &   H_4(B^s;\Z^w)\ar[from=u]
   \\
   \bbF_2\otimes_{\bbF_2[\pi']}H_4(B';\bbF_2[\pi'])& &
\end{tikzcd}
\]
The composition $B'\xrightarrow{\iota_k} B^s \xrightarrow{p_k} B'$ is homotopic to the identity, and hence $(p_k)_* \circ (\iota_k)_* =\Id$. It follows that the top left vertical map is injective. By \cref{cor:ker-B}, the kernel of the top right horizontal map has a $\Z/2$ subgroup, call it $K$, which is mapped nontrivially to $\bbF_2\otimes_{\bbF_2[\pi']}H_4(B';\bbF_2[\pi'])$, and onward by the vertical injection $(\iota_k)_*$ to $\bbF_2\otimes_{\bbF_2[\pi']}H_4(B^s;\bbF_2[\pi'])$. By the commutativity of the diagram, $K$ is mapped nontrivially to $\Z^w\otimes_{\Z\pi'}H_4(B^s;\Z\pi')$ under the middle vertical $(\iota_k)_*$. By definition, $K$ maps trivially to $H_4(B';\Z^w)$, and so maps trivially to $H_4(B^s;\Z^w)$ in the bottom right. Hence by the commutativity of the diagram, the image of $K$ in $\Z^w\otimes_{\Z\pi'}H_4(B^s;\Z\pi')$, which we know is nontrivial, maps trivially to  $H_4(B^s;\Z^w)$.
In other words, the map $\Z^w\otimes_{\Z\pi'}H_4(B^s;\Z\pi')\to H_4(B^s;\Z^w)$ has nontrivial kernel, as desired.
\end{proof}

The following corollary shows that \cref{thm:homotopyclass}\,\eqref{item:2} holds for $4$-manifolds with fundamental group $\pi$ as in \eqref{eq:pi-decomp-new}. We will also use it in the proof of \cref{cor:propH-3mfd}, showing that for a $3$-manifold group $\pi$ and homomorphism $w\colon \pi\to C_2$ such that $(\pi,w)$ is admissible, the pair $(\pi,w)$ has {\propH}.

\begin{corollary}\label{cor:kernelB-equals-kernelphi}
    Let $M$ be a closed $4$-manifold with fundamental group $\pi$ as in \eqref{eq:pi-decomp-new} and orientation character $w \colon \pi \to C_2$ such that $(\pi,w)$ is admissible.
    Let $t' \leq t$ be such that, for some identification as in \eqref{eq:pi-decomp-new}, the image of the fundamental class $[M]$ in $H_4(\pi;\Z^w)\cong \bigoplus_{k=1}^tH_4(H_k;\Z^w)\cong (\Z/2)^t$ is trivial in the first $t'$ summands and nontrivial for $k>t'$.
    Let $B:=P_2(M)$. Then the kernel of $\B_{H_2(B;\Z\pi)} \circ \Upsilon$ equals the kernel of $\varphi_B\colon \Z^w\otimes_{\Z\pi}H_4(B;\Z\pi)\to H_4(B;\Z^w)$.
\end{corollary}

\begin{proof}
    By the commutativity of the diagram in~\eqref{eq:main-square}, we know that $\ev^*\circ \mB_{H_2(B;\Z\pi)} \circ \Upsilon = \Theta_B\circ \varphi_B$. Since $\pi$ is a $3$-manifold group, by \cref{prop:ev-inj} $\ev^*$ is an isomorphism. Therefore, the kernel of $\varphi_B$ is contained in the kernel of $\mB_{H_2(B;\Z\pi)} \circ \Upsilon$.

    By \cref{prop:kernelB}, the kernel of $\B_{H_2(B;\Z\pi)} \circ \Upsilon$ is a subgroup of $(\Z/2)^{t'}$, where the image of the fundamental class $[M]$ in $H_4(\pi;\Z^w)\cong \bigoplus_{k=1}^tH_4(H_k;\Z^w)\cong (\Z/2)^t$ is trivial in the first $t'$ summands and nontrivial for $k>t'$. On the other hand, the kernel of $\varphi_B$ needs at least $t'$ generators by \cref{lem:kernel}. Therefore the kernel of $\B_{H_2(B;\Z\pi)} \circ \Upsilon$ equals the kernel of $\varphi_B$, as needed.
\end{proof}

\section{\propH}\label{section:Property-H}
In this section we discuss the \emph{$4$th homology lifting property}, henceforth known as \propH.  This  will help us to show that \cref{thm:homotopyclass}\,\eqref{item:1} holds. More precisely, it seeks to establish criteria for deciding whether an  element in the codomain of $\varphi_X \colon \Z^w\otimes_{\Z\pi} H_4(X;\Z\pi)\to H_4(X;\Z^w)$ lies in the image, for an arbitrary CW complex $X$ with $\pi_1(X)=\pi$.
We will first define a version for $3$-coconnected CW complexes and homomorphisms $w\colon \pi\to C_2$, and later we will define a version for groups. We will end this section by showing that $(\pi,w)$ has {\propH}, where $\pi$ is a $3$-manifold group and $w\colon \pi\to C_2$ is a homomorphism, such that $(\pi,w)$ is admissible (\cref{cor:propH-3mfd}).

\subsection{\propH\ for CW complexes}

\begin{definition}[\propH]\label{defn:Property-H-spaces}
Let $X$ be a connected $CW$ complex and write $\pi:= \pi_1(X)$. Let $w\colon \pi\to C_2$ be a homomorphism. 
We say that $(X,w)$ has \emph{\propH} if for every 2-connected map $c \colon X \to B\pi$ and for every
\[x\in \ker(c_* \colon H_4(X;\Z^w)\to H_4(\pi;\Z^w))\] such that \[\langle \alpha,\beta\cap x\rangle =0\in \Z\pi\] for all $\alpha,\beta\in H^2(X;\Z\pi)$, we have that $x$ lies in the image of $\varphi_X \colon \Z^w\otimes_{\Z\pi} H_4(X;\Z\pi)\xrightarrow{} H_4(X;\Z^w)$.
\end{definition}

\begin{remark}
    Different choices of 2-connected maps $c \colon X \to B\pi$ are related by automorphisms of $\pi$, and therefore they determine the same kernel $K:=\ker(H_4(X;\Z^w)\to H_4(\pi;\Z^w))$.  So to verify that \propH\ holds it suffices to fix one choice of $c$.

    We also note that the image of $\varphi_X$ equals the image of the map $H_4(X;\Z\pi)\to H_4(X;\Z^w)$, so we could have equivalently asked for $x$ to lie in the image of the latter map in the definition of {\propH}.
\end{remark}

The next lemma shows that {\propH} is independent of stabilisations for $3$-coconnected CW complexes.

\begin{lemma}
\label{lem:H-stable}
Let $B$ be a connected, $3$-coconnected $CW$ complex and let $\pi:=\pi_1(B)$. Let $c_B \colon B \to B\pi$ be a 2-connected map.  Let $w\colon \pi\to C_2$ be a homomorphism. Let $B^s:=P_2(B\vee S^2)$. Then $(B,w)$ has {\propH} if and only if $(B^s,w)$ has {\propH}.
\end{lemma}

\begin{proof}
First we need some setup. Let $K:=\ker(H_4(B;\Z^w)\to H_4(\pi;\Z^w))$. We consider the Leray--Serre spectral sequence for the fibration $\wt{B} \to B \xrightarrow{c_B} B\pi$, where $\wt{B}$ is the universal cover, converging to $H_*(B;\Z^w)$. Note that $H_1(\wt{B};\Z) =0$. In addition, since $B$ is $3$-coconnected, we know $\wt{B} \simeq K(\pi_2(B),2)$, and so $\pi_3(\wt{B})=0$. By the Hurewicz theorem, the Hurewicz map $\pi_3(\wt{B}) \to H_3(\wt{B};\Z)$ is surjective, and hence $H_3(\wt{B};\Z)=0$.  Note that $K$ is the term $F_{2,2}$ in the filtration of $H_4(B;\Z^w)$ given by convergence of the spectral sequence and there is an exact sequence
\[\Z^w \otimes_{\Z\pi} H_4(\wt{B};\Z)\to K\to E\to 0\]
with $E:= E_{2,2}^{\infty}$ a quotient of $H_2(\pi;\pi_2(B)^w)$. Here we used that $\pi_2(B)\cong H_2(\wt{B};\Z)$.
Similarly, for $K^s:=\ker(H_4(B^s;\Z^w)\to H_4(\pi;\Z^w))$ we have an exact sequence
\[\Z^w \otimes_{\Z\pi} H_4(B^s;\Z\pi) \to K^s\to E^s \to 0\]
with $E^s$ the $E^{\infty}_{2,2}$ term of the corresponding Leray--Serre spectral sequence for $\wt{B}^s \to B^s \xrightarrow{c_{B^s}} B\pi$. Again $E^s$ is a quotient of $H_2(\pi;\pi_2(B^s)^w)$. The isomorphism $H_2(\pi;\pi_2(B^s)^w)\cong H_2(\pi;\pi_2(B)^w\oplus \Z\pi)\cong H_2(\pi;\pi_2(B)^w)$ induces an isomorphism $E^s\xrightarrow{\cong} E$. To see this use the definition of $E^s$ and $E$ as the quotients of $H_2(\pi;\pi_2(B^s)^w)\cong H_2(\pi;\pi_2(B)^w)$ by the image of the $d_3$ map on $H_4(B\pi;\Z^w)$. Then naturality of the spectral sequences shows that the quotients agree.

Consider the inclusion and collapse maps $B\to B\vee S^2\to B$. Passing to the Postnikov $2$-types yields maps $B\xrightarrow{\iota}B^s \xrightarrow{\rho} B$.
Then, comparing the spectral sequences, using $\iota$ and $\rho$, we obtain the following induced commutative diagram with exact rows.
\begin{equation}\label{eq:stab-proph-comm-diag}
\begin{tikzcd}
\Z^w \otimes_{\Z\pi} H_4(B;\Z\pi)\ar[r]\ar[d, "\iota_*"]&K\ar[r]\ar[d]&E\ar[r]\ar[d,"\cong"]&0\\
\Z^w \otimes_{\Z\pi} H_4(B^s;\Z\pi)\ar[r]\ar[d,"\rho_*"]&K^s\ar[r]\ar[d]&E^s\ar[r]\ar[d,"\cong"]&0\\
\Z^w \otimes_{\Z\pi} H_4(B;\Z\pi)\ar[r]&K\ar[r]& E \ar[r]&0
\end{tikzcd}
\end{equation}

Now we are ready to prove that $(B,w)$ has {\propH} if and only if $(B^s,w)$ has {\propH}. First suppose that $(B^s,w)$ has {\propH}. We will show that $(B,w)$ has {\propH}. So let $x\in K$ and assume that $\langle \alpha,\beta\cap x\rangle=0\in \Z\pi$ for all $\alpha,\beta\in H^2(B;\Z\pi)$. Define $x^s:= \iota_*(x)\in H_4(B^s,\Z^w)$. Then $x^s$ lies in $K^s$, because $c_B\colon B \to B\pi$ factors through~$\iota$. Moreover for all $\alpha^s,\beta^s\in H^2(B^s;\Z\pi)$ we have, in $\Z\pi$, that
\[\langle \alpha^s,\beta^s\cap x^s\rangle
= \langle \alpha^s,\beta^s\cap \iota_*(x)\rangle
= \langle \alpha^s,\iota_*(\iota^*(\beta^s)\cap x)\rangle
= \langle \iota^*(\alpha^s),\iota^*(\beta^s)\cap x\rangle
=0.\]
Then, since $(B^s,w)$ has {\propH}, we know that $x^s\in K^s$ lies in the image of $\Z^w \otimes_{\Z\pi} H_4(B^s;\Z\pi)$. So $x^s$ has trivial image in $E^s$. By the top right square of \eqref{eq:stab-proph-comm-diag}, $x$ has trivial image in $E$, and so lies in the image of $\Z^w \otimes_{\Z\pi} H_4(B;\Z\pi)$. Thus $(B,w)$ has {\propH}.

It remains to show that \propH\ for $(B,w)$ implies \propH\ for $(B^s,w)$. We will apply essentially the same argument as above, but now using the bottom two rows of \eqref{eq:stab-proph-comm-diag}. Assume that $(B,w)$ has {\propH}. Let $x^s \in K^s$ such that $\langle \alpha^s,\beta^s\cap x^s\rangle=0\in \Z\pi$ for all $\alpha^s,\beta^s\in H^2(B^s;\Z\pi)$. Define $x:=\rho_*(x^s)\in H_4(B;\Z^w)$. Then $x\in K$ since the map $c_{B^s}$ factors through $\rho$. We also have
\[\langle \alpha,\beta\cap x\rangle
= \langle \alpha,\beta\cap \rho_*(x^s)\rangle
= \langle \alpha,\rho_*(\rho^*(\beta)\cap x^s)\rangle
= \langle \rho^*(\alpha),\rho^*(\beta)\cap x^s\rangle
=0,\]
for all $\alpha,\beta\in H^2(B;\Z\pi)$. So by \propH\ for $(B,w)$, the element $x$ lies in the image of $\Z^w \otimes_{\Z\pi} H_4(B;\Z\pi)$, and thus has trivial image in $E$. By the bottom right square of \eqref{eq:stab-proph-comm-diag}, the element $x^s \in K^s$ also has trivial image in $E^s$, and hence lies in the image of $\Z^w \otimes_{\Z\pi} H_4(B^s;\Z\pi)$. Thus \propH\ for $(B,w)$ implies \propH\ for~$(B^s,w)$. This completes the proof.
\end{proof}

\begin{lemma}
\label{lem:H-prod}
For $i=1,\dots,n$, let $B_i$ be a connected, $3$-coconnected CW complex with fundamental group $G_i$. Let $X:=\bigvee_{i=1}^nB_i$ and let $\pi:=\pi_1(X)\cong \ast_{i=1}^nG_i$. Let $w\colon \pi \to C_2$ be a homomorphism. If each $(B_i,w|_{G_i})$ has {\propH}, for $i=1,\dots,n$, then $(X,w)$ has {\propH}.
\end{lemma}

\begin{proof}
We will write $w_i$ for $w|_{G_i}$ to make the proof more readable.
Let \[K_i:=\ker(H_4(B_i;\Z^{w_i})\to H_4(G_i;\Z^{w_i})) \text{ and } K:=\ker(H_4(X;\Z^w)\to H_4(\pi;\Z^w)).\]

Using the Mayer--Vietoris sequence, we have that $K\cong \bigoplus_{i=1}^n K_i$ and $\Z^w\otimes_{\Z\pi}H_4(\wt{X};\Z)\cong \bigoplus_{i=1}^n\Z^{w_i}\otimes_{\Z G_i}H_4(\wt{B}_i;\Z)$, where $\wt{X}$ and $\wt{B_i}$ are the universal covers. For each $x\in K$ with $\langle \alpha,\beta\cap x\rangle=0\in \Z\pi$ for all $\alpha,\beta\in H^2(X;\Z\pi)$, write $x_i$ for the image of $x$ in  $K_i$. We can view $\alpha_i,\beta_i\in H^2(B_i;\Z\pi)$ as elements of $H^2(X;\Z\pi)$, again using a Mayer--Vietoris sequence. Then
\[\langle \alpha_i,\beta_i\cap x_i\rangle=\langle \alpha_i,\beta_i\cap x_i\rangle+\sum_{j\neq i}\langle 0,0\cap x_j\rangle=\langle \alpha_i,\beta_i\cap x\rangle=0.\]
By assumption $(B_i,{w_i})$ has {\propH} and hence for each $i$ there exists a preimage $y_i\in \Z^w\otimes_{\Z G_i}H_4(\wt{B}_i;\Z)$ of $x_i \in K_i$. By naturality of the Mayer--Vietoris sequence, taking $\sum_i y_i \in \bigoplus_{i=1}^n K_i \cong K$ we have that $\varphi_X(\sum_i y_i) = x \in H_4(X;\Z^w)$.
Hence $(X,w)$ has {\propH}.
\end{proof}

We will need the following lemma on $3$-coconnected CW complexes later in this section to establish {\propH} for $\Z\times \Z/2$, as well as in the proof of \cref{cor:2-dim}.

\begin{lemma}
\label{lem:sequence-for-pi2-projective}
Let $B$ be a connected, $3$-coconnected CW complex with fundamental group $\pi:=\pi_1(B)$ and let $w\colon \pi\to C_2$ be a homomorphism. Assume that $\pi_2(B)$ is projective as a $\Z\pi$-module. Then there is an exact sequence
\[H_5(\pi;\Z^w)\to \Z^w\otimes_{\Z\pi}H_4(B;\Z\pi)\xrightarrow{\varphi_B} H_4(B;\Z^w)\to H_4(\pi;\Z^w)\to 0,\]
where the third map is induced by some classifying map $B\to B\pi$ inducing an isomorphism on fundamental groups.
\end{lemma}

\begin{proof}
Since $\pi_2(B)$ is projective, there is a decomposition $\pi_2(B) \oplus Q \cong \Z\pi^m$   for some $\Z\pi$-module $Q$ and for some $m \in \mathbb{N}$, and hence
\[H_n(\pi;\pi_2(B)^w) \leq  H_n(\pi;\pi_2(B)^w) \oplus H_n(\pi;Q^w) \cong H_n(\pi;(\pi_2(B) \oplus Q)^w) \cong H_n(\pi;\Z\pi^m) =0\]
for all $n>0$. Furthermore, $\pi_2(B)$ is free as an abelian group, and therefore projective. Thus for the universal cover $\wt B\simeq K(\pi_2(B),2)\simeq \prod \CP^\infty$ we have $H_{2k+1}(B;\Z\pi)\cong H_{2k+1}(\wt B;\Z)=0$ for all $k$.

Consider the Leray--Serre spectral sequence for the fibration $\wt B\to B\to B\pi$. On the $4$-line the only nontrivial terms are $H_0(\pi;H_4(\wt B;\Z)^w)\cong \Z^w\otimes_{\Z\pi}H_4(B;\Z\pi)$ and $H_4(B\pi;H_0(\wt B;\Z)^w)\cong H_4(\pi;\Z^w)$.
The codomains of all differentials going out of $H_4(\pi;\Z^w)$ are trivial, while the only possibly nontrivial differential into $\Z^w\otimes_{\Z\pi}H_4(B;\Z\pi)$ is $d_5\colon H_5(\pi;\Z^w)\to \Z^w\otimes_{\Z\pi}H_4(B;\Z\pi)$.

Hence $H_4(B;\Z^w)$ fits into an extension
\[0\to \Z^w\otimes_{\Z\pi}H_4(B;\Z\pi)/\im d_5\to H_4(B;\Z^w)\to H_4(\pi;\Z^w)\to 0\]
and the lemma follows.
\end{proof}

\subsection{\propH\ for groups}
\label{sec:proph-groups-general}
\begin{definition}
Let $\pi$ be a finitely presented group and let $w\colon\pi\to C_2$ be a homomorphism. We say that $(\pi,w)$ has \emph{\propH} if the pair $(P_2(M),w)$ has {\propH}, for every closed $4$-manifold $M$ with fundamental group $\pi$ and orientation character $w$.
\end{definition}

\begin{lemma}
\label{lem:propH-finite}
Let $\pi$ be a finite group and let $w\colon \pi\to C_2$ be a homomorphism.
Then $(\pi,w)$ has {\propH}.
\end{lemma}

\begin{proof}
    Let $M$ be a closed $4$-manifold with fundamental group $\pi$ and orientation character $w$. Let $B:=P_2(M)$.
    Let $x\in H_4(B;\Z^w)$ be such that $\langle \alpha,\beta\cap x\rangle=0$ for all $\alpha,\beta\in H^2(B,\Z\pi)$. We will show that $x$ lies in the image of the map $\varphi_B\colon \Z^w \otimes_{\Z\pi} H_4(B;\Z\pi)\to H_4(B;\Z^w)$.
    We will use the commutative diagram
    \[
\begin{tikzcd}
\Z^w\otimes_{\Z\pi}H_4(B;\Z\pi)	\arrow[r,"\varphi_B"]\arrow[d, "\mB_{H_2(B;\Z\pi)} \circ \Upsilon"]	&H_4(B;\Z^w)\arrow[d,"\Theta_B"]\\
\Her^w(H_2(B;\Z\pi)^\dagger)\arrow[r,"\ev^*", hook]	&\Her^w(H^2(B;\Z\pi)),
\end{tikzcd}
\]		
analogous to~\eqref{eq:main-square}. Recall that we know $\ev^*$ is injective by \cref{cor:ev*-pi-finite}.
Since $\pi$ is finite, there exists $k\in \N$ such that $kx=\varphi_B(y)$ -- see for example \cite{KT}*{(3.3)}. By \cref{prop:finitepi}, $y$ is in the torsion subgroup of $\Z^w\otimes_{\Z\pi}H_4(B;\Z\pi)$ since its image in $\Her^w(H^2(B;\Z\pi))$ is zero and $\ev^*$ is injective. As a result, $kx$, and therefore also $x$, is in the torsion subgroup of $H_4(B;\Z^w)$. Teichner \cite{teichner-phd} showed that the torsion subgroup of $H_4(B;\Z^w)$ lies in the image of $\varphi_B$ for every finite group (see also \cite{KT}*{Proof of Theorem~3.4}). This completes the proof of the lemma.
\end{proof}

\begin{lemma}\label{lem:Z-proph}
    Let $\pi=\Z$ and let $w\colon \pi\to C_2$ be a homomorphism. Then $(\pi,w)$ has {\propH}.
\end{lemma}

\begin{proof}
    Let $M$ be a closed $4$-manifold with fundamental group $\pi=\Z$ and orientation character $w$. Let $B:=P_2(M)$. By \cref{lem:pi2-z}, we know that $\pi_2(M)\cong\pi_2(B)$ is stably free and therefore projective. So we can apply \cref{lem:sequence-for-pi2-projective} to yield the exact sequence
    \[H_5(\pi;\Z^w)\to \Z^w\otimes_{\Z\pi}H_4(B;\Z\pi)\xrightarrow{\varphi_B} H_4(B;\Z^w)\to H_4(\pi;\Z^w)\to 0.\]
    Since $\pi$ is geometrically 1-dimensional, it follows that the map $\varphi_B$ is an isomorphism. This implies that $(\pi,w)$ has {\propH}.
\end{proof}

\subsection{\propH\ for \texorpdfstring{$PD_3$-groups}{PD3-groups}}
\label{sec:proph-pd3}

\begin{lemma}
\label{lem:PD3H}
Let $\pi$ be a $PD_3$-group and let $w\colon \pi\to C_2$ be a homomorphism. Then $(\pi,w)$ has {\propH}.
\end{lemma}

\begin{proof}
Let $X$ be an aspherical $PD_3$-complex with fundamental group $\pi$, orientation character $v'\colon \pi\to C_2$, and a single top cell. Then $K:=X^{(2)}$ has $H_2(K;\Z\pi)\cong \Z\pi$ and $H^2(K;\Z\pi)\cong I\pi^{v'}$. Let $v:=wv'$. By \cite{KPT-long}*{Lemma~7.12}, there is a $4$-manifold $N$ with fundamental group $\pi$, orientation character $w$, second homotopy group $\pi_2(N)\cong I\pi^v\oplus \Z\pi$, and, under the decomposition \[\Her^w(I\pi^v\oplus \Z\pi)\cong \Her^w(I\pi^v)\oplus \Her^w(\Z\pi)\oplus \Hom_{\Z\pi}(I\pi^v,\Z\pi^\dagger)\cong \Her^w(I\pi)\oplus \Her^w(\Z\pi)\oplus \Z\pi,\]
where the second isomorphism uses \cite{KLPT}*{Lemma 7.5}, the intersection form $\lambda_N$ maps to $(0,*,1)$.

By~\cite{kreck} (see also \cite{KPT}*{Theorem~1.2, Section~1.5}), and since $H_4(\pi;\Z^w)=0$, every other $4$-manifold with fundamental group $\pi$ and orientation character $w$ is homeomorphic to $N$ modulo  connected sum with copies of $\CP^2, \overline{\CP}^2$, and $\star \CP^2$. Using \cref{lem:H-stable}, it suffices to show that $(B,w)$ has {\propH} for $B:=P_2(N)$.

The Leray--Serre spectral sequence for the fibration $\wt B\to B\to B\pi$ yields the short exact sequence
\[0\to \Z^w\otimes_{\Z\pi}H_4(\wt B;\Z)\to H_4(B;\Z^w)\xrightarrow{q} \Z\to 0.\]
Let $y\in H_4(B;\Z^w)$ be a class such that for all $\alpha,\beta\in H^2(B;\Z\pi)$ we have $\langle \alpha,\beta\cap y\rangle=0$, i.e.\ $y \in \ker \Theta_B$.  Since $\pi$ is a $PD_3$-group we have that $H_4(\pi;\Z^w)=0$ so there is no corresponding condition on $y$.
 Let $k := q(y)$ be the image of $y$ in $\Z$.
We then have to show that $k$ is zero, so that $y$ lies in the image of $\Z^w\otimes_{\Z\pi}H_4(\wt B;\Z)$.  It will follow that $(\pi,w)$ has {\propH}.

By comparing the above Leray--Serre spectral sequence for $\wt B\to B\to B\pi$ with the corresponding Leray--Serre spectral sequence for $\wt N\to N\to B\pi$, via a $3$-connected map $f\colon N\to B$, we obtain a commutative diagram with exact rows:
\[
\begin{tikzcd}
&0\arrow[r]\arrow[d]	&H_4(N;\Z^w)\arrow[r, "\cong"]\arrow[d,"f_*"]	&\Z\arrow[r]\arrow[d,"="]	&0\\
0\arrow[r]	&\Z^w\otimes_{\Z\pi}H_4(\wt B;\Z)\arrow[r]	&H_4(B;\Z^w)\arrow[r]	&\Z\arrow[r]	& 0.
\end{tikzcd}
\]		
Therefore,  $f_*[N]\in H_4(B;\Z^w)$ maps to $\pm 1\in \Z$. Thus $kf_*[N]\mp y$ is in the image of~$H_4(B;\Z\pi)$.
We have the commutative diagram
\[
\begin{tikzcd}
\Z^w\otimes_{\Z\pi}H_4(B;\Z\pi)	\arrow[r,"\varphi_B"] \arrow[d, "\mB_{H_2(B;\Z\pi)} \circ \Upsilon"]	&H_4(B;\Z^w)\arrow[d,"\Theta_B"]\\
\Her^w(H_2(B;\Z\pi)^\dagger)\arrow[r,"\ev^*"]	&\Her^w(H^2(B;\Z\pi)),
\end{tikzcd}
\]		
again following~\eqref{eq:main-square}.
Since $y \in \ker \Theta_B$ we have that $\Theta_B(kf_*[N]\pm y) = kf_*\lambda_N$. Let $\lambda:=f_*\lambda_N$. Then we have that $k\lambda$ lies in the image of $\Theta_B \circ \varphi_B$. By the diagram, $k\lambda$ lies in the image of $\ev^*\circ\,\, \B_{H_2(B;\Z\pi)} \circ \Upsilon$.  Since $\Upsilon$ is an isomorphism this implies $k\lambda$ lies in the image of $\ev^*\circ\,\, \B_{H_2(B;\Z\pi)}$.

Under the decomposition \[\Her^w(I\pi^v\oplus \Z\pi)\cong \Her^w(I\pi^v)\oplus \Her^w(\Z\pi)\oplus \Hom_{\Z\pi}(I\pi^v,\Z\pi^\dagger)\cong \Her^w(I\pi^v)\oplus \Her^w(\Z\pi)\oplus \Z\pi,\] where the second isomorphism uses \cite{KLPT}*{Lemma 7.5}, the element $k\lambda$ maps to $(0,*,k)$ using the description of $\lambda$ at the start of the proof.
By \cref{lem:splitting-of-B}, the component of the image of $k\lambda$ in $\Hom_{\Z\pi}(I\pi^v,\Z\pi^\dagger)\cong \Z\pi$ lies in the image of $I\pi^v\otimes_{\Z\pi} \Z\pi\to \Hom_{\Z\pi}(I\pi^v,\Z\pi^\dagger)$. Under the isomorphisms $I\pi^v\otimes_{\Z\pi} \Z\pi\cong I\pi^v$ and $\Hom_{\Z\pi}(I\pi^v,\Z\pi^\dagger)\cong \Z\pi$, this map corresponds to the inclusion $I\pi^v\hookrightarrow \Z\pi$. In particular, $k$ lies in the image only for $k=0$. Since above we had that $kf_*[N] \mp y$ lies in the image of $\varphi_B$, and we now know that $k=0$, we deduce that $y$ lies in the image of $\varphi_B$. Hence $(\pi,w)$ has {\propH}.
\end{proof}

\subsection{\propH\ for \texorpdfstring{$\Z\times\Z/2$}{Z+Z/2}}
\label{sec:proph-ZZ2}
In this subsection let $\pi=\Z\times \Z/2=\langle t,T\mid [T,t],T^2\rangle$ and let $v'\colon \Z\pi\to C_2$ be given by $v'(t)=1$ and $v'(T)=-1$. Let $w\colon \pi\to C_2$ be such that $w(T)=1$. We will show in \cref{lem:propH-ZxZ2} that $(\pi,w)$ has \propH, but we will need a
couple of preliminary lemmas.
Let $v=wv'$. We will need to consider both the untwisted augmentation ideal $I\pi$, as well as the twisted augmentation ideals $I\pi^w$, $I\pi^{v'}$, and $I\pi^v$.

\begin{lemma}
\label{lem:extension-ZxZ2}
Let $M$ be a closed $4$-manifold with fundamental group $\pi$ and orientation character $w$ such that $w(T)=1$. Let $B$ be a connected, $3$-coconnected CW complex, and let $f\colon M\to B$ be $3$-connected.
Then there is an exact sequence \[\Z^w\otimes_{\Z\pi}H_4(\wt B;\Z)\to H_4(B;\Z^w)\to \Z/2\to 0,\]
and  the image of $f_*[M]$ in $\Z/2$ is nontrivial.
\end{lemma}

\begin{proof}
If $\pi_2(M)\cong \pi_2(B)$ is stably free as a $\Z\pi$-module, then there is an exact sequence
\[\Z^w\otimes_{\Z\pi}H_4(B;\Z\pi)\to H_4(B^w;\Z)\to H_4(\pi;\Z^w)\to 0\]
by \cref{lem:sequence-for-pi2-projective}. By \cref{lem:pi2-ZZ2}, $\pi_2(M)$ is stably free if and only if the image of $[M]$ in $H_4(\pi;\Z^w)\cong \Z/2$ is nontrivial. This completes the proof of the lemma in this case.

If $\pi_2(M)$ is not stably free, then it is stably isomorphic to $I\pi\oplus I\pi^v$ by \cref{lem:pi2-ZZ2}, where $v=wv'$. We now prove the lemma in this case.	
Let $\wt{M}$ denote the universal cover of $M$. Note that $H_3(M;\Z^w)\cong H^1(M;\Z)\cong H^1(\pi;\Z)\cong \Z$ and $H_3(\wt M;\Z)\cong H_3(M;\Z\pi)\cong H^1(M;\Z\pi^w)\cong H^1(\pi;\Z\pi^w)\cong \Z^w$, where the last isomorphism follows from the fact that $\pi$ has two ends~\cite{Geoghegan}*{Theorem~13.5.5}. Consider the spectral sequence with $E^2$ term $H_p(\pi;H_q(\wt M;\Z)^w)$ converging to $H_{p+q}(M;\Z^w)$. On the $3$-line of the $E^2$ page we have
\begin{align*}
\Z^w\otimes_{\Z\pi}H_3(\wt M;\Z)&\cong \Z^w\otimes_{\Z\pi}\Z^w\cong \Z;\\
H_1(\pi;I\pi^w\oplus I\pi^{v'})&\cong H_1(\pi;I\pi^w)\oplus H_1(\pi;I\pi^{v'})\cong H_2(\pi;\Z^w)\oplus H_2(\pi;\Z^{v'})\cong  \Z/2\oplus\Z/2;\\
H_3(\pi;\Z^w)&\cong \Z/2.
\end{align*}
On the $E^\infty$ page, these terms will produce the associated graded group for the cyclic group $H_3(M;\Z^w)$. Since only cyclic groups can arise in such an associated graded group, it follows that the $d_3$-differential $\Z/2\cong H_4(\pi;\Z^w)\to H_1(\pi;I\pi^w\oplus I\pi^{v'})$ is nontrivial and thus injective.

On the $4$-line of the $E^2$ page we have
\begin{align*}
H_1(\pi;H_3(\wt M;\Z)^w)&\cong H_1(\pi;\Z)\cong \Z\oplus \Z/2;\\
H_2(\pi;I\pi^w\oplus I\pi^{v'})&\cong H_3(\pi;\Z^w)\oplus H_3(\pi;\Z^{v'})\cong \Z/2\oplus\Z/2; \\
H_4(\pi;\Z^w)&\cong \Z/2.
\end{align*}
We already saw that the $d_3$ map out of $H_4(\pi;\Z^w)$ is injective, so that term does not survive to the $E^\infty$ page. The $d_2$-differential $H_2(\pi;I\pi^w\oplus I\pi^{v'})\cong \Z/2\oplus \Z/2\to H_0(\pi;H_3(\wt{M};\Z)^w)\cong \Z$ is trivial, so $H_2(\pi;I\pi^w\oplus I\pi^{v'})$ survives to the $E^3$ page. Since $H_4(M;\Z^w)$ is again cyclic, the $d_3$-differential $H_5(\pi;\Z^w)\cong \Z/2\to H_2(\pi;I\pi^w\oplus I\pi^{v'})\cong \Z/2\oplus \Z/2$ is nontrivial and therefore injective. Hence from the $E^\infty$ page we obtain an extension $\Z\to H_4(M;\Z^w)\to H_2(\pi;I\pi^w\oplus I\pi^{v'})/\im d_2\cong \Z/2\to 0$.

Now we  consider the analogous spectral sequence with $E^2$ term $H_p(\pi;H_q(\wt B;\Z)^w)$ converging to $H_{p+q}(B;\Z^w)$, where $\wt{B}$ is the universal cover of $B$. In this case the $d_2$-differential $H_2(\pi;I\pi^w\oplus I\pi^{v'})\cong \Z/2\oplus \Z/2\to H_0(\pi;H_3(\wt{B};\Z)^w)= 0$ is necessarily trivial. Since $f\colon M\to B$ is $3$-connected, the same $d_3$-differentials as above are nontrivial, and we have the following commutative diagram with exact rows.
\[
\begin{tikzcd}
\Z\arrow[r]	&H_4(M;\Z^w)\arrow[r]\arrow[d,"f_*"]	&\Z/2\arrow[r]\arrow[d,"="]	&0\\
\Z^w\otimes_{\Z\pi}H_4(\wt{B};\Z)\arrow[r]	&H_4(B;\Z^w)\arrow[r]	&\Z/2\arrow[r]	&0
\end{tikzcd}
\]
Since $[M]$ is a generator of $H_4(M;\Z^w)$ it follows that the image of $f_*[M]$ is nontrivial in $H_2(\pi;I\pi^w\oplus I\pi^{v'})/\im d_2\cong \Z/2$ as claimed.
\end{proof}

\begin{lemma}\label{lem:e-iso}
    The map $e\colon I\pi\to I\pi^{\dagger\dagger}$ sending $x \mapsto (f \mapsto f(x))$ is an isomorphism.
\end{lemma}
\begin{proof}
	The statement of the lemma is independent of $w$, but $w$ appears indirectly in the proof, in the involution used to consider $I\pi^\dagger$ as a left module. For simplicity, we assume that $w$ is trivial. 
	
    By \cite{KPT-long}*{Lemma~7.14}, the map $\theta \colon I\pi^{v'}  \xrightarrow{\cong} I\pi^\dagger$ sending
	  \begin{align*}
    1-t &\mapsto \Bigg\{\begin{array}{rcl}
	      1-t &\mapsto & 1-t  \\
	      1-T &\mapsto& 1-T
	    \end{array}\Bigg\} \text{ and }
	    1+T  \mapsto
	    \Bigg\{\begin{array}{rcl}
	      1-t &\mapsto & 1+T  \\
	      1-T &\mapsto& 0
	    \end{array}\Bigg\}
	  \end{align*}
    is an isomorphism. Similarly one sees that the map $\theta' \colon I\pi  \xrightarrow{\cong} (I\pi^{v'})^\dagger$ sending
   \begin{align*}
     1-t \mapsto \Bigg\{\begin{array}{rcl}
	      1-t &\mapsto & 1-t  \\
	      1+T &\mapsto& 1+T
	    \end{array}\Bigg\} \text{ and }
	    1-T \mapsto
	    \Bigg\{\begin{array}{rcl}
	      1-t &\mapsto & 1-T  \\
	      1+T &\mapsto& 0
	    \end{array}\Bigg\}
	  \end{align*}
   is an isomorphism. By a straightforward computation, \[\theta' = \theta^\dagger \circ e \colon I\pi \to  (I\pi^{v'})^\dagger.\]
   It follows that $e$ is an isomorphism as claimed.
\end{proof}

Now we are ready to show that the group $\Z\times \Z/2$ has \propH\ with respect to a certain map $\Z\times \Z/2\to C_2$.

\begin{lemma}
\label{lem:propH-ZxZ2}
Let $w\colon \Z\times \Z/2\to C_2$ be such that $w(T)=1$.
The pair $(\Z\times \Z/2,w)$ has {\propH}.
\end{lemma}

\begin{proof}
Let $\pi:=\Z\times \Z/2$. Since $w(T)=1$, the orientation character is determined by $w(t) = \pm 1 \in C_2$.
For both $w(t) = \pm 1$, we have that $H_4(\pi;\Z^w)\cong \Z/2$ and thus by ~\cite{kreck} (see also \cite{KPT}*{Theorem~1.2, Section~1.5}) up to connected sum with $\CP^2, \overline{\CP}^2$ and $\star \CP^2$ there are only two homeomorphism classes with fundamental group $\pi$ and orientation character $w$. By \cref{lem:H-stable}, it suffices to consider one manifold from both classes.

First we consider the case of $w$ trivial.
The nontrivial class in $H_4(\pi;\Z)\cong \Z/2$ is represented by $M = S^1 \times \RP^3$, and the trivial class by a $4$-manifold $N$ obtained as the boundary of a 5-dimensional thickening of the 2-skeleton of $B\pi$.

Consider the nontrivial class in $H_4(\pi;\Z)$ and $B:=P_2(M)$, where $M=S^1\times \RP^2$. Since $\pi_2(B)=0$ and so is certainly projective, by \cref{lem:sequence-for-pi2-projective}  there is an exact sequence
\[\Z\otimes_{\Z\pi}H_4(\wt B;\Z)\xrightarrow{\varphi_B} H_4(B;\Z)\to H_4(\pi;\Z)\to 0.\]
In particular, every class in the kernel of $H_4(B;\Z)\to H_4(\pi;\Z)$ belongs to the image of $\Z^w \otimes_{\Z\pi} H_4(\wt B;\Z)$ and thus $(B,w)$ has {\propH}.

Now we consider the trivial class in $H_4(\pi;\Z)$.
By \cite{KPT-long}*{Lemma~7.12}, the manifold $N$ has $\pi_1(N)\cong \pi$, the second homotopy group $\pi_2(N)\cong I\pi\oplus I\pi^\dagger$, and hyperbolic intersection form. Let $P:=P_2(N)$. We need to show that $(P,w)$ has {\propH}.

Let $f\colon N\to P$ be $3$-connected. By \cref{lem:extension-ZxZ2}, there is an exact sequence
\begin{equation}\label{eq:exact-seq-H4P-Z2}
\Z\otimes_{\Z\pi}H_4(\wt P;\Z)\xrightarrow{\varphi_P} H_4(P;\Z)\to \Z/2\to 0
\end{equation}
and the image of $f_*[N]$ in $\Z/2$ is nontrivial. Let $y\in H_4(P;\Z)$ be a class such that $\langle \alpha,\beta\cap y\rangle=0$ for all $\alpha,\beta\in H^2(P;\Z\pi)$, i.e.\ $y \in \ker \Theta_P$. It follows from the spectral sequence computation in the proof of \cref{lem:extension-ZxZ2} that the map $H_4(P;\Z) \to H_4(\pi;\Z)$ is the trivial map, so we must consider all $y \in \ker \Theta_P$.
By exactness of \eqref{eq:exact-seq-H4P-Z2}, we have to show that the image of $y$ in $\Z/2$ is trivial. Assume for a contradiction that $y$ maps nontrivially to $\Z/2$.  Then $f_*[N]-y$ maps to $0 \in \Z/2$ and hence lies in the image of $\varphi_P \colon \Z\otimes_{\Z\pi}H_4(\wt P;\Z) \to H_4(P;\Z)$.
We have the commutative diagram
\[
\begin{tikzcd}
\Z\otimes_{\Z\pi}H_4(P;\Z\pi)	\arrow[r,"\varphi_P"]\arrow[d, "\mB_{H_2(P;\Z\pi)} \circ \Upsilon"]	&H_4(P;\Z)\arrow[d,"\Theta_P"]\\
\Her(H_2(P;\Z\pi)^\dagger)\arrow[r,"\ev^*"]	&\Her(H^2(P;\Z\pi)).
\end{tikzcd}
\]		
We have that $\Theta_P(f_*[N]-y) = f_*\lambda_N$. Let $\lambda:=f_*\lambda_N$. Then $\Theta_P(f_*[N] - y)=\lambda$, and so we deduce that $\lambda$ lies in the image of $\Theta_P \circ \varphi_P$.
By the diagram, it follows that $\lambda$ lies in the image of $\ev^*\circ\,\, \B_{H_2(P;\Z\pi)} \circ \Upsilon$, or equivalently, since $\Upsilon$ is an isomorphism, that $\lambda$ lies in the image of $\ev^*\circ\,\, \B_{H_2(P;\Z\pi)}$.

Since $P = P_2(N)$, we have that $H^2(P;\Z\pi) \cong H^2(N;\Z\pi) \cong H_2(N;\Z\pi) \cong I\pi \oplus I\pi^\dagger$.
Under the decomposition $\Her(I\pi\oplus I\pi^\dagger)\cong \Her(I\pi)\oplus \Her(I\pi^\dagger)\oplus \Hom_{\Z\pi}(I\pi,I\pi^{\dagger\dagger})$, the element $\lambda$ maps to $(0,0,e)$, where $e \colon I\pi \to I\pi^{\dagger\dagger}$ is as in \cref{lem:e-iso}, since $\lambda$ is hyperbolic.

We will obtain a contradiction by showing that  $\lambda$ is not in the image of $\ev^*\circ\,\,\B_{H_2(B;\Z\pi)}$. The desired contradiction will prove that the image of $y$ in $\Z/2$ is trivial after all, and hence by exactness of \eqref{eq:exact-seq-H4P-Z2} that $y$ lies in the image of $\varphi_P$.

We examine the pre-image of $\lambda$ under $\ev^*$.
Note that $(I\pi\oplus I\pi^\dagger)^\dagger \cong I\pi^\dagger \oplus I\pi^{\dagger\dagger}$. The bottom left group of Hermitian forms $\Her(H_2(P;\Z\pi)^\dagger)$ has a similar  decomposition as in \cref{remark:splitting-of-Her-for-A-plus-A'}, with
\[\Her(I\pi^\dagger \oplus I\pi^{\dagger\dagger}) \cong \Her(I\pi^\dagger)\oplus \Her(I\pi^{\dagger\dagger})\oplus \Hom_{\Z\pi}(I\pi^{\dagger\dagger},I\pi^{\dagger\dagger}).\]
The map $\ev \colon H^2(P;\Z\pi) \to H_2(P;\Z\pi)^\dagger$ translates to a map $\ev\colon I\pi\oplus I\pi^\dagger \to  I\pi^\dagger \oplus I\pi^{\dagger\dagger}$, given by $(x,g)\mapsto (g,e(x))$.
The coordinates are switched because we used Poincar\'{e} duality in the identification $H^2(P;\Z\pi) \cong I\pi \oplus I\pi^\dagger$.   A straightforward check shows that $e$ lies in the image of $\ev^*$, and that
$\ev^*(\id_{I\pi^{\dagger\dagger}}) = e$.
That is, one has to check that, $(0,0,\id_{I\pi^{\dagger\dagger}})$ induces, under the map $\ev$, the form corresponding to $(0,0,e)$.
Moreover $\ev^*$ is injective by \cref{prop:ev-inj}, so $(0,0,\id_{I\pi^{\dagger\dagger}})$ represents the unique pre-image of $\lambda$ in $\Her(H_2(P;\Z\pi)^\dagger)$. So we have to show that $(0,0,\id_{I\pi^{\dagger\dagger}})$ is not in the image of $\mB_{H_2(P;\Z\pi)}$.

Now we consider the map $\mB_{H_2(P;\Z\pi)}$.  The compatibility of the decomposition in~\cref{lemma:baues-splitting-of-Gamma-A-plusA'} and the decomposition from \cref{remark:splitting-of-Her-for-A-plus-A'} with respect to $\mB_{H_2(P;\Z\pi)}$, as in \cref{lem:splitting-of-B}, means it suffices to show that $\id_{I\pi^{\dagger\dagger}}$ does not lie in the image of the map
\[I\pi\otimes_{\Z\pi}I\pi^\dagger \to \Hom_{\Z\pi}(I\pi^{\dagger\dagger},I\pi^{\dagger\dagger})\]
from \cref{lem:splitting-of-B}.

The map $e\colon I\pi\to I\pi^{\dagger\dagger}$ is an isomorphism by  \cref{lem:e-iso}. Using this we have an isomorphism $\Hom_{\Z\pi}(I\pi^{\dagger\dagger},I\pi^{\dagger\dagger}) \cong \Hom_{\Z\pi}(I\pi,I\pi)$ sending $\id_{I\pi^{\dagger\dagger}}$ to $\id_{I\pi}$,  and a commutative diagram
\[\begin{tikzcd}
I\pi\otimes_{\Z\pi}I\pi^\dagger\ar[r,"j \otimes \id"] \ar[d]&\Z\pi\otimes_{\Z\pi}I\pi^\dagger\ar[d,"\cong"]\ar[d,"\cong"]\\
\Hom_{\Z\pi}(I\pi^{\dagger\dagger},I\pi^{\dagger\dagger})\ar[r,"j^{\dagger\dagger}_*"] \ar[d,"\cong"] & \Hom_{\Z\pi}(I\pi^{\dagger\dagger},\Z\pi^{\dagger\dagger}) \ar[d,"\cong"] \\
\Hom_{\Z\pi}(I\pi,I\pi)\ar[r,"j_*"] & \Hom_{\Z\pi}(I\pi,\Z\pi).
\end{tikzcd}\]
The top vertical maps are those induced by $\B_{I\pi\oplus I\pi^\dagger}$ and $\B_{\Z\pi\oplus I\pi^\dagger}$, described in \cref{lem:splitting-of-B}. The lower vertical maps are induced by $e$. The horizontal maps are induced by the inclusion $j \colon I\pi \to \Z\pi$.  We want to see that $\Id_{I\pi}$ is not in the image of the left vertical composition.
Under the lower horizontal map, $\id_{I\pi}$ maps to the inclusion $j\colon I\pi\to \Z\pi$, which lifts to $1 \otimes j$ in the top right $\Z\pi \otimes_{\Z\pi} I\pi$.
Thus if we show that $1\otimes j$ is not in the image of the top horizontal map $I\pi\otimes_{\Z\pi}I\pi^\dagger \to \Z\pi\otimes_{\Z\pi}I\pi^\dagger$, it follows from the diagram that $\id_{I\pi}$ is not in the image of the left vertical composition, and hence that $\id_{I\pi^{\dagger\dagger}}$ is not in the image of the top left vertical map, as desired.

Recall that $v'\colon \pi \to C_2$ is given by $v'(t)=1$ and $v'(T)=-1$.
Using the group presentation $\pi\cong \langle t,T\mid [t,T],T^2\rangle$, there is an isomorphism $\theta^{-1} \colon I\pi^\dagger \to I\pi^{v'}$ sending $j$ to  $(1-t)$; see the proof of \cref{lem:e-iso}.
The composition
\[I\pi \otimes_{\Z\pi} I\pi^\dagger \xrightarrow{j \otimes \Id} \Z\pi \otimes_{\Z\pi} I\pi^\dagger \xrightarrow{\varepsilon \otimes \theta^{-1}} \Z \otimes_{\Z\pi} I\pi^{v'}\]
is the zero homomorphism, where $\varepsilon \colon \Z\pi \to \Z$ is the augmentation, and $(\varepsilon \otimes \theta^{-1})(1 \otimes j) = 1 \otimes (1-t)$. Thus if $1 \otimes j$ is in the image of $j \otimes \Id$, it follows that $1 \otimes (1-t) = 0 \in \Z \otimes_{\Z\pi} I\pi^{v'}$.  It therefore  suffices to show that $1\otimes(1-t)$ is nontrivial. For this we consider the exact sequence
\[0\to \Tor_1^{\Z\pi}(\Z,\Z^{v'})\to \Z\otimes_{\Z\pi} I\pi^{v'}\to \Z\otimes_{\Z\pi}\Z\pi\to \Z\otimes_{\Z\pi}\Z^{v'}\to 0.\]
This can be identified with the exact sequence
\[0\to \Z/2\to \Z\otimes_{\Z\pi} I\pi^{v'}\to \Z\to \Z/2\to 0\]
where $1\otimes(1-t)$ maps to $0 \in \Z$ and $1\otimes(1+T)$ maps to $2 \in \Z$. Since $1\otimes(1-t)$ and $1\otimes(1+T)$ generate $\Z\otimes_{\Z\pi}I\pi^{v'}$, it follows that $1\in \Z/2$ maps to $1\otimes (1-t)$. In particular, the latter is nontrivial as needed.

Recapping, this means that $1 \otimes j$ is not in the image of $j \otimes \Id$, which implies that $\Id_{I\pi^{\dagger\dagger}}$ is not in the image of the restriction of $\mB_{H_2(P;\Z\pi)}$ to the map $I\pi \otimes_{\Z\pi} I\pi^{\dagger} \to \Hom_{\Z\pi}(I\pi^{\dagger\dagger},I\pi^{\dagger\dagger})$ that is the top left vertical map of the previous diagram. This in turn implies that $\lambda = f_*\lambda_N$ is not in the image of $\ev^* \circ \mB_{H_2(P;\Z\pi)}$, which as explained above leads to the desired contradiction. So $y$ maps to $0 \in \Z/2$, and therefore $y$ lies in the image of $\varphi_B$, as required.  Thus $(\pi,w)$ has \propH\ when $w$ is trivial.

The case where $w(t)=-1$ is similar. Here one has to consider the mapping torus $M'$ of the orientation reversing involution on $\RP^3$ in place of $S^1 \times \RP^3$, representing the nontrivial element of $H_4(\pi;\Z^w) \cong \Z/2$. Since $\pi_2(M')=0$, letting $B' := P_2(M')$, we again have an exact sequence
\[\Z\otimes_{\Z\pi}H_4(\wt B';\Z^w)\xrightarrow{\varphi_{B'}} H_4(B';\Z^w)\to H_4(\pi;\Z^w)\to 0\]
by \cref{lem:sequence-for-pi2-projective}.  Again, this implies that every element in the kernel of $H_4(B';\Z^w)\to H_4(\pi;\Z^w)$ lies in the image of $\varphi_{B'}$.

For $0 \in H_4(\pi;\Z^w)$ there is a  $4$-manifold $N'$ with $\pi_2(N')\cong I\pi\oplus I\pi^{wv'}\cong I\pi\oplus I\pi^\dagger$ and hyperbolic intersection form. Here we use that the involution of $\Z\pi$ in this case is given by $g\mapsto w(g)g^{-1}$ so that as a left module $I\pi^{\dagger}$ is isomorphic to $I\pi^{wv'}$.
Let $P' := P_2(N')$. We just observed that $H_2(P';\Z\pi)$ is isomorphic to  $H_2(P;\Z\pi)$. We also have the exact sequence from \cref{lem:extension-ZxZ2} for $w(t) =-1$. The analysis above applies unchanged, to show that for $y \in \ker \Theta_{P'}$ we have that $y \mapsto 0 \in \Z/2$, and hence $y \in \im \varphi_{P'}$.  This shows that $(\pi,w)$ has \propH~for $w(T)=1$ and $w(t)=-1$.
\end{proof}

\subsection{{\propH} for \texorpdfstring{$3$}{3}-manifold groups that are nontrivial free products}\label{section:prop-H-for-3-mfld-groups-free-products}

In this section we establish \propH\ for all admissible $(\pi,w)$,
deducing it for $3$-manifold groups that are nontrivial free products of factors for which we already proved {\propH} in \cref{sec:proph-groups-general,sec:proph-pd3,sec:proph-ZZ2}.

We work initially with an arbitrary finitely presented group $\pi$, a subgroup $G \leq \pi$, and a character $w \colon \pi \to C_2$.  Only at the end of the section, in \cref{cor:propH-3mfd}, do we restrict to admissible $3$-manifold groups and orientation characters. For $G\leq \pi$, we also prove results for an arbitrary left $\Z G$-module $A$. Recall that $\Sigma_2$ denotes the symmetric group on two elements.

\begin{definition}
\label{def:psi-and-chi}
    Let $G\leq \pi$ be a subgroup. Let $w\colon \pi\to C_2$ be a homomorphism and define the composition $\ol{w}\colon G\hookrightarrow \pi \xrightarrow{w} C_2$. For every $\Z G$-module $A$ we define
    \begin{align*}
        \psi\colon \Ind_G^\pi(A\otimes_\Z A)&\to (\Ind_G^\pi A)\otimes_{\Z\pi} (\Ind_G^\pi A)\\
        \gamma\otimes (a\otimes a')&\mapsto (\gamma\otimes a)\otimes (\gamma\otimes a')
    \end{align*}
    and
    \begin{align*}
        \xi\colon \Sesq^{\overline{w}}(A)&\to \Sesq^w(\Ind_G^\pi A)\\
        \lambda &\mapsto \big((\gamma\otimes a,\gamma'\otimes a')\mapsto \gamma\lambda(a,a')\ol{\gamma'}\big).
    \end{align*}
    In the case that $A$ is free as an abelian group, recall that $\Gamma(A)$ is isomorphic to the group of fixed points of $A\otimes_{\Z} A$ under the $\Sigma_2$-action permuting the two copies of $A$. Therefore, taking fixed points, we obtain
    \begin{align*}
        \psi^{\Sigma_2}\colon \Ind_G^\pi \Gamma(A) &\to \Gamma(\Ind_G^\pi A) \text{ and}\\
    \xi^{\Sigma_2}\colon \Her^{\overline{w}}(A) &\to \Her^w(\Ind_G^\pi A).
    \end{align*}
\end{definition}

Recall that for $G\leq \pi$, and a $\Z G$-module $A$, we denote the module $\Hom_{\Z G}(A,\Z G)$ by $A^\star$.

\begin{lemma}\label{lem:ind-sesq-square}
    Let $G\leq \pi$ be a subgroup, let $A$ be a $\Z G$-module, and let $w\colon \pi\to C_2$ be a homomorphism. Define the composition $\ol{w}\colon G\hookrightarrow \pi \xrightarrow{w} C_2$. Then we have the commutative diagram
    \begin{equation}\label{eq:ind-sesq-square}
    \begin{tikzcd}[row sep =scriptsize, column sep =scriptsize]
        \Ind_G^\pi(A\otimes_\Z A)\ar[r,"\psi"]\ar[d,"\wt{\mB}_A"]   &\Ind_G^\pi A \otimes_{\Z\pi} \Ind_G^\pi A 
        \ar[d,"\wt\mB_{\Ind_G^\pi A}"] 
        \\
        \Sesq^{\ol{w}}(A^\star)\ar[r, "\xi"] &\Sesq^w((\Ind_G^\pi A)^\dagger),
    \end{tikzcd}
    \end{equation}
    where
    \begin{enumerate}
        \item $\wt\mB_A(\gamma\otimes(a\otimes a'))=\big( (f,f')\mapsto w(\gamma)\overline{f(a)}f(a')\big)$, and
        \item $\wt\mB_{\Ind_G^\pi A}((\gamma\otimes a)\otimes(\gamma'\otimes a'))=\big((h,h')\mapsto \overline{h(\gamma\otimes a)}h'(\gamma'\otimes a')\big)$,
    \end{enumerate}
    and the maps $\psi$ and $\xi$ are from \cref{def:psi-and-chi}.

    Here we used the isomorphism $\Sesq^w((\Ind_G^\pi A)^\dagger)\cong \Sesq^w(\Ind_G^\pi(A^\star))$ induced by the isomorphism from \cref{lem:ind-star-dagger-isom}, in order to use the map $\xi$.
\end{lemma}

\begin{proof}
    We have \[\wt\mB_{\Ind_G^\pi(A)}(\psi(\gamma\otimes(a\otimes a')))=\big((h,h')\mapsto \overline{h(\gamma\otimes a)}h'(\gamma\otimes a')\big).\]
    Applying the isomorphism from \cref{lem:ind-star-dagger-isom}, we obtain the element of $\Sesq^w(\Ind_G^\pi(A^\star))$ given by
    \[(\delta\otimes f,\delta'\otimes f')\mapsto (\overline{\gamma  f(a)\overline{\delta}})(\gamma f'(a')\overline{\delta'})=w(\gamma)\delta \overline{f(a)}f'(a')\overline{\delta'}.\]
    On the other hand, we also see
    \[
    \xi(\wt\mB_A(\gamma\otimes(a\otimes a')))=\xi\big((h,h')\mapsto w(\gamma)\ol{h(a)}h'(a')\big)=\big((\delta\otimes f,\delta'\otimes f')\mapsto  w(\gamma)\delta\ol{f(a)}f'(a')\ol{\delta'}\big)
    \]
    as well.
    Hence the diagram commutes as claimed.
\end{proof}

\begin{lemma}
\label{lem:baby-bees}
    Let $G\leq \pi$ be a subgroup, let $A$ be a $\Z G$-module, and let $w\colon \pi\to C_2$ be a homomorphism. Define the composition $\ol{w}\colon G\hookrightarrow \pi \xrightarrow{w} C_2$. Consider the diagram \eqref{eq:ind-sesq-square}. Let $x\in \Ind_G^\pi(A)\otimes_\Z\Ind_G^\pi(A)$ such that $\wt{\mB}_{\Ind_G^\pi(A)}(x)$ is in the image of $\xi$. Then there exists $b\in \Ind_G^\pi(A\otimes_\Z A)$ such that $x-\psi(b)$ is in the kernel of $\wt{\mB}_{\Ind_G^\pi(A)}$.

    Furthermore, if $x$ is a $\Sigma_2$-fixed point, then $b$ can be chosen to be a $\Sigma_2$-fixed point.
\end{lemma}

\begin{proof}
    Let $x=\sum_{i=1}^n(g_i\otimes a_i)\otimes (g_i'\otimes a_i')$. Using the isomorphism \[\Sesq^w((\Ind_G^\pi A)^\dagger)\cong \Sesq^w(\Ind_G^\pi(A^\star))\] coming from \cref{lem:ind-star-dagger-isom}, we have that
    \[\wt{\mB}_{\Ind_G^\pi}(A)(x)=((\gamma\otimes f,\gamma'\otimes f')\mapsto\sum_{i=1}^n \gamma \overline{f(m_i)g_i}g_i'f'(m'_i)\overline{\gamma'}).\]
    Since $\wt{\mB}_{\Ind_G^\pi(A)}(x)$ is in the image of $\xi$, there exists a $\lambda\in \Sesq^w(A^\star)$ such that for all $f,f'\in A^\star$ we have
    \[\lambda(f,f')=\sum_{i=1}^n\overline{f(m_i)g_i}g_i'f'(m'_i).\]
    Up to reordering, we can assume that there exists an $m$ such that $g_i^{-1}g_i'\in G$ for $i\leq m$ and $g_i^{-1}g_i'\notin G $ for $i> m$.
    Note that $\overline{f(m_i)g_i}g_i'f'(m'_i)\in \Z[G]$ if and only if $g_i^{-1}g_i'\in G$. Since $\lambda(f,f')\in \Z[G]$, it follows that
    \[\lambda(f,f')=\sum_{i=1}^m\overline{f(m_i)g_i}g_i'f'(m'_i),\]
    and that $\sum_{i=m+1}^n(g_i\otimes a_i)\otimes (g_i'\otimes a_i')$ is in the kernel of $\wt{\mB}_{\Ind_G^\pi(A)}$.

    By assumption, for each $i\leq m$ there exists $h_i\in G$ such that $g_i'=g_ih^{-1}$. Then $(g_i\otimes m_i)\otimes(g_i'\otimes m_i')=(g_i\otimes a_i)\otimes(g_i\otimes ha_i')$ and hence
    \[\sum_{i=1}^m(g_i\otimes a_i)\otimes (g_i'\otimes a_i')=\psi(\sum_i g_i\otimes(a_i\otimes ha_i')).\]
    Thus setting $b:=\sum_i g_i\otimes(a_i\otimes ha_i')$ we have that $x-\psi(b)$ is in the kernel of $\wt{\mB}_{\Ind_G^\pi(A)}$ as required.

    Since $g_i^{-1}g_i'\in G$ if and only if $(g_i')^{-1}g_i\in G$, the element $b$ defined above is a $\Sigma_2$-fixed point if $x$ is a $\Sigma_2$-fixed point.
\end{proof}

Now we consider the case that $\pi\cong *_{i=1}^n G_i$, where the groups $G_i$, for $i=1,\dots,n$ are finitely presented, and for each $i$ we consider left $\Z G_i$-modules $A_i$.
Let $A:=\bigoplus\limits_i\Ind_{G_i}^\pi(A_i)$ and let $s_i \colon \Ind_{G_i}^\pi(A_i) \to A$ denote the canonical inclusion map for each $i$.
Let $w\colon \pi\to C_2$ be a homomorphism, and for each $i$, consider the composition $\ol{w}_i\colon G_i\hookrightarrow \pi\xrightarrow{w}C_2$.

\begin{lemma}
\label{lem:the-square}
     Let $x\in \Z^w\otimes_{\Z\pi} \Gamma(A)$. The diagram
    \[
    \begin{tikzcd}        \bigoplus\limits_{i}\Z^w\otimes_{\Z\pi}\Ind_{G_i}^\pi \Gamma(A_i) \ar[r,"\oplus_i\psi_i^{\Sigma_2}"]\ar[d,"\oplus_i\mB_{A_i}", start anchor={[yshift=2mm]}] &\bigoplus\limits_{i}\Z^w\otimes_{\Z\pi}\Gamma(\Ind_{G_i}^\pi A_i )\ar[d,"\oplus_i\mB_{\Ind_{G_i}^\pi A}"]\ar[r,"s
"]&\Z^w\otimes_{\Z\pi} \Gamma(A)\ar[d,"\mB_{A}"]\\
        \bigoplus\limits_i\Her^{\ol{w}_i}(A_i^\star)\ar[r,"\oplus_i\xi_{i}^{\Sigma_2}"]  &\bigoplus\limits_i\Her^w((\Ind_{G_i}^\pi A_i)^\dagger)\ar[r,"s^\dagger
"]&\Her^w(A^\dagger)
    \end{tikzcd}
    \]
   commutes, where $s$ and $s^\dagger$ are the maps induced by $(\Gamma(s_i))_i$ and $(s_i^\dagger)_i$ respectively, and each $\psi_i$ and $\xi_i$ comes from \cref{def:psi-and-chi}.
   We denote the top composition by $\Psi^{\Sigma_2}:= s \circ (\oplus_i\psi_i^{\Sigma_2})$ and the bottom composition by $\Xi^{\Sigma_2} := s^\dagger \circ (\oplus_i\xi_{i}^{\Sigma_2})$.
     If $\mB_{A}(x)$ lies in the image of $\Xi^{\Sigma_2}$, then there exists $b\in\bigoplus\limits_{i} \Z^w\otimes_{\Z\pi}\Ind_{G_i}^\pi \Gamma(A_i)$ such that $x-\Psi^{\Sigma_2}(b) \in \ker \mB_{A}$.
\end{lemma}

\begin{proof}
For every $\Z\pi$-module $U$, the map $\mB_U$ is obtained from $\wt\mB_U$ by taking $\Sigma_2$-fixed points and noticing that the result factors through $\Z^w\otimes_{\Z\pi}\Gamma(U)$. The left hand square thus commutes by \cref{lem:ind-sesq-square}. The right hand square commutes by naturality of $\mB$.

Now assume that $\mB_A(x)$ lies in the image of $\Xi^{\Sigma_2}$. Let $y \in \bigoplus\limits_i\Her^{\ol{w}_i}(A_i^\star)$ be a preimage. Then $s^\dagger(\oplus_i\xi_{i}^{\Sigma_2}(y)) = \mB_A(x)$. By \cref{lem:splitting-of-B}, there exists  $x'\in \bigoplus\limits_{i}\Z^w\otimes_{\Z\pi}\Gamma(\Ind_{G_i}^\pi A_i)$ with
\[\oplus_i\mB_{\Ind_{G_i}^\pi A}(x') =\oplus_i\xi_{i}^{\Sigma_2}(y)\]
such that
$x-s(x') \in \ker \mB_A \subseteq \Z^w\otimes_{\Z\pi} \Gamma(A)$. By \cref{lem:baby-bees}, there exists $b\in \bigoplus\limits_{i}\Z^w\otimes_{\Z\pi}\Ind_{G_i}^\pi(\Gamma(A_i))$ such that $x'-\bigoplus_i\psi_i^{\Sigma_2}(b) \in \ker \big(\oplus_i\mB_{\Ind_{G_i}^\pi A}\big)$.
Then
\begin{align*}
    \mB_{A}(x-\Psi^{\Sigma_2}(b))&=\mB_{A}(x-s(x')+s(x'-\oplus_i\psi_i^{\Sigma_2}(b))) \\
    &= 0 + s^\dagger(\oplus_i\mB_{\Ind_{G_i}^\pi(A)}(x'-\oplus_i\psi_i^{\Sigma_2}(b)))\\
    &=0.
\end{align*}
So $x-\Psi^{\Sigma_2}(b)$ lies in the kernel of $\mB_{A}$ as needed.
\end{proof}

\begin{lemma}
\label{lem:top-right-square}
    For $i=1,\dots, n$, let $B_i$ be a connected, $3$-coconnected CW complex. Let $B:=P_2(\bigvee_i B_i)$. Let $f^i\colon B_i\to B$ be the composition $B_i\hookrightarrow \bigvee_i B_i\to B$.
    Let $\pi:=\pi_1(B)$ and $G_i:=\pi_1(B_i)$ for each $i$.
    Then the square
    \[\begin{tikzcd}
    \Ind_{G_i}^\pi \Gamma(H_2(B_i;\Z G_i))\ar[r,"\psi_i^{\Sigma_2}"]\ar[d,"\cong"] & \Gamma(\Ind_{G_i}^\pi H_2(B_i,\Z G_i))\ar[r,"\Gamma(f^i_*)"]&\Gamma(H_2(B;\Z\pi))\ar[d,"\cong"]\\
        \Ind_{G_i}^\pi H_4(B_i;\Z G_i)\ar[rr,"f^i_*"] &&H_4(B;\Z\pi)
    \end{tikzcd}\]
    commutes, where $\psi_i^{\Sigma_2}$ is as in \cref{def:psi-and-chi}.
\end{lemma}

\begin{proof}
    By naturality of Whitehead's exact sequence~\cite{whitehead-certainsequence}, the square
    \[\begin{tikzcd}
    \Gamma(H_2(\wt B_i;\Z))\ar[r,"{\Gamma(\wt f^i_*)}"]&\Gamma(H_2(\wt B;\Z))\\
        H_4(\wt B_i;\Z)\ar[r,"\wt f^i_*"]\ar[u,"\cong"']&H_4(\wt B;\Z)\ar[u,"\cong"']
    \end{tikzcd}\]
    commutes.
    We use the isomorphisms $H_*(\wt B_i;\Z)\cong H_*(B_i;\Z G_i)$ and $H_*(\wt B;\Z)\cong H_*(B;\Z\pi)$. Under these, $\wt f^i_*\colon H_2(\wt B_i;\Z)\to H_2(\wt B;\Z)$ agrees with the composition
    \[H_2(B_i;\Z G_i)\xrightarrow{u} \Ind_{G_i}^\pi H_2(B_i;\Z G_i)\cong H_2(B_i;\Z\pi)\xrightarrow{f^i_*}H_2(B;\Z\pi),\]
    where we omit $\Res_{G_i}^\pi$ from the notation for readability.
    Also inverting the vertical isomorphisms, we obtain the commutative diagram
    \[\begin{tikzcd}
    \Gamma(H_2(B_i;\Z G_i)))\ar[r,"\Gamma(u)"]\ar[d,"\cong"] & \Gamma(\Ind_{G_i}^\pi H_2(B_i,\Z G_i)) \ar[r,"\Gamma(f^i_*)"]&\Gamma(H_2(B;\Z\pi))\ar[d,"\cong"]\\
        H_4(B_i;\Z G_i)\ar[rr,"f^i_*\circ u"] &&H_4(B;\Z\pi)
    \end{tikzcd}\]
    By naturality, the diagram is $G_i$-equivariant, and thus induces the diagram from the statement by applying $\Ind_{G_i}^\pi$, and also noticing that by doing so $\Gamma(u)$ agrees with $\psi_i^{\Sigma_2}$.
\end{proof}

\begin{lemma}
\label{lem:bottom-right-square}
    For $i=1,\dots, n$, let $B_i$ be a connected, $3$-coconnected CW complex. Let $B:=P_2(\bigvee_i B_i)$. Let $f^i\colon B_i\to B$ be the composition $B_i\hookrightarrow \bigvee_i B_i\to B$.
    Let $\pi:=\pi_1(B)$ and $G_i:=\pi_1(B_i)$ for each $i$. For each $i$, let $\ol{w}_i$ denote the composition $G_i \hookrightarrow \pi\xrightarrow{w} C_2$.    Then the diagram
    \[\begin{tikzcd}
        H_4(B_i;\Z^{\overline{w}_i})\ar[rr,"f^i_*"]\ar[d,"\Theta_{B_i}"]&&H_4(B;\Z^w)\ar[d,"\Theta_B"]\\
        \Her^{\overline{w}_i}(H^2(B_i;\Z G_i))\ar[r,"\xi_{i}^{\Sigma_2}"]&\Her^w(\Ind_{G_i}^\pi H^2(B_i;\Z G_i))\ar[r,"f^i_*"]&\Her^w(H^2(B;\Z\pi))
    \end{tikzcd}\]
    commutes, where $\xi_i^{\Sigma_2}$ is as in \cref{def:psi-and-chi}.
\end{lemma}

\begin{proof}
We can extend the diagram as follows.
    \[\begin{tikzcd}
        H_4(B_i;\Z^{\overline{w}_i})\ar[rr,"f^i_*"]\ar[dr,"{x\mapsto\langle-,-\cap x\rangle}"]\ar[d,"\Theta_{B_i}"]&&H_4(B;\Z^w)\ar[d,"\Theta_B"]\\
        \Her^{\overline{w}_i}(H^2(B_i;\Z G_i))\ar[r,"\xi_i^{\Sigma_2}"]&\Her^w(\Ind_{G_i}^\pi H^2(B_i;\Z G_i))\ar[r,"f^i_*"]&\Her^w(H^2(B;\Z\pi))
    \end{tikzcd}\]
    The right hand square then commutes by naturality of the Kronecker product since $\Theta_{B}(f^i_*(x))=\langle - , - \cap f^i_*(x)\rangle$.
    To see that the left hand triangle commutes, note that for all $x\in H_4(B_i;\Z^{\overline{w}_i})$, $\alpha,\beta\in H^2(B_i;\Z G_i)$ and $g,g'\in \pi$ we have \[\langle g\otimes \alpha,(g'\otimes \beta)\cap x\rangle=g\langle \alpha,\beta\cap x\rangle \overline{g'}=\xi_i^{\Sigma_2}(\Theta_{B_i}(x))(g\otimes \alpha,g'\otimes \beta).\qedhere\]
\end{proof}

\begin{proposition}
    \label{cor:propH-3mfd}
    Let $\pi$ be a $3$-manifold group and let $w\colon \pi\to C_2$ be a homomorphism such that $(\pi,w)$ is admissible.
    Then $(\pi,w)$ has {\propH}.
\end{proposition}

\begin{proof}
    Let $M$ be a closed $4$-manifold with fundamental group $\pi$ and orientation character $w$. We have to show that $(P_2(M),w)$ has {\propH}. By \cref{lem:H-stable}, it suffices to show this for some stabilisation of $M$. By \cref{lemma:stable-splitting}, we can thus assume that $M= \#_i M_i$, where each $M_i$ is a closed $4$-manifold with fundamental group $\pi_1(M_i)=:G_i$ either cyclic, isomorphic to $\Z\times \Z/2$, or a $PD_3$-group. Let $B_i:=P_2(M_i)$ and let $B:=P_2(M)\simeq P_2(\bigvee_iB_i)$.
    For each $i$, denote by $\ol{w}_i$ the composition $G_i\hookrightarrow \pi\xrightarrow{w} C_2$.

    Consider the following diagram with exact rows. The top left and the second-from-bottom left square are given in \cref{lem:top-right-square,lem:bottom-right-square}, respectively. The maps $\varphi_{B_i}$, $\varphi_B$, $\Theta_{B_i}$, and~$\Theta_B$ are as in \eqref{eq:main-square}.
        The map $\wh\Xi^{\Sigma_2}$ is the map induced by $(f^i_* \circ \xi_i^{\Sigma_2})_i$, where the terms come from \cref{lem:bottom-right-square}.
    The second and third rows are part of the long exact sequences for the pair $(B,\bigvee_i B_i)$. The diagram commutes by naturality of $\varphi$, $\xi^{\Sigma_2}$ and~$\Theta$, together with \cref{lem:top-right-square,lem:bottom-right-square}.
    \[
    \begin{tikzcd}[column sep=scriptsize]
        \Z^w\otimes_{\Z\pi}\bigoplus\limits_{i}\Ind_{G_i}^\pi \Gamma(H_2(B_i;\Z G_i)) \ar[r,"\Psi^{\Sigma_2}"]\ar[d,"\cong", start anchor={[yshift=2mm]},"\oplus_i \Upsilon_i^{-1}"'] &\Z^w\otimes_{\Z\pi} \Gamma(H_2(B;\Z\pi))\ar[r]\ar[d,"\cong","\Upsilon^{-1}"'] &\coker(\Psi^{\Sigma_2})\ar[d]\\
        \Z^w\otimes_{\Z\pi}H_4(\bigvee_i B_i;\Z\pi)\ar[r,"\iota_*"]\ar[d,"\vee_i \varphi_{B_i}"] &\Z^w\otimes_{\Z\pi} H_4(B;\Z\pi)\ar[r, twoheadrightarrow]\ar[d, "\varphi_B"] &\Z^w\otimes_{\Z\pi}H_4(B,\bigvee_i B_i;\Z\pi)\ar[d,"\cong"]\\
        H_4(\bigvee_i B_i;\Z^w)\ar[r,"\iota_*"]\ar[d,"\vee_i \Theta_{B_i}"] &H_4(B;\Z^w)\ar[r,]\ar[d, "\Theta_B"] &H_4(B,\bigvee_i B_i;\Z^w)\ar[d]\\
        \bigoplus\limits_i\Her^{\ol{w}_i}(H^2(B_i;\Z G_i))\ar[r,"\wh\Xi^{\Sigma_2}"]  &\Her^w(H^2(B;\Z\pi))\ar[r] &\coker(\wh\Xi^{\Sigma_2})\\        \bigoplus\limits_i\Her^{\ol{w}_i}(H_2(B_i;\Z G_i)^\star)\ar[r,"\Xi^{\Sigma_2}"]\ar[u, end anchor={[yshift=2mm]}, "\oplus \ev^*","\cong"']  &\Her^w(H_2(B;\Z\pi)^\dagger)\ar[r]\ar[u, "\ev^*", "\cong"']   &\coker(\Xi^{\Sigma_2})\ar[u, "\cong"']
    \end{tikzcd}
    \]
    The vertical map second from the top on the right is an isomorphism by the relative Leray--Serre spectral sequence \cite{switzer}*{Theorem~15.27, Remark~2, p.~351} since $H_j(B,\bigvee_i B_i;\Z\pi)=0$ for $j<4$. The maps denoted by $\ev^*$ are isomorphisms by \cref{prop:ev-inj} since $\pi$ is a $3$-manifold group. Note that the composition of the five maps in the middle column is the map $\mB_{H_2(B;\Z\pi)}$.

    Note that \[H_3\big(\bigvee_i B_i;\Z\pi\big)\cong \bigoplus_iH_3(B_i;\Z\pi)\cong \bigoplus_i\Ind_{G_i}^\pi H_3(B_i;\Z G_i),\]
    and $0=\pi_3(B_i)\to H_3(B_i;\Z G_i)$ is surjective for each $i$ by the Hurewicz theorem. Therefore $H_3\big(\bigvee_i B_i;\Z\pi\big)=0$ and the map $\Z^w\otimes_{\Z\pi}H_4(B;\Z\pi)\to\Z^w\otimes_{\Z\pi}H_4(B,\bigvee_i B_i;\Z G_i)$ on the second row is surjective.

    Let $x\in \ker(H_4(B;\Z^w)\to H_4(\pi;\Z^w))$ be such that $\langle \alpha,\beta\cap x\rangle=0$ for all $\alpha,\beta\in H^2(B;\Z\pi)$. We will show that $x$ admits a lift to $\Z^w\otimes_{\Z\pi} H_4(B;\Z\pi)$. This will complete the proof that $(B,w)$ has {\propH}.

    Let $\overline{x}$ be the image of $x$ in $H_4(B,\bigvee_i B_i;\Z^w)$
    and let $\ol{z}$ be the preimage of $\ol{x}$ in $\Z^w\otimes_{\Z\pi}H_4(B,\bigvee_i B_i;\Z\pi)$ under the isomorphism $\Z^w\otimes_{\Z\pi}H_4(B,\bigvee_i B_i;\Z\pi)\xrightarrow{\cong} H_4(B,\bigvee_i B_i;\Z^w)$. Let $z\in \Z^w\otimes_{\Z\pi}H_4(B;\Z\pi)$ be a preimage of $\ol{z}$ and let $y\in \Z^w\otimes_{\Z\pi}\Gamma(H_2(B;\Z\pi))$ be a preimage under the top middle isomorphism.  Let $u$ denote the image of $y$ in $\Her^w(H^2(B;\Z\pi))$. By the definition of the map $\Theta_B$, the element $x$ maps trivially to $\Her^w(H^2(B;\Z\pi))$. Therefore, $z$ maps to zero in $\coker(\wh\Xi^{\Sigma_2})$, which implies that $u$ lies in the image of the map $\wh\Xi^{\Sigma_2}$. Consider the image $v$ of $y$ in $\Her^w(H_2(B;\Z\pi)^\dagger)$ under the map $\mB_{H_2(B;\Z\pi)}$. Using the vertical isomorphisms between the two bottom rows of the diagram, we see that $v$ lies in the image of $\Xi^{\Sigma_2}$, since $u$ lies in the image of $\wh\Xi^{\Sigma_2}$.

    By \cref{lem:the-square}, there exists $\wt{b}\in \Z^w\otimes_{\Z\pi} \bigoplus_i\Gamma(H_2(B;\Z\pi))$ such that $y-\Psi^{\Sigma_2}(\wt{b})\in \Z^w\otimes_{\Z\pi} \Gamma(H_2(B;\Z\pi))$ maps trivially to $\Her^w(H_2(B;\Z\pi)^\dagger)$.
    In other words $y-\Psi^{\Sigma_2}(\wt{b})$ is in the kernel of $\mB_{H_2(B;\Z\pi)}$. We know from \cref{cor:kernelB-equals-kernelphi} that in our setting the kernel of $\mB_{H_2(B;\Z\pi)}$ equals the kernel of $\varphi_B$. Therefore, the element $y-\Psi^{\Sigma_2}(\wt{b})$ also maps trivially to $H_4(B;\Z^w)$. Mapping further to $H_4(B,\bigvee_i B_i;\Z^w)$, it follows that $\ol{z}=0$. This implies that $\ol{x}$ is zero, so $x$ lies in the image of $H_4(\bigvee_iB_i;\Z^w)$. Let $\wt{x}\in H_4(\bigvee_iB_i;\Z^w)$ be a preimage of $x$.

    We know that each individual $(B_i, w\vert_{\pi_1(B_i})$ has {\propH} by \cref{lem:propH-finite,lem:Z-proph,lem:PD3H,lem:propH-ZxZ2}. Therefore by \cref{lem:H-prod}, $(\bigvee_i B_i,w)$ has {\propH} as well. We will apply this momentarily. First we check that $\wt{x}$ satisfies the desired conditions. We note that $\wt{x}$ lies in $\ker(H_4(\bigvee_i B;\Z^w)\to H_4(\pi;\Z^w))$, since the map $\bigvee_i B\to B\pi$ factors through $B$ and since $x$ maps to $0$ in $H_4(\pi;\Z^w)$. Moreover, for any $\alpha, \beta\in H^2(\bigvee_i B_i;\Z\pi)$, we have
    \[
    \langle \alpha,\beta\cap \wt{x}\rangle
    =\langle \alpha,\beta\cap \iota_*(x)\rangle
    = \langle \alpha,\iota_*(\iota^*(\beta)\cap x)\rangle
    = \langle \iota^*(\alpha),\iota^*(\beta)\cap x\rangle
    =0,
    \]
    where the final equality follows by our assumption on $x$. Now we can apply the fact that $(\bigvee_i B, w)$ has {\propH} to conclude that $\wt{s}$ admits a lift $\wt{t}$ in $\Z^w\otimes_{\Z\pi}H_4(\bigvee_i B_i;\Z\pi)$. The image of $\wt{s}$ in $\Z^w\otimes_{\Z\pi}H_4(B;\Z\pi)$ is the required lift of $x$. It follows that $(B,w)$ has {\propH} as needed.
\end{proof}

\section{Proof of \texorpdfstring{\cref{thm:main-homotopy}}{the main theorem}}\label{sec:mainhomotopy-proof}

With all our preliminary results in hand, we are finally able to prove the main theorem.
We recall the statement for the convenience of the reader.

\begin{reptheorem}{thm:main-homotopy}
Let $M$ and $M'$ be closed $4$-manifolds with fundamental group $\pi$ and orientation character $w$, such that $\pi$ is a $3$-manifold group and $(\pi,w)$ is admissible.

Then every isomorphism $\smash{Q(M) \xrightarrow{\cong} Q(M')}$ between the quadratic $2$-types of $M$ and $M'$ is realised by a homotopy equivalence.
In particular, $M$ and $M'$ are homotopy equivalent if and only if they have isomorphic quadratic $2$-types.
Here, homotopy equivalences are assume to be basepoint and local orientation preserving.
\end{reptheorem}

\begin{proof}
We will apply \cref{cor:main-strategy}.
Assume that $M$ and $M'$ have isomorphic quadratic $2$-types. By definition this means that for $B$ the Postnikov $2$-type of $M$ (and equivalently of $M'$), we have $3$-connected maps $f\colon M\to B$ and $f'\colon M'\to B$ inducing a given identification of the quadratic $2$-type. Let $\pi:=\pi_1(B)$.
By \cref{prop:H2FA-3groups}, $\pi_2(M)$ and $\pi_2(M')$ are free as abelian groups, so \cref{thm:homotopyclass}\,\eqref{item:0} holds. \cref{thm:homotopyclass}\,\eqref{item:2} holds by \cref{cor:kernelB-equals-kernelphi}.
Also by \cref{prop:ev-inj}, $\ev^*$ is injective (and in fact an isomorphism), so \cref{thm:homotopyclass}\,\eqref{item:3} holds.

Next we show that $M$ and $M'$ satisfy \cref{thm:homotopyclass}\,\eqref{item:1}, for $k=1$. Consider the element \[x:= f_*[M]-(f')_*[M']\in H_4(B;\Z^w).\]
By \cref{cor:images-fund-classes-in-group-homology}, which applies since $M$ and $M'$ have isomorphic quadratic $2$-types, $x$ maps to the trivial element in $H_4(\pi;\Z^w)$.
Moreover, since $M$ and $M'$ have isomorphic equivariant intersection forms, we also know that $\langle \alpha, \beta\cap x\rangle=0$ for every $\alpha, \beta\in H^2(B;\Z\pi)$.
By \cref{cor:propH-3mfd}, $\pi$ has \propH.
Then by \propH, we conclude that $x\in H_4(B;\Z^w)$ is contained in the image of $\Z^w\otimes_{\Z\pi} H_4(B;\Z\pi)$. Therefore the given $M$ and $M'$ satisfy \cref{thm:homotopyclass}\,\eqref{item:1}, for $k=1$.

Since all the requirements of \cref{thm:homotopyclass} are satisfied, we see by \cref{cor:main-strategy} that $M$ and $M'$ are homotopy equivalent, via a homotopy equivalence inducing the desired maps $(f')_*^{-1}\circ f_*\colon \pi_1(M) \to \pi_1(M')$ and $(f')_*^{-1}\circ f_*\colon \pi_2(M) \to \pi_2(M')$.
\end{proof}

\section{Applications to geometrically 2-dimensional and finite groups}\label{sec:oldresults}
In this section we reprove a couple of results of Hambleton, Kreck, and Teichner, in order to demonstrate the wide applicability of our methods. That is, as well as dealing with new fundamental groups, our methods recover previously known results.

\subsection{Geometrically 2-dimensional groups}\label{sec:geom2d}

In \cite{HKT09}*{Theorem~C}, Hambleton, Kreck, and Teichner showed that closed, oriented $4$-manifolds with geometrically 2-dimensional fundamental group $\pi$ that satisfies the Farrell-Jones conjecture, are classified up to $s$-cobordism by their quadratic $2$-type together with their Kirby-Siebenmann invariant and their $w_2$-type. It follows that such $4$-manifolds are classified up to homotopy equivalence by their quadratic $2$-type together with the $w_2$-type.

The $w_2$-type is determined by the homotopy type but, as we show in the next example, the $w_2$-type is not determined by the quadratic $2$-type. Thus the $w_2$-type has to be included in the data for a complete homotopy classification, and the analogue of the statement of \cref{thm:main-homotopy} does not hold for geometrically 2-dimensional groups.

\begin{example}\label{example-Z-squared}
Let $S^2\to S^2\wt{\times} T^2\to T^2$ be the unique $S^2$-bundle over $T^2$ that is orientable but not spin. Then both $S^2\times T^2$ and $S^2\wt{\times} T^2$ have fundamental group $\Z^2$ and second homotopy group $\Z$ with the trivial action of $\Z^2$. The $k$-invariants are trivial since $\Z^2$ is geometrically $2$-dimensional. Since the radical of the intersection form is $H^2(\Z^2;\Z[\Z^2])\cong \Z$, the intersection forms are also trivial.

In particular, $S^2\times T^2$ and $S^2\wt{\times} T^2$ have quadratic $2$-type $(\Z^2,\Z,0,0)$, but are not homotopy equivalent since only one of them is spin.

We discuss the failure of the conditions in \cref{thm:homotopyclass}, which is that \eqref{item:2} does not hold. Thus these examples are consistent with \cref{thm:homotopyclass}.
The $2$-type $B$ is $\CP^{\infty} \times T^2$, so $H_*(B;\Z[\Z^2]) \cong H_*(\CP^{\infty};\Z)$, with a trivial $\Z^2$-action.
We therefore see using the K\"{u}nneth theorem that  \[\varphi_B \colon \Z \otimes_{\Z[\Z^2]} H_4(B;\Z[\Z^2]) \cong \Z \to H_4(B;\Z) \cong \Z^2\]  is injective.
On the other hand the map
\[\mathcal{B}_{H_2(B;\Z[\Z^2])} \circ \Upsilon \colon \Z \otimes_{\Z[\Z^2]} H_4(B;\Z[\Z^2]) \cong \Z \to \Her(H_2(B;\Z[\Z^2])^\dagger) \cong \Her(\Z^{\dagger})\]
is zero, because $\Z^{\dagger} := \Hom_{\Z[\Z^2]}(\Z,\Z[\Z^2]) =0$ and hence the codomain is trivial.  So it is not possible for the nontrivial kernel of $\mathcal{B}_{H_2(B;\Z[\Z^2])} \circ \Upsilon$ to be contained in the trivial kernel of $\varphi_B$.
\end{example}

The \emph{reduced Postnikov $2$-type} $P$ of a manifold $M$ is a $3$-coconnected CW complex that is determined up to homotopy equivalence by the existence of a $2$-connected map $c_M\colon M\to P$ whose kernel on $\pi_2$ is the radical $R$ of the equivariant intersection form $\lambda_M$. In particular, $\pi_2(P)\cong \pi_2(M)/R$. Such a map $c_M\colon M\to P$ is called a reduced $3$-equivalence.

Despite \cref{example-Z-squared}, the quadratic $2$-type determines the image of the fundamental class in the homology of the reduced Postnikov $2$-type, when the fundamental group is geometrically $2$-dimensional~\cite{HKT09}*{Theorem~5.13}.  This theorem is a key step in the proof of \cite{HKT09}*{Theorem~C}.  We perceive a problem with the proof of the former theorem, in particular with the diagram at the start of the proof. The fixed points $\Gamma(A)^{\pi}$ can be trivial, and the map $\Gamma(A)_{\pi} \to \Gamma(A)^{\pi}$ need not be defined. This map is used in showing that the diagram commutes.
We give a new proof of \cite{HKT09}*{Theorem~5.13} below, as a corollary of \cref{thm:homotopyclass}.

\begin{corollary}[{\cite{HKT09}*{Theorem~5.13}}]
\label{cor:2-dim}
Let $\pi$ be a geometrically 2-dimensional group and let $M$ and $N$ be closed $4$-manifolds with fundamental group $\pi$, orientation character $w$, and the same reduced Postnikov $2$-type $P$. Two reduced $3$-equivalences $c_M\colon M \to P$ and $c_N\colon N \to P$ satisfy
$(c_M)_*[M] = (c_N )_*[N]\in H_4(P;\Z^w)$
if and only if $(c_M)^*\lambda_M=(c_N)^*\lambda_N$.
\end{corollary}
\begin{proof}
If $\pi$ is geometrically 2-dimensional, then $H^3(\pi;\Z\pi)=0$. Then $\ev^*$ is injective by~\cref{lem:ev-inj}, so \cref{thm:homotopyclass}\,\eqref{item:3} holds.  Furthermore, $\pi_2(M)/R\cong \pi_2(P)\cong H_2(P;\Z\pi)$ is stably free for any $4$-manifold with geometrically 2-dimensional fundamental group by \cite{HKT09}*{Corollary~4.4}.
Hence $H_2(P;\Z\pi)$ is projective as a $\Z\pi$-module, and is therefore free as an abelian group, so \cref{thm:homotopyclass}\,\eqref{item:0} holds. In addition $\B_{H_2(P;\Z\pi)}$ is injective by \cref{prop:free}, since $H_2(P;\Z\pi)$ is projective. Hence \cref{thm:homotopyclass}\,\eqref{item:2} holds automatically.

Since $\pi$ is geometrically 2-dimensional and $H_2(P;\Z\pi)$ is projective, $\varphi_P \colon \Z^w\otimes_{\Z\pi} H_4(P;\Z\pi)\to H_4(P;\Z^w)$ is an isomorphism by \cref{lem:sequence-for-pi2-projective}. Hence \cref{thm:homotopyclass}\,\eqref{item:1} holds, and in particular for $k=1$.  Thus the corollary follows from \cref{cor:main-strategy}.
\end{proof}

\subsection{Finite groups}\label{sec:finite}

We end this section by reproving the result, mentioned just below~\cref{prop:fundclass}, of Hambleton--Kreck~\cite{Hambleton-Kreck:1988-1}*{Theorem~1.1 (i)} and Teichner~\cite{teichner-phd} (cf.~\cite{KT}*{Corollary~1.6}).  In contrast with the previous subsection, no negative inference should be drawn about the published proofs.  Rather, the purpose is to demonstrate that our method is consistent with, and an extension of, the previously known methods.

\begin{corollary}[\cites{Hambleton-Kreck:1988-1,teichner-phd,KT}]\label{cor:finitepi}
Let $\pi$ be a finite group. Let $M$ and $M'$ be closed $4$-manifolds with fundamental group~$\pi$ and orientation character $w$.  Assume that $\Z^w\otimes_{\Z\pi}\Gamma(\pi_2(M))$ is torsion-free.

Then every isomorphism $Q(M) \xrightarrow{\cong} Q(M')$ between the quadratic $2$-types of $M$ and $M'$ is realised by a $($basepoint and local orientation preserving$)$ homotopy equivalence.
\end{corollary}

\begin{proof}
Certainly if $M$ and $M'$ are homotopy equivalent then they have isomorphic quadratic $2$-types. For the converse, suppose that $M$ and $M'$ have isomorphic quadratic $2$-types. We will apply \cref{cor:main-strategy}. Let $B$ denote the Postnikov  $2$-type of $M$ (and therefore also of $M'$). Therefore there are $3$-connected maps $f\colon M\to B$ and $f'\colon M'\to B$.

We need to check that the conditions of \cref{thm:homotopyclass} are satisfied.
Since $\pi$ is finite, the dual of the evaluation map $\ev^*\colon \Her(H_2(B;\Z\pi)^\dagger)\to \Her(H^2(B;\Z\pi))$ is injective by \cref{cor:ev*-pi-finite}. This establishes  \cref{thm:homotopyclass}\,\eqref{item:3}. Since $M$ is closed and has finite fundamental group, as an abelian group $H_2(B;\Z\pi)\cong H_2(M;\Z\pi)\cong H_2(\wt{M};\Z) \cong H^2(\wt{M};\Z)$ where $\wt{M}$ is the universal cover. This is free by the universal coefficient theorem, since $H_1(\wt{M};\Z)=0$.  This establishes \cref{thm:homotopyclass}\,\eqref{item:0}.

Further, there is an exact sequence $0\to \Z^w\otimes_{\Z\pi} H_4(B;\Z\pi)\to H_4(B;\Z^w)\to \Z/|\pi|$ by~\cite{Hambleton-Kreck:1988-1}*{p.~89}. Teichner~\cite{teichner-phd} (see also~\cite{KT}*{Theorem~3.4}) showed that when $M$ and $M'$ have isomorphic quadratic $2$-types, the  elements $f_*[M]$ and $f'_*[M']$ in  $H_4(B;\Z^w)$ map to the same element in $\Z/|\pi|$. Therefore, $f_*[M]-f'_*[M']$ lies in the image of the map $\Z^w\otimes_{\Z\pi} H_4(B;\Z\pi)\to H_4(B;\Z^w)$, giving \cref{thm:homotopyclass}\,\eqref{item:1} with $k=1$. Finally since the map $ \Z^w\otimes_{\Z\pi} H_4(B;\Z\pi)\to H_4(B;\Z^w)$ is injective, in order to establish \cref{thm:homotopyclass}\,\eqref{item:2} we have to show that the map  $\B_{H_2(B;\Z\pi)}$ is injective. But this follows from \cref{prop:finitepi} when $\Z^w\otimes_{\Z\pi}\Gamma(H_2(B;\Z\pi))$ is torsion-free, which holds by hypothesis.  The result now follows from \cref{cor:main-strategy}.
\end{proof}

\section{A homeomorphism classification for oriented \texorpdfstring{$4$-manifolds}{4-manifolds} with infinite dihedral fundamental group}\label{sec:homeo}

Let $D_\infty:=\langle a,b\mid a^2,b^2\rangle \cong \Z/2 * \Z/2$ be the infinite dihedral group. In this section we will prove \cref{thm:homeo-class} from the introduction. The following theorem will be a key ingredient.

\begin{theorem}
\label{thm:homeo-homotopy-stable}
Let $M$ and $M'$ be closed, oriented $4$-manifolds with fundamental group $D_\infty$. Then $M$ and $M'$ are homeomorphic over $D_\infty$ if and only if they are homotopy equivalent and stably homeomorphic over~$D_\infty$.
\end{theorem}

Here and throughout the section, we assume that all homotopy equivalences, homeomorphisms, and stable homeomorphisms are basepoint and orientation preserving, as per our conventions.

 In \cref{subsection:D-inftystable-homeo-classn} we compute the stable classification using modified surgery.
In \cref{subsection:D-infty-structure-set} we compute the structure set using the surgery exact sequence.
In \cref{subsection:D-infty-homeo-classn} we then combine these two computations to obtain \cref{thm:homeo-homotopy-stable} and hence \cref{thm:homeo-class}.
Finally in \cref{subsection:D-infty-replacing-k-invariant} we show that in some situations one need not compute the $k$-invariant, and instead computing the $w_2$-type suffices.

\subsection{The stable homeomorphism classification}\label{subsection:D-inftystable-homeo-classn}

By modified surgery \cite{kreck}, the stable classification is determined by the bordism group of the normal 1-type of the manifold. For $M$ oriented this only depends on $\pi_1(M)$ and the second Stiefel--Whitney class $w_2(M)$. More precisely, it is given by $\pi_1(M)$ and the $w_2$-type. By definition, the $w_2$-type is $\infty$ if the universal cover $\wt M$ is not spin and otherwise the $w_2$-type is  given by an element $\theta\in H^2(\pi_1(M);\Z/2)$ such that $c^*\theta=w_2(M)$, where $c\colon M\to B\pi_1(M)$ is a classifying map. Up to automorphisms of $D_\infty$, there are four $w_2$-types, as follows.
Consider the standard projections $p_a,p_b \colon D_\infty\to \Z/2$ determined by $p_a(a) = 1, p_a(b) =0$ and $p_b(a)=0, p_b(b)=1$.  Let $\psi\in H^1(\Z/2;\Z/2)$ be the generator.
Define elements of $H^1(D_\infty;\Z/2)$ by \[x:= p_a^*(\psi) \text{ and } y := p_b^*(\psi).\]
Then the possibilities for the $w_2$-type are \[\theta\in\{\infty,0,x^2,x^2+y^2\}.\]

\begin{remark}
\label{khldflkhfglufdalkgualkguvf}
For $\theta=0$, the stable classification is determined by the signature.
For the cases $\theta=\infty,x^2$, the stable classification is determined by the signature and the Kirby--Siebenmann invariant. For $\theta=0$ or $x^2$ this is similar to the computation in \cite{teichner-star}*{Lemma~2} using that $\Omega_4^{\TopSpin}(B\Z/2)\cong 8\Z$ via the signature \cite{teichner-phd}*{Section~4.2}. For $\theta=\infty$, bordism over the normal $1$-type is given by oriented bordism over $BD_\infty$ \cite{teichner-phd}*{Example~2.1.2}, and we have $\Omega_4 (BD_\infty)\cong \Z$ via the signature.
\end{remark}

As mentioned in the introduction, we will use the invariants from \cref{def:s-invt}. If the $w_2$-type of $M$ is $0,\infty$, or $x^2$, this stable homeomorphism invariant is determined by the signature and the Kirby--Siebenmann invariant, since those determine the stable classification (see \cref{khldflkhfglufdalkgualkguvf}). Hence we focus on the case of $w_2$-type $x^2+y^2$.

\begin{lemma}\label{lem:stable-classes}
Let $\theta=x^2+y^2$. There are precisely four distinct stable homeomorphism classes with fixed identification of the fundamental group with $D_\infty$, $w_2$-type $\theta$ and a given signature $z\in 8\Z$. These classes are represented by homotopy equivalent manifolds and distinguished by $s$; recall that $s$ takes values in $\Z/2 \times \Z/2$.
\end{lemma}

\begin{proof}
By \cite{kreck}*{Theorem~C}, two $4$-manifolds with normal $1$-type $(D_\infty,x^2+y^2)$ are stably homeomorphic if they admit bordant normal 1-smoothings. By \cite{teichner-star}*{Lemma~2}, the bordism group over this normal 1-type is isomorphic to $8\Z\oplus \Z/2\oplus \Z/2$, where the first summand is given by the signature. Hence there are at most four stable homeomorphism classes for a given signature.

Consider the bundle $S^2\to E\to \RP^2$ with orientable but not spin total space. By \cite{teichner-star}*{Proposition~1}, there exists a manifold $\star E$ that is homotopy equivalent to $E$ but has nontrivial Kirby--Siebenmann invariant. For $z\in \Z$, the manifolds $E\#E\#^zE_8$, $\star E\# \star E\#^zE_8$, $E\#\star E\#^zE_8$ and $\star E\# E\#^z E_8$ have normal $1$-type determined by $(D_\infty,\theta)$. Since $\sigma(E)=0$, these manifolds have signature $8z$. We have \[s(E\#E\#^zE_8)=(\sigma(E)/8+\ks(E),\sigma(E\#^z E_8)/8 +\ks(E\#^zE_8))=(0+0,z+z)=(0,0),\]
and similarly $s(\star E\#\star E\#^z E_8)=(1,1), s(E\#\star E\#^z E_8)=(0,1)$, and $s(\star E\# E\#^z E_8)=(1,0)$. It follows that there are precisely four stable homeomorphism classes with signature $8z$ and that they are distinguished by $s$, as asserted.
\end{proof}

\begin{remark}
If we also consider stable homeomorphisms that induce a nontrivial automorphism on $D_\infty$, then two of the stable classes are identified, namely those with $s=(1,0)$ and $s=(0,1)$. Actually, $\Out(D_\infty)\cong C_2$ with the nontrivial element given by the map that swaps $a$ and $b$. To see this, it is easier to use the presentation $\langle t,a\mid a^2,atat\rangle$, where $t=ab$. Then we see that $\Aut(D_\infty)=\{(m,\epsilon)\mid m\in \Z, \epsilon \in \{\pm 1\}\}$, with $(m,\epsilon)$ mapping $t$ to $t^\epsilon$ and $a$ to $at^m$. Then $(m,\epsilon)\circ (n,\eta)=(m+\epsilon n,\epsilon\eta)$ and hence $D_\infty \cong \Aut(D_\infty)$ by $t \mapsto (1,0)$ and $a \mapsto (0,1)$. The inner automorphisms are generated by the conjugations $c_t=(-2,1)$ and $c_a=(0,-1)$, where $c_g \colon x \mapsto gxg^{-1}$ denotes conjugation by $g$. Hence $\Out(D_\infty)\cong C_2$ is generated by $(1,-1)$, which is the map that swaps $a$ and $b$.
\end{remark}

\subsection{The structure set}\label{subsection:D-infty-structure-set}

The next step in the proof of \cref{thm:homeo-homotopy-stable} is to calculate the relevant structure set.
We will make use of the stable homeomorphism classification in this computation.

\begin{proposition}
\label{prop:structure-set-dihedral}
Let $M$ be a closed, oriented $4$-manifold with fundamental group $D_\infty$. Then the structure set $\sS(M)$ is isomorphic to $H_2(M;\Z/2)$.
\end{proposition}

For the proof we need the following lemma.
The Whitehead group of $\Z/2$ vanishes, and the Whitehead group is additive with respect to free products by~\cite{Stallings-Whitehead-additive}. Hence $\Wh(D_\infty)=\Wh(\Z/2)\oplus \Wh(\Z/2) =0$, and so throughout this section we may and shall omit decorations from $L$-groups and structure sets.

\begin{lemma}[Connolly--Davis~\cite{davis-connolly}]
\label{lem:connolly-davis}
$L_5(\Z D_\infty)=0$ and $L_4(\Z D_\infty)\cong \Z^3$.
\end{lemma}

\begin{proof}
By \cite{davis-connolly}*{p.~1046} we have $L_n(\Z D_\infty)\cong L_n(\Z[\Z/2])\oplus \wt L_n(\Z[\Z/2])\oplus \UNil_n(\Z,\Z,\Z)$. By Wall~\cite{Wall-surgery-book}*{Theorem~13A.1}, $L_5(\Z[\Z/2])=0$ and $L_4(\Z[\Z/2])\cong \Z^2$. Furthermore, $\UNil_5(\Z;\Z,\Z)=\UNil_1(\Z;\Z,\Z)=0$ and $\UNil_4(\Z;\Z,\Z)=\UNil_0(\Z;\Z,\Z)=0$ by \cite{davis-connolly}*{Theorem~1.10}. Here Davis--Connolly relied on prior computations by Connolly--Ranicki~\cite{connolly-ranicki}, Connolly--Ko\'zniewski~\cite{connolly-kozniewski}, and Cappell~\cite{cappell-bams}.
\end{proof}

\begin{proof}[Proof of \cref{prop:structure-set-dihedral}]
The group $D_\infty$ is an extension of abelian groups $0 \to \Z \to D_\infty \to \Z/2 \to 0$, and thus a good group (see e.g.~\cites{Freedman-Teichner:1995-1, DET-book-goodgroups}). Hence we have the surgery exact sequence
\[L_5(\Z D_\infty)\to \sS(M)\to \cN(M)\to L_4(\Z D_\infty).\]
As stated above, decorations are irrelevant for this fundamental group, so we omit them.
The set of normal invariants $\cN(M)$ is isomorphic to
\[\cN(M)\cong [M,\G/\Top]\cong H^4(M;\Z)\oplus H^2(M;\Z/2)\cong \Z\oplus H_2(M;\Z/2),\]
where the $\Z$-summand is detected in $L_4(\Z D_\infty)$ by the signature.
By \cref{lem:connolly-davis}, $L_4(\Z D_\infty)$ is torsion-free and hence the kernel of the surgery obstruction map is $H_2(M;\Z/2)$. Since $L_5(\Z D_\infty)=0$ by \cref{lem:connolly-davis}, $\sS(M)\cong H_2(M;\Z/2)$ as claimed.
\end{proof}

\subsection{The homeomorphism classification}\label{subsection:D-infty-homeo-classn}

Let $M$ be a closed, oriented $4$-manifold together with an identification of $\pi_1(M)$ with $D_\infty$. Let
$\hAut(M,D_\infty)$ denote the group of homotopy self-equivalences of $M$ that act as the identity on $\pi_1(M)$. Again, recall that all homotopy self-equivalences, homeomorphisms, and stable homeomorphisms are assumed to be orientation preserving.

\begin{theorem}
\label{thm:stableclasss}
Let $M$ be a closed, oriented $4$-manifold with fundamental group $D_\infty$. Then the set of homeomorphism classes over $D_\infty$ of manifolds homotopy equivalent to $M$, which by surgery theory is isomorphic to $\sS(M)/\hAut(M,D_\infty)$, has:
\begin{enumerate}[leftmargin=1cm,font=\normalfont]
\item\label{it:dinfty-i} a single element if $M$ is spin;
\item\label{it:dinfty-ii} two elements distinguished by the Kirby--Siebenmann invariant if $M$ has $w_2$-type $\infty$ or $x^2$;
\item\label{it:dinfty-iii} four elements if $M$ has $w_2$-type $x^2+y^2$, distinguished by the invariant~$s$.
\end{enumerate}
In all cases, the different classes are pairwise not stably homeomorphic.
\end{theorem}

\begin{proof}
As mentioned above, since $\Wh(D_\infty) =0$, the forgetful map $\sS^s(M) \to \sS^h(M)$ from the  simple to the  non-simple structure set is an isomorphism, and so we can consider $\sS(M)$ as the simple structure set. Moreover since $D_\infty$ is good, the $s$-cobordism theorem holds, and so we can identify the set of homeomorphism classes over $D_\infty$ of manifolds homotopy equivalent to $M$ with~$\sS(M)/\hAut(M,D_\infty)$.

By \cref{prop:structure-set-dihedral}, $\sS(M)\cong H_2(M;\Z/2)$. It remains to deduce the action of $\hAut(M,D_\infty)$ on the structure set.
By Stong~\cite{Stong-conn-sum}*{Proposition~3.2}, every class in $H_2(M;\Z/2)\cong \sS(M)$ that is represented by a map $R\colon \RP^2\to M$ with $R^*w_2(M)=0$ can be represented by a self-homotopy equivalence of $M$ (which is homotopic to the identity on the 2-skeleton of $M$).

We first consider the case $w_2(\wt M)=0$. Then every element of $\pi_2(M)$ can be represented by a map $R\colon \RP^2\to M$ with $R^*w_2(M)=0$ and we have $H_2(M;\Z/2)/\pi_2(M)\cong H_2(D_\infty;\Z/2)\cong (\Z/2)^2$.

For elements $a,b \in D_\infty=\pi_1(M)$, we can choose maps $R_a,R_b \colon \RP^2\to M$ that map to the elements $(1,0)$ and $(0,1)$ in $H_2(D_\infty;\Z/2) \cong \Z/2 \oplus \Z/2$. The images in $H_2(D_\infty;\Z/2)$ are determined by the elements of $\pi_1(M)$ represented by the image of the generator of $\pi_1(\RP^2)$, $a$ and $b$ respectively, since the composition $\RP^2 \to M \xrightarrow{c} BD_\infty$ is determined by the induced map on fundamental groups.

If $w_2(M)=0$, then both $(R_a)_*[\RP^2]$ and $(R_b)_*[\RP^2]$ in $H_2(M;\Z/2) \cong \sS(M)$ can be represented by self-homotopy equivalences by \cite{Stong-conn-sum}*{Proposition~3.2}, as in the first paragraph of the proof. Thus \eqref{it:dinfty-i} follows. If $w_2(M)=c^*x^2$, then again using Stong's method, the map $R_b \colon \RP^2\to M$ can be represented by a self-homotopy equivalence of $M$. Thus $\sS(M)/\hAut(M,D_\infty)$ has at most two elements. In this case, there exists a manifold $M'$ homotopy equivalent to $M$ with $\ks(M')\neq\ks(M)$ by \cite{Stong} (see also \cite{table}*{Proposition~5.11}). This implies \eqref{it:dinfty-ii} in the case of $w_2$-type $x^2$.

Now we consider the case of $w_2$-type $x^2+y^2$. As before $\sS(M)/\hAut(M)$ has at most four elements. Let $f\colon M'\to M$ be a homotopy equivalence. Then $f$ and $f\#\id\colon M'\#(S^2\times S^2)\to M\#(S^2\times S^2)$ have the same image under $\sS(M) \xrightarrow{\cong} H_2(M;\Z/2) \to  H_2(D_\infty;\Z/2)$.
By \cref{lem:stable-classes}, the four classes are pairwise not stably homeomorphic and are distinguished by the values of $s$. Thus \eqref{it:dinfty-iii} follows.

It remains to show \eqref{it:dinfty-ii} in the case of $w_2$-type $\infty$. 		
If $w_2(\wt M)\neq 0$, choose a map $S\colon S^2\to M$ with $S^*w_2(M)\neq 0$. Then there is a basis of $x\in H_2(M;\Z/2)$ such that either $x$ or $x+[S]$ can be represented by a map $R\colon \RP^2\to M$ with $R^*w_2(M)=0$. Hence $\sS(M)/\hAut(M,D_\infty)$ has at most two elements. As before, by \cites{FQ,Stong-conn-sum} there exists a manifold $M'$ homotopy equivalent to $M$ with $\ks(M')\neq\ks(M)$. This implies \eqref{it:dinfty-ii} in the case of $w_2$-type $\infty$.
\end{proof}

\cref{thm:homeo-homotopy-stable} is a direct consequence of \cref{thm:stableclasss} as follows.

\begin{proof}[Proof of \cref{thm:homeo-homotopy-stable}]
For $M' \simeq M$, we deduce from the contrapositive of the last sentence of \cref{thm:stableclasss} that $M'$ is homeomorphic to $M$ if $M$ is stably homeomorphic to $M$.   Thus \cref{thm:homeo-homotopy-stable} follows from \cref{thm:stableclasss}.
\end{proof}

We now prove \cref{thm:homeo-class} from the introduction as an application of \cref{thm:homeo-homotopy-stable}. We restate the theorem for the convenience of the reader.

\begin{reptheorem}{thm:homeo-class}
Let $M_1$ and $M_2$ be closed, oriented $4$-manifolds with isomorphisms $\alpha_i\colon \pi_1(M_i)\xrightarrow{\cong} D_\infty$. Then $M_1$ and $M_2$ are orientation preserving homeomorphic over $D_\infty$ if and only if
\begin{enumerate}[leftmargin=1cm,font=\normalfont]
\item\label{item-D-infty-1-rep} $M_1$ and $M_2$ have isomorphic quadratic $2$-types over $D_\infty$,
\item\label{item-D-infty-2-rep} $\ks(M_1)=\ks(M_2)$, and
\item\label{item-D-infty-3-rep}
$s(M_1,\alpha_1)=s(M_2,\alpha_2) \in \Z/2 \times \Z/2$.
\end{enumerate}
Moreover, if  conditions \eqref{item-D-infty-2} and \eqref{item-D-infty-3} hold, then every isomorphism of the quadratic $2$-types over $D_\infty$
is realised by a homeomorphism $M_1 \to M_2$.
\end{reptheorem}

\begin{proof}[Proof of \cref{thm:homeo-class}]
By \cref{thm:main-homotopy}, the homotopy type is determined by the quadratic $2$-type. The homotopy type determines the signature, and by  homotopy invariance of Stiefel-Whitney classes, the homotopy type determines the $w_2$-type.
Hence by \cref{thm:stableclasss}, two oriented $4$-manifolds $M_0$ and $M_1$ with the same quadratic $2$-types are orientation preserving stably homeomorphic if and only if:
\begin{enumerate}[label=(\roman*)]
\item they have equal Kirby--Siebenmann invariants, $\ks(M_0)=\ks(M_1)$, and
 \item $s(M_0,\alpha_0) = s(M_1,\alpha_1) \in \Z/2 \times \Z/2$.
\end{enumerate}
If the common $w_2$-type differs from $x^2 +y^2$, then by \cref{khldflkhfglufdalkgualkguvf} the second item $s(M_0,\alpha_0) = s(M_1,\alpha_1)$ holds automatically, since the stable homeomorphism classification and hence the invariants $s(M_i,\alpha_i)$ are determined by the signatures and the Kirby-Siebenmann invariants, which by the other assumptions already agree. If the $w_2$-type is $x^2+y^2$, then the additional assumption on $s(M_i,\alpha_i)$ is necessary.

Then by \cref{thm:homeo-homotopy-stable}, two oriented $4$-manifolds with fundamental group $D_\infty$ are homeomorphic over $D_\infty$ if and only if they are homotopy equivalent and stably homeomorphic over $D_\infty$. This completes the proof of \cref{thm:homeo-class}.
\end{proof}

\subsection{Replacing the \texorpdfstring{$k$}{k}-invariant with the second Stiefel--Whitney class}\label{subsection:D-infty-replacing-k-invariant}

If one seeks to apply \cref{thm:homeo-class} in practice, it may be challenging to decide whether two quadratic $2$-types are isomorphic. We do not know whether there is an algorithm to decide this problem. To ease the burden, we demonstrate next that in some cases one need not compute the $k$-invariant, and can instead either compute the $w_2$-type or take connected sum with either $\CP^2$ or $\ol{\CP}^2$.

For an arbitrary group $\pi$ and the augmentation ideal $I\pi$, let $H(I\pi)$ denote the hyperbolic form on $I\pi$, i.e.\ the form
\begin{align*}
    (I\pi \oplus I\pi^\dagger) \times (I\pi \oplus I\pi^\dagger) &\to \Z\pi\\
    ((x,\varphi),(y,\psi)) &\mapsto \ol{\varphi(y)} + \psi(x).
\end{align*}

\begin{proposition}\label{prop:D-infty-k-invariant-free-statement}
    Let $M$ and $M'$ be oriented $4$-manifolds with fundamental group $\pi:=D_\infty$, with corresponding identifications $\alpha$ and $\alpha'$. Suppose that the equivariant intersection forms are both isomorphic to $H(I\pi)\oplus\lambda$, where $\lambda$ is a nonsingular Hermitian form on a stably free $\Z\pi$-module.
    \begin{enumerate}[(i)]
        \item\label{item:prop:dinfty-k-inv-i}  Then $M\#\CP^2$ and $M'\#\CP^2$ have isomorphic quadratic $2$-types, and hence by \cref{thm:main-homotopy} are homotopy equivalent.
        \item\label{item:prop:dinfty-k-inv-ii} If $M$ and $M'$ are almost spin and have the same $w_2$-type $w_2^\pi\in H^2(\pi;\Z/2)\cong \Z/2 \times \Z/2$, then they have isomorphic quadratic $2$-types, and hence by \cref{thm:main-homotopy} are homotopy equivalent.
    \end{enumerate}
If in addition $\ks(M)=\ks(M')$, and if in case \ref{item:prop:dinfty-k-inv-ii} we have moreover that $s(M,\alpha) = s(M',\alpha')$,  then \cref{thm:homeo-class} implies that $M$ and $M'$ are homeomorphic.
\end{proposition}

The condition on $s$-invariants is automatic in case~\ref{item:prop:dinfty-k-inv-i} because the universal covers are not spin.
This quickly leads to  \cref{cor:intro-smooth-dihedral}, which we restate.

\begin{repcorollary}{cor:intro-smooth-dihedral}
    Let $M$ and $M'$ be closed, oriented, smooth $4$-manifolds with fundamental group $\pi:=D_\infty$ and equivariant intersection forms both isomorphic to $H(I\pi)\oplus\lambda$, where $\lambda$ is a nonsingular Hermitian form on a stably free $\Z\pi$-module.
    Then $M\#\CP^2$ and $M'\#\CP^2$ are homeomorphic, as are $M\#\ol{\CP}^2$ and $M'\#\ol{\CP}^2$.
\end{repcorollary}

\begin{proof}
    Both $M\#\CP^2$ and $M'\#\CP^2$ have trivial Kirby--Siebenmann invariant as they are smooth. Since their universal covers are not spin, they also have trivial $s$-invariant by definition. By \cref{prop:D-infty-k-invariant-free-statement}, $M\#\CP^2$ and $M'\#\CP^2$ have isomorphic quadratic $2$-types. Hence $M\#\CP^2$ and $M'\#\CP^2$ are homeomorphic by \cref{thm:homeo-class}.

    Since $-H(I\pi)$ is isometric to $H(I\pi)$, we can apply the same argument to show that $\ol{M}\#\CP^2$ and $\ol{M'}\#\CP^2$ are homeomorphic. Changing the orientation, it follows that also $M\#\ol{\CP}^2$ and $M'\#\ol{\CP}^2$ are homeomorphic.
\end{proof}

Now we begin working towards the proof of \cref{prop:D-infty-k-invariant-free-statement}. First we introduce a pair of useful $4$-manifolds.

\begin{example}\label{ex:E-and-F}
    There are two important examples of smooth, oriented $4$-manifolds with fundamental group $\Z/2$, denoted by $E$ and $F$, that arise as the total spaces of $S^2$-bundles over $\RP^2$. Let $\eta$ be the canonical line bundle over $\RP^2$ and let $\varepsilon$ be the trivial bundle. Then, as in \cite{Hambleton-Kreck:1988-1}*{Remark~4.5}, we define $E = S(3\eta)$ and $F=S(\eta \oplus 2\varepsilon)$.
    We already used $E$ in the proof of \cref{lem:stable-classes}.  Kirby diagrams for $E$ and $F$ are shown in \cref{fig:kirby-diagrams-for-E-F}; they also appear in \cite{GompfStip}*{Example~4.6.5, Figure~5.4.6}.
\end{example}

\begin{figure}[htb]
         \centering
        \begin{tikzpicture}
        \node[anchor=south west,inner sep=0] at (0,0){\includegraphics[width=0.3\textwidth]{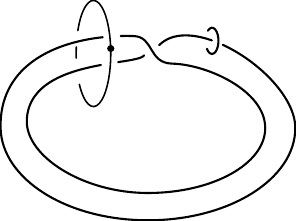}};
	\node at (4.65,1.5) {$1$};
        \node at (3.25,3.15) {$0$};
        \node[anchor=south west,inner sep=0] at (6,0){\includegraphics[width=0.3\textwidth]{E.pdf}};
	\node at (10.65,1.5) {$0$};
        \node at (9.25,3.15) {$0$};
	\end{tikzpicture}
        \caption{Kirby diagrams for $E$ (left) and $F$ (right).}
        \label{fig:kirby-diagrams-for-E-F}
\end{figure}

    \begin{lemma}\label{lemma:E-F-spin-or-not}
    The manifold $F$ is spin while $E$ is not. Further, if $p \colon E \to \RP^2$ is the bundle projection and $x \in H^2(\RP^2;\Z/2)\cong \Z/2$ is the nontrivial element, then $w_2(E) = p^*(x)$.
        \end{lemma}

\begin{proof}
    This can be computed by the Whitney sum formula as follows. Writing $E(3\eta)$ for the total space of $3\eta$, we have $w_i(TE \oplus \nu_{E \subseteq E(3\eta)}) = j^*w_i(E(3\eta))$, where $j \colon E \to E(3\eta)$ is the inclusion of the sphere bundle into the total space.
    Since $E$ and $3\eta$ have orientable total spaces, it follows that $w_1(\nu_{E \subseteq E(3\eta)})=0$, and hence the line bundle $\nu_{E \subseteq E(3\eta)}$ is trivial. Therefore $w_2(TE) = j^*w_2(E(3\eta))$.
    Since  \[w(T\RP^2) = (1+x)^3 = 1+x+x^2,\] where $x \in H^1(\RP^2;\Z/2)$ is the generator,  we compute that
    \[w_2(T\RP^2 \oplus 3\eta) = w_2(T\RP^2) + w_1(T\RP^2)w_1(3\eta) + w_2(3\eta) = x^2 + x^2 +x^2 = x^2 \in H^2(\RP^2;\Z/2), \]
    so $w_2(T\RP^2 \oplus 3\eta)$ is nontrivial.      Then since \[w_2(T\RP^2 \oplus 3\eta) =z^*w_2(E(3\eta)) \in H^2(\RP^2;\Z/2),\] where $z \colon \RP^2 \to E(3\eta)$ is the zero section, we deduce that $w_2(E(3\eta))$ is nontrivial in
    \[H^2(E(3\eta);\Z/2) \cong H^2(\RP^2;\Z/2) \cong \Z/2,\]
    and hence $w_2(E)$ is the nontrivial element of $j^*(H^2(E(3\eta);\Z/2)) \subseteq H^2(E;\Z/2)$.  The computation for $F$ is similar, and yields that $w_2(F)=0$.

    Alternatively, the computation of $w_2$ follows by \cite{GompfStip}*{Corollary~5.7.2} from the Kirby diagrams for $E$ and $F$ in \cref{fig:kirby-diagrams-for-E-F}.
    \end{proof}

\begin{lemma}\label{lemma-univ-cover-E-F}
    Both $E$ and $F$ have $($spin$)$ universal cover $S^2 \times S^2$.
\end{lemma}
 \begin{proof}
         This can also be computed geometrically using Kirby calculus. Alternatively, note that the universal covers $\wt{F}$ and $\wt{E}$ are both $S^2$-bundles over $S^2$. We have to argue that the nontrivial bundle does not arise. For this recall that $w_2$ pulls back under the covering map to $w_2$ of the cover. Since $w_2(F)=0$, this implies immediately that $\wt{F}$ is spin, and hence is $S^2 \times S^2$.  For $E$, the $\RP^2$ in $E$ is double covered by the base $S^2$ in $\wt{E}$,  and hence the nontrivial $w_2(E)$ pulls back trivially to $\wt{E}$. Thus $\wt{E} $ is also spin and hence is  $S^2 \times S^2$.
 \end{proof}

Next we compute some algebraic topological invariants of $E$ and $F$.

\begin{lemma}
    Let $M \in \{E,F\}$. Then $M$ is a rational homology 4-sphere with $H_1(M;\Z) \cong \Z/2 \cong H_2(M;\Z)$, $H_0(M;\Z) \cong \Z \cong H_4(M;\Z)$, and $H_i(M;\Z)=0$ otherwise.
\end{lemma}

\begin{proof}
Both $E$ and $F$ have integral handle chain complexes
\[\Z \xrightarrow{(0)} \Z \xrightarrow{(0,2)} \Z\oplus \Z \xrightarrow{(2,0)} \Z \xrightarrow{(0)} \Z,\]
so in particular they are rational homology $4$-spheres. Here the attaching maps of the $3$-handles can be seen by turning the handle decomposition upside down.
The $\Z$-homology can be read off to be as claimed.
\end{proof}

\begin{remark}\label{rem:E-F-homology-nonlemma}
The fact that $E$ and $F$ are rational homology spheres can be deduced from $\chi(E) = 2 = \chi(F)$ and $\pi_1(E) = \Z/2 = \pi_1(F)$.

With $\Z/2$ coefficients, we have $H_2(M;\Z/2) \cong \Z/2 \oplus \Z/2$, with one summand generated by an embedded $\RP^2$, denoted by $R$, corresponding to a section of the $S^2$-bundle. It arises by capping off the M\"obius band that is visible in the Kirby diagram with the core of the $2$-handle.  The other summand is generated by an $S^2$ fibre and denoted by $S$, and is again visible in the Kirby diagram as the disc bounded by the `helper circle' union the core of that $2$-handle.  The fibre $S$ also represents the nontrivial element in $H_2(M;\Z)$. Identifying $H^2(E;\Z/2) \cong \Hom(H_2(E;\Z/2),\Z/2)$ we have $w_2(E)(R) = 1$ and $w_2(E)(S)=0$, while $w_2(F)(R) = 0 = w_2(F)(S)$, which accords with $F$ being spin.
\end{remark}

The $\Z[\Z/2]$-module chain complexes for $E$ and $F$ can also be read off from the Kirby diagrams to be
\begin{equation}\label{eq:d-infty-cc}
    \Z[\Z/2] \xrightarrow{(T-1)} \Z[\Z/2] \xrightarrow{(0,T+1)} \Z[\Z/2] \oplus \Z[\Z/2] \xrightarrow{(T+1, 0)} \Z[\Z/2] \xrightarrow{(T-1)}\Z[\Z/2],
\end{equation}
where $T$ is the nontrivial element in $\Z/2$ and corresponds to the covering transformation. This is immediate from the diagrams and Fox calculus for the handle $2$-skeleton, and then the remaining boundary maps can be seen by turning the handle decomposition upside down.
We deduce, again for $M \in \{E,F\}$, that $\pi_2(M) \cong H_2(M;\Z[\Z/2])\cong\ker(T+1) \oplus \coker(T+1) \cong \Z^- \oplus \Z^-$.
We then compute \[H^3(\Z/2;\pi_2(M)) \cong H^3(\Z/2;\Z^- \oplus \Z^-) \cong \Z/2 \oplus \Z/2.\] 

\begin{lemma}\label{lemma:E-F-k-invariants}
    Let $M\in \{E,F\}$. The intersection form $\lambda_M \colon \pi_2(M) \times \pi_2(M) \to \Z[\Z/2]$ is hyperbolic. With respect to a hyperbolic basis for $\pi_2(M) \cong \Z^- \oplus \Z^-$, we have an identification $H^3(\Z/2;\pi_2(M)) \cong \Z/2 \oplus \Z/2$. Then $k_F = (1,0)$ and $k_E = (1,1)$.
\end{lemma}

Note that this result agrees with the assertion in \cite{Hambleton-Kreck:1988-1}*{Remark~4.5}.

\begin{proof}
Recall from~\cite{EM49}*{(4.1)~and~Section~9} that the $k$-invariant can be computed by constructing a partial chain map up to degree $2$ of the standard free $\Z[\Z/2]$-module resolution $C_*^{\Z/2}:=C_*(\Z/2;\Z/2[\Z/2])$ of $\Z$ into $C_*^M := C_*(M;\Z[\Z/2])$, as follows:
\[
\begin{tikzcd}[column sep=large]
\Z[\Z/2]\ar[r,"(T-1)"] \ar[dr,"{k=(T-1,0)}"' near start] & \Z[\Z/2]\ar[r,"{(T+1)}"] \ar[d,"{(1,0)}"] & \Z[\Z/2]\ar[r,"(T-1)"]\ar[d,"(1)"] & \Z[\Z/2] \ar[d,"(1)"] \\
    \Z[\Z/2] \ar[r,"{(0,T+1)}"'] & \Z[\Z/2] \oplus \Z[\Z/2] \ar[r,"{(T+1, 0)}"']  & \Z[\Z/2] \ar[r,"{(T-1)}"']  & \Z[\Z/2].
\end{tikzcd}
\]
The image of the map $k$ under \[\Hom_{\Z[\Z/2]}\big(C_3^{\Z/2},\ker \big(d_2 \colon C_2^M \to C_1^M\big)\big) \to \Hom_{\Z[\Z/2]}\big(C_3^{\Z/2},\pi_2(M)\big) \cong \Hom_{\Z[\Z/2]}\big(C_3^{\Z/2},\Z^- \oplus \Z^-\big)\]  represents the $k$-invariant in $H^3(\Z/2;\Z^- \oplus \Z^-)$.
The cycles $(1-T,0)$ and $(0,1)$ in $C_2^M$ generate $\pi_2(M)\cong \Z^-\oplus\Z^-$. In this basis, the $k$-invariant of $M$ is $(1,0)\in H^3(\Z/2;\pi_2(M))\cong \Z/2 \oplus \Z/2$.

We compute the intersection forms from the diagrams of $E$ and $F$ with respect to this basis of $\pi_2$.
Note that $\Hom_{\Z[\Z/2]}(\Z^-,\Z[\Z/2]) \cong \Z^-$, with the generator given by the homomorphism that sends $1 \mapsto 1-T$. Hence, chasing through the definition, the hyperbolic form on $\Z^-$ is identified with the form on $\Z^- \oplus \Z^-$ that sends $((a,b),(c,d)) \mapsto (1-T)(ad+bc)$.

A geometric computation using the Kirby diagram shows that for $M \in \{E,F\}$ the form $\lambda_M$ is represented by
\begin{equation}\label{eqn:int-form-E-F}
    \left(\begin{smallmatrix}
    * & 1-T \\ 1-T & 0
\end{smallmatrix}\right) \colon (\Z^- \oplus \Z^-) \times (\Z^- \oplus \Z^-) \to \Z[\Z/2].
\end{equation}
The top left entry $*$ is of the form $n(1-T)$, because every element in $\Hom_{\Z[\Z/2]}(\Z^-,\Z[\Z/2])$ sends $1$ to a multiple of $1-T$.
If we ignore all $T$ terms we obtain $\left(\begin{smallmatrix}
    n & 1 \\ 1 & 0
\end{smallmatrix}\right)$, which gives the intersection pairing of the universal cover, which we know to be homeomorphic to $S^2 \times S^2$ by \cref{lemma-univ-cover-E-F}. Hence $n$ is even. However, unless $n=0$, our given basis is not a hyperbolic basis. Nonetheless we can change the basis of $\pi_2(M)$ to a hyperbolic basis, so that the intersection form is represented by the matrix in \eqref{eqn:int-form-E-F} with $*=0$.
We want to describe the $k$-invariant with respect to this hyperbolic basis of $\pi_2(M)$.

If $n \equiv 0 \mod{4}$ then the basis change matrix reduces to the identity modulo 2 and the $k$-invariant is again given by $(1,0) \in \Z/2 \oplus \Z/2 \cong H^3(\Z/2;\pi_2(M))$ with respect to the hyperbolic basis. On the other hand if $n \equiv 2 \mod{4}$ then the $k$-invariant is given by $(1,1) \in \Z/2 \oplus \Z/2$.

We show that $n \equiv 0 \mod{4}$ for $M=F$ and $n \equiv 2 \mod{4}$ for $M=E$, so that, as in the statement of the lemma, $k_F = (1,0)$ and $k_E=(1,1)$ with respect to the hyperbolic bases as claimed.
For this we use the intersection form $\lambda_M^{\Z^-}$ on $M$ with $\Z^-$ coefficients and its relationship with the intersection forms with $\Z[\Z/2]$ and with $\Z/2$ coefficients.
Let $\varepsilon^- \colon \Z[\Z/2] \to \Z^-$ be the twisted augmentation and let $\red_2 \colon \Z^- \to \Z/2$ be reduction modulo 2. Let \[\varepsilon^-_* \colon H_2(M;\Z[\Z/2]) \to H_2(M;\Z^-) \text{ and } (\red_2)_* \colon H_2(M;\Z^-) \to H_2(M;\Z/2)\] be the induced maps on homology.
By naturality of the cap product and the Kronecker pairing we have \[\varepsilon^-\lambda_M(x,y)=\lambda_M^{\Z^-}(\varepsilon^-_*(x),\varepsilon^-_*(y))\] for all $x,y \in H_2(M;\Z[\Z/2])$ and  \[\red_2\lambda_M^{\Z^-}(u,v)=\lambda_M^{\Z/2}\big((\red_2)_*(u),(\red_2)_*(v)\big)\] for all $u,v \in H_2(M;\Z^-)$.

Using the chain complex \eqref{eq:d-infty-cc} we compute that $H_2(M;\Z^-) \cong \Z^- \oplus \Z^-$, with basis elements $\{R,S\}$ from \cref{rem:E-F-homology-nonlemma} again corresponding to the two $2$-handles. Consider the composition
\begin{equation}\label{eq:disp-1}
H_2(M;\Z^-) \xrightarrow{(\red_2)_*} H_2(M;\Z/2) \xrightarrow{w_2} \Z/2.
\end{equation}
For $M=F$ this is the zero map since $w_2$ is trivial, while for $M= E$ this is given by $(1,0)$, by \cref{lemma:E-F-spin-or-not}.
We can also consider the composition
\begin{equation}\label{eq:disp-2}
H_2(M;\Z^-) \xrightarrow{x \mapsto\lambda^{\Z^-}_M(x,x)} \Z^- \xrightarrow{\red_2} \Z/2.
\end{equation}
The two maps $H_2(M;\Z^-) \to \Z/2$ in \eqref{eq:disp-1} and \eqref{eq:disp-2} coincide, because \[\red_2\lambda^{\Z^-}_M(x,x) = \lambda_M^{\Z/2}((\red_2)_* x,(\red_2)_*x) = w_2((\red_2)_* x)\]
by the Wu formula. So $\lambda^{\Z^-}_M((1,0),(1,0)) = m$, for some $m$ that is even for  $M=F$ and odd for $M=E$.

The element $(1,0) \in \Z^- \oplus \Z^- \cong \pi_2(M)\cong H_2(M;\Z[\Z/2])$ is represented  by \[(1-T,0) \in \Z[\Z/2] \oplus \Z[\Z/2] \cong C_2(M;\Z[\Z/2]);\] this is the element whose square with respect to $\lambda_M$ gives the entry $*$ in \eqref{eqn:int-form-E-F} we seek to understand.
It maps to $(2,0)$ in $\Z^- \oplus \Z^- \cong H_2(M;\Z^-)$.
Therefore, with $\varepsilon^- \colon \Z[\Z/2] \to \Z^-$ the twisted augmentation, we have
\[4m = \lambda^{\Z^-}_M((2,0),(2,0)) = \varepsilon^-(\lambda_M((1,0),(1,0))) = \varepsilon^-(n(1-T)) = 2n.\]
So $2m=n$. For $M=F$ we have that $m=2\ell$ for some $\ell$, thus $n=4\ell$ and hence $n\equiv 0 \mod{4}$. For~$M=E$ we have $m=2\ell+1$, thus $n = 4\ell+2$ and hence $n \equiv 2 \mod{4}$.

In fact we also verified geometrically using the Kirby diagram that for $M=F$ we have $n=-4$ and for $M=E$ we have $n = -2$, but we prefer to give the algebraic argument here as it is easier for the reader to verify.

This completes the computation that $k_F = (1,0)$ and $k_E=(1,1)$ with respect to the hyperbolic bases of $\pi_2$.
\end{proof}

\begin{lemma}\label{lemma:E-F}
 There is a diffeomorphism  $E \# \CP^2 \cong F \# \CP^2$.
\end{lemma}

\begin{proof}
We give a Kirby calculus argument in \cref{fig:kirby-calculus-E-F}.
\end{proof}

\begin{figure}[htb]
     \centering
     \begin{subfigure}[b]{0.3\textwidth}
        \centering
        \begin{tikzpicture}
        \node[anchor=south west,inner sep=0] at (0,0){\includegraphics[height=3cm]{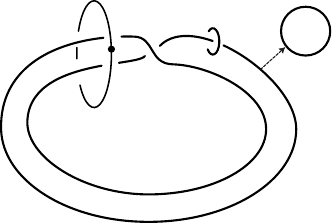}};
	\node at (4.15,1.3) {$0$};
        \node at (2.9,2.8) {$0$};
        \node at (3.75,2.9) {$1$};
	\end{tikzpicture}
         \caption{}
     \end{subfigure}
     \hfill
     \begin{subfigure}[b]{0.3\textwidth}
         \centering
        \begin{tikzpicture}
        \node[anchor=south west,inner sep=0] at (0,0){\includegraphics[height=3cm]{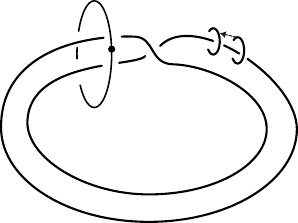}};
	\node at (4.15,1.3) {$1$};
        \node at (2.9,2.8) {$0$};
        \node at (3.25,2.7) {$1$};
	\end{tikzpicture}
         \caption{}
     \end{subfigure}
     \hfill
     \begin{subfigure}[b]{0.3\textwidth}
         \centering
        \begin{tikzpicture}
        \node[anchor=south west,inner sep=0] at (0,0){\includegraphics[height=3cm]{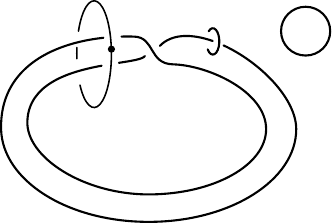}};
	\node at (4.15,1.3) {$1$};
        \node at (2.9,2.8) {$0$};
        \node at (3.75,2.9) {$1$};
	\end{tikzpicture}
         \caption{}
     \end{subfigure}
        \caption{Starting with $F\#\CP^2$ in (a), perform the handle slide indicated by a dashed arrow to produce (b). After a second handle slide the result is $E\#\CP^2$ in~(c).}
        \label{fig:kirby-calculus-E-F}
\end{figure}
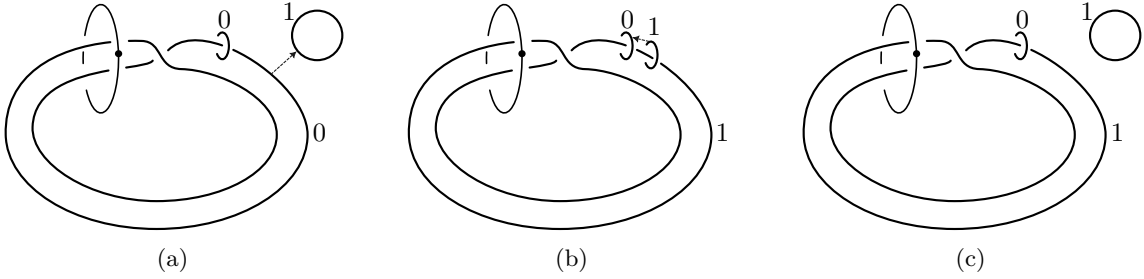

Having gathered all this information about $E$ and $F$ we can now begin considering the $4$-manifolds with infinite dihedral fundamental group arising as their connected sum.

\begin{lemma}\label{lemma:k-invariants-E-conn-sum-F-etc}
    Let $\pi:= D_\infty$. The equivariant intersection form of each of the manifolds $F\# F$, $E\# F$, $F\# E$, and $E\#E$ is isomorphic to the hyperbolic form $H(I\pi)$.
    Their $k$-invariants lie in $H^3(\pi;I\pi\oplus I\pi^\dagger)\cong (\Z/2)^2\oplus(\Z/2)^2$.
    In some hyperbolic bases, their $k$-invariants are $((1,0),(1,0))$, $((1,1),(1,0))$, $((1,0),(1,1))$,  and $((1,1),(1,1))$, respectively.
\end{lemma}

\begin{proof}
Let $M \in \{F\# F, E\# F, F\# E,E\#E\}$.   Let $G_i \cong \Z/2$ denote the two factors of $\pi$, that is $\pi:=D_\infty = G_1 * G_2$.
Using \cref{lem:adding-I,lem:ind-star-dagger-isom} we compute that
        \begin{align*}
            H(I\pi)&\cong H(\Ind_{\Z/2}^\pi \Z^- \oplus\Ind_{\Z/2}^\pi \Z^-)\\
            &\cong H(\Ind_{\Z/2}^\pi \Z^-)\oplus H(\Ind_{\Z/2}^\pi \Z^-)\\
            &\cong \Ind_{\Z/2}^\pi H(\Z^-)\oplus \Ind_{\Z/2}^\pi H(\Z^-).
        \end{align*}
Hence by \cref{lemma:E-F-k-invariants}, the equivariant intersection form of $M$
is isomorphic to the hyperbolic form $H(I\pi)$.  In particular
\[\pi_2(M) \cong I\pi\oplus I\pi^\dagger \cong \Ind_{G_1}^\pi(\Z^-\oplus(\Z^-)^*) \oplus \Ind_{G_2}^\pi(\Z^-\oplus(\Z^-)^*) \cong \Ind_{G_1}^\pi(\Z^-\oplus \Z^-) \oplus \Ind_{G_2}^\pi(\Z^-\oplus \Z^-).\]
Therefore, the $k$-invariant lies in
 \[   H^3(\pi;\pi_2(M)) \cong H^3(\pi; \Ind_{G_1}^\pi(\Z^-\oplus\Z^-)) \oplus H^3(\pi; \Ind_{G_2}^\pi(\Z^-\oplus\Z^-)).
   \]
  By Shapiro's lemma (see e.g.~\cite{brown:cohomology-group}*{Proposition~III.6.2}),
    \[H^3(\pi;\Ind_{G_i}^\pi(\Z^-\oplus\Z^-))\cong H^3(\Z/2;\Z^-\oplus \Z^-) \cong H^3(\Z/2;\Z^-)\oplus H^3(\Z/2;\Z^-)\cong (\Z/2)^2,\]
    for each $i$.
    We only give the computation of the $k$-invariant for $M=E\# F$; the other three cases are similar. The inclusion of the 3-skeleton $E^{(3)}$ into $M$ induces the inclusion $G_1\to \pi$ on fundamental groups. Hence the $k$-invariant of $M$ maps to the image of the $k$-invariant of $E^{(3)}$ in $H^3(G_1;\Res_{G_1}^\pi \pi_2(M) )$. Note that $\Res_{G_i}^\pi \Ind_{G_i}^\pi \Z^- \cong \Z^-\oplus F_i$, where $F_i$ is a free $\Z[G_i]$-module, for each $i$, while $\Res_{G_j}^\pi \Ind_{G_i}^\pi\Z^-$ is a free $\Z[G_j]$-module $F'_j$ for $i\neq j$ by \cite{brown:cohomology-group}*{Proposition~III.5.6}. Using that $H^3(G_1;F_i)=H^3(G_1;F_j')=0$ for all $i,j$, we deduce that
    \begin{align*}
        H^3(G_1;\Res_{G_1}^\pi \pi_2(M)) & \cong H^3\big(G_1;\Res_{G_1}^\pi\big(\Ind_{G_1}^\pi(\Z^-\oplus\Z^-) \oplus \Ind_{G_2}^\pi(\Z^-\oplus\Z^-)\big)\big) \\
        &\cong  H^3(G_1;\Z^-\oplus\Z^-)\\
        &\cong (\Z/2)^2.
        \end{align*}
    The map $\Z^-\oplus \Z^-\cong \pi_2(E^{(3)})\to \Res_{G_1}^\pi \pi_2(M) \cong \Z^-\oplus F_1\oplus \Z^-\oplus F_1\oplus F'_1$ is given by the inclusion of the two $\Z^-$-summands.
    Identifying $\pi_2(E^{(3)}) = \pi_2(E)$, the $k$-invariant of $E^{(3)}$ equals the $k$-invariant of $E$.
    Hence the image of the $k$-invariant of $E^{(3)}$ in $H^3(G_1;\Res_{G_1}^\pi \pi_2(M) )$ is $(1,1)$ by \cref{lemma:E-F-k-invariants}. Similarly, the image of the $k$-invariant of $F^{(3)}$ in $H^3(G_2;\Res_{G_2}^\pi \pi_2(M) )$ is $(1,0)$, again using \cref{lemma:E-F-k-invariants}. Hence the $k$-invariant of $E\#F$ is $((1,1),(1,0))$ as claimed.
\end{proof}

The next proposition proves \cref{prop:D-infty-k-invariant-free-statement}\,\ref{item:prop:dinfty-k-inv-i}.

\begin{proposition}
    Let $M$ and $M'$ be 
    oriented $4$-manifolds with fundamental group $\pi:=D_\infty$ and equivariant intersection forms both isometric to $H(I\pi)\oplus\lambda$, where $\lambda$ is a nonsingular Hermitian form on a stably free $\Z\pi$-module.
    Then $M\#\CP^2$ and $M'\#\CP^2$ have isomorphic quadratic $2$-types.
\end{proposition}

\begin{proof}
	It suffices to show that there is an isometry from $\lambda_M\oplus\langle 1\rangle$ to $\lambda_{M'}\oplus\langle 1\rangle$ sending the $k$-invariant of $M\#\CP^2$ to that of $M'\#\CP^2$.  By modified surgery theory \cite{kreck} (see also \cite{KPT}*{Theorem~1.2~and~Remark~1.3}) and since $H_4(\pi;\Z)=0$, there exist $p,p',q,q'\in\N$ such that we have a homeomorphism \[M\#\ks(M)(\star\CP^2)\#p\CP^2\#q\ol{\CP}^2 \cong F\#F\#p'\CP^2\#q'\ol{\CP}^2.\]
    Since the $k$-invariant of $F\# F$ is nontrivial by \cref{lemma:k-invariants-E-conn-sum-F-etc}, so is the $k$-invariant of $M$.

    Indeed we will need slightly more than just nontriviality.
    Let $G_1$ and $G_2$ both denote $\Z/2$, so that $\pi=G_1* G_2$.
    As in the proof of \cref{lemma:k-invariants-E-conn-sum-F-etc}, the $k$-invariant of $M$ lies in
    \begin{equation}\label{eq:k-invariant-home}
    H^3(\pi;\pi_2(M))\cong H^3(\pi; \Ind_{G_1}^\pi(\Z^-\oplus\Z^-)) \oplus H^3(\pi; \Ind_{G_2}^\pi(\Z^-\oplus\Z^-)).
    \end{equation}
    We showed above that $M$ is $\CP^2$-stably homeomorphic (possibly with an extra $\star \CP^2$ connected-summand) to $F\# F$. The $k$-invariant of $F\# F$ is nontrivial in both summands of \eqref{eq:k-invariant-home} by \cref{lemma:k-invariants-E-conn-sum-F-etc}.
    Therefore, not only is the $k$-invariant of $M$ nontrivial, it restricts nontrivially to each direct summand in \eqref{eq:k-invariant-home}. By the same argument, this is also true for the $k$-invariant of~$M'$.

    Let $\alpha_1$ and $\alpha_2$ denote arbitrary elements of $H^3(\pi;I\pi\oplus I\pi^\dagger)$ which are nontrivial in each summand under the isomorphism \eqref{eq:k-invariant-home}.
    Since the equivariant intersection form of $M$ is isometric to $H(I\pi)\oplus\lambda$ and the $\alpha_i$ are arbitrary, showing that there is an isometry of $H(I\pi)\oplus\langle 1\rangle$ that maps $\alpha_1\in H^3(\pi;I\pi\oplus I\pi^\dagger\oplus\Z\pi)\cong H^3(\pi;I\pi\oplus I\pi^\dagger)$ to $\alpha_2$ will complete the proof.

    Up to the isometry of $H(\Z^-)$ that interchanges the two summands of $\Z^-\oplus(\Z^-)^*\cong \Z^-\oplus\Z^-$, for both $i=1,2$, each $\alpha_i\in H^3(\pi;I\pi\oplus I\pi^\dagger)$ is the $k$-invariant of one of $F\#F$, $E\#F$, $F\#E$ or $E\#E$ by \cref{lemma:k-invariants-E-conn-sum-F-etc}.
    Since
    \[F\#F\#\CP^2\cong F\#E\# \CP^2\cong E\#F\#\CP^2\cong E\# E\#\CP^2\]
    by \cref{lemma:E-F}, there exist enough isometries of $H(I\pi)\oplus\langle 1\rangle$.
\end{proof}

Finally, the next proposition proves \cref{prop:D-infty-k-invariant-free-statement}\,\ref{item:prop:dinfty-k-inv-ii}.

\begin{proposition}
	Let $M$ and $M'$ be oriented, almost spin $4$-manifolds with fundamental group $\pi:=D_\infty$, and the same $w_2^\pi\in H^2(\pi;\Z/2)\cong (\Z/2)^2$. Assume that the intersection forms of $M$ and $M'$ are both isometric to $H(I\pi)\oplus\lambda$,
    where $\lambda$ is a nonsingular Hermitian form on a stably free $\Z\pi$-module.
    Then $M$ and $M'$ are homotopy equivalent.
\end{proposition}
	\begin{proof}
		By \cref{thm:main-homotopy}, it suffices to show that there is an isometry $\pi_2(M)\to \pi_2(M')$ that sends $k_M$ to $k_{M'}$. By assumption there is an isometry $\varphi\colon \pi_2(M)\to \pi_2(M')$, but we have no control over the behaviour of the $k$-invariants.
		
		By the stable classification (\cref{khldflkhfglufdalkgualkguvf,lem:stable-classes}), there are simply connected, spin $4$-manifolds $L$ and $K$, as well as $N_i\in\{E,F\}$, for $i=1,2$, such that
        \[
        \begin{tikzcd}
        M\#L\ar[r,"f","\simeq"'] &N_1\#N_2\#K &M'\#L\ar[l,"\simeq","{f'}"'].
        \end{tikzcd}
        \]
        Then
        \[\Psi:=f'_*\circ(\varphi \oplus \Id_{\Z\pi \otimes \pi_2(L)})\circ f_*^{-1}\colon \pi_2(N_1\#N_2\#K)\to \pi_2(N_1\#N_2\#K)\]
        is an isometry.
        Note that $\Psi$ preserves the $k$-invariant of $N_1\#N_2\#K$ if and only if
         $\varphi \oplus \Id_{\Z\pi \otimes \pi_2(L)}$ sends $k_{M \# L}$ to $k_{M' \# L}$, which in turn holds if and only if $\varphi$ sends  $k_M$ to $k_{M'}$.

		We carry out the argument in the case that $N_1\#N_2=E\#F$. The other cases are similar.
       As before  let $G_i \cong \Z/2$ denote the two factors of $\pi$, that is $\pi:=D_\infty = G_1 * G_2$.
        By \cref{lemma:k-invariants-E-conn-sum-F-etc}, since $K$ is simply connected, the $k$-invariant of $E\#F\#K$ is
		\begin{align*}
        ((1,1),(1,0)) \in H^3(\pi;\pi_2(E\#F\#K))&\cong H^3(\pi;\Ind_{G_1}^\pi(\Z^-\oplus\Z^-)\oplus \Ind_{G_2}^\pi(\Z^-\oplus\Z^-))\\
        &\cong (\Z/2)^2\oplus(\Z/2)^2.
        \end{align*}
        As in the proof of \cref{lemma:k-invariants-E-conn-sum-F-etc}, the inclusions $G_i\to \pi$ map the $k$-invariant to
		\[(1,1)\in(\Z/2)^2\cong H^3(G_1;\Z^-\oplus\Z^-)\cong H^3(G_1;\Res^\pi_{G_1}\pi_2(E\#F\#K)),\]
		and
		\[(1,0)\in(\Z/2)^2\cong H^3(G_2;\Z^-\oplus\Z^-)\cong H^3(G_2;\Res^\pi_{G_2}\pi_2(E\#F\#K)).\]
        The inclusion $j_i\colon G_i\to \pi$ induces the following commutative square.
        \[\begin{tikzcd}[column sep = large]
            H^3(\pi;\pi_2(E\# F\# K))\ar[r,"\Psi","\cong"']\ar[d,"j_i^*"] & H^3(\pi;\pi_2(E\# F\# K)) \ar[d,"j_i^*"] \\
            H^3(G_i;\Res_{G_i}^\pi\pi_2(E\# F\# K))\ar[r,"\Res_{G_i}^\pi(\Psi)","\cong"'] & H^3(G_i;\Res_{G_i}^\pi \pi_2(E\# F\# K))
        \end{tikzcd}\]
        Since $j_i^*(k_{E\#F\#K})$ is nontrivial, so is \[(x_i,y_i):=j_i^*\Psi(k_{E\#F\#K})\in (\Z/2)^2\cong H^3(G_i;\Res^\pi_{G_1}\pi_2(E\#F\#K))\] for $i=1,2$.

		As in the proof of \cref{lemma:k-invariants-E-conn-sum-F-etc}, using \cref{lem:adding-I,lem:ind-star-dagger-isom} we compute that
        \[H(I\pi)\cong 
        \Ind_{\Z/2}^\pi H(\Z^-)\oplus \Ind_{\Z/2}^\pi H(\Z^-),\]
        We can therefore change $\varphi$ (and hence $\Psi$) by the isometry that permutes the two copies of $\Z^-$ in $H(\Z^-)$, if necessary, to arrange that $(x_i,y_i)$ is the $k$-invariant of a $4$-manifold $N_i' \in \{E,F\}$ for $i=1,2$.
        Thus we can construct an isometry \[\Psi' \colon \pi_2(E\#F\# K)\xrightarrow{\Psi}\pi_2(E\#F\# K)\cong \pi_2(N_1'\# N_2'\# K)\] that preserves the $k$-invariant, using that the intersection forms of $E$ and $F$ are isometric.
        Hence by \cref{thm:main-homotopy}, the isometry $\Psi'$ is induced by a homotopy equivalence $E\# F \# K \to N_1' \# N_2' \# K$ over $\pi$.
        Since a homotopy equivalence over $\pi$ preserves the second Stiefel--Whitney class $w_2^\pi$, it follows that $N_1'=E$ and $N_2'=F$. We deduce that $\Psi$ must have already preserved the $k$-invariant of $k_{E\# F\# K}$. From the construction of $\Psi$ it follows that the isometry $\varphi$ maps the $k$-invariant of $M$ to that of $M'$. Again applying \cref{thm:main-homotopy}, $\varphi$ is realised by a homotopy equivalence $M\xrightarrow{\simeq}M'$, as desired.
        \end{proof}

\section{Additional information on classifications for torsion-free \texorpdfstring{$3$-manifold}{3-manifold} groups}\label{section:additional-information}

In this section we prove results primarily about $4$-manifolds with torsion-free $3$-manifold fundamental group. First we consider some aspects of the realisation question, which asks: for a fixed fundamental group, which values of the quadratic $2$-type are realised as the invariants of some $4$-manifold?  We focus on characterising the modules that arise as $\pi_2(M)$ and the sesquilinear forms that are realisable as equivariant intersection forms.

After that we investigate what can be said about the $s$-cobordism and homeomorphism classifications for $4$-manifolds in a fixed homotopy type. In particular, we  analyse the surgery exact sequence and give upper bounds on the number of $s$-cobordism and homeomorphism classes.

\subsection{Realisation of \texorpdfstring{$\pi_2(M)$}{the second homotopy group}}\label{section:realisation-pi-2}

We  say that two $\Z\pi$-modules $L$ and $L'$ are \emph{strictly stably
isomorphic} if $L\oplus \Z\pi^r \cong L'\oplus\Z\pi^r$ for some $r\in \N$.

\begin{theorem}
Let $M$ be a closed $4$-manifold such that $\pi:=\pi_1(M) \cong G_1*\ldots* G_m*F(r)$, where $G_i$ is a  $PD_3$-group for each $i$,  $F(r)$ is free of rank $r$, and $m, r \geq 0$. Let $\beta_1:=\beta_1(\pi;\bbF_2)$.
Then
\begin{equation}\label{eq:euler-char}
\chi(M)=2+\dim_{\bbF_2}(\bbF_2\otimes_{\Z\pi}\pi_2(M))-\beta_1-m-r.
\end{equation}
Let $w\colon \pi\to C_2$ be the orientation character of $M$. Let $v'\colon \pi\to C_2$ be trivial on $F(r)$ and the orientation character of a $PD_3$-complex with fundamental group $G_j$ for each $j$. Then $\pi_2(M)$ is strictly stably isomorphic to $\Z\pi^s
\oplus I\pi^v$, where $s = \chi(M) + m + r - 2$ and $v=wv'\colon \pi\to C_2$.
\end{theorem}

\begin{proof}
	Let $c\colon M\to B\pi$ be a classifying map. Consider the Leray--Serre spectral sequence for $\wt M\to M\to B\pi$ with $E^2$-page $H_{p}(\pi;H_q(\wt M;\bbF_2))$ converging to $H_{p+q}(M;\bbF_2)$. This gives rise to an exact sequence
\begin{equation}
	\label{eq:F2-sequence}
	H_3(M;\bbF_2)\xrightarrow{c_*} H_3(\pi;\bbF_2)\to \bbF_2\otimes_{\Z\pi}\pi_2(M)\to H_2(M;\bbF_2)\xrightarrow{c_*} H_2(\pi;\bbF_2)\to 0.
\end{equation}
By naturality of the cap product, the composition
\begin{equation}\label{eq:composition}
H^1(\pi;\bbF_2)\xrightarrow{c^*} H^1(M;\bbF_2)\xrightarrow{-\cap [M]} H_3(M;\bbF_2)\xrightarrow{c_*}H_3(\pi;\bbF_2)
\end{equation}
agrees with the map $-\cap c_*[M]$. Since $H_4(\pi;\Z)=0$, in particular $c_*[M]=0$ and thus the above composition is trivial. The first two maps of the composition \eqref{eq:composition} are isomorphisms and hence $H_3(M;\bbF_2)\xrightarrow{c_*}H_3(\pi;\bbF_2)$ is trivial.
It now follows from \eqref{eq:F2-sequence}, that \[\beta_2(M;\bbF_2)=\beta_2+\dim_{\bbF_2}(\bbF_2\otimes_{\Z\pi}\pi_2(M))-\beta_3,\]
where $\beta_i:=\beta_i(\pi;\bbF_2)$.
By Poincar\'e duality, $\beta_3(M;\bbF_2)=\beta_1(M;\bbF_2)=\beta_1$. Hence we have
\[\chi(M)=2+\beta_2+\dim_{\bbF_2}(\bbF_2\otimes_{\Z\pi}\pi_2(M))-\beta_3-2\beta_1.\]
Now we investigate the quantity $\beta_2 - \beta_3 -\beta_1$.
If $G$ is a $PD_3$-group, then
\[\beta_2(G;\bbF_2)-\beta_3(G;\bbF_2)-\beta_1(G;\bbF_2)=\beta_1(G;\bbF_2)-1-\beta_1(G;\bbF_2)=-1\]
by Poincar\'e duality. Furthermore,
\[\beta_2(\Z;\bbF_2)-\beta_3(\Z;\bbF_2)-\beta_1(\Z;\bbF_2)=0-0-1=-1.\]
Since the Betti numbers of a free product are the sums of the Betti numbers of the factors, and we have $m$ factors that are $PD_3$-groups and $r$ factors isomorphic to  $\Z$, this implies that $\beta_2-\beta_3-\beta_1=-m-r$.
Hence
\[\chi(M)=2+\beta_2+\dim_{\bbF_2}(\bbF_2\otimes_{\Z\pi}\pi_2(M))-\beta_3-2\beta_1=2+\dim_{\bbF_2}(\bbF_2\otimes_{\Z\pi}\pi_2(M))-\beta_1-m-r\]
as claimed.

Moreover, there exist $\ell,\ell'$ such that $\pi_2(M\#\ell(S^2\times S^2))\cong \pi_2(M)\oplus \Z\pi^{2\ell}$ is isomorphic to $\Z\pi^{\ell'}\oplus I\pi^v$ by \cref{cor:pi2-decomp}. Then $\pi_2(M)$ is strictly stably isomorphic to $\Z\pi^{\ell'-2\ell}\oplus I\pi$, and we want to show that $\ell'-2\ell=s$, where $s=\chi(M) + m + r - 2$ as in the statement of the proposition.
We have
\[\dim_{\bbF_2}(\bbF_2 \otimes_{\Z\pi}\pi_2(M \# \ell(S^2 \times S^2))) = \ell' + \dim_{\bbF_2}(\bbF_2 \otimes_{\Z\pi} I\pi^v) = \ell' + \beta_1.\]
 Here we used that $\dim_{\bbF_2}(\bbF_2 \otimes_{\Z\pi} I\pi^v)  = \beta_1$. To see this, consider the short exact sequence
 $0 \to I\pi^v \to \Z\pi \to \Z^v \to 0$. Tensoring with $\bbF_2 \otimes_{\Z\pi} -$ yields the long exact sequence
 \[\Tor_1^{\Z\pi}(\bbF_2 ,\Z\pi) \to \Tor_1^{\Z\pi}(\bbF_2,\Z^v) \to \bbF_2 \otimes_{\Z\pi} I\pi^v \to \bbF_2 \otimes_{\Z\pi} \Z\pi \to \bbF_2 \otimes_{\Z\pi} \Z^v \to 0.\]
 Since $\Tor_1^{\Z\pi}(\bbF_2 ,\Z\pi) =0$, $\Tor_1^{\Z\pi}(\bbF_2,\Z^v) \cong H_1(\pi;\bbF_2^v) \cong H_1(\pi;\bbF_2)$
 and $\bbF_2 \otimes_{\Z\pi} \Z\pi = \bbF_2 = \bbF_2 \otimes_{\Z\pi} \Z^v$, so we obtain the exact sequence $0 \to H_1(\pi;\bbF_2) \to \bbF_2 \otimes_{\Z\pi} I\pi^v \to \bbF_2 \to \bbF_2 \to 0$, from which the desired equality $\dim_{\bbF_2}(\bbF_2 \otimes_{\Z\pi} I\pi^v)  = \beta_1$ follows.

Now by \eqref{eq:euler-char} we therefore have
\[\chi(M)+2\ell=\chi(M\#\ell(S^2\times S^2))= 2 + (\ell' + \beta_1) - \beta_1 -m-r = 2+\ell'-m-r.\]
Hence $\ell'-2\ell=\chi(M) + m + r - 2=s$, as asserted.
\end{proof}

\subsection{Realisation of the equivariant intersection forms}\label{section:realisation-int-forms}

Let $N$ be a $3$-manifold with fundamental group $\pi:=\pi_1(N)$. Let $N_*:=\overline{N\setminus D^3}$. Then \[M_0:=\partial(N_*\times D^2)\cong (N_*\times S^1)\cup_{S^2\times S^1}(S^2\times D^2)\] is a $4$-manifold with fundamental group $\pi$. Similarly, we can form the Gluck twist $M_1:=(N_*\times S^1)\cup_{\tau}(S^2\times D^2)$ of $M_0$. The manifolds $M_0$ and $M_1$ are called the \emph{spin} and \emph{twist spin} of $N$ respectively.

Plotnick showed that $\pi_2(M_\varsigma)\cong I\pi\oplus \pi_2(N_*)$ for $\varsigma=0,1$ but that the intersection forms $\lambda_0$ and $\lambda_1$ of $M_0$ and $M_1$ are in general not isomorphic \cite{Plotnick}*{Propositions~2.1 and~2.2}. If $N$ is aspherical, then $\pi_2(N)\cong \Z\pi$ and $I\pi^\dagger\cong \Z\pi^\dagger$ \cite{KLPT}*{Lemma~7.5}.
If $N$ is aspherical and orientable, the intersection form on $M_\varsigma$ is given by $\left(\begin{smallmatrix}
	\varsigma&1\\1&0
\end{smallmatrix}\right)$. In particular, $\lambda_0(\alpha,\beta)=0$ and $\lambda_1(\alpha,\beta)=\alpha\overline{\beta}$ for $\alpha,\beta\in I\pi$ by \cite{Plotnick}*{Section~3}; see also \cite{KLPT}*{Section~7.2}.

For $\pi$ the fundamental group of a closed, orientable, aspherical $3$-manifold, \cite{KLPT}*{Theorem~9.6} gives a complete list of stable isomorphism classes of intersection forms of closed, oriented $4$-manifolds with fundamental group $\pi$.  Each of the forms on their list is of the form $\lambda\oplus \lambda_\varsigma$ for $\varsigma=0$ or $1$. The next theorem shows that every form $\lambda\oplus \lambda_\varsigma$ is moreover realised as the intersection form of some $PD_4$-complex. If the group $\pi$ is solvable, then it is in particular good, in which case we can also realise the form by a topological $4$-manifold.

\begin{theorem}
	Let $N$ be a $3$-manifold, let $\pi:=\pi_1(N)$ and let $w_N \colon \pi \to C_2$ be the orientation character of~$N$. Suppose that there are no elements of order two in $\ker(w_N)$.  Let $\lambda$ be a nonsingular Hermitian form on a stably free $\Z\pi$-module and let $\varsigma=0,1$. Then there is a $PD_4$-complex $Z$ with fundamental group $\pi$,  orientation character $w_Z=w_N$,  and equivariant intersection form $\lambda\oplus\lambda_\varsigma$. If $N$ is orientable and $\pi$ is torsion-free and solvable, then $Z$ is homotopy equivalent to a topological manifold.
\end{theorem}

\begin{proof}
	By \cite{hillman-matrix}*{Theorem~10}, there is a $PD_4$-complex with fundamental group $\pi$ and a $2$-connected degree one map $f\colon Z\to M_\varsigma$ such that the equivariant intersection form on the surgery kernel is isomorphic to $\lambda$. (The hypothesis on 2-torsion is needed here.) Hence it follows from~\cite{AGS-book}*{Proposition~10.21} that $Z$ has equivariant intersection form $\lambda\oplus\lambda_\varsigma$.

    Now assume that $N$ is orientable and $\pi$ is torsion-free and solvable, then $\pi:=\pi_1(N)$ is either a  $PD_3$-group or $\Z$.
    It follows that $\pi$ is cohomologically 3-dimensional.

	If $N$ is orientable, then $M_\varsigma$ and $Z$ are orientable and thus $Z$ has a $\Top$ reduction \cites{Hambleton-red,Land-red}. Hence there is a 2-connected degree one normal map $g\colon M\to Z$. By adding copies of $E_8$ to $M$ we can assume that the signatures of $M$ and $Z$ agree.
	In the composition $\mathbb{L}\langle1\rangle_4(B\pi)\to \mathbb{L}_4(B\pi) \to  L_4(\Z\pi)$, the first map is an isomorphism since $\pi$ is 3-dimensional, and the second map is an isomorphism since $\pi$ satisfies the Farrell--Jones conjecture~\cite{Wegner}.
    It follows from a straightforward consideration of naturality in the Atiyah--Hirzebruch spectral sequence that the map $\cN(M)\cong \mathbb{L}\langle 1\rangle_4(M)\to \mathbb{L}\langle1\rangle_4(B\pi)$ is surjective. Hence the surgery obstruction map $\cN(M)\xrightarrow{\sigma} L_4(\Z\pi)$ is surjective and there exists a 2-connected degree one normal map $g'\colon M'\to M$ with surgery obstruction $\sigma(g' \colon M' \to M) = -\sigma(g \colon M \to Z)$.

   We claim that $g\circ g'\colon M'\to Z$ is a degree one normal map, as we show next. As a composition of degree one maps, $g\circ g'$ is degree one, so we only need to show the compatibility of normal structures. Since the difference of the signatures of $M$ and $Z$ is zero, and $\sigma(g')=-\sigma(g)$, the difference of the signatures of $M'$ and $M$ is also zero. By the Hirzebruch signature theorem and since $g'$ is degree one, we have that $(g')^*p_1(\nu M) = p_1 (\nu M')$. Since the Stiefel--Whitney classes depend only on the underlying spherical fibrations, $(g')^*w_2(\nu M) = w_2 (\nu M')$. Dold--Whitney~\cite{Dold-Whitney}*{Theorems~1 and~2} proved that (stable) orientable bundles over an orientable $4$-manifold (or more generally any $PD_4$-complex $X$ with no $2$-torsion in $H^4(X;\Z)$) are determined by $w_2$ and $p_1$.
    It follows that the map $g'$ pulls back the normal bundle of $M$ to the normal bundle of $M'$.  Hence  $g\circ g'\colon M'\to Z$ is a degree one normal map as claimed.

    The homology of the domain of a $2$-connected degree one map splits as the direct sum of the surgery kernel and the homology of the codomain. It follows that the surgery kernel of $g\circ g'$ is the direct sum of the surgery kernels of $g$ and $g'$, the surgery obstruction of $g\circ g'$ is trivial. Since solvable groups are good \cites{Freedman-Teichner:1995-1,Krushkal-Quinn:2000-1} (see also~\cite{DET-book-goodgroups}*{Example~19.6}), by surgery~\citelist{\cite{DET-book-surgery-chapter}\cite{FQ}*{Chapter~11}}, $g \circ g'$ is normally bordant to a homotopy equivalence, and hence there exists a manifold homotopy equivalent to $Z$ as claimed.
\end{proof}

\subsection{The \texorpdfstring{$s$-cobordism}{s-cobordism} and homeomorphism  classifications}\label{sec:homeo-s-cob-classification}

For $4$-manifolds with torsion-free $3$-manifold groups, we do not have a complete homeomorphism classification.  However by applying surgery we can obtain some information, which in the case that the group is also solvable comes very close.  We begin by considering arbitrary torsion-free $3$-manifold groups. These are in general not known to be good, so we can only obtain conclusions on the set of $4$-manifolds up to $s$-cobordism.

\begin{theorem}\label{thm:s-cob-classn-3-mfld-group}
Let $M$ be a closed $4$-manifold whose fundamental group $\pi$ is a torsion-free $3$-manifold group. Let $\beta_3:=\beta_3(\pi;\Z/2)$ and let $\beta_1:=\beta_1(\pi;\Z/2)$.
\begin{enumerate}
  \item\label{s-cob-bounds-1} There are at most $2^{\beta_3}$ topological $s$-cobordism classes of $4$-manifolds homotopy equivalent to $M$ and with the same Kirby--Siebenmann invariant.
  \item\label{s-cob-bounds-2} If $M$ is smooth and orientable, then there are at most $2^{\beta_3 + \beta_1}$ smooth $s$-cobordism classes of smooth $4$-manifolds homotopy equivalent to $M$.
  \item\label{s-cob-bounds-3} If $M$ is smooth and nonorientable, then there are at most $2^{\beta_3 + \beta_1+1}$ smooth $s$-cobordism classes of smooth $4$-manifolds homotopy equivalent to $M$.
\end{enumerate}
\end{theorem}

Note that if $\pi$ is torsion-free and $\pi = \pi_1(Y)$, for $Y$ a closed $3$-manifold, then $\beta_3$ equals the number of irreducible factors in a prime decomposition of $Y$.

\begin{proof}
	We will only sketch the proof since the strategy is well-known. We refer to \cite{Kasprowski-Land} for a more detailed treatment and further citations. We give the argument for the topological case and provide the necessary modifications for the smooth case at the end.
	
Every $3$-manifold group satisfies the Farrell--Jones conjecture by \cite{BFL-3d}*{Corollary~1.3}.
By \cite{Kasprowski-Land}*{Lemma~2.3}, every torsion-free $3$-manifold group $\pi$ thus satisfies the following properties for every homomorphism $w \colon \pi \to C_2$.
	\begin{enumerate}
		\item The Whitehead group $Wh(\pi)$ vanishes,
		\item the assembly map $\mathbb{L}\langle 1\rangle_4^w(B\pi)\to L_4(\Z\pi,w)$ is injective, and
		\item the assembly map $\mathbb{L}\langle 1\rangle_5^w(B\pi)\to L_5(\Z\pi,w)$  is surjective.
	\end{enumerate}
In particular by the first item every homotopy equivalence is simple and every $h$-cobordism is an $s$-cobordism.
	
Fix $w$ to be the orientation character of $M$.  Let $f\colon N\to M$ be a homotopy equivalence.
Let $\eta(f) \in \cN(M)$ denote its normal invariant.  Recall that
\[\cN(M) \cong [M,\G/\Top] \cong [\Sigma^{\infty} M_+,\mathbb{L}\langle1\rangle] \cong \mathbb{L}\langle1\rangle_4^w(M).\]
 Injectivity of $\mathbb{L}\langle 1\rangle_4^w(B\pi)\to L_4(\Z\pi,w)$ implies that the normal invariant $\eta(f)$ of $f$ is contained in the kernel of $\cN(M)\cong \mathbb{L}\langle 1\rangle_4^w(M)\to \mathbb{L}\langle 1\rangle_4^w(B\pi)$. Using the twisted Atiyah--Hirzebruch spectral sequence, the kernel of $\mathbb{L}\langle 1\rangle_4^w(M)\to \mathbb{L}\langle 1\rangle_4^w(B\pi)$ is isomorphic to the kernel of $H_2(M;\Z/2)\to H_2(\pi;\Z/2)$. Hence, with respect to this identification, $\eta(f)$ can be represented by an immersed $2$-sphere $\alpha$ in $M$, i.e.\ corresponds to an element in the image of $\pi_2(M) \to H_2(M;\Z/2)$. If $w_2(M)$ vanishes on $\alpha$, there exists a homotopy equivalence $f'\colon N\to M$ with trivial normal invariant using Novikov pinching; see e.g.~\cite{Kasprowski-Land}*{Lemma~3.3} for more details.
	
	We will now show that if $w_2(M)$ is nontrivial on $\alpha$, then $N$ and $M$ have different Kirby--Siebenmann invariant, cf.~\cite{Kasprowski-Land}*{Proof of Lemma~3.4}. There is a homotopy equivalence $g\colon \star\CP^2\to \CP^2$ with normal invariant represented by $\CP^1$. The homotopy equivalence $f\#g\colon N\#\star\CP^2\to M\#\CP^2$ has normal invariant represented by $\alpha+\CP^1$. On this 2-sphere $w_2$ now vanishes. Hence $N\#\star\CP^2$ and $M\#\CP^2$ are normally cobordant over $M\#\CP^2$, again using Novikov pinching, as above. In particular, they have the same Kirby--Siebenmann invariant. It follows that the Kirby--Siebenmann invariants of $N$ and $M$ are different, as claimed.
	
	So far we have shown that every manifold $N$ that is homotopy equivalent to $M$ and has the same Kirby--Siebenmann invariant is normally cobordant to $M$. Let $W$ be a choice of normal cobordism. It has a surgery obstruction $\sigma(W)\in L_5(\Z\pi,w)$. If $\sigma(W)=0$, then $N$ and $M$ are $s$-cobordant.
	
We claim that the number of $s$-cobordism classes is bounded by the order of the cokernel of the map $\mathbb{L}\langle1\rangle_5^w(M)\to L_5(\Z\pi,w)$.
To see this we need to show that if two normal cobordisms $W_1$ and $W_2$ over $M$, from $M$ to $N_1$ and $N_2$ respectively, are such that  $\sigma(W_1) = \sigma(W_2) \in \coker (\mathbb{L}\langle1\rangle_5^w(M)\to L_5(\Z\pi,w))$, then $N_1$ and $N_2$ are $s$-cobordant.
By the first paragraph of the proof of \cite{Kasprowski-Land}*{Theorem~3.1}, we can modify $W_1$ to a normal cobordism $W'_1$, also from $M$ to $N_1$, with surgery obstruction $\sigma(W'_1)=\sigma(W_2)$.
Stack $W_1'$ and $W_2$, gluing them along $M$, to obtain a normal cobordism $\ol{W} := -W_1' \cup_M W_2$ between $N_1$ and $N_2$ with surgery obstruction $\sigma(\ol{W}) = -\sigma(W_1') + \sigma(W_2) =0$. Hence we can surger $\ol{W}$ to an $s$-cobordism between $N_1$ and $N_2$, completing the proof of the claim.

Since the assembly map $\mathbb{L}\langle 1\rangle_5^w(B\pi)\to L_5(\Z\pi,w)$  is surjective, the order of the cokernel of $\mathbb{L}\langle1\rangle_5^w(M)\to L_5(\Z\pi,w)$ is the same as that of the cokernel $\mathbb{L}\langle1\rangle_5^w(M)\to \mathbb{L}\langle 1\rangle_5^w(B\pi)$.
	Considering the twisted Atiyah-–Hirzebruch spectral sequences for $M$ and $B\pi$, as in the proof of \cite{Kasprowski-Land}*{Theorem~3.1}, the cokernel of $\mathbb{L}\langle1\rangle_5^w(M)\to \mathbb{L}\langle 1\rangle_5^w(B\pi)$ is isomorphic to the cokernel of $H_3(M;\Z/2)\to H_3(\pi;\Z/2)$. In particular it has order at most $2^{\beta_3}$. This completes the proof of~\eqref{s-cob-bounds-1}.

For \eqref{s-cob-bounds-2}, the smooth orientable case, let $f \colon N \to M$ be a homotopy equivalence between smooth $4$-manifolds with fundamental group $\pi$.  The forgetful map $\cN^{\mathrm{Diff}}(M) \to \cN^{\Top}(M)$ is described by Kirby--Taylor in \cite{Kirby-Taylor}*{Lemma~7}.  In the orientable case the forgetful map is injective.  We showed above that $N$ and $M$ are topologically normally bordant over $M$. Hence by injectivity $N$ and $M$ are smoothly normally bordant over $M$.  The cokernel of $\cN^{\mathrm{Diff}}(M \times I,\partial) \to \cN^{\Top}(M \times I,\partial)$ can be identified with the cokernel of $[(M \times I,\partial),(\G/ O,*)] \to [(M \times I,\partial),(\G /\Top,*)]$, which is a subgroup of \[[(M\times I,\partial),(B(\Top/O),*)] = H^4(M \times I,\partial;\Z/2) \cong H_1(M;\Z/2) \cong H_1(\pi;\Z/2),\] and therefore has order at most $2^{\beta_1}$. Combining this with the argument above, the  cokernel of $\cN^{\mathrm{Diff}}(M \times I,\partial) \to L_5(\Z\pi,w)$ has order at most $2^{\beta_1 + \beta_3}$, and hence there are at most this many smooth $s$-cobordism classes of manifolds homotopy equivalent to $M$, proving \eqref{s-cob-bounds-2}.

For \eqref{s-cob-bounds-3}, the nonorientable case the fibre of $\Id_M$ under  $\cN^{\mathrm{Diff}}(M) \to \cN^{\Top}(M)$ has cardinality two, again by Kirby--Taylor \cite{Kirby-Taylor}*{Lemma~7}. In Kirby and Taylor's lemma there are two cases, depending on whether the Wu class $v_2(TM)$ vanishes or not, but in both cases the fibre of $\Id_M$ has cardinality two. Hence $N$ lies in one of two smooth normal bordism classes, each of which has at most $2^{\beta_1 + \beta_3}$ smooth $s$-cobordism classes, by the previous argument for \eqref{s-cob-bounds-2}, which already used either $w$-twisted or $\Z/2$-coefficients, and hence applies in the nonorientable case.
\end{proof}

\begin{corollary}
Let $M$ be a closed $4$-manifold whose fundamental group $\pi$ is a torsion-free, solvable $3$-manifold group.
There are at most two homeomorphism classes of closed $4$-manifolds homotopy equivalent to~$M$ and with the same Kirby--Siebenmann invariant.
\end{corollary}

\begin{proof}
Such groups do not decompose nontrivially as free products, so  $H_3(\pi;\Z/2) \cong \Z/2$ and hence $\beta_3=1$.
Since $\pi$ is solvable it is a good group, so the topological $s$-cobordism theorem implies that $s$-cobordant manifolds are homeomorphic. Therefore \cref{thm:s-cob-classn-3-mfld-group}\,\eqref{s-cob-bounds-1} implies that within a fixed homotopy class, and fixing the Kirby--Siebenmann invariant, we have at most $2^{\beta_3} = 2$ homeomorphism classes.
\end{proof}

\def\MR#1{}
\bibliographystyle{alpha}
\bibliography{bib.bib}

\begin{bibdiv}
\begin{biblist}

\bib{bass-finitistic}{article}{
      author={Bass, Hyman},
       title={Finitistic dimension and a homological generalization of
  semi-primary rings},
        date={1960},
        ISSN={0002-9947,1088-6850},
     journal={Trans. Amer. Math. Soc.},
      volume={95},
       pages={466\ndash 488},
         url={https://doi.org/10.2307/1993568},
      review={\MR{157984}},
}

\bib{baues}{book}{
      author={Baues, Hans-Joachim},
       title={Homotopy type and homology},
      series={Oxford Mathematical Monographs},
   publisher={The Clarendon Press, Oxford University Press, New York},
        date={1996},
        ISBN={0-19-851482-4},
        note={Oxford Science Publications},
      review={\MR{1404516}},
}

\bib{Baues-Bleile}{article}{
      author={Baues, Hans~Joachim},
      author={Bleile, Beatrice},
       title={Poincar{\'e} duality complexes in dimension four},
        date={2008},
        ISSN={1472-2747},
     journal={Algebr. Geom. Topol.},
      volume={8},
      number={4},
       pages={2355\ndash 2389},
}

\bib{BDK07}{article}{
      author={Brookman, Jeremy},
      author={Davis, James~F.},
      author={Khan, Qayum},
       title={Manifolds homotopy equivalent to {$P^n\#P^n$}},
        date={2007},
        ISSN={0025-5831},
     journal={Math. Ann.},
      volume={338},
      number={4},
       pages={947\ndash 962},
         url={https://doi.org/10.1007/s00208-007-0099-x},
      review={\MR{2317756}},
}

\bib{bieri-eckmann}{article}{
      author={Bieri, Robert},
      author={Eckmann, Beno},
       title={Groups with homological duality generalizing {P}oincar\'e{}
  duality},
        date={1973},
        ISSN={0020-9910,1432-1297},
     journal={Invent. Math.},
      volume={20},
       pages={103\ndash 124},
         url={https://doi.org/10.1007/BF01404060},
      review={\MR{340449}},
}

\bib{BFL-3d}{article}{
      author={Bartels, Arthur~C.},
      author={Farrell, F.~T.},
      author={L{\"u}ck, Wolfgang},
       title={The {Farrell}-{Jones} conjecture for cocompact lattices in
  virtually connected {Lie} groups},
        date={2014},
        ISSN={0894-0347},
     journal={J. Am. Math. Soc.},
      volume={27},
      number={2},
       pages={339\ndash 388},
}

\bib{brunner-ratcliffe}{article}{
      author={Brunner, Andrew~M.},
      author={Ratcliffe, John~G.},
       title={Finite 2-complexes with infinitely-generated groups of
  self-homotopy- equivalences},
        date={1982},
        ISSN={0002-9939},
     journal={Proc. Am. Math. Soc.},
      volume={86},
       pages={525\ndash 530},
}

\bib{brown:cohomology-group}{book}{
      author={Brown, Kenneth~S.},
       title={Cohomology of groups},
      series={Graduate Texts in Mathematics},
   publisher={Springer-Verlag, New York},
        date={1994},
      volume={87},
        ISBN={0-387-90688-6},
        note={Corrected reprint of the 1982 original},
      review={\MR{1324339}},
}

\bib{cappell-bams}{article}{
      author={Cappell, Sylvain~E.},
       title={Unitary nilpotent groups and {Hermitian} {K}-theory. {I}},
        date={1974},
        ISSN={0002-9904},
     journal={Bull. Am. Math. Soc.},
      volume={80},
       pages={1117\ndash 1122},
}

\bib{davis-connolly}{article}{
      author={Connolly, Frank},
      author={Davis, James~F.},
       title={The surgery obstruction groups of the infinite dihedral group},
        date={2004},
        ISSN={1465-3060},
     journal={Geom. Topol.},
      volume={8},
       pages={1043\ndash 1078},
}

\bib{connolly-kozniewski}{article}{
      author={Connolly, Frank},
      author={Ko{\'z}niewski, Tadeusz},
       title={Nil groups in {{\(K\)}}-theory and surgery theory},
        date={1995},
        ISSN={0933-7741},
     journal={Forum Math.},
      volume={7},
      number={1},
       pages={45\ndash 76},
}

\bib{connolly-ranicki}{article}{
      author={Connolly, Frank},
      author={Ranicki, Andrew},
       title={On the calculation of {UNil}},
        date={2005},
        ISSN={0001-8708},
     journal={Adv. Math.},
      volume={195},
      number={1},
       pages={205\ndash 258},
}

\bib{davis}{article}{
      author={Davis, Michael~W.},
       title={The cohomology of a {C}oxeter group with group ring
  coefficients},
        date={1998},
        ISSN={0012-7094,1547-7398},
     journal={Duke Math. J.},
      volume={91},
      number={2},
       pages={297\ndash 314},
         url={https://doi.org/10.1215/S0012-7094-98-09113-X},
      review={\MR{1600586}},
}

\bib{Dold-Whitney}{article}{
      author={Dold, A.},
      author={Whitney, H.},
       title={Classification of oriented sphere bundles over a {$4$}-complex},
        date={1959},
     journal={Ann. of Math. (2)},
      volume={69},
       pages={667\ndash 677},
}

\bib{EM49}{article}{
      author={Eilenberg, Samuel},
      author={MacLane, Saunders},
       title={Homology of spaces with operators. {II}},
    language={English},
        date={1949},
        ISSN={0002-9947},
     journal={Trans. Am. Math. Soc.},
      volume={65},
       pages={49\ndash 99},
}

\bib{guide}{book}{
      author={Friedl, Stefan},
      author={Nagel, Matthias},
      author={Orson, Patrick},
      author={Powell, Mark},
       title={The foundations of four-manifold theory in the topological
  category,},
    language={English},
      series={NYJM Monogr.},
   publisher={Albany, NY: State University of New York, University at Albany},
        date={2025},
      volume={6},
}

\bib{FQ}{book}{
      author={Freedman, Michael},
      author={Quinn, Frank},
       title={Topology of $4$-manifolds},
      series={Princeton Mathematical Series},
   publisher={Princeton University Press},
        date={1990},
      volume={39},
}

\bib{F}{article}{
      author={Freedman, Michael},
       title={The topology of four-dimensional manifolds},
        date={1982},
        ISSN={0022-040X},
     journal={J. Differential Geom.},
      volume={17},
      number={3},
       pages={357\ndash 453},
      review={\MR{679066}},
}

\bib{Freedman-Teichner:1995-1}{article}{
      author={Freedman, Michael},
      author={Teichner, Peter},
       title={{$4$}-manifold topology. \textup{{I}}. {S}ubexponential groups},
        date={1995},
        ISSN={0020-9910},
     journal={Invent. Math.},
      volume={122},
      number={3},
       pages={509\ndash 529},
      review={\MR{MR1359602 (96k:57015)}},
}

\bib{Geoghegan}{book}{
      author={Geoghegan, Ross},
       title={Topological methods in group theory},
      series={Graduate Texts in Mathematics},
   publisher={Springer, New York},
        date={2008},
      volume={243},
        ISBN={978-0-387-74611-1},
         url={https://doi.org/10.1007/978-0-387-74614-2},
      review={\MR{2365352}},
}

\bib{GompfStip}{book}{
      author={Gompf, Robert},
      author={Stipsicz, Andras},
       title={{$4$}-manifolds and {K}irby calculus},
      series={Graduate Studies in Mathematics},
   publisher={American Mathematical Society, Providence, RI},
        date={1999},
      volume={20},
        ISBN={0-8218-0994-6},
      review={\MR{1707327}},
}

\bib{Ha09}{inproceedings}{
      author={Hambleton, Ian},
       title={Intersection forms, fundamental groups and 4-manifolds},
        date={2009},
   booktitle={Proceedings of {G}\"{o}kova {G}eometry-{T}opology {C}onference
  2008},
   publisher={G\"{o}kova Geometry/Topology Conference (GGT), G\"{o}kova},
       pages={137\ndash 150},
}

\bib{Hambleton-red}{article}{
      author={Hambleton, Ian},
       title={Orientable {$4$}-dimensional {P}oincar\'e{} complexes have
  reducible {S}pivak fibrations},
        date={2019},
        ISSN={0002-9939,1088-6826},
     journal={Proc. Amer. Math. Soc.},
      volume={147},
      number={7},
       pages={3177\ndash 3179},
         url={https://doi.org/10.1090/proc/14465},
      review={\MR{3973916}},
}

\bib{hambleton-stabilityrange}{article}{
      author={Hambleton, Ian},
       title={A stability range for topological 4-manifolds},
        date={2023},
        ISSN={0002-9947,1088-6850},
     journal={Trans. Amer. Math. Soc.},
      volume={376},
      number={12},
       pages={8769\ndash 8793},
         url={https://doi.org/10.1090/tran/9020},
      review={\MR{4669310}},
}

\bib{Hambleton-Hildum}{article}{
      author={Hambleton, Ian},
      author={Hildum, Alyson},
       title={Topological 4-manifolds with right-angled {A}rtin fundamental
  groups},
        date={2019},
     journal={J. Topol. Anal.},
      volume={11},
      number={4},
       pages={777\ndash 821},
}

\bib{Hambleton-Hillman}{article}{
      author={Hambleton, Ian},
      author={Hillman, Jonathan~A.},
       title={Quotients of {$S^2\times S^2$}},
        date={2023},
     journal={J. Lond. Math. Soc. (2)},
      volume={108},
      number={4},
       pages={1393\ndash 1416},
}

\bib{FMGK}{book}{
      author={Hillman, Jonathan~A.},
       title={Four-manifolds, geometries and knots},
      series={Geometry \& Topology Monographs},
   publisher={Geometry \& Topology Publications, Coventry},
        date={2002},
      volume={5},
      review={\MR{1943724}},
}

\bib{hillman}{article}{
      author={Hillman, Jonathan~A.},
       title={{{\(PD_{4}\)}}-complexes with strongly minimal models},
        date={2006},
        ISSN={0166-8641},
     journal={Topology Appl.},
      volume={153},
      number={14},
       pages={2413\ndash 2424},
}

\bib{hillman-matrix}{incollection}{
      author={Hillman, Jonathan~A.},
       title={{$PD_4$}-complexes and 2-dimensional duality groups},
        date={2021},
   booktitle={2019--20 {MATRIX} {Annals}},
      series={MATRIX Book Ser.},
      volume={4},
   publisher={Springer, Cham},
       pages={57\ndash 109},
         url={https://doi.org/10.1007/978-3-030-62497-2_3},
      review={\MR{4294762}},
}

\bib{hillman-COAT}{article}{
      author={{Hillman}, Jonathan~A.},
       title={{Homotopy types of 4-manifolds with 3-manifold fundamental
  groups}},
        date={2023-07},
     journal={arXiv e-prints},
       pages={arXiv:2307.15292},
      eprint={2307.15292},
}

\bib{hillman-stable-connsum}{article}{
      author={Hillman, Jonathan~A.},
       title={Free products and {$4$}-dimensional connected sums},
        date={1995},
        ISSN={0024-6093,1469-2120},
     journal={Bull. London Math. Soc.},
      volume={27},
      number={4},
       pages={387\ndash 391},
         url={https://doi.org/10.1112/blms/27.4.387},
      review={\MR{1335291}},
}

\bib{Hambleton-Kreck:1988-1}{article}{
      author={Hambleton, Ian},
      author={Kreck, Matthias},
       title={On the classification of topological {$4$}-manifolds with finite
  fundamental group},
        date={1988},
        ISSN={0025-5831},
     journal={Math. Ann.},
      volume={280},
      number={1},
       pages={85\ndash 104},
      review={\MR{928299}},
}

\bib{Hambleton-Kreck-93}{article}{
      author={Hambleton, Ian},
      author={Kreck, Matthias},
       title={Cancellation, elliptic surfaces and the topology of certain
  four-manifolds},
        date={1993},
        ISSN={0075-4102},
     journal={J. Reine Angew. Math.},
      volume={444},
       pages={79\ndash 100},
      review={\MR{1241794}},
}

\bib{HKT09}{article}{
      author={Hambleton, Ian},
      author={Kreck, Matthias},
      author={Teichner, Peter},
       title={Topological 4-manifolds with geometrically two-dimensional
  fundamental groups},
        date={2009},
        ISSN={1793-5253},
     journal={J. Topol. Anal.},
      volume={1},
      number={2},
       pages={123\ndash 151},
}

\bib{HKT94}{article}{
      author={Hambleton, Ian},
      author={Kreck, Matthias},
      author={Teichner, Peter},
       title={Nonorientable {$4$}-manifolds with fundamental group of order
  {$2$}},
        date={1994},
        ISSN={0002-9947},
     journal={Trans. Amer. Math. Soc.},
      volume={344},
      number={2},
       pages={649\ndash 665},
         url={https://doi.org/10.2307/2154500},
      review={\MR{1234481}},
}

\bib{jahren-kwasik}{article}{
      author={Jahren, Bj{\o}rn},
      author={Kwasik, S{\l}awomir},
       title={Manifolds homotopy equivalent to
  {$\mathbb{RP}^4\#\mathbb{RP}^4$}},
        date={2006},
        ISSN={0305-0041},
     journal={Math. Proc. Cambridge Philos. Soc.},
      volume={140},
      number={2},
       pages={245\ndash 252},
         url={https://doi.org/10.1017/S0305004105008893},
      review={\MR{2212277}},
}

\bib{kim-kojima-raymond}{article}{
      author={Kim, Myung~Ho},
      author={Kojima, Sadayoshi},
      author={Raymond, Frank},
       title={Homotopy invariants of nonorientable {$4$}-manifolds},
        date={1992},
        ISSN={0002-9947,1088-6850},
     journal={Trans. Amer. Math. Soc.},
      volume={333},
      number={1},
       pages={71\ndash 81},
         url={https://doi.org/10.2307/2154099},
      review={\MR{1028758}},
}

\bib{Kasprowski-Land}{article}{
      author={Kasprowski, Daniel},
      author={Land, Markus},
       title={Topological 4-manifolds with 4-dimensional fundamental group},
        date={2022},
        ISSN={0017-0895},
     journal={Glasg. Math. J.},
      volume={64},
      number={2},
       pages={454\ndash 461},
}

\bib{KLPT}{article}{
      author={Kasprowski, Daniel},
      author={Land, Markus},
      author={Powell, Mark},
      author={Teichner, Peter},
       title={Stable classification of 4-manifolds with 3-manifold fundamental
  groups},
        date={2017},
        ISSN={1753-8416},
     journal={J. Topol.},
      volume={10},
      number={3},
       pages={827\ndash 881},
         url={https://doi.org/10.1112/topo.12025},
      review={\MR{3797598}},
}

\bib{KLT}{article}{
      author={Kreck, Matthias},
      author={L{\"u}ck, Wolfgang},
      author={Teichner, Peter},
       title={Counterexamples to the {Kneser} conjecture in dimension four},
        date={1995},
        ISSN={0010-2571},
     journal={Comment. Math. Helv.},
      volume={70},
      number={3},
       pages={423\ndash 433},
}

\bib{KLT-stable-connsum}{incollection}{
      author={Kreck, Matthias},
      author={L\"{u}ck, Wolfgang},
      author={Teichner, Peter},
       title={Stable prime decompositions of four-manifolds},
        date={1995},
   booktitle={Prospects in topology ({P}rinceton, {NJ}, 1994)},
      series={Ann. of Math. Stud.},
      volume={138},
   publisher={Princeton Univ. Press, Princeton, NJ},
       pages={251\ndash 269},
      review={\MR{1368662}},
}

\bib{KNR}{article}{
      author={Kasprowski, Daniel},
      author={Nicholson, John},
      author={Ruppik, Benjamin},
       title={Homotopy classification of 4-manifolds whose fundamental group is
  dihedral},
        date={2022},
        ISSN={1472-2747},
     journal={Algebr. Geom. Topol.},
      volume={22},
      number={6},
       pages={2915\ndash 2949},
}

\bib{KNV}{misc}{
      author={Kasprowski, Daniel},
      author={Nicholson, John},
      author={Veselá, Simona},
       title={Stable equivalence relations on 4-manifolds},
        date={2024},
        note={Preprint: available at ar{X}iv:2405.06637},
}

\bib{DET-book-goodgroups}{incollection}{
      author={Kim, Min~Hoon},
      author={Orson, Patrick},
      author={Park, JungHwan},
      author={Ray, Arunima},
       title={Good groups},
        date={2021},
   booktitle={The disc embedding theorem},
   publisher={Oxford University Press},
}

\bib{table}{article}{
      author={Kasprowski, Daniel},
      author={Powell, Mark},
      author={Ray, Arunima},
       title={Counterexamples in 4-manifold topology},
        date={2022},
        ISSN={2308-2151},
     journal={EMS Surv. Math. Sci.},
      volume={9},
      number={1},
       pages={193\ndash 249},
}

\bib{KPR}{article}{
      author={Kasprowski, Daniel},
      author={Powell, Mark},
      author={Ruppik, Benjamin},
       title={Homotopy classification of 4-manifolds with finite abelian
  2-generator fundamental groups},
    language={English},
        date={2024},
        ISSN={0305-0041},
     journal={Math. Proc. Camb. Philos. Soc.},
      volume={177},
      number={2},
       pages={263\ndash 283},
}

\bib{KPT-long}{misc}{
      author={Kasprowski, Daniel},
      author={Powell, Mark},
      author={Teichner, Peter},
       title={Algebraic criteria for stable diffeomorphism of spin
  4-manifolds},
        date={2021},
        note={Preprint: available at ar{X}iv:2006.06127, to appear in Memoirs
  of the AMS},
}

\bib{KPT}{article}{
      author={Kasprowski, Daniel},
      author={Powell, Mark},
      author={Teichner, Peter},
       title={Four-manifolds up to connected sum with complex projective
  planes},
        date={2022},
        ISSN={0002-9327},
     journal={Am. J. Math.},
      volume={144},
      number={1},
       pages={75\ndash 118},
}

\bib{Krushkal-Quinn:2000-1}{article}{
      author={Krushkal, Vyacheslav~S.},
      author={Quinn, Frank},
       title={Subexponential groups in 4-manifold topology},
        date={2000},
        ISSN={1465-3060},
     journal={Geom. Topol.},
      volume={4},
       pages={407\ndash 430},
      review={\MR{MR1796498 (2001i:57031)}},
}

\bib{kreck}{article}{
      author={Kreck, Matthias},
       title={Surgery and duality},
        date={1999},
        ISSN={0003-486X},
     journal={Ann. Math. (2)},
      volume={149},
      number={3},
       pages={707\ndash 754},
         url={www.math.princeton.edu/~annals/issues/1999/149_3.html},
}

\bib{Kirby-Taylor}{incollection}{
      author={Kirby, Robion~C.},
      author={Taylor, Laurence~R.},
       title={A survey of 4-manifolds through the eyes of surgery},
        date={2001},
   booktitle={Surveys on surgery theory, {V}ol. 2},
      series={Ann. of Math. Stud.},
      volume={149},
   publisher={Princeton Univ. Press, Princeton, NJ},
       pages={387\ndash 421},
      review={\MR{1818779}},
}

\bib{KT}{article}{
      author={Kasprowski, Daniel},
      author={Teichner, Peter},
       title={{{\( \mathbb{CP}^2\)}}-stable classification of 4-manifolds with
  finite fundamental group},
        date={2021},
        ISSN={1945-5844},
     journal={Pac. J. Math.},
      volume={310},
      number={2},
       pages={355\ndash 373},
}

\bib{lam}{book}{
      author={Lam, Tsit~Yuen},
       title={Serre's problem on projective modules},
      series={Springer Monographs in Mathematics},
   publisher={Springer-Verlag, Berlin},
        date={2006},
        ISBN={978-3-540-23317-6; 3-540-23317-2},
         url={https://doi.org/10.1007/978-3-540-34575-6},
      review={\MR{2235330}},
}

\bib{Land-red}{article}{
      author={Land, Markus},
       title={Reducibility of low-dimensional {P}oincar\'e{} duality spaces},
        date={2022},
        ISSN={1867-5778,1867-5786},
     journal={M\"unster J. Math.},
      volume={15},
      number={1},
       pages={47\ndash 81},
      review={\MR{4476490}},
}

\bib{levine:modules}{article}{
      author={Levine, Jerome},
       title={Knot modules. {I}},
        date={1977},
        ISSN={0002-9947,1088-6850},
     journal={Trans. Amer. Math. Soc.},
      volume={229},
       pages={1\ndash 50},
         url={https://doi.org/10.2307/1998498},
      review={\MR{461518}},
}

\bib{Milnor-simply-connected-4-manifolds}{incollection}{
      author={Milnor, John},
       title={On simply connected {$4$}-manifolds},
        date={1958},
   booktitle={Symposium internacional de topolog\'\i a algebraica},
   publisher={Universidad Nacional Aut\'onoma de M\'exico and UNESCO, Mexico
  City},
       pages={122\ndash 128},
}

\bib{Nicholson}{article}{
      author={Nicholson, John},
       title={Projective modules and the homotopy classification of {{\(({G},
  n)\)}}-complexes},
    language={English},
        date={2024},
        ISSN={1472-2747},
     journal={Algebr. Geom. Topol.},
      volume={24},
      number={4},
       pages={2245\ndash 2284},
}

\bib{DET-book-surgery-chapter}{incollection}{
      author={Orson, Patrick},
      author={Powell, Mark},
      author={Ray, Arunima},
       title={Surgery theory and the classification of closed, simply connected
  4-manifolds},
    language={English},
        date={2021},
   booktitle={The disc embedding theorem. with an afterword by michael h.
  freedman},
   publisher={Oxford: Oxford University Press},
       pages={331\ndash 351},
}

\bib{Plotnick}{article}{
      author={Plotnick, Steven~P.},
       title={Equivariant intersection forms, knots in {$S^4$}, and rotations
  in {$2$}-spheres},
        date={1986},
        ISSN={0002-9947,1088-6850},
     journal={Trans. Amer. Math. Soc.},
      volume={296},
      number={2},
       pages={543\ndash 575},
         url={https://doi.org/10.2307/2000379},
      review={\MR{846597}},
}

\bib{Quinn-annulus}{article}{
      author={Quinn, Frank},
       title={Ends of maps. {III}. {D}imensions {$4$}\ and {$5$}},
        date={1982},
        ISSN={0022-040X},
     journal={J. Differential Geom.},
      volume={17},
      number={3},
       pages={503\ndash 521},
      review={\MR{679069}},
}

\bib{AGS-book}{book}{
      author={Ranicki, Andrew},
       title={Algebraic and geometric surgery},
    language={English},
      series={Oxford Math. Monogr.},
   publisher={Oxford: Oxford University Press},
        date={2002},
        ISBN={0-19-850924-3},
}

\bib{Stallings-Whitehead-additive}{article}{
      author={Stallings, John},
       title={Whitehead torsion of free products},
        date={1965},
     journal={Ann. of Math. (2)},
      volume={82},
       pages={354\ndash 363},
}

\bib{Stong}{article}{
      author={Stong, Richard},
       title={Existence of {$\pi_1$}-negligible embeddings in {$4$}-manifolds.
  {A} correction to {T}heorem 10.5 of {F}reedmann and {Q}uinn},
        date={1994},
        ISSN={0002-9939},
     journal={Proc. Amer. Math. Soc.},
      volume={120},
      number={4},
       pages={1309\ndash 1314},
      review={\MR{1215031}},
}

\bib{Stong-conn-sum}{article}{
      author={Stong, Richard},
       title={Uniqueness of connected sum decompositions in dimension $4$},
        date={1994},
     journal={Topology and its applications},
      volume={56},
       pages={277\ndash 291},
}

\bib{swan-laurent-pols}{article}{
      author={Swan, Richard~G.},
       title={Projective modules over {L}aurent polynomial rings},
        date={1978},
        ISSN={0002-9947,1088-6850},
     journal={Trans. Amer. Math. Soc.},
      volume={237},
       pages={111\ndash 120},
         url={https://doi.org/10.2307/1997613},
      review={\MR{469906}},
}

\bib{switzer}{book}{
      author={Switzer, Robert~M.},
       title={Algebraic topology---homotopy and homology},
      series={Classics in Mathematics},
   publisher={Springer-Verlag, Berlin},
        date={2002},
        ISBN={3-540-42750-3},
        note={Reprint of the 1975 original [Springer, New York; MR0385836 (52
  \#6695)]},
      review={\MR{1886843}},
}

\bib{teichner-phd}{book}{
      author={Teichner, Peter},
       title={Topological four-manifolds with finite fundamental group},
   publisher={Shaker Verlag},
        date={1992},
        ISBN={3-86111-182-9},
        note={PhD Thesis, University of Mainz, Germany},
}

\bib{teichner-star}{incollection}{
      author={Teichner, Peter},
       title={On the star-construction for topological 4-manifolds},
        date={1997},
   booktitle={Geometric topology. 1993 {G}eorgia international topology
  conference, {A}ugust 2--13, 1993, {A}thens, {GA}, {USA}},
   publisher={Providence, RI: American Mathematical Society; Cambridge, MA:
  International Press},
       pages={300\ndash 312},
}

\bib{Wall-surgery-book}{book}{
      author={Wall, C. Terence~C.},
       title={Surgery on compact manifolds},
     edition={Second},
      series={Mathematical Surveys and Monographs},
   publisher={American Mathematical Society, Providence, RI},
        date={1999},
      volume={69},
        ISBN={0-8218-0942-3},
        note={Edited and with a foreword by A. A. Ranicki},
      review={\MR{1687388}},
}

\bib{wang95}{article}{
      author={Wang, Zhenghan},
       title={Classification of closed nonorientable {$4$}-manifolds with
  infinite cyclic fundamental group},
        date={1995},
        ISSN={1073-2780},
     journal={Math. Res. Lett.},
      volume={2},
      number={3},
       pages={339\ndash 344},
         url={https://doi.org/10.4310/MRL.1995.v2.n3.a11},
      review={\MR{1338793}},
}

\bib{Wegner}{article}{
      author={Wegner, Christian},
       title={The {F}arrell-{J}ones conjecture for virtually solvable groups},
        date={2015},
     journal={J. Topol.},
      volume={8},
      number={4},
       pages={975\ndash 1016},
}

\bib{Whitehead-4-complexes}{article}{
      author={Whitehead, J. H.~C.},
       title={On simply connected, {$4$}-dimensional polyhedra},
        date={1949},
        ISSN={0010-2571},
     journal={Comment. Math. Helv.},
      volume={22},
       pages={48\ndash 92},
      review={\MR{0029171}},
}

\bib{whitehead-certainsequence}{article}{
      author={Whitehead, J. H.~C.},
       title={A certain exact sequence},
        date={1950},
        ISSN={0003-486X},
     journal={Ann. of Math. (2)},
      volume={52},
       pages={51\ndash 110},
         url={https://doi.org/10.2307/1969511},
      review={\MR{35997}},
}

\end{biblist}
\end{bibdiv}

\end{document}